\documentclass[12pt, a4paper, twoside, openright]{book} 

\usepackage{amsmath}
\usepackage{amssymb}
\usepackage{amsthm}
\usepackage[british]{babel}
\usepackage{booktabs}
\usepackage{cjhebrew}
\usepackage{enumitem}
\usepackage{fancyhdr}
\usepackage[inner=40mm,outer=20mm]{geometry}
\usepackage{graphicx}
\usepackage{imakeidx}
\makeindex[intoc]
\usepackage{mathtools} 
\usepackage{microtype}
\usepackage{multirow}
\usepackage{pgfplots}
\usepackage{setspace}
\usepackage{stmaryrd}
\usepackage[labelformat=simple]{subcaption}
\usepackage{tabularx}
\usepackage{tikz}
\usepackage{titlesec}
\usepackage{xifthen}
\usepackage{xparse}
\usepackage{xstring}
\usepackage[colorlinks=true,linkcolor=blue,citecolor=blue,pagebackref]{hyperref}
\usetikzlibrary{arrows}
\usetikzlibrary{decorations.pathmorphing}
\usetikzlibrary{math}
\usetikzlibrary{patterns}
\usetikzlibrary{plotmarks}

\tikzset{wavy/.style={decorate, decoration={snake}}}

\pgfplotsset{compat=newest}

\theoremstyle{plain}
\newtheorem{theorem}{Theorem}
\newtheorem{lemma}[theorem]{Lemma}
\newtheorem{corollary}[theorem]{Corollary}
\newtheorem{proposition}[theorem]{Proposition}
\newtheorem{fact}[theorem]{Fact}
\newtheorem{observation}[theorem]{Observation}

\theoremstyle{definition}

\newtheorem{conjecture}[theorem]{Conjecture}

\newtheorem{problem}[theorem]{Problem}
\newtheorem{question}[theorem]{Question}

\theoremstyle{remark}
\newtheorem{remark}[theorem]{Remark}

\newenvironment{subproof}[1][Proof]%
{\begin{proof}[#1]}{\end{proof}}

\renewcommand{\interleave}{\obslash} 

\renewcommand*{\backref}[1]{}
\renewcommand*{\backrefalt}[4]
{
    \ifcase #1 Not cited.
    \or        Cited on page~#2.
    \else      Cited on pages~#2.
    \fi
}

\newcommand{\extindex}[2][]{%
    \ifthenelse{\equal{#1}{}}%
    {\index{#2}}%
    {\index{#1!#2}\index{#2|see {#1}}}%
}
\newcommand{\extindof}[3]{\index{#1!#2}\index{#2 #3|see {#1}}}

\DeclareCaptionLabelFormat{cont}{#1~#2\alph{ContinuedFloat}}
\newcommand{\ball}[2]{#1 \balloon #2}
\newcommand{\ballgen}[3]{#2 \balloon_{#1} #3}
\newcommand{\ballij}[2]{#1 \balloon_{i,j} #2}
\newcommand{\balloon}{\circledcirc}
\newcommand{\ballwedgek}[2]{\ballwedgex{k}{#1}{#2}}
\newcommand{\ballwedgex}[3]{#2\,\triangle_{#1}\,#3}
\newcommand{\bbN}{\mathbb{N}} 
\newcommand{\bbR}{\mathbb{R}} 

\newcommand{\bphi}{{\phi^*}}
\newcommand{\btau}{\tau^*}
\newcommand{\chain}[1]{\mathcal{#1}}

\newcommand{\chainb}{\chain{B}}
\newcommand{\chainc}{\chain{C}}
\newcommand{\chaing}{\chain{G}}
\newcommand{\chainm}{\chain{M}}
\newcommand{\chainr}{\chain{R}}

\newcommand{\contrib}[2]{\mathfrak{C}_{#1,#2}}
\newcommand{\cO}{O} 

\newcommand{\cprime}{c^\prime}
\newcommand{\cS}{\mathcal{S}} 
\newcommand{\dtwo}{21}
\newcommand{\dwmtitleline}[2][12]{{\Huge \textbf{\MakeUppercase{#2}}}\\[#1pt]}
\newcommand{\e}{\mathrm{e}}
\newcommand{\E}{\mathrm{E}} 
\newcommand{\emptyperm}{\epsilon}
\newcommand{\ex}[1]{\overline{#1}}
\newcommand{\familyil}[1]{\mathcal{F}_{\interleave}(#1)}
\newcommand{\familysum}[1]{\mathcal{F}_{\oplus}(#1)}

\newcommand{\ildtwokb}{\left(\ildtwok\right)}
\newcommand{\ildtwok}{\nils{k}{\dtwo}}
\newcommand{\ildtwokpb}{\left(\ildtwokp\right)}
\newcommand{\ildtwokp}{\nils{{k+1}}{\dtwo}}
\newcommand{\ilone}{\interleave 1}
\newcommand{\inflateall}[2]{#1[#2]}
\newcommand{\inflatesome}[3]{#1_{#2}[#3]}
\newcommand{\minusone}{\ominus 1}
\newcommand{\mobfn}[2]{\mu[#1,#2]}
\newcommand{\mob}{M\"{o}bius }
\newcommand{\mobmaxn}{\mobmaxx{n}}
\newcommand{\mobmaxx}[1]{\absmax(#1)}
\newcommand{\mobp}[1]{\mu[#1]}
\newcommand{\mobxfn}[3]{\mu_#1[#2,#3]}
\newcommand{\NE}{\mathrm{NE}} 
\newcommand{\nils}[2]{\interleave^{#1}#2}
\newcommand{\nsums}[2]{\oplus^{#1}#2}
\newcommand{\oneil}{1 \interleave}
\newcommand{\oneminus}{1 \ominus} 
\newcommand{\oneplus}{1 \oplus} 
\newcommand{\oneplusone}{1 \oplus 1}
\newcommand{\order}[1]{\ensuremath{\left\lvert#1\right\rvert}}

\newcommand{\plusone}{\oplus 1}
\newcommand{\redset}{\redx{\pi}}
\newcommand{\redx}[1]{\mathsf{R_{#1}}}
\newcommand{\sE}{\mathcal{E}} 
\newcommand{\shapes}{\mathcal{S}}
\newcommand{\simpcomp}{\Delta}
\newcommand{\skewinterleave}{\oslash}
\newcommand{\sNEle}{\sNE_\lambda(\mathrm{even},\pi)}
\newcommand{\sNElo}{\sNE_\lambda(\mathrm{odd},\pi)}
\newcommand{\sNEl}{\sNE_\lambda(*,\pi)}
\newcommand{\sNE}{\mathcal{NE}} 

\newcommand{\sumrab}{\left(\sumra\right)}
\newcommand{\sumra}{\nsums{r}{\alpha}}
\newcommand{\weightgen}[3]{W(#1,#2,#3)}
\newcommand{\weightosc}[3]{W_{io}(#1,#2,#3)}

\newcommand{\zpmfp}{39.95\%}
\newcommand{\zpmfr}{0.3995}

\DeclareMathOperator{\Av}{Av}

\DeclareMathOperator{\Dim}{Dim} 

\DeclareMathOperator{\Img}{Img} 
\DeclareMathOperator{\maxk}{MaxK}
\DeclareMathOperator{\mink}{MinK}
\DeclareMathOperator{\absmax}{AbsMax_{\mu}}
\DeclareMathOperator{\rawmink}{RawMinK}
\DeclareMathOperator{\src}{src}
\DeclareMathOperator{\truemobmax}{Max_{\mu}}
\DeclareMathOperator{\truemobmin}{Min_{\mu}}

\makeatletter
\def\mymod#1{\allowbreak\mkern10mu({\operator@font mod}\,\,#1)}
\makeatother

\makeatletter
\def\cleardoublepage{\clearpage\if@twoside \ifodd\c@page\else
    \hbox{}
    \vspace*{\fill}
    \vspace{\fill}
    \thispagestyle{empty}
    \newpage
    \if@twocolumn\hbox{}\newpage\fi\fi\fi}
\makeatother

\newcommand{\tinydot}[1]{
    \node at #1 {\tiny $\bullet$};
}
\newcommand{\evensmallerdot}[1]{
    \node at #1 {\scriptsize $\bullet$};
}
\newcommand{\smallerdot}[1]{
    \node at #1 {\footnotesize $\bullet$};
}
\newcommand{\smalldot}[1]{
    \node at #1 {\small $\bullet$};
}

\newcommand{\normaldot}[1]{
    \node at #1 {\normalsize $\bullet$};
}
\newcommand{\hiddendot}[1]{
    \node at #1 {\normalsize \textcolor{white}{$\bullet$}};
}
\newcommand{\opendot}[1]{
    \node at #1 {\normalsize \textcolor{white}{$\bullet$}};
    \node at #1 {\normalsize $\circ$};
}
\newcommand{\embeddot}[2]{
    \node at #1 {\normalsize \textcolor{#2}{$\bullet$}};
    \node at #1 {\normalsize \textcolor{#2}{$\circ$}};
}
\newcommand{\embedcir}[2]{
    \node at #1 {\normalsize \textcolor{white}{$\bullet$}};
    \node at #1 {\normalsize \textcolor{#2}{$\circ$}};
}
\newcommand{\embedast}[2]
{
    \node at #1 {\normalsize \textcolor{white}{$\bullet$}};
    \node at #1 {\normalsize \textcolor{#2}{$\ast$}};
}
\newcommand{\plotgrid}[2]{
    \foreach \i in {0,1,...,#2}{
        \draw [color=darkgray] (0.5, {\i+0.5})--({#1+0.5}, {\i+0.5});
    };
    \foreach \i in {0,1,...,#1}{
        \draw [color=darkgray] ({\i+0.5}, 0.5)--({\i+0.5}, {#2+0.5});
    };
}
\newcommand{\plotopengrid}[2]{
	\foreach \i in {0,1,...,#2}{
		\draw [color=darkgray] (0.5, {\i+0.5})--({#1+1.5}, {\i+0.5});
	};
	\foreach \i in {0,1,...,#1}{
		\draw [color=darkgray] ({\i+0.5}, 0.5)--({\i+0.5}, {#2+1.5});
	};
}
\newcommand{\plotperm}[1]{
    \foreach \j [count=\i] in {#1} {
        \normaldot{(\i,\j)}{};
    };
}
\newcommand{\darkhline}[2]{
    \draw [color=red, thick]
    (0.5,{#1+0.5}) -- ({#2+0.5},{#1+0.5});
}
\newcommand{\darkvline}[2]{
    \draw [color=red, thick]
    ({#1+0.5},0.5) -- ({#1+0.5},{#2+0.5});
}
\newcommand{\plotpermgrid}[1]{
    \foreach \i [count=\permsize] in {#1} {\global\let\n=\permsize};
    \foreach \i in {0,1,2,...,\n}{
        \draw [color=darkgray] ({\i+0.5}, 0.5)--({\i+0.5}, {\n+0.5});
        \draw [color=darkgray] (0.5, {\i+0.5})--({\n+0.5}, {\i+0.5});
    }
    \plotperm{#1};
}

\newcommand{\point}[2]{
    \filldraw (#1,#2) circle (9pt);%
}
\newcommand{\permnode}[3]{
    \StrLen{#3}[\n];
    \foreach \i in {0,...,\n}{
        \draw [color=darkgray] ({\i+0.5+#1-\n/2}, {0.5+#2})--({\i+0.5+#1-\n/2}, {\n+0.5+#2});
        \draw [color=darkgray] ({0.5+#1-\n/2}, {\i+0.5+#2})--({\n+0.5+#1-\n/2}, {\i+0.5+#2});
    };
    \foreach \i in {1,...,\n}{%
        \StrChar{#3}{\i}[\j];
        \point{\i+#1-\n/2}{\j+#2};%
    };
    \node (l#3) [align=center, below] at ({\n/2+0.5+#1-\n/2},{#2+0.3}) {\scriptsize{#3}}; %
    \node (u#3) [align=center, above] at ({\n/2+0.5+#1-\n/2},{\n+#2-0.2}) {}; %
}
\newcommand{\link}[2]{
    \draw (l#1) -- (u#2);	
}
\newcommand{\dnode}[3]{\node (p#1) at (#2,#3)  {\scriptsize{$\bullet$}}}
\newcommand{\tnode}[4]{%
    \node (p#1) at (#2,#3)%
    {#4};%
}    
\newcommand{\pnode}[3]{%
    \node (p#1) at (#2,#3)%
    {#1};%
}
\newcommand{\ddline}[2]{
    \foreach \i in {#2} {
        \draw [->] (p#1)  -- (p\i);
    };
}
\newcommand{\dline}[2]{
    \foreach \i in {#2} {
        \draw (p#1)  -- (p\i);
    };
}
\newcommand{\sqat}[3]{
    \foreach \i in {0,1,...,#1}{
        \draw [color=darkgray] ({\i+0.5+#2}, 0.2+#3)--({\i+0.5+#2}, {#1+0.8+#3});
    };
    \foreach \i in {0,1,...,#1}{
        \draw [color=darkgray] ({0.3+#2}, {\i+0.5+#3})--({#1+0.8+#2}, {\i+0.5+#3});
    };
}

\newcommand{\spoint}[2]{
    \filldraw (#1,#2) circle (3pt);
}

\newcommand{\cell}[3]{
    \node at (#1,#2) {#3};
}
\newcommand{\scell}[3]{
    \cell{#1}{#2}{\small{#3}}
}

\renewcommand{\chaptermark}[1]{\markboth{\MakeUppercase{#1}}{}}
\allowdisplaybreaks
\microtypecontext{spacing=nonfrench}

\begin{document}
    
    \frontmatter
    
        \begin{center}
        \thispagestyle{empty}   
        \hspace{0pt}
        \vfill
        \dwmtitleline{On the \mob}
        \dwmtitleline{function of}
        \dwmtitleline{permutations}
        \dwmtitleline{under}
        \dwmtitleline{the pattern}
        \dwmtitleline[36]{containment order}
        \textbf{\MakeUppercase{David William Marchant}} \\[48pt]
        \includegraphics[width=0.4\textwidth]{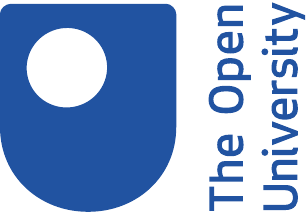}
        \\[36pt] 
        A thesis submitted to The Open University \\
        for the degree of Doctor of Philosophy \\
        in Mathematics
        \\[36pt] 
        April 2020
        \vfill
    \end{center}

    \chapter*{Abstract}
\addcontentsline{toc}{chapter}{Abstract}

We study several aspects of the \mob function,
$\mobfn{\sigma}{\pi}$, 
on the poset of permutations under the pattern containment order.

First, we consider cases where 
the lower bound of the poset is indecomposable.
We show that $\mobfn{\sigma}{\pi}$
can be computed by considering just the
indecomposable permutations contained
in the upper bound.
We apply this to the case where the
upper bound is an increasing oscillation,
and give a method for computing
the value of the \mob function 
that only involves evaluating
simple inequalities.

We then consider conditions on an interval which guarantee that
the value of the \mob function is zero.
In particular, we show that if a permutation $\pi$ contains 
two intervals of length 2, 
which are not order-isomorphic to one another,
then 
$\mobfn{1}{\pi} = 0$.
This allows us to prove that the proportion of 
permutations of length $n$ with principal \mob 
function equal to zero is asymptotically 
bounded below by $(1-1/e)^2\ge\zpmfr$. 
This is the first 
result determining the value of $\mobfn{1}{\pi}$ 
for an asymptotically positive proportion of 
permutations~$\pi$. 

Following this, we 
use ``2413-balloon'' permutations to
show that the growth of the 
principal \mob 
function on the permutation poset
is exponential.  
This improves on previous work, which 
has shown that the growth 
is at least polynomial.

We then generalise 2413-balloon permutations,
and find a recursion for the value
of the principal \mob function of these
generalisations.

Finally, we look back at the results
found, and discuss ways to relate 
the results from each chapter.
We then consider further research avenues.

    \chapter*{Dedication}
\addcontentsline{toc}{chapter}{Dedication}

This thesis is dedicated to Jo.
Five years ago she agreed that I could stop work
in order to study for a PhD --
which has been the hardest I have ever worked.
There is no way for me to adequately express my gratitude for
her support,
her tolerance for my mathematical adventures,
or the way that she makes me complete.
\begin{center}
    \cjRL{\Large 'a:niy l:dwodiy w:dwodiy liy} \\[2pt]   
    \cjRL{\normalsize +sir ha+s*iyriym w;g}
\end{center} 

    \chapter*{Acknowledgements}
\addcontentsline{toc}{chapter}{Acknowledgements}

I would like to thank my supervisor, Robert Brignall, 
for introducing me to the \mob function
on the permutation pattern poset, and for
offering me the chance to study under his supervision.
Robert has always been patient and forbearing in our
interactions, and has allowed me the freedom
to find my own mathematical ``voice''.

As a part-time student, living some distance 
from The Open University campus, my opportunities 
to meet with other PhD students have been
somewhat limited.  Where these opportunities
have arisen, I have been warmly welcomed.
I would especially like to thank 
Grahame Erskine,
Jakub Slia\v{c}an,
James Tuite,
Margaret Stanier,
Olivia Jeans,
and
Rob Lewis
for their support and discussions.

Part-time students are not entitled to any funding support
from the School of Mathematics and Statistics at The Open University.
I am therefore most grateful to the School for
providing funding for conferences and visits over the last
four years.  I am also grateful for the support received
from the National Science Foundation for 
contributions towards travel and accommodation costs 
for attending conferences, 
and for the support from Charles University in Prague
for my visit in 2019.

The Permutation Patterns community is small and vibrant.
I was made to feel welcome by everyone I encountered.
I would particularly like to thank
Einar Steingr\'imsson,
V{\'i}t Jel{\'i}nek,
and
Jan Kyn{\v c}l 
for their support and encouragement.
I would also like to thank the 
anonymous referees of the papers 
which underpin this thesis
for all their work.

    \chapter*{Declarations}
\addcontentsline{toc}{chapter}{Declarations}

Some chapters of this thesis are based on
work that has been published.
The relevant chapters are as follows:
\begin{enumerate}
    \item Chapter~\ref{chapter_incosc_paper}
    is based on published joint work with
    Robert Brignall.
    The paper~\cite{Brignall2017a}
    was published in 
    \emph{Discrete Mathematics}.
    \item Chapter~\ref{chapter_oppadj_paper}
    is based on published joint work with
    Robert Brignall,
    V{\'i}t Jel{\'i}nek
    and
    Jan Kyn{\v c}l.
    The paper~\cite{Brignall2020}
    was published in
    \emph{Mathematika}.
    I thank the London Mathematical Society
    for granting permission to 
    include edited extracts from the published
    article in this thesis.
    \item Chapter~\ref{chapter_2413_balloon_paper}
    is based on published sole work by the author.
    The paper~\cite{Marchant2020}
    was published in
    \emph{The Electronic Journal of Combinatorics}.
\end{enumerate}    
None of the results appear in any other thesis,
and all co-authors 
have agreed with the inclusion of joint work in this thesis.
Where work has been published,
the publishers have given permission for
edited extracts of the published article
to be included in this thesis.

This thesis is approximately 
49,000
words.

    \microtypesetup{protrusion=false} 
    \tableofcontents 
    \microtypesetup{protrusion=true}  
    
    \mainmatter 
    
    \chapter{Overview}
\label{chapter_overview}

This thesis is primarily concerned with the 
\mob function on the poset of
permutations
ordered by classic pattern containment.

This thesis consists of eight chapters.

\section{Introductory material}

This chapter (Chapter~\ref{chapter_overview}), 
is an overview of the thesis.
Following this overview, 
we have Chapter~\ref{chapter_common_definitions},
which defines the terminology and notation for the
subject area as a whole.  
Subsequent chapters will also include 
definitions of  
terminology and notation that is only used in those chapters.
This is then followed, 
in Chapter~\ref{chapter_background_and_history},
with a brief overview of the history of
the subject area, 
and a description of the motivation for the work
described in this thesis.  

\section{Chapters based on published material}

Chapters~\ref{chapter_incosc_paper},
\ref{chapter_oppadj_paper}, and
\ref{chapter_2413_balloon_paper}
are based on material that has been published
in peer-reviewed journals.  
Chapter~\ref{chapter_balloon_permutations_preprint}
is based on material currently
being prepared for publication.

These chapters start with a section (``Preamble'') that introduces the subject matter.
This is based on the abstract of
the published paper, but may include additional
material 
to help place the subject into the context
of this thesis.
This is then followed by sections  
that are based on the published material.
We then conclude each chapter with a section (``Chapter summary'')
that summarises the impact of the results,
discusses how they relate to this thesis,
and considers possible avenues for further research
following on from the 
results described in the chapter. 

\section{Conclusion}

Chapter~\ref{chapter_conclusion}
looks back at the results from
chapters~\ref{chapter_incosc_paper},
\ref{chapter_oppadj_paper}, 
\ref{chapter_2413_balloon_paper},
and~\ref{chapter_balloon_permutations_preprint}.
Here we summarise the results that we presented in the 
preceding four chapters,
and discuss whether it is possible
to find a common theme (beyond the obvious 
``related to the \mob function'')
in the work presented.
We then consider possible avenues for 
future research.

\section{Details of chapters based on published material}

We now provide a more detailed 
description of the chapters based
on published material,
or on material being prepared for publication.

In the description 
that follows,
we may use terminology
that is in common use in the field,
but which will not be formally
defined until Chapter~\ref{chapter_common_definitions}.

\subsection{The \mob function of permutations with an indecomposable lower bound}

Chapter~\ref{chapter_incosc_paper}
is based on a published paper,
``The \mob function of permutations 
with an indecomposable lower bound''~\cite{Brignall2017a},
which is joint work with 
Robert Brignall.
In this paper we show that,
given some interval $[\sigma, \pi]$
in the permutation poset,
if $\sigma$ is sum (resp. skew) indecomposable, 
then the value of the \mob function
$\mobfn{\sigma}{\pi}$ depends solely
on the sum (resp. skew) indecomposable
permutations contained in the upper bound $\pi$.

The basic methodology is to 
first use existing results to show that 
certain permutations that are contained
in the interval do not contribute
to the value of the \mob function.
We then show that the permutations that remain
can be partitioned into families, 
defined by a single sum (resp. skew) indecomposable 
permutation $\alpha$, and that the net contribution
of a family will be in
$\{
\pm \mobfn{\sigma}{\alpha},
0
\}$.
We derive a $\{ \pm 1, 0\}$ weighting function 
$\weightgen{\sigma}{\alpha}{\pi}$,
and using this, we then show that
$\mobfn{\sigma}{\pi}$ can be calculated 
by summing the value of 
$\mobfn{\sigma}{\alpha} \weightgen{\sigma}{\alpha}{\pi}$, 
over all permutations $\alpha$
that are sum (resp. skew) indecomposable
and contained in the interval.

We then set $\pi$ to be 
an increasing oscillation.
This allows us to define a revised weighting function
specific to these intervals
which can be computed 
by using simple inequalities.
This then leads to a fast 
algorithm for calculating
$\mobfn{\sigma}{\pi}$, 
where $\pi$ is an
increasing oscillation.

We then have some conjectures relating
to the long-term behaviour of 
the absolute value of $\mobfn{1}{\pi}$,
where $\pi$ is an increasing oscillation.

The chapter concludes by summarising the 
impact of the results,
particularly from a computational perspective.

\subsection{Zeros of the \mob function of permutations}

Chapter~\ref{chapter_oppadj_paper}
is based on a published paper
``Zeros of the \mob function of permutations''~\cite{Brignall2020},
which is joint work with
Robert Brignall,
V{\'i}t Jel{\'i}nek
and
Jan Kyn{\v c}l.

In this paper we show that if a permutation $\pi$
contains two opposing adjacencies,
then $\mobfn{1}{\pi} = 0$.  
We then use this result to show that
the proportion of permutations
of length $n$ with 
principal \mob function equal to zero
is, asymptotically,
bounded below by \zpmfr.

We start by showing that if a poset $P$
has a particular structure,
then $\mobp{P} = 0$.
We then show that if a
permutation $\pi$
contains two opposing adjacencies,
then the poset interval $[1, \pi]$
has the required structure,
and it follows that 
$\mobfn{1}{\pi} = 0$.
We then provide a second proof of the same result
based on normal embeddings.
The techniques used in both proofs are
used in later, more complicated, settings.

We then show that if $\sigma$ is any permutation,
and 
$\phi$ meets certain requirements,
then any permutation $\pi$ that contains
an interval order-isomorphic to $\phi$
has $\mobfn{\sigma}{\pi} = 0$.  
We use this result to show 
if $\sigma$ meets a particular condition,
and $\pi$ contains an interval copy in the form
$\alpha \oplus 1 \oplus \beta$,
then $\mobfn{\sigma}{\pi} = 0$.
We then show that if
$\sigma = 1$, then $\sigma$ meets the condition
required,
and thus we prove that
if a permutation contains
an interval copy in the form
$\alpha \oplus 1 \oplus \beta$,
then $\mobfn{1}{\pi} = 0$.

In the next part of this chapter, we show that,
asymptotically, the proportion
of permutations of length $n$ that
have a principal \mob function value of zero, $d_n$,
is bounded below by 
$\left( 1 - \dfrac{1}{\e} \right)^2 \approxeq \zpmfr$.

We then use the techniques already introduced 
to show that there are pairs of permutations,
$\alpha, \beta$,
such that if $\pi$ contains interval copies of $\alpha$ 
and $\beta$,
then $\mobfn{1}{\pi} = 0$.
We further show that there are individual permutations
$\alpha$ with the property
that 
if $\pi$ contains an interval copies of $\alpha$,
then $\mobfn{1}{\pi} = 0$.

We then discuss 
further ways in which we could find
a permutation $\pi$ where
the presence of a specific interval
or intervals in $\pi$ 
would guarantee $\mobfn{1}{\pi} = 0$.
We discuss $d_n$, including 
a conjecture on an upper bound for $d_n$.

The chapter concludes 
by summarising the 
impact of the results.
We show that there is some numerical evidence
that a large proportion of permutations
with multiple non-opposing adjacencies
have a principal \mob function value of zero,
and show that if we could 
prove this for a positive proportion 
of these permutations, then
we could improve the lower bound
of $d_n$.

We discuss extending the opposing adjacency
result to more general poset intervals
and show that this is not possible in all cases.
We then present two minor results.
The first shows that certain intervals 
$[\sigma, \pi]$ have $\mobfn{\sigma}{\pi} = 0$.
The second result shows that if 
$\sigma$ is adjacency-free, 
and $\pi$ is an inflation of $\sigma$,
and $\mobfn{\sigma}{\pi} = 0$,
then we have some information about 
the permutations used in the inflation.

\subsection{2413-balloons and the growth of the \mob function}

Chapter~\ref{chapter_2413_balloon_paper}
is based on a published paper
``2413-balloons and the growth of the \mob function''~\cite{Marchant2020},
which is sole work by the author.
In this paper we show that
the growth of the principal 
\mob function on the permutation poset is exponential.

We start by defining the ``2413-balloon'' of some permutation $\beta.$
The resulting permutation has extremal points
that are order-isomorphic to 2413, 
and the non-extremal points are an interval copy of $\beta$.
A double 2413-balloon is the result of 
ballooning a permutation that is already a 2413-balloon.

We take a poset where the upper bound is a double 2413-balloon,
and the lower bound is 1, and we show how we can partition
the chains in the poset into three sets.
We then show that two of these subsets
contribute zero to the value of the
\mob function.

The remaining set of chains has the property that
the second-highest element of every chain
is in a particular set of permutations.
We show that the Hall sum over
the remaining chains is equivalent
to summing the \mob function over the set of permutations.
The permutations in this set have the property
that they can all be formed 
as the sum of $\beta$ and either one, two or three 
copies of the permutation $1$.
For example, one of these permutations is 
$1 \oplus \beta$, and another is
$1 \ominus ((\beta \ominus 1) \oplus 1)$.
If we let $\rho$ be one of the
permutations in the set, then
this means that we can use a well-known result
to show that
$\mobfn{1}{\rho} = \pm \mobfn{1}{\beta}$.

We use this to show that if
$\beta$ is a 2413-balloon,
and $\pi$ is the 2413-balloon of $\beta$, then
$\mobfn{1}{\pi} = 2 \mobfn{1}{\beta}$,
and this 
in turn allows us to show that 
the growth of the principal 
\mob function on the permutation poset is exponential.

We then consider 2413-balloons where the permutation
being ballooned is not itself a 2413-balloon.
Using a similar argument to that used for
double 2413-balloons, 
we derive an expression for 
$\mobfn{1}{\pi}$, where $\pi$ is the 2413-balloon
of some permutation $\beta$, and $\beta$ is not a 2413-balloon.
For all but trivial cases we prove
that $\mobfn{1}{\pi} = \mobfn{1}{\beta}$.

We discuss generalising the ``balloon''
operation.
We provide two conjectures which,
up to symmetry, cover all generalised 2413-balloons.

The chapter concludes 
with a brief discussion 
of a set of permutations
where it is believed that 
the growth of the principal 
\mob function is also exponential, 
but the ``growth rate'' is faster than that found 
for double 2413-balloons.
We also discuss generalised balloons.

\subsection{The principal \mob function of balloon permutations}

Chapter~\ref{chapter_balloon_permutations_preprint}
is based on an unpublished paper which is
being prepared for submission in parallel with this thesis,
and which is sole work by the author.
In this paper we generalise the
2413-balloon permutations used in Chapter~\ref{chapter_2413_balloon_paper},
and derive an expression for the value
of the principal \mob function of these permutations.

We start by defining a method of
constructing a permutation from two smaller permutations
$\alpha$ and $\beta$.  This construction method 
requires that the constructed permutation
contains $\beta$ as an interval copy,
and that the remaining points are
order-isomorphic to $\alpha$.
We call such a permutation a ``balloon'' permutation,
which we write as $\ballij{\alpha}{\beta}$.
We describe several sub-types of balloon permutation,
including one that we call a ``wedge'' permutation.

We show that the chains in the poset interval $[1, \ballij{\alpha}{\beta}]$
can be partitioned into three sets.  
We further show that one of these subsets contributes zero
to the value of the \mob function.
We then prove that
the contribution of a second subset
can be written as the sum of 
the principal \mob function of a set of permutations,
all of which contain $\beta$ as an interval copy.
This leads to an expression for  
$\mobfn{1}{\ballij{\alpha}{\beta}}$.
This expression includes a ``correction factor'',
expressed as a sum over a particular set of (hard to handle) chains.

We then consider wedge permutations,
and we show that the correction factor is always zero,
thus leading to a simplified expression for the 
principal \mob function of a wedge permutation.

We further show that the principal \mob function
of a wedge permutation is always a multiple of 
the principal \mob function of $\beta$.

We discuss some of the problems that need to be overcome
in order to extend our result to any interval
where the upper bound is a balloon permutation.

    
    \newcommand{\sectionbreak}{\clearpage} 

    \chapter{Common definitions}
\label{chapter_common_definitions}

%
%
A \emph{permutation}\extindex{permutation}
of length $n$ is 
an ordering of the natural numbers $1, \ldots, n$.
For short (length less than 10) permutations,
we write the permutation without delimiters,
so $2413$ represents a permutation of length 4,
where the first element has value $2$.
For longer permutations, we use commas 
to delimit values, so,
for example, we would write
$1,7,4,3,10,9,2,5,8,6$.
We may occasionally use commas as delimiters
in short permutations,
where this will aid the reader.
We use $\pi_i$ to refer to the 
$i$-th element of the permutation $\pi$, 
so, for example, if $\pi = 2413$, then $\pi_1 = 2$.
We let $\emptyperm$ denote the unique permutation of length~$0$.
We write the length of a permutation $\pi$ as $\order{\pi}$.

%
%
If $L$ is a list of distinct integers,
then we can treat $L$ as a permutation
by replacing the $i$-th smallest entry with $i$.
As an example, if $L = 4,9,2,6$, then 
$L$ represents the permutation 2413.

%
%
A permutation $\pi$ can be represented graphically
by plotting the points $(i, \pi_i)$,
with $1 \leq i \leq \order{\pi}$, as 
shown in Figure~\ref{figure_example_permutation_plot}.
Throughout we treat a permutation and its plot
interchangeably.
\begin{figure}
    \centering
    \begin{tikzpicture}[scale=0.3]
        \plotpermgrid{1,7,4,3,10,9,2,5,8,6}
        \opendot{(3,4)}
        \opendot{(6,9)}
        \opendot{(7,2)}
        \opendot{(10,6)}
    \end{tikzpicture}
    \caption{The permutation $1,7,4,3,10,9,2,5,8,6$,
        highlighting one of the copies of $2413$ that it contains.}
    \label{figure_example_permutation_plot}
\end{figure}
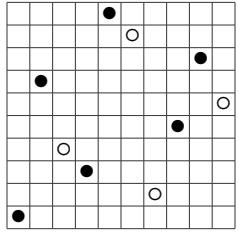

%
%
The set of all permutations of length $n$ is written $\cS_n$.
A sequence of numbers $a_1,a_2,\dots,a_n$ is 
\emph{order-isomorphic}\extindex{order-isomorphic}
to a sequence $b_1,b_2,\allowbreak\dots,b_n$ 
if for every $i,j \in [1,n]$ we have $a_i<a_j \Leftrightarrow b_i<b_j$. 
A permutation $\pi \in\cS_n$ 
\emph{contains}\extindex[permutation]{contains}
a permutation 
$\sigma\in\cS_k$ 
as a 
\emph{pattern}\extindex[permutation]{pattern} 
if $\pi$ has a subsequence of length $k$  
order-isomorphic to~$\sigma$.
We say that $\pi$ 
\emph{avoids}\extindex[permutation]{avoids}
$\sigma$ if $\pi$ does not contain $\sigma$.
%
%
There are other ways to define pattern containment, 
some of which are more specific, and others
more general.  We discuss some of these
in Chapter~\ref{chapter_background_and_history},
and in that chapter we refer to containment as defined above
as 
\emph{classic pattern containment}\extindex{classic pattern containment}.
Throughout the rest of this document we omit 
the qualifier ``classic''.

%
%
As an example of containment,
the permutation $2413$ is contained in the permutation
$1,7,4,3,10,9,2,5,8,6$, as shown
in Figure~\ref{figure_example_permutation_plot}.

%
%
If $\sigma$ is contained in $\pi$,
then there will be at least 
one set of points of $\pi$, with cardinality $\order{\sigma}$,
such that the set of points 
is order-isomorphic to $\sigma$.
We call such a set of points an 
\emph{embedding}\extindex[permutation]{embedding}
of $\sigma$ into $\pi$.
The points highlighted in
Figure~\ref{figure_example_permutation_plot},
(4,9,2,6),
represent one possible embedding of $2413$
in $1,7,4,3,10,9,2,5,8,6$.
%
%
%
Typically, where embeddings are used,
the arguments used 
require that only some of the embeddings 
are counted, and these are generally referred
to as 
\emph{normal embeddings}\extindex{normal embedding}.
It is notable that the precise definition of a
normal embedding varies between papers, and indeed,
in Brignall et al~\cite{Brignall2020},
several different definitions of normal embedding
are used.

We note here that 
one problem with embeddings
arises in cases such as $\mobfn{1}{24153}$.
Here there are plainly only five ways
to embed the permutation $1$ 
into $24153$, however
$\mobfn{1}{24153} = 6$,
and thus the embedding approach
is not sufficient.
One possible solution to this issue
is to count the 
normal embeddings and then
add a correction factor.

%
%
The set of all permutations,
ordered by pattern containment, 
is a 
\emph{poset}\extindex{poset} 
(partially ordered set).

%
%
If we have two permutations $\sigma$ and $\pi$
such that $\sigma$ is not contained in $\pi$,
and $\pi$ is not contained in $\sigma$,
then we say that $\sigma$ and $\pi$ are 
\emph{incomparable}\extindex[permutation]{incomparable}.

%
%
\extindex[poset]{interval}
A closed interval $[\sigma, \pi]$ in a poset is the set
defined as $\{ \tau : \sigma \leq \tau \leq \pi \}$.
A half-open interval $[\sigma, \pi)$ is the set
$\{ \tau : \sigma \leq \tau < \pi \}$,
and the open interval $(\sigma, \pi)$ is the set
$\{ \tau : \sigma < \tau < \pi \}$,
%
%
The 
\emph{\mob function}\extindex[\mob function]{$\mobfn{\sigma}{\pi}$}, 
$\mobfn{\sigma}{\pi}$, 
is defined for an ordered pair of elements $(\sigma, \pi)$
from any poset.
If $\sigma \not\leq \pi$, then $\mobfn{\sigma}{\pi} = 0$,
and if $\sigma = \pi$, then $\mobfn{\sigma}{\pi} = 1$.  The remaining
possibility is that $\sigma < \pi$,
and in this case we have
\begin{align}
\mobfn{\sigma}{\pi} &
= 
- \sum_{\lambda \in [\sigma, \pi)} \mobfn{\sigma}{\lambda}.
\label{equation_mobius_function}
\end{align}
%
%
If we have $\sigma < \pi$, then from the definition above
we also have
\begin{align*}
\sum_{\lambda \in [\sigma, \pi]} \mobfn{\sigma}{\lambda}
& = 0.
\end{align*}
%
%
For posets that possess a unique smallest element $\hat{0}$,
we define the 
\emph{principal \mob function}\extindex[principal \mob function]{$\mobp{\pi}$}, 
$\mobp{\pi} = \mobfn{\hat{0}}{\pi}$.
%
%
We will occasionally want to discuss
the \mob function of a poset $P$ 
with a unique minimal element $\hat{0}$, 
and unique maximal element $\hat{1}$.  Here, we define
$
\mobp{P} 
= 
\mobfn{\hat{0}}{\hat{1}}
=
\sum_{\hat{0} \leq x < \hat{1}} \mobfn{\hat{0}}{x}
$.

%
%
The 
\emph{Hasse diagram}\extindex{Hasse diagram} 
of a poset $P$ is a directed graph,
where two vertices $u$ and $w$ are connected from $u$ to $w$
by a directed arc 
if and only if $w < u$, and there is no $v$ such that $w < v < u$.
The Hasse diagram of the permutation poset $[1, 13524]$ is
shown in Figure~\ref{figure_example_hasse_diagram}.
Note that we sometimes omit the arrowheads for clarity,
as Hasse diagrams, by convention, are always drawn
with the arc direction downwards.
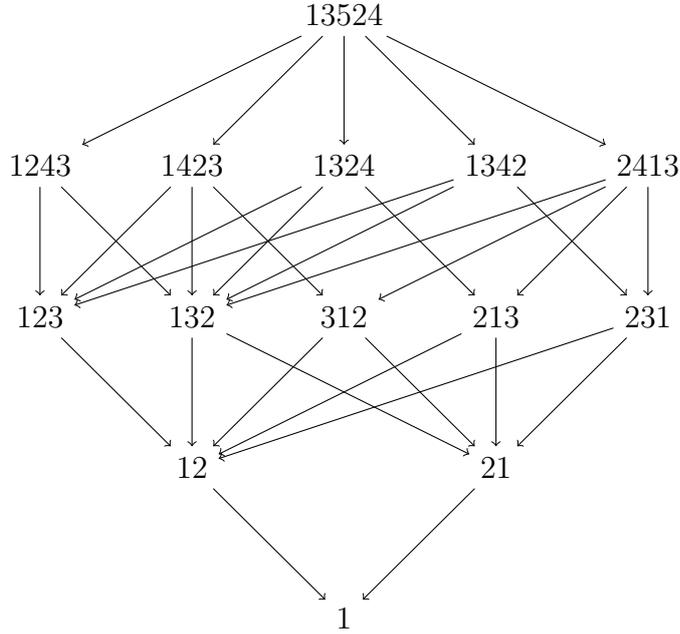
\begin{figure}
    \centering
        \begin{tikzpicture}[xscale=1,yscale=2]
        \pnode{13524}{0}{5};
        \pnode{1243}{-4}{4};
        \pnode{1423}{-2}{4};
        \pnode{1324}{ 0}{4};
        \pnode{1342}{ 2}{4};
        \pnode{2413}{ 4}{4};
        \pnode{123}{-4}{3};
        \pnode{132}{-2}{3};
        \pnode{312}{ 0}{3};
        \pnode{213}{ 2}{3};
        \pnode{231}{ 4}{3};
        \pnode{12}{-2}{2};
        \pnode{21}{ 2}{2};
        \pnode{1}{0}{1};
        \ddline{13524}{1243,1423,1324,1342,2413};
        \ddline{1243}{123,132}
        \ddline{1423}{123,132,312}
        \ddline{1324}{123,132,213}
        \ddline{1342}{123,132,231}
        \ddline{2413}{132,213,231,312}
        \ddline{123}{12}
        \ddline{132}{12,21}
        \ddline{213}{12,21}
        \ddline{312}{12,21}
        \ddline{231}{12,21}
        \ddline{12}{1}
        \ddline{21}{1}
        \end{tikzpicture}		
    \caption{The Hasse diagram of the poset interval $[1, 13524]$.}
    \label{figure_example_hasse_diagram}
\end{figure}

%
%
A 
\emph{chain}\extindex{chain} 
in a poset interval $[\sigma, \pi]$ is,
for our purposes,
a subset of the elements in the interval $[\sigma, \pi]$,
where the subset includes the elements $\sigma$ and $\pi$,
and every distinct pair of elements of the subset are comparable.
This last clause means that the subset has a total order.
If a chain $c$ has $k$ elements,
then we say that the length of $c$, written $\order{c}$,
is $k - 1$.
%
%
One way to visualise a chain is to first choose
a path in the Hasse diagram from the highest entry
to the lowest,
and then a chain is found by 
choosing a (possibly improper)
subset of the elements on the path, 
ensuring that the first and last elements 
(the upper and lower bounds of the poset)
are included in the subset.

%
%
Chains in a poset interval are related to
the \mob function by 
Hall's Theorem~\cite[Proposition 3.8.5]{Stanley2012},
which says 
that 
\begin{align*}
\mobfn{\sigma}{\pi} = 
\sum_{c \in \chainc(\sigma, \pi)} (-1)^{\order{c}}  =
\sum_{i=1}^{\order{\pi} - 1} (-1)^i K_i
\end{align*}
where $\chainc(\sigma, \pi)$ is the set of chains 
in the poset interval $[\sigma, \pi]$, 
and $K_i$ is the number of chains of length $i$.

%
%
If $\chainc$ is a subset of the chains 
in some poset interval $[\sigma, \pi]$,
then the 
\emph{Hall sum}\extindex[chain]{Hall sum} 
of $\chainc$ is
$\sum_{c \in \chainc} (-1)^{\order{c}}$.

%
%
A \emph{parity-reversing involution}\extindex{parity-reversing involution}, 
$\Phi: \chainc \mapsto \chainc$,
is an involution 
such that for any $c \in \chainc$,
the parities of $c$ and $\Phi(c)$
are different.

%
%
A simple corollary to Hall's Theorem is
\begin{corollary}
    \label{corollary-halls-corollary}
    If we can find a set of chains $\chainc$
    with a parity-reversing involution,
    then the Hall sum of $\chainc$ is zero.
\end{corollary}
\begin{proof}
    Because there is a parity-reversing involution, 
    the number of chains in $\chainc$ with odd length is
    equal to the number of chains with even length,
    so $\sum_{c \in \chainc} (-1)^{\order{c}}  = 0$.  	
\end{proof}

%
%
In Chapters~\ref{chapter_2413_balloon_paper} 
and~\ref{chapter_balloon_permutations_preprint}
we will want to
show that there is a parity-reversing involution 
on a set of chains $\chainc$.
Our basic methodology, given
a set of chains $\chainc$,
and a chain $c \in \chainc$,
will be to construct a chain $\cprime$
by using a parity-reversing involution $\Phi$.
Strictly speaking, $\Phi$ 
is a function that maps a set of permutations
(which is a chain)
to a set of permutations
(which may not be a chain).
As examples, if $\Phi(c)$
removes the largest or smallest element of $c$,
or adds an element so that $\Phi(c)$ does not have
a total order, then $\Phi(c)$ is not a chain.
To show that $\Phi$ is a parity-reversing involution
we will need to show that $\Phi(c)$ is a chain in $\chainc$,
and that $c$ and $\Phi(c)$ have opposite parities.
In our discussions, we will typically set
$\cprime = \Phi(c)$, 
and then show that the 
set of permutations $\cprime$ is a chain.
We will then, without further comment, 
treat $\cprime$ as a chain.

%
%
When discussing chains, in general we will only be interested
in a small subset of the chain containing two or three elements.
We say that a 
\emph{segment}\extindex[chain]{segment}
of some chain $c$ 
is a non-empty subset of the elements in $c$ with the property that
any element not in the segment is 
either less than every element in the segment,
or is greater than every element in the segment.

%
%
A 
\emph{direct sum}\extindex[permutation]{direct sum} 
of two permutations $\alpha$ and $\beta$
of lengths $m$ and $n$ respectively
is the permutation 
$
\alpha_1, \ldots, \alpha_m,
\beta_1 + m, \ldots, \beta_n + m
$.
We write a sum as $\alpha \oplus \beta$.
A 
\emph{skew sum}\extindex[permutation]{skew sum}, 
$\alpha \ominus \beta$,
is the permutation
$
\alpha_1 +n, \ldots, \alpha_m + n,
\beta_1, \ldots, \beta_n
$.
As examples,
$321 \oplus 213 = 321546$, 
and 
$321 \ominus 213 = 654213$, 
and these are shown in 
Figure~\ref{figure_examples-of-sums}.
%
%
A 
\emph{sum-indecomposable}\extindex[permutation]{sum-indecomposable} 
(resp. 
\emph{skew-indecomposable}\extindex[permutation]{skew-indecomposable})
permutation is a permutation
that cannot be written as the 
direct sum (resp. skew sum) of two smaller permutations.
If a permutation is not sum-indecomposable,
then it is 
\emph{sum-decomposable}\extindex[permutation]{sum-decomposable},
and 
if a permutation is not skew-indecomposable,
then it is 
\emph{skew-decomposable}\extindex[permutation]{skew-decomposable},
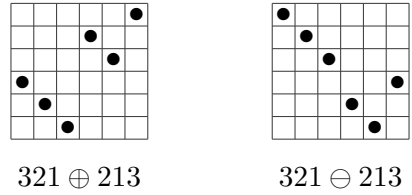
\begin{figure}
    \begin{center}
        \begin{subfigure}[t]{0.22\textwidth}
            \centering
            \begin{tikzpicture}[scale=0.3]   
            \plotpermgrid{3,2,1,5,4,6}
            \end{tikzpicture}
            \caption*{$321 \oplus 213$}
        \end{subfigure}
        \begin{subfigure}[t]{0.22\textwidth}
            \centering
            \begin{tikzpicture}[scale=0.3]                    
            \plotpermgrid{6,5,4,2,1,3}
            \end{tikzpicture}
            \caption*{$321 \ominus 213$}
        \end{subfigure}
    \end{center}
    \caption{Examples of direct and skew sums.} 
    \label{figure_examples-of-sums}	
\end{figure}    

%
%
Given a permutation $\pi$, 
the 
\emph{finest}\extindex{finest sum decomposition} 
sum decomposition  (resp. skew decomposition)
of $\pi$ is a decomposition into the maximum number
of sum-indecomposable (resp. skew-indecomposable)
permutations.
As examples, using Figure~\ref{figure_examples-of-sums},
the finest sum decomposition of $321546$
is $321 \oplus 21 \oplus 1$,
and the finest skew-decomposition of 
$654213$
is
$1 \ominus 1 \ominus 1 \ominus 213$.

%
%
Let $\alpha$ be a permutation, 
and $r$ a positive integer.
Then $\sumra$\extindex{$\sumra$}
is $\alpha \oplus \alpha \oplus \ldots \oplus \alpha \oplus \alpha$,
with $r$ occurrences of $\alpha$.
If $S$ is a set of permutations, then
$\nsums{r}{S} = \cup_{\lambda \in S} \{ \nsums{r}{\lambda} \}$.

%
%
A 
\emph{layered}\extindex[permutation]{layered} 
permutation 
is a permutation that can be written
as the direct sum of one or more
decreasing permutations.
Egge and Mansour, in~\cite{Egge2004}, show that
layered permutations can also be defined as
permutations that avoid the permutations 
$231$ and $312$.
The first example in
Figure~\ref{figure_examples-of-sums}
is a layered permutation.

%
%
If a permutation $\pi$ can be written as 
$\oneplus\oneplus\tau$,
$\oneminus\oneminus\tau$,
$\tau\plusone\plusone$, or
$\tau\minusone\minusone$,
where $\tau$ is non-empty 
(so $\order{\pi} \geq 3$),
then we say that $\pi$
has a \emph{long corner}\extindex{long corner}.

%
%
We will occasionally want to discuss
situations where some permutation $\pi$ is
known to
have a sum-decomposable (resp. skew-decomposable) decomposition,
but we do not know exactly which 
permutations form the decomposition.
In such cases we will write
$\pi = \alpha_1 \oplus \ldots \oplus \alpha_n$
or
$\pi = \alpha_1 \ominus \ldots \ominus \alpha_n$.
Similarly, if we want to discuss 
an arbitrary set of permutations 
then we will write
$\{ \alpha_1, \ldots, \alpha_n \}$.
It will always be clear from the context
whether $\alpha_i$ refers to 
the $i$-th element of the permutation $\alpha$,
the $i$-th permutation in a sum,
or the $i$-th permutation in a set of permutations.

%
%
An 
\emph{interval}\extindex[permutation]{interval} 
in a permutation $\pi$ is a 
non-empty contiguous set of indexes $i, i+1, \ldots, j$
such that the set of values
$\{ \pi_i, \pi_{i+1}, \ldots, \pi_j \}$ is also contiguous.
Every permutation $\pi$ has intervals of 
length 1 and of length $\order{\pi}$,
which we call 
\emph{trivial intervals}\extindex{trivial interval}.
A 
\emph{simple}\extindex[permutation]{simple} 
permutation 
is a permutation that only has 
trivial intervals.
As examples, $1324$ is not simple, as, for example, 
the second and third points $(32)$
form a non-trivial interval,
whereas $2413$ is simple.

%
%
We say that 
$\pi$ has an 
\emph{interval copy}\extindex[permutation]{interval copy} 
of a permutation $\alpha$ if it contains an interval 
of length $\order{\alpha}$ whose 
elements form a subsequence order-isomorphic to~$\alpha$.

%
%
We note here that the term ``interval'' is used
in relation to both posets and permutations.
This is standard terminology in the field, 
and when we use the term ``interval''
it will be clear from the context
whether we are referring to a
poset interval or an interval in a permutation.

%
%
An interval of length 2 is termed an 
\emph{adjacency}\extindex{adjacency}
in this thesis.  
An adjacency is clearly order-isomorphic
to either $12$, an 
\emph{up-adjacency}\extindex{up-adjacency},
or to $21$, a 
\emph{down-adjacency}\extindex{down-adjacency}.
If a permutation contains 
at least one up-adjacency
and at least one down-adjacency,
then we say that the permutation has
\emph{opposing adjacencies}\extindex{opposing adjacencies}.
An interval of length 3 that is monotonic,
that is, order-isomorphic to either 123 or 321,
is a 
\emph{triple-adjacency}\extindex{triple-adjacency}.
We note here that some sources 
use ``adjacency''
to refer to a non-trivial interval 
of any length
that is monotonic.

%
%
A 
\emph{descent}\extindex{descent} 
in a permutation $\pi$ is a position $i$
such that $\pi_i > \pi_{i+1}$.  
Similarly, an 
\emph{ascent}\extindex{ascent} 
in a permutation $\pi$ is
a position $i$ such that $\pi_i < \pi_{i+1}$.

%
%
A 
\emph{permutation class}\extindex{permutation class}
is a set of permutations $C$
with the property that if $c \in C$, and $d < c$, then $d \in C$.
Every permutation class can be defined by the minimal
set of permutations that are not contained in the class,
and this minimal set is referred to as the 
\emph{basis}\extindex[permutation class]{basis}.
If a permutation class $C$ has basis $B$, then we write
$C = \Av (B)$.
Where we want to discuss a permutation
class $C$ that contains a specific set 
of (normally simple) permutations,
we refer to $C$ as a 
\emph{hereditary class}\extindex[permutation class]{hereditary class}.
There is no difference between a permutation class
and a hereditary class, 
the distinction is simply used to draw attention
to the properties of the class that we are discussing.

%
%
Permutations can be represented
by plotting points in a square grid,
as described earlier.
It is clear that any symmetry of the square,
if applied to a permutation plot,
will result in another  
permutation plot.  
If $\alpha$ is a permutation,
then a reflection in a vertical bisector of the plot
is called a 
\emph{reversal}\extindex[permutation]{reversal}, 
written $\alpha^R$,
a reflection in a horizontal bisector of the plot
is called a 
\emph{complement}\extindex[permutation]{complement}, 
written $\alpha^C$.
The inverse of a permutation, written $\alpha^{-1}$
is also a symmetry.
These three operations are the generating set
of the group of symmetries of permutations.

We can now see that
for any permutations $\sigma$ and $\pi$, and any symmetry $S$,
\[
\mobfn{\sigma}{\pi}
=
\mobfn{\sigma^S}{\pi^S}.
\]

We will occasionally want to discuss 
permutations
where we want a unique representative
from the symmetries.  We say that such
a representative is the 
\emph{canonical}\extindex[permutation]{canonical}
form of the permutation,
and for our purposes we choose the
symmetry which is smallest under the lexicographic order.
As an example, 
2413 and 3142 are symmetries of one another,
and the canonical representation is 2413.

    \chapter{Background and history}
\label{chapter_background_and_history}

\section{Permutations}

%
%
The first, albeit implicit, reference to permutations
in the literature appears to be
due to Euler in~\cite{Euler1755},
where he describes polynomials
which essentially define what are now known as 
the Eulerian numbers $A_{n,m}$.
In permutational terms, $A_{n,m}$ 
counts the number of permutations of length $n$
that have $m$ descents.
%
%
The next significant set of results
comes some 150 years later,
where MacMahon~\cite{MacMahon1915} 
has a result that,
interpreted in permutational terms,
shows that $\Av(123)$ is counted by the Catalan numbers.
%
%
The Erd\H{o}s-Szekeres theorem~\cite{Erdos1935} 
can be interpreted as saying that
a permutation of length $(a-1)(b-1)+1$
must contain either an increasing sequence of length $a$ 
or a decreasing sequence of length $b$.

%
%
The study of pattern avoidance in permutations
can be said to have started with exercise 2.2.1(5)
in Knuth~\cite{Knuth1968}, where readers are
essentially asked to show that 
a permutation that can be stack-sorted
must avoid 231.
%
%
This work was further developed
in the 1970s and 1980s
in papers by 
Knuth~\cite{Knuth1970},  
Rogers~\cite{Rogers1978},
Rotem~\cite{Rotem1981},
and
Simion and Schmidt~\cite{Simion1985}.

%
%
This initial development then turned into
a veritable explosion of papers, 
most of which are too specific to
relate to this general background.
%
%
A good summary of the 
way in which the field has developed
can be found in 
the book by Kitaev~\cite{Kitaev2011},
the book by Bona~\cite{Bona2016Book},
the survey article by Steingr{\'{i}}msson~\cite{Steingrimsson2013},
and the chapter by Vatter on permutation classes in~\cite{Handbook2015}.

\section{The \mob function}

%
%
The \mob function was first defined
in the context of number theory
by August \mob in 1832 in~\cite{Mobius1832}.
In that paper, \mob defines
$\mu(n): \bbN \mapsto \bbN$ 
as 0 if $n$ has a repeated prime factor, and as
$(-1)^k$ if $n$ is the product of $k$ distinct prime factors.
If we say that a positive integer $a$ 
is contained in a positive integer $b$
if $a$ divides $b$, then the integers 
under this relationship form a poset,
and $\mu(n) = \mobfn{1}{n}$.

%
%
The number-theoretic \mob function has been extensively
studied since its definition.
The combinatorial \mob function does not seem
to have any significant presence in the literature
until a seminal paper by Rota in 1964~\cite{Rota1964a},
which made an explicit link between
the principle of inclusion--exclusion
and the combinatorial \mob function.

%
%
While there are many papers that
have results relating to the \mob function on a variety of
posets, we refer the reader to 
Cameron~\cite{Cameron1994}
or
Stanley~\cite{Stanley2012}
for a general background to the area.

%
%
The classic definition of the \mob function,
as given in Equation~\ref{equation_mobius_function},
is, essentially, a recursive sum over the elements of the poset.
This thesis, in general, restricts itself to this view.  
A simple consequence of the fundamental definition
is
Hall's Theorem~\cite[Proposition 3.8.5]{Stanley2012} 
which defines the \mob function as a sum
over the chains in the poset.
There are, however, other ways in which we can understand the
\mob function, and 
in order to provide a broad background,
we briefly describe two of them here.

\subsection{Simplicial complexes}

%
%
Given a set of vertices $V$, 
a 
\emph{simplicial complex}\extindex{simplicial complex} 
$\simpcomp$
is a non-empty set of subsets of $V$ such that
if $v \in V$, then $\{ v \} \in \simpcomp$;
and
if $G \in \simpcomp$, and
$F \subset G$, then
$F \in \simpcomp$.
If $F \in \simpcomp$,
then we say that $F$ has dimension
$\order{F} - 1$.  
We then refer to $F$ as a 
\emph{face}\extindex[simplicial complex]{face} of $V$.
Note that the empty subset $\emptyset$
is a face of $V$.
If we have two elements of a poset $\sigma$ and $\pi$,
with $\sigma < \pi$, and
$(\sigma, \pi)$ is non-empty, 
then we can set $V$ to be 
the set of chains in the 
open interval $(\sigma, \pi)$,
and this gives a simplicial complex $\simpcomp$.
%
%
Given a simplicial complex $\simpcomp$,
the reduced Euler characteristic of $\simpcomp$,
$\chi (\simpcomp)$ is defined as
\[
\chi (\simpcomp)
=
\sum_{k=-1}^{\Dim \simpcomp} (-1)^k f_k (\simpcomp),
\]
where $f_k (\simpcomp)$ is the number of faces of dimension $k$.
%
%
If we have a poset $P$ with 
unique minimal and maximal elements 
$\hat{0}$ and $\hat{1}$ respectively,
and set $\simpcomp$ to be the 
chains in the open interval $(\hat{0}, \hat{1})$,
then
Hall's 
Theorem~(see, for example, Stanley~\cite[Proposition 3.8.5]{Stanley2012} 
or Wachs~\cite[Proposition 1.2.6]{Wachs2006})
gives us that
$
\mobp{P} = \chi (\simpcomp)
$.

%
%
Since a simplicial complex is a topological entity, 
in addition to the possibility of using the 
\mob function to determine the value of the 
reduced Euler characteristic, it is possible
to pose questions about the topology of the poset.
While this approach has been taken in some papers 
(discussed in 
Section~\ref{section_background_permutation_poset_and_mobius} 
below), 
this thesis does not use this approach or provide 
any topological results.  
The interested reader is referred to
the material in~\cite{Wachs2006}
for further details.

\subsection{Incidence algebras and incidence matrices}

%
%
A poset $P$ is 
\emph{locally finite}\extindex[poset]{locally finite}
if, for every $\sigma, \pi \in P$,
the interval $[\sigma, \pi]$
has a finite number of elements.
%
%
Following Rota~\cite{Rota1964a}, we define
an 
\emph{incidence algebra}\extindex{incidence algebra}
by first taking
a locally finite partially ordered set $P$,
and considering the set of all 
real-valued functions $f(x,y)$,
where
$x, y \in P$ and 
$f(x,y) = 0$ if $x \not\leq y$.
We then define the incidence algebra 
of $P$ by convolution:
\[
(f * g)(x,y) = \sum_{x \leq z \leq y} f(x,z) g(z,y).
\]
%
%
This algebra has an identity element,
normally written as $\delta(x,y)$, 
which is defined as
\[
\delta(x,y) = 
\begin{cases}
1 & \text{If } x = y \\
0 & \text{Otherwise}.
\end{cases}
\]
%
%
The 
\emph{zeta function}\extindex[incidence algebra]{zeta function}
is defined as
\[
\zeta(x,y)  = 
\begin{cases}
1 & \text{If } x \leq y \\
0 & \text{Otherwise}.
\end{cases}
\]
%
%
With these definitions, 
it can be shown that the \mob function is the convolutional inverse
of $\zeta$, so
\[
    \zeta * \mu (x,y) = \mu * \zeta (x,y) = \delta(x,y).
\]
%
%
Let $Z$ be a square matrix, 
with rows and columns indexed by the elements of a poset $P$,
and with $Z_{x,y} = \zeta(x,y)$.
We call this the 
\emph{zeta matrix}\extindex[incidence algebra]{zeta matrix}
of $P$.
%
%
It is now possible to show that if $x$ and $y$ are 
elements of the poset, then
\[
    \mobfn{x}{y} = (Z^{-1})_{x,y}.
\]

%
%
We remark here that the result above implies that
we can determine the value of the \mob function
for every interval in a poset
by (simply) calculating the inverse of the 
zeta matrix.  
%
%
We have some computational evidence that
using the ``matrix inverse'' method
to determine the value of the \mob function
for a significant number of intervals
in a poset is computationally more efficient 
than using the fundamental definition
given in Equation~\ref{equation_mobius_function}.
On the other hand, if we want to
determine the value of the \mob function 
for a single interval, then the 
fundamental definition seems to be
significantly faster than the matrix inverse method.
Of course, our ideal is to find
methods that can determine
the value of the \mob function
faster than either the matrix inverse method,
or using the fundamental definition.

\section{The \mob function for general posets}
\label{section_background_mobius_various_posets}

%
%
Before we move on to consider the 
\mob function of the permutation poset under classic 
pattern containment, we divert slightly to review
some results relating to the \mob function
on other posets.  
We start by remarking that,
for a general poset, 
using the recursive definition of
the \mob function is computationally hard.
Our purpose in this section is to 
establish that determining the \mob function 
need not be computationally hard in some cases.
We refer the reader to~\cite{Kitaev2011} for a good overview of 
most of the containment types discussed in this section.

%
%
There are some well-known cases where an explicit
formula exists for the \mob function.
For example, 
the \mob function on 
a Boolean algebra is given by 
$\mobfn{R}{S} = (-1)^{\#(S-R)}$
(see, for instance, Example 3.8.3 in~\cite{Stanley2012}).

%
%
A slightly more complex example is given
by the poset of subspaces of a vector space 
$V \subseteq GF(q)^n$.
If we have $U \subseteq W \subseteq V$, then
\[
\mobfn{U}{W} =
(-1)^{k} q^{\binom{k}{2}},
\text{ where }
k = \dim(W) - \dim(U).    
\]
This result is attributed to Hall
in Rota~\cite{Rota1964a}.
%
%
%
The poset of subspaces of a vector space
is an example of a lattice,
and the proof given in Rota utilises this fact.
The \mob function for general lattices 
is also well-known
(see, for instance, Section 3.9 in~\cite{Stanley2012}).

%
%
We now turn to posets that are
defined by free monoids over alphabets,
or by permutations using a containment 
other than classic pattern containment.

%
%
Bj\"orner completely determined the \mob function of subword order 
in~\cite{BjornerSubword}, and then 
completely determined the \mob function for factor order
in~\cite{Bjorner1993}.

%
%
Sagan and Vatter, 
in~\cite{Sagan2006},
considered ordered partitions (compositions) of an integer,
with a partial order given by subwords,
and completely determined the \mob function on this poset.
This paper also has the first result for
the permutation poset under classic pattern containment,
which we discuss in 
Section~\ref{section_background_permutation_poset_and_mobius}.

%
%
Bernini, Ferrari and Steingr\'{\i}msson,
in~\cite{Bernini2011},
considered permutations using consecutive pattern containment.
For most intervals $[\sigma, \pi]$ they have a set of explicit
formulae for $\mobfn{\sigma}{\pi}$,
based mainly on how many times $\sigma$ occurs in $\pi$.
For intervals not covered by their formulae,
they provide a polynomial algorithm to calculate the 
\mob function.  
We consider that the \mob function
of permutations under consecutive pattern containment
is, therefore, completely known.
Sagan and Willenbring,
in~\cite{Sagan2012},
reproduced this result using a 
technique known as discrete Morse theory.

%
%
Bernini and Ferrari, in~\cite{Bernini2017},
introduced the quasi-consecutive pattern poset
of permutations,
where $\sigma$ is contained in $\pi$ if
$\pi$ contains an occurrence of $\sigma$ where all entries are adjacent,
except possibly the first and second.
They completely determine the \mob function
for any interval $[\sigma, \pi]$ where
$\sigma$ occurs exactly once in $\pi$.

%
%
A recent preprint by 
Bernini, 
Cervetti, 
Ferrari and
Steingr\'{\i}msson~\cite{Bernini2019}
considers the poset of Dyck paths,
where we say that a path $P$ contains a path $Q$
if the steps in $Q$ are a subsequence 
of the steps in $P$.  
The preprint includes 
expressions for  
the \mob function 
of some specific intervals in this poset.

%
%
The posets discussed so far are, in some way,
simpler than the 
permutation poset under classic pattern containment,
and we have seen that 
the \mob function has either been completely determined,
or, as in the last two examples,  has been determined for a 
particular subset of intervals in the poset.
We now consider posets that, in some sense, generalize 
classic pattern containment.

%
%
Mesh patterns are a generalization of classic pattern containment
on permutations,
and the poset of mesh patterns contains 
the poset of permutations as an induced subposet.
We refer the reader 
to~\cite{Branden2011} for a formal definition of mesh patterns.
%
%
In~\cite{Smith2018a},
Smith and Ulfarsson present some 
initial results on the \mob
function of the mesh pattern poset, 
and show that 
as $n \to \infty$, the proportion
of mesh patterns $p$ of length $n$
with $\mobfn{1^\emptyset}{p} = 0$
approaches 1,
where $1^\emptyset$ is the unshaded
singleton mesh pattern.
%
%
The mesh pattern $1^\emptyset$
corresponds to the permutation $1$ in
the induced poset of permutations.
In the permutation pattern poset, 
we know~\cite{Albert2003} that the number of   
simple permutations of length $n$
is, asymptotically, $\frac{n!}{\e^2}$,
and it is generally believed
that for most simple permutations
the value of the principal 
\mob function is non-zero,
thus in the permutation pattern poset
we do not expect
the proportion
of permutations $\pi$ of length $n$
with $\mobp{\pi} = 0$
to approach 1.

%
%
Smith generalised pattern containment in~\cite{Smith2019},
and found some explicit formulae for the \mob function.
These formulae have the general form
\[
\mobfn{\sigma}{\pi}
=
(-1)^{\order{\pi} - \order{\sigma}} E(\sigma, \pi)
+
\sum_{\lambda \in [\sigma, \pi)} 
\mobfn{\sigma}{\lambda} 
\mu[ \hat{P}(\lambda, \pi) ],
\]
where 
$E(\sigma, \pi)$ counts specific types of embeddings
of $\sigma$ into $\pi$,
and
$\hat{P}(\lambda, \pi)$ is a poset derived from 
the interval $[\lambda, \pi]$.

Although this last example does have a complete
characterisation of the \mob function
on all intervals of the poset,
from a computational perspective 
the result is only useful if, outside the result given,
we can show that the second term is zero.
This result is a generalisation
of some of the results given
by Smith in~\cite{Smith2016a},
which we discuss in the following
section.

\section{The \mob function of the permutation poset under pattern containment}
\label{section_background_permutation_poset_and_mobius}

%
%
The study of the \mob function in the (classic) permutation poset
was introduced by
Wilf~\cite{Wilf2002}, who wrote
\begin{quote}
    We can partially order the set of all permutations 
    of all numbers of letters by
    declaring that $\sigma \leq \tau$ if
    $\sigma$ is contained as a pattern in $\tau$. 
    It would be interesting to study
    this as a poset. 
    For example, what can be said about its \mob function?
\end{quote}  

%
%
The first result in this area was by Sagan and Vatter~\cite{Sagan2006}.
Their paper primarily concerns itself with 
the \mob function of a composition poset.
An integer composition can be thought of
as an ordered list of positive integers,
and a layered permutation 
can be completely specified by 
such a list,
so there is a bijection between
integer compositions and layered permutations.
From this it follows that
there is a bijection between
the poset of compositions of integers and 
the poset of layered permutations.
In the final section of their paper, 
they use 
this bijection
to essentially give
an expression for the
\mob function on 
intervals in the poset of layered permutations
under classic pattern containment.

%
%
Steingr\'{\i}msson and Tenner~\cite{Steingrimsson2010} found a 
large class of pairs of permutations $(\sigma, \pi)$ 
where $\mobfn{\sigma}{\pi} = 0$.  
They show that the (poset) interval
$[\sigma, \pi]$ has 
$\mobfn{\sigma}{\pi} = 0$
if $\pi$ contains a non-trivial interval
where none of the elements of
the (permutation) interval 
are part of an embedding of 
$\sigma$ into $\pi$.
They also show that 
if there is exactly one
embedding of $\sigma$ into $\pi$,
and the complement of the embedding satisfies
certain conditions, then
$\mobfn{\sigma}{\pi} \in \{0, \pm 1\}$.

%
%
In a seminal paper,
Burstein, Jel{\'{i}}nek, Jel{\'{i}}nkov{\'{a}} 
and Steingr{\'{i}}msson~\cite{Burstein2011} found
a recursion for the \mob function
for sum/skew decomposable permutations
in terms of the sum/skew indecomposable 
permutations in the lower and upper bounds.
They also found a method to determine
the \mob function for separable permutations
by counting embeddings.
The recursions for decomposable permutations
are used extensively
in this thesis.

%
%
McNamara and Steingr{\'{i}}msson~\cite{McNamara2015}
investigated the topology of intervals in the 
permutation poset, and in doing so
found a single recurrence
equivalent to the recursions in~\cite{Burstein2011}.

%
%
We now compare the recursions from
Burstein et al~\cite{Burstein2011} 
with those from 
McNamara and Steingr{\'{i}}msson~\cite{McNamara2015}.
The recursions from Burstein, Jel{\'{i}}nek, Jel{\'{i}}nkov{\'{a}} 
and Steingr{\'{i}}msson~\cite{Burstein2011}
can be written as follows.
\begin{quotation}
    \begin{proposition}[{%
            McNamara and 
            Steingr{\'{i}}msson
            \cite[Proposition 8.3]{McNamara2015},
            following 
            Burstein, 
            Jel{\'{i}}nek, 
            Jel{\'{i}}nkov{\'{a}} and 
            Steingr{\'{i}}msson 
            \cite[Proposition 1]{Burstein2011}%
        }]
        \label{BJJS-proposition-1-as-in-mcnamara}
        Let $\sigma$ and $\pi$ be non-empty permutations with
        finest decompositions $\sigma= \sigma_1 \oplus \ldots \oplus \sigma_s$
        and
        $\pi = \pi_1 \oplus \ldots \oplus \pi_t$,
        where $t \geq 2$. 
        Suppose that $\pi_1 = 1$.  
        Let $k \geq 1$ be the largest integer such that all the components 
        $\pi_1, \ldots, \pi_k$ are equal
        to 1, 
        and let $\ell \geq 0$ be the largest integer such that all the components 
        $\sigma_1, \ldots, \sigma_\ell$ are equal
        to 1. Then
        \[
        \mobfn{\sigma}{\pi}
        =
        \begin{cases}
        0 & \text{if $k-1 > \ell$,} \\
        -\mobfn{\sigma_{> k-1}}{\pi_{> k}} & \text{if $k-1 = \ell$,} \\
        \mobfn{\sigma_{> k}}{\pi_{> k}} - \mobfn{\sigma_{> k-1}}{\pi_{> k}} & \text{if $k-1 < \ell$.}
        \end{cases}
        \]    
    \end{proposition}
    The remaining case is $\pi_1 > 1$, and is covered by the next proposition.
    \begin{proposition}[{%
            McNamara and 
            Steingr{\'{i}}msson
            \cite[Proposition 8.4]{McNamara2015},
            following 
            Burstein, 
            Jel{\'{i}}nek, 
            Jel{\'{i}}nkov{\'{a}} and 
            Steingr{\'{i}}msson 
            \cite[Proposition 2]{Burstein2011}%
        }]
        Let $\sigma$ and $\pi$ be non-empty permutations with
        finest decompositions 
         $\sigma= \sigma_1 \oplus \ldots \oplus \sigma_s$
        and
        $\pi = \pi_1 \oplus \ldots \oplus \pi_t$,
        where $t \geq 2$.
        Suppose that
        $\pi_1 > 1$. Let $k \geq 1$ be the largest integer such that all the 
        $\pi_1, \ldots, \pi_k$ are equal
        to $\pi_1$. Then
        \[
        \mobfn{\sigma}{\pi} =
        \sum_{i=1}^s
        \sum_{j=1}^k
        \mobfn{\sigma_{\leq i}}{\pi_{1}}
        \mobfn{\sigma_{> i}}{\pi_{> j}}.
        \]
    \end{proposition}
\end{quotation}
The recursion found by McNamara and Steingr{\'{i}}msson
can be written as follows.
\begin{quotation}
    \begin{proposition}[{%
            McNamara and 
            Steingr{\'{i}}msson
            \cite[Proposition 8.1]{McNamara2015}%
        }]
    Consider permutations $\sigma$ and $\pi$ and let 
    $\pi = \pi_1 \oplus \ldots \oplus \pi_t$
    be the finest
    decomposition of $\pi$. Then
    \[
    \mobfn{\sigma}{\pi}
    =
    \sum_{\sigma = \varsigma_1 \oplus \ldots \oplus \varsigma_t}
    \prod_{1 \leq m \leq t}
    \begin{cases}
    \mobfn{\varsigma_m}{\pi_m} + 1 & 
    \text{If $\varsigma_m = \emptyperm$ and $\pi_{m-1} = \pi_m$,} \\
    \mobfn{\varsigma_m}{\pi_m} &
    \text{otherwise, }
    \end{cases}
    \]
    where the sum is over all direct sums 
    $\sigma = \varsigma_1 \oplus \ldots \oplus \varsigma_t$,
    such that
    $\emptyperm \leq \varsigma_m \leq \pi_m$
    for all $1 \leq m \leq t$.
    \end{proposition}
\end{quotation}
We remark here that the recursion from 
McNamara and Steingr{\'{i}}msson
is, in some sense, a nicer recursion than that found by 
Burstein, Jel{\'{i}}nek, Jel{\'{i}}nkov{\'{a}} 
and Steingr{\'{i}}msson.
Despite this, in 
this thesis
we use the 
Burstein et al recursions 
as these
are easier to work with
in the context of our results.

%
%
Smith~\cite{Smith2013}
found an explicit formula for the \mob function on the interval
$[1, \pi]$ for all permutations $\pi$ with a single descent.
Smith~\cite{Smith2016}
has explicit expressions for the 
\mob function $\mobfn{\sigma}{\pi}$
when $\sigma$ and $\pi$ have the same number of descents.
In~\cite{Smith2016a}, Smith found
an expression that determines the
\mob function for all intervals in
the poset.  
The main result is
\begin{align}
\label{equation_smith_all_intervals}
\mobfn{\sigma}{\pi}
&=
(-1)^{\order{\pi} - \order{\sigma}}
\NE(\sigma, \pi)
+
\sum_{\lambda \in [\sigma, \pi)}
\mobfn{\sigma}{\lambda}
\sum_{S \in EZ^{\lambda,\pi}} (-1)^{\order{S}}
\end{align}
where 
$\NE(\sigma, \pi)$
is the number of normal embeddings of $\sigma$ into $\pi$,
and
$EZ^{\lambda,\pi}$
is a set of sets of embeddings of $\lambda$ into $\pi$
that satisfy a particular condition.

One view of this result is that
it tells us that the value of the \mob function
on an interval $[\sigma, \pi]$
is given by the number of normal embeddings
of $\sigma$ into $\pi$, plus a correction factor.
Smith notes~\cite[Remark 22]{Smith2016a} that 
95\% of intervals with $\order{\pi} \leq 8$ have
$\mobfn{\sigma}{\pi}
=
(-1)^{\order{\pi} - \order{\sigma}}
\NE(\sigma, \pi)$,
so in these cases
the correction factor is zero.

Smith remarks~\cite[Remark 23]{Smith2016a} that
where we can show that 
\[
\sum_{\lambda \in [\sigma, \pi)}
\mobfn{\sigma}{\lambda}
\sum_{S \in EZ^{\lambda,\pi}} (-1)^{\order{S}}
= 0,
\] 
the normal embedding approach 
can determine the value of the \mob function in
polynomial time, whereas using the recursive
formula of Equation~\ref{equation_mobius_function}
is exponential complexity.

One approach to showing that
$
\sum_{\lambda \in [\sigma, \pi)}
\mobfn{\sigma}{\lambda}
\sum_{S \in EZ^{\lambda,\pi}} (-1)^{\order{S}}
= 0
$
would be to find permutations $\sigma$ and $\pi$
such that for any 
$\lambda \in [\sigma, \pi)$,
$EZ^{\lambda,\pi}$ is empty, 
as this would force the second term to be zero.
Some small-scale experiments by the author
suggest that it is significantly more likely that
some of the sets
$EZ^{\lambda,\pi}$
are non-empty, 
and therefore when
$
\sum_{\lambda \in [\sigma, \pi)}
\mobfn{\sigma}{\lambda}
\sum_{S \in EZ^{\lambda,\pi}} (-1)^{\order{S}}
= 0
$
it is likely to be because,
taken across every $\lambda \in [\sigma, \pi)$,
the number of sets in 
$EZ^{\lambda,\pi}$ with even order
is the same as the number of sets with odd order.

%
%
Brignall and Marchant~\cite{Brignall2017a}
showed that if the lower bound of an interval is indecomposable,
then the \mob function depends only on the indecomposable permutations 
contained in the upper bound. 
They then used this result to find a fast polynomial algorithm
for computing $\mobp{\pi}$ where $\pi$
is an increasing oscillation.  
This paper forms the basis of Chapter~\ref{chapter_incosc_paper} 
of this thesis.

%
%
Brignall, Jel{\'{i}}nek, Kyn{\v{c}}l and Marchant~\cite{Brignall2020}
prove that if a permutation $\pi$ contains
opposing adjacencies, then
$\mobp{\pi} = 0$.
They then use this to
show that the proportion of permutations of length $n$ with principal \mob 
function equal to zero is asymptotically bounded below by 
$(1-1/e)^2 \ge 0.3995$.
This paper forms the basis of Chapter~\ref{chapter_oppadj_paper} of this thesis.

%
%
Jel{\'{i}}nek, Kantor, Kyn{\v{c}}l and Tancer~\cite{Jelinek2020}
show how to construct a sequence of permutations
$\pi_n$
with length $2n + 2$,
and they show that for $n \geq 2$,
\[
\mobp{\pi_n}
=
-\binom{n+2}{7}
-\binom{n+1}{7}
+ 2\binom{n+2}{5}
-\binom{n+2}{3}
-\binom{n}{2}
-2n,
\]
and thus demonstrate that 
the absolute value of the \mob function
grows according to the seventh power of the length.
In their paper they also show that 
if $f:[\sigma, \pi] \to \mathbb{R}$ is any
function satisfying $f(\pi) = 1$, then
\[
\mobfn{\sigma}{\pi}
=
f(\sigma) 
- 
\sum_{\lambda \in [\sigma, \pi)}
\mobfn{\sigma}{\lambda}
\sum_{\tau \in [\lambda, \pi]}
f(\tau)
\]
and as a corollary, they then show that
\[
\mobfn{\sigma}{\pi}
=
(-1)^{\order{\pi} - \order{\sigma}} \E(\sigma, \pi)
- 
\sum_{\lambda \in [\sigma, \pi)}
\mobfn{\sigma}{\lambda}
\sum_{\tau \in [\lambda, \pi]}
(-1)^{\order{\pi} - \order{\tau}} \E(\tau, \pi),
\]
where $\E(\alpha, \beta)$ is the number of
embeddings of $\alpha$ into $\beta$.
We note that the formula in the corollary
has a similar structure to
Equation~\ref{equation_smith_all_intervals}
described above,
although in general it suffers from the same restrictions
as Smith's equation.

%
%
Finally, Marchant~\cite{Marchant2020} showed how to
construct a sequence of permutations
$\pi_1$, $\pi_2$, $\pi_3, \ldots$
with lengths $n, n+4, n+8, \ldots$     
such that $\mobfn{1}{\pi_{i+1}} = 2 \mobfn{1}{\pi_{i}}$,
and this gives us that the
growth of the 
principal \mob 
function on the permutation poset
is exponential.  
This paper forms the basis of Chapter~\ref{chapter_2413_balloon_paper} 
of this thesis.

\section{Motivation}

%
%
It seems reasonably clear that the \mob function
of the permutation poset under classic pattern containment
is a non-trivial problem.  This contrasts with 
some of the posets described in~\ref{section_background_mobius_various_posets},
where the \mob function is completely determined.

%
%
As we have described, the study of the 
permutation poset under classic pattern containment
was initiated by Wilf in 2002 in~\cite{Wilf2002}.
Anecdotally, Wilf is believed to have later said
that the \mob function on the permutations pattern poset
was \textit{``A mess. Don't touch it''}.

We state here that we think that
Wilf's reported view is somewhat pessimistic.
While we think that it is unlikely that
there is a polynomial-time procedure
for computing the value of the \mob
on an arbitrary interval of the permutation poset,
we believe, and we hope to show in this
thesis, that there is considerable scope
for further research in this area.

%
%
The permutation pattern poset
is the subject of considerable research activity
outside of the \mob function,
and we claim that the permutation 
pattern poset is the underlying
object for many studies 
related to patterns in permutations.
This then means that 
research into the 
\mob function on the permutation pattern poset
may lead to a better understanding of this
poset, and hence to results in
other areas related to permutation patterns.

%
%
We can summarise our motivation for research in this area 
by saying that
the \mob function of the permutation poset under
classic pattern containment
is not well-understood, and indeed
up until recently
the proportion of 
permutations where we had a (computationally) simple
way to determine the value of the principal \mob function
was, asymptotically, zero.  

The paper which forms the basis of Chapter~\ref{chapter_oppadj_paper}
shows that, asymptotically, 
the proportion of permutations
where the principal \mob function is zero is at least \zpmfr.
The corollary to this result, however, is that we do not yet
have an effective means to compute the 
principal \mob function for 60\% of all permutations.

Further, research into the 
\mob function on the permutation pattern poset
may lead to a better understanding of the 
intrinsic properties of the poset,
which in turn may lead to results
in related areas.

    \chapter{The \mob function of permutations with an indecomposable lower bound}
\chaptermark{Permutations with an indecomposable lower bound}
\label{chapter_incosc_paper}

\section{Preamble}

This chapter
is based on a published paper~\cite{Brignall2017a},
which is joint work with 
Robert Brignall.

In this chapter, we
show that the \mob function
of an interval in a permutation poset
where the lower bound is sum (resp. skew) indecomposable depends solely
on the sum (resp. skew) indecomposable permutations contained
in the upper bound,
and that this can simplify the calculation
of the \mob sum.
For increasing oscillations,
we give a recursion for the \mob sum
which only involves evaluating
simple inequalities.

\section{Introduction}

Recall that the \mob function on a poset interval $[\sigma, \pi]$
is defined by
\begin{align}
\label{incosc_equation_mobius_function}
\mobfn{\sigma}{\pi}
& =
\begin{cases}
1 & \text{if $\sigma = \pi$,} \\
- \sum_{\lambda \in [\sigma, \pi)} \mobfn{\sigma}{\lambda} &\text{otherwise.}
\end{cases}
\end{align}

%
%
Our motivation for this 
chapter 
is to find 
a 
\emph{contributing set}\extindex{contributing set} 
$\contrib{\sigma}{\pi}$
that is significantly smaller than
the poset interval $[\sigma, \pi)$,
and a $\{0, \pm1 \}$ weighting function $\weightgen{\sigma}{\alpha}{\pi}$
such that
\begin{align}
\label{equation_contributing_sum}
\mobfn{\sigma}{\pi} 
& = 
- \sum_{\alpha \in \contrib{\sigma}{\pi}}
\mobfn{\sigma}{\alpha} \weightgen{\sigma}{\alpha}{\pi}.
\end{align}
Plainly, in Equation~\ref{equation_contributing_sum},
we could set 
$\contrib{\sigma}{\pi} = [\sigma, \pi)$,
and $\weightgen{\sigma}{\alpha}{\pi} = 1$,
which is equivalent to
Equation~\ref{incosc_equation_mobius_function}.

One approach here would be to
take a permutation $\beta$ 
such that
$\sigma < \beta < \pi$.
We could then set 
$\contrib{\sigma}{\pi} 
= 
\{
\lambda : 
\lambda \in [\sigma, \pi) 
\text{ and } 
\lambda \not\in [\sigma, \beta]
\}$,
and $\weightgen{\sigma}{\alpha}{\pi} = 1$,
since, from Equation~\ref{incosc_equation_mobius_function},
$\sum_{\lambda \in [\sigma, \beta]} \mobfn{\sigma}{\lambda} = 0$.
This approach was used in 
Smith~\cite{Smith2013},
who determined the \mob function on the interval
$[1, \pi]$ for all permutations $\pi$ with a single descent.
Smith's paper is unusual, in that it provides an explicit formula
for the value of the \mob function.

Our approach is different.  
We identify individual elements (say $\lambda$), of the poset
that have $\mobfn{\sigma}{\lambda} = 0$.  
We also
show that there are pairs of elements, 
$\lambda$ and $\lambda^\prime$,
where $\mobfn{\sigma}{\lambda} = - \mobfn{\sigma}{\lambda^\prime}$,
and so we can exclude these pairs of elements.
Finally, we show that there are
quartets of permutations $\lambda_1, \ldots, \lambda_4$
where $\sum_{i=1}^{4} \mobfn{\sigma}{\lambda_i} = 0$;
and that we can systematically identify these quartets.
By excluding these permutations from
$\contrib{\sigma}{\pi}$
we can significantly reduce the number of elements
in $\contrib{\sigma}{\pi}$ 
compared to the number of elements in the
interval $[\sigma, \pi)$. 
This approach results in
the ability to compute 
$\mobfn{\sigma}{\pi}$, 
where
$\sigma$ is indecomposable,
much faster than evaluating 
Equation~\ref{incosc_equation_mobius_function}.
For increasing oscillations, we will show
that the elements of $\contrib{\sigma}{\pi}$
can be determined using simple inequalities,
and that as a consequence
$\mobfn{\sigma}{\pi}$ can be determined
using inequalities.
With this approach, we 
have computed $\mobfn{1}{\pi}$,
where $\pi$ is an increasing oscillation,
up to $\order{\pi} = \text{2,000,000}$. 

Our main tool in the first part of
this chapter comes from the results of
Burstein, Jel{\'{i}}nek, Jel{\'{i}}nkov{\'{a}} 
and Steingr{\'{i}}msson~\cite{Burstein2011}.
They found
a recursion for the \mob function
for sum/skew decomposable permutations
in terms of the sum/skew indecomposable 
permutations in the lower and upper bounds.
They also found a method to determine
the \mob function for separable permutations
by counting embeddings.
We use the recursions for 
decomposable permutations to underpin the first
part of this chapter.

In this chapter  
we
show that the \mob function
on intervals with a sum indecomposable lower bound
depends only on the sum indecomposable permutations
contained in the upper bound.
We provide a weighting function that
determines which sum indecomposable permutations
contribute to the \mob sum.
We then consider increasing oscillations.
For these permutations, we show how we can
find all of the permutations that contribute to the \mob sum
by applying simple numeric inequalities,
which leads to a fast polynomial algorithm
for determining the \mob function.

We start with some essential definitions and notation
relevant to this chapter
in Section~\ref{incosc_section-definitions-and-notation},
then in Section~\ref{section-preliminary-lemmas} 
we provide a number of preliminary lemmas.
We conclude this section with a theorem
that gives $\mobfn{\sigma}{\pi}$,
where $\sigma$
is a sum indecomposable permutation,
for all $\pi$.
In Section~\ref{section-increasing-oscillations}
we consider $\mobfn{\sigma}{\pi}$
where $\sigma$
is a sum indecomposable permutation,
and $\pi$ is an increasing oscillation.
We finish with some concluding remarks in 
Section~\ref{incosc_section-concluding-remarks}.

\section{Definitions and notation}\label{incosc_section-definitions-and-notation}

When discussing the \mob function, 
$\mobfn{\sigma}{\pi} = - \sum_{\lambda \in [\sigma, \pi)} \mobfn{\sigma}{\lambda}$,
we will frequently be examining the value of
$\mobfn{\sigma}{\lambda}$ for a specific permutation 
$\lambda$.  We say that this is the 
\emph{contribution}\extindex{contribution} 
that 
$\lambda$ makes to the sum.  If we have a set of permutations
$S \subseteq [\sigma, \pi)$
such that
$\sum_{\lambda \in S} \mobfn{\sigma}{\lambda} = 0$, 
then we say that the set $S$ makes
\emph{no net contribution}\extindex{net contribution}
to the sum.

The 
\emph{interleave}\extindex[permutation]{interleave} 
of two permutations
$\alpha$ and $\beta$ is formed by taking
the sum $\alpha \oplus \beta$,
and then exchanging the value of the largest point from
$\alpha$ with the value of the smallest point from $\beta$.
We can also view this as increasing the largest
point from $\alpha$ by 1, and simultaneously
decreasing the smallest point from $\beta$ by 1.
We write an interleave as $\alpha \interleave \beta$.
For example, $321 \interleave 213 = 421536$,
see Figure~\ref{figure_examples-of-sums-and-interleaves}.

For completeness, we also define
a 
\emph{skew interleave}\extindex[permutation]{skew interleave},
$\alpha \skewinterleave \beta$,
which is formed 
by taking the skew sum $\alpha \ominus \beta$,
and then exchanging the smallest point from
$\alpha$ with the largest point from $\beta$.
As an example,
$321 \skewinterleave 213 = 653214$,
as shown in 
Figure~\ref{figure_examples-of-sums-and-interleaves}.

The interleave operations, 
$\interleave$ and $\skewinterleave$,
are not associative, as
$\oneil \oneil 1$ could represent
$231$ or $312$.  To avoid this ambiguity,
we require that the permutation $1$
can either be interleaved to the left
or to the right, but not both.  
It is easy to see that this restriction
establishes associativity.
We note here that, with this restriction,
an expression involving $\oplus$ and $\interleave$
represents a unique permutation
regardless of the order in which the operations
are applied.

Let $\alpha$ be a permutation with length greater than 1.  
We will frequently want to refer to 
permutations that 
have the form
$\alpha \interleave \alpha \interleave 
\ldots
\interleave \alpha \interleave \alpha$.
If there are $n$ copies of $\alpha$ being
interleaved, then we will write this as
$\nils{n}{\alpha}$\extindex{$\nils{n}{\alpha}$}, 
so, for example, 
we have $\nils{3}{(21)} = 21 \interleave 21 \interleave 21 = 315264$.

\begin{figure}
    \begin{footnotesize}
        \begin{center}
            \begin{subfigure}[t]{0.22\textwidth}
                \begin{center}
                    \begin{tikzpicture}[scale=0.3]   
                    \plotpermgrid{3,2,1,5,4,6}
                    \end{tikzpicture}
                \end{center}
                \caption*{$321 \oplus 213$}
            \end{subfigure}
            \begin{subfigure}[t]{0.22\textwidth}
                \begin{center}
                    \begin{tikzpicture}[scale=0.3]                   
                    \plotpermgrid{6,5,4,2,1,3}
                    \end{tikzpicture}
                \end{center}
                \caption*{$321 \ominus 213$}
            \end{subfigure}
            \begin{subfigure}[t]{0.22\textwidth}
                \begin{center}
                    \begin{tikzpicture}[scale=0.3]                    
                    \plotpermgrid{4,2,1,5,3,6}
                    \end{tikzpicture}
                \end{center}
                \caption*{$321 \interleave 213$}
            \end{subfigure}
            \begin{subfigure}[t]{0.22\textwidth}
                \begin{center}
                    \begin{tikzpicture}[scale=0.3]                    
                    \plotpermgrid{6,5,3,2,1,4}
                    \end{tikzpicture}
                \end{center}
                \caption*{$321 \skewinterleave 213$}
            \end{subfigure}
        \end{center}
    \end{footnotesize}
    \caption{Examples of direct and skew sums and interleaves.} 
    \label{figure_examples-of-sums-and-interleaves}	
\end{figure}
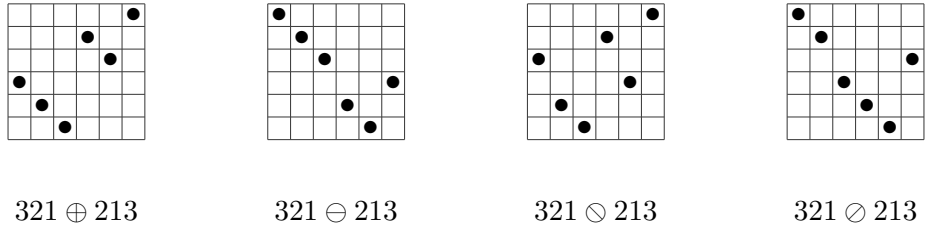    

For the remainder of this chapter, 
by symmetry it suffices to 
discuss permutations in relation to 
sums and interleaves only.
For the same reason, 
references to (in)decomposable permutations 
may omit the ``sum'' qualifier.

The 
\emph{increasing oscillating sequence}\extindex{increasing oscillating sequence} 
is the sequence
\[
4, 1, 6, 3, 8, 5, 10, 7, \ldots, 2k+2, 2k-1, \ldots.
\]
The start of the sequence is depicted in 
Figure~\ref{figure_increasing_oscillating_sequence}.  
\begin{figure}
    \centering
    \begin{subfigure}[c]{0.4\textwidth}
        \centering
        \begin{tikzpicture}[scale=0.3]
        \plotopengrid{12}{14}
        \plotperm{4, 1, 6, 3, 8, 5}
        \smalldot{(7,10)}  
        \smalldot{(8,7)}
        \smallerdot{(9, 12)}
        \smallerdot{(10,  9)}
        \evensmallerdot{(11, 14)}
        \evensmallerdot{(12, 11)}
        \tinydot{(13, 16)}
        \tinydot{(14, 13)}
        \end{tikzpicture}
    \end{subfigure}
    \caption{A depiction of the start of the increasing oscillating sequence.}
    \label{figure_increasing_oscillating_sequence}
\end{figure}
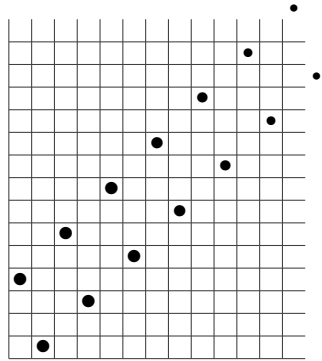  
An 
\emph{increasing oscillation}\extindex[permutation]{increasing oscillation}
is a simple permutation contained in
the increasing oscillating sequence.
For lengths greater than three, 
there are exactly two increasing oscillations
of each length.  
Let $W_n$ be the increasing oscillation 
with $n$ elements which starts with a descent,
and
let $M_n$ be the increasing oscillation
with $n$ elements which starts with an ascent.  Then
\begin{align*}
W_{2n}   & = \nils{n}{\dtwo},  &
M_{2n}   & = \oneil \left( \nils{n-1}{\dtwo} \right) \ilone,  \\
W_{2n-1} & = \left( \nils{n-1}{\dtwo} \right) \ilone, \qquad \text{and} &
M_{2n-1} & = \oneil \left( \nils{n-1}{\dtwo} \right). \\
\end{align*}
Note that $W_n = M_n^{-1}$.

There are instances where,
for some permutation $\alpha$, 
we are interested in the set of
permutations 
$
\{\alpha, \;
\oneplus \alpha, \;
\alpha \plusone, \;
\oneplus \alpha \plusone
\}$.
Given a permutation $\alpha$, we
refer to this set 
as $\familysum{\alpha}$,
and we 
say that this set 
is the 
\emph{family}\extindex{family} 
of $\alpha$.
If $S$ is a set of permutations, then
$\familysum{S}= \cup_{\alpha \in S} \{\familysum{\alpha} \}$.

There are also some instances where we 
are interested in the set of
permutations 
$
\familyil{\alpha} = 
\{\alpha, \;
\oneil \alpha, \;
\alpha \ilone, \;
\oneil \alpha \ilone
\}$.
Note that every increasing oscillation
is an element of 
$
\familyil{\ildtwok}
$
for some $k \geq 1$.

\section{Preliminary lemmas and main theorem}\label{section-preliminary-lemmas}

In this section our aim is to show that 
if $\sigma$ is indecomposable, then for any $\pi \geq \sigma$
there is a $\{0, \pm 1\}$ weighting function
$\weightgen{\sigma}{\alpha}{\pi}$ 
and a set of permutations
$\contrib{\sigma}{\pi}$, such that
\[
\mobfn{\sigma}{\pi} 
= 
- \sum_{\alpha \in \contrib{\sigma}{\pi}}
\mobfn{\sigma}{\alpha} \weightgen{\sigma}{\alpha}{\pi}.
\]

If $\pi$ is 
the identity permutation
$1,2, \ldots, n$ or its reverse,
then $\mobfn{\sigma}{\pi}$ is trivial for any $\sigma$, and we
exclude the identity and its reverse from
being the upper bound of any interval under consideration.

As noted earlier, our approach is to show that 
there are permutations, pairs of permutations, and
quartets of permutations in 
$[\sigma, \pi)$ 
that make no net contribution to the sum.

We use Proposition 1 and 2, and Corollary 3 
from
Burstein, 
Jel{\'{i}}nek, 
Jel{\'{i}}nkov{\'{a}} and 
Steingr{\'{i}}msson~\cite{Burstein2011}.
Note that we have already introduced these 
propositions on page~\pageref{BJJS-proposition-1-as-in-mcnamara},
but we repeat them here for ease of use.
We start with some required notation.
If $\pi$ is a non-empty permutation with decomposition
$\pi_\oneplus \ldots \oplus \pi_n$, then 
for any integer $i$ with $0 \leq i \leq n$, 
$\pi_{\leq i}$ is the permutation 
$\pi_\oneplus \ldots \oplus \pi_i$,
and 
$\pi_{> i}$ is the permutation
$\pi_{i+1} \oplus \ldots \oplus \pi_n$.
An empty sum of permutations is defined as $\varepsilon$, 
and in particular $\pi_{\leq 0} = \pi_{> n} = \varepsilon$.
We can see that 
$\mobfn{\varepsilon}{\varepsilon} = 1$,
$\mobfn{\varepsilon}{1} = -1$ and
$\mobfn{\varepsilon}{\tau} = 0$ for any $\tau > 1$.
We now recall the results from 
Burstein, 
Jel{\'{i}}nek, 
Jel{\'{i}}nkov{\'{a}} and 
Steingr{\'{i}}msson:
\begin{proposition}[{%
    Burstein, 
    Jel{\'{i}}nek, 
    Jel{\'{i}}nkov{\'{a}} and 
    Steingr{\'{i}}msson 
    \cite[Proposition 1]{Burstein2011}}]
    \label{BJJS_proposition_1}
    Let $\sigma$ and $\pi$ be non-empty permutations
    with decompositions 
    $\sigma = \sigma_\oneplus \ldots \oplus \sigma_m$
    and
    $\pi = \pi_\oneplus \ldots \oplus \pi_n$,
    with $n \geq 2$.
    Assume that $\pi_1 = 1$, 
    and let $k$ be the largest integer such that
    $\pi_1, \pi_2, \ldots ,\pi_k$ are all equal to $1$.
    Let $l \geq 0$ be the largest integer such that
    $\sigma_1, \sigma_2, \ldots, \sigma_l$ are all equal to $1$.
    Then
    \begin{align*}
    \mobfn{\sigma}{\pi} & =
    \begin{cases*}
    0 & \text{if $k-1 > l$,} \\
    -\mobfn{\sigma_{> k-1}}{\pi_{> k}} & \text{if $k-1 = l$,} \\
    \mobfn{\sigma_{> k}}{\pi_{> k}} - \mobfn{\sigma_{> k-1}}{\pi_{> k}} & \text{if $k-1 < l$.}
    \end{cases*}
    \end{align*}
\end{proposition}
\begin{proposition}[{%
    \cite[Proposition 2]{Burstein2011}}]
    \label{BJJS_proposition_2}
    Let $\sigma$ and $\pi$ be non-empty permutations
    with decompositions 
    $\sigma = \sigma_1 \oplus \ldots \oplus \sigma_m$
    and
    $\pi = \pi_1 \oplus \ldots \oplus \pi_n$,
    with $n \geq 2$.
    Assume that $\pi_1 \neq 1$, 
    and let $k$ be the largest integer such that
    $\pi_1, \pi_2, \ldots ,\pi_k$ are all equal to $\pi_1$.
    Then
    \begin{align*}
    \mobfn{\sigma}{\pi} & =
    \sum_{i=1}^m
    \sum_{j=1}^k
    \mobfn{\sigma_{\leq i}}{\pi_{1}}
    \mobfn{\sigma_{> i}}{\pi_{> j}}.
    \end{align*}    
\end{proposition}
\begin{corollary}[{%
    \cite[Corollary 3]{Burstein2011}}]
    \label{BJJS_corollary_3}
    Let $\sigma$ and $\pi$ be as in Proposition~\ref{BJJS_proposition_2}.
    Suppose that $\sigma$ is sum indecomposable, so $m = 1$.
    Then
    \begin{align*}
    \mobfn{\sigma}{\pi} & =
    \begin{cases*}
    \mobfn{\sigma}{\pi_1} & \text{if $\pi = \nsums{k}{\pi_1$},} \\
    -\mobfn{\sigma}{\pi_1} & \text{if $\pi = \left(\nsums{k}{\pi_1} \right) \plusone$,} \\
    0 & \text{otherwise,}
    \end{cases*}
    \end{align*}
\end{corollary}
A simple consequence of 
Propositions~\ref{BJJS_proposition_1}
and~\ref{BJJS_proposition_2}
is the identification
of some intervals of permutations where 
the value of the \mob function is zero.
\begin{lemma}
    \label{lemma_mobius_function_is_zero}
    Let $\pi \in 
    \{
    \oneplus \oneplus \tau,
    \tau \plusone \plusone,
    \familysum{\sumrab \oplus \tau^\prime}
    \}
    $,
    where $\tau$ is any permutation,
    $r$ is maximal,
    $\alpha$ is sum indecomposable,
    and
    $\tau^\prime$ is any permutation
    greater than $1$.
    Let $\sigma$ be a sum indecomposable permutation.
    Then
    $
    \mobfn{\sigma}{\pi}  = 0
    $.
\end{lemma}
\begin{proof}
    Consider $\pi = \oneplus \oneplus \tau$.
    We use Proposition~\ref{BJJS_proposition_1}.  
    If $\tau_1 = 1$,
    then $k \geq 3$, and $l \leq 1$, 
    and the result follows immediately.
    Now assume that $\tau_1 \neq 1$.
    Then $k=2$.  
    If $\sigma > 1$, then again the
    result follows immediately.
    If $\sigma=1$, then we have
    $\mobfn{\sigma}{\pi} 
    = - \mobfn{\sigma_{> k-1}}{\pi_{>k}} 
    = - \mobfn{\varepsilon}{\tau} 
    = 0$.
    The case for $\pi = \tau \plusone \plusone$
    follows by symmetry.
    
    Now consider 
    $\pi = \familysum{\sumrab \oplus \tau^\prime}$.    
    If $\pi = \sumrab \oplus \tau$,
    or $\pi = \sumrab \oplus \tau \plusone$,
    then
    we use Proposition~\ref{BJJS_proposition_2}.
    In that context we have $m = 1$ and $k = r$,
    and so
    $
    \mobfn{\sigma}{\pi}
    =
    \sum_{j=1}^r
    \mobfn{\sigma}{\pi_1}
    \mobfn{\varepsilon}{\pi_{> j}} 
    $
    For every value of $j$, $\pi_{> j}$ 
    is non-empty and greater than $1$, and so
    $\mobfn{\varepsilon}{\pi_{> j}} = 0$ for all $j$,
    and hence every term in the sum is zero.     
    If $\pi = \oneplus \sumrab \oplus \tau$
    or
    $\pi = \oneplus \sumrab \oplus \tau \plusone$,
    then
    we use Proposition~\ref{BJJS_proposition_1},
    which reduces to one of the previous cases.    
\end{proof}
We now turn to identifying pairs and quartets of 
permutations
that make no net contribution to the \mob sum.
We start by showing that 
if $\sigma$ and $\alpha$ are indecomposable,
and $r \geq 1$, and
with $\pi \in \familysum{\sumra}$,
then
$\mobfn{\sigma}{\pi}$ and $\mobfn{\sigma}{\alpha}$
have the same magnitude.
\begin{lemma}
    \label{lemma_pi_has_single_block}
    Let $\pi \in \familysum{\sumra}$,
    where $r \geq 1$ and $\alpha > 1$ is sum indecomposable.
    Let $\sigma$ be a sum indecomposable permutation. 
    Then 
    \[
    \mobfn{\sigma}{\pi} = 
    \begin{cases*}
    \mobfn{\sigma}{\alpha} & 
    \text{if $\pi = \sumra$\; or\; $\oneplus \sumrab \plusone $},
    \\
    -\mobfn{\sigma}{\alpha} &
    \text{if $\pi = \oneplus \sumrab $\; or\; $\sumrab \plusone$}.
    \end{cases*}
    \]
    
    As a consequence, 
    if $\familysum{\sumra} \subseteq [\sigma,\pi)$,
    then $\familysum{\sumra}$ makes no net contribution to 
    $\mobfn{\sigma}{\pi}$.
\end{lemma}
\begin{proof}
    If $\pi = \sumra$
    or $\pi = \sumrab \plusone$,
    then this is immediate from Corollary~\ref{BJJS_corollary_3}.
    If $\pi = \oneplus \sumrab$
    or $\pi = \oneplus \sumrab \plusone$,    
    then we use Proposition~\ref{BJJS_proposition_1}.
    
    For the net contribution of $\familysum{\sumra}$,
    $\sum_{\lambda \in \familysum{\sumra}} \mobfn{\sigma}{\lambda} = 0$.
\end{proof}
We now have a lemma that adds a further restriction to
the permutations that have a non-zero
contribution to the \mob sum.
\begin{lemma}
    \label{lemma_only_r_and_rplusone_count}
    If $\sigma \leq \pi$, 
    and $\alpha \in [\sigma, \pi]$ is sum indecomposable,
    and $r$ is the smallest integer such that 
    $\oneplus \sumrab \plusone \not\leq \pi$, then
    $\familysum{\nsums{k}{\alpha}} \subseteq [\sigma, \pi)$
    for all $k \in [1, r)$.
\end{lemma}
\begin{proof}
        For any $k < r$, 
        $\sigma \leq 
        \nsums{k}{\alpha} < 
        \oneplus \left(\nsums{k}{\alpha}\right) \plusone \leq
        \pi$.    
        Note that by Lemma~\ref{lemma_pi_has_single_block}
        the net contribution of the family 
        $\familysum{\nsums{k}{\alpha}}$ to $\mobfn{\sigma}{\pi}$ is zero.
\end{proof}
\begin{observation}
    \label{observation_only_r_and_rplusone_count}
    Using the same terminology as 
    Lemma~\ref{lemma_only_r_and_rplusone_count}, if $k > r+1$
    then we must have $\nsums{k}{\alpha} \not\leq \pi$.
    As a consequence, for each indecomposable 
    $\alpha \in [\sigma, \pi]$, the only families of $\alpha$
    that can have a non-zero net contribution
    to $\mobfn{\sigma}{\pi}$ are
    $\familysum{\sumra}$ 
    and $\familysum{\nsums{r+1}{\alpha}}$.
\end{observation}
We now eliminate two specific permutations from the \mob sum.
\begin{lemma}
    \label{lemma_order_greater_than_three_ignore_o_opo}
    If $\pi$ is any permutation with $\order{\pi} > 3$
    apart from the identity permutation and its reverse,
    and $\sigma$ is sum indecomposable,
    then the permutations $1$ and $\oneplusone$ make no net contribution
    to the \mob sum $\mobfn{\sigma}{\pi}$.
\end{lemma}
\begin{proof}
    If $\sigma = 1$, then the interval contains
    both $1$ and $\oneplusone$.  Since 
    $\mobfn{1}{1} = 1$ and $\mobfn{1}{12} = -1$,
    there is no net contribution to $\mobfn{\sigma}{\pi}$.
    If $\sigma > 1$, then $\sigma \neq 12$,
    and so neither $1$ nor $12$ is in the interval.
\end{proof}

Before we present the main theorem for this section,
we formally define the weight function
and the contributing set.
Let $\alpha$ be a sum indecomposable permutation.
The weight function, $\weightgen{\sigma}{\alpha}{\pi}$,
is defined as 
\begin{align}
\label{equation_general_weight_function}
\weightgen{\sigma}{\alpha}{\pi}
& =
\begin{cases*}
1 & 
If
$
\left\lbrace
\begin{array}{l}
\sigma \leq \sumra \leq \pi
\text{ and } \\
\oneplus \sumrab \not \leq \pi
\text{ and } \\
\sumrab \plusone \not \leq \pi,
\end{array}
\right.
$
\\
-1 & 
If
$
\left\lbrace
\begin{array}{l}
\sigma \leq \sumra \leq \pi
\text{ and } \\
\oneplus \sumrab \leq \pi
\text{ and } \\
\sumrab \plusone \leq \pi 
\text{ and } \\
\nsums{r+1}{\alpha} \not\leq \pi,
\end{array}
\right.
$
\\
0 &
Otherwise,
\end{cases*}
\end{align}
where $r$ is the smallest integer 
such that $\oneplus \sumrab \plusone \not \leq \pi$.

The contributing set $\contrib{\sigma}{\pi}$ 
is defined as
\begin{align*}
\contrib{\sigma}{\pi}
& = 
\left\lbrace
\alpha :
\begin{array}{l}
\alpha \in [\sigma, \pi), \\
\alpha \text{ is sum indecomposable, and } \\
\weightgen{\sigma}{\alpha}{\pi} \neq 0
\end{array}
\right\rbrace.
\end{align*}

We have one last lemma before we move on
to the main theorem.
\begin{lemma}
    \label{lemma_weight_of_alpha}
    If $\sigma$ and $\alpha$ are sum indecomposable,
    then for any permutation $\pi$,
    $\mobfn{\sigma}{\alpha}\weightgen{\sigma}{\alpha}{\pi}$
    gives the contribution of the set of 
    families
    $\familysum{\sumra}$ to the \mob sum,
    where $r$ is any positive integer.
\end{lemma}
\begin{proof}
    By Observation~\ref{observation_only_r_and_rplusone_count},
    we only need consider the contribution made
    by $\sumra$ and $\nsums{r+1}{\alpha}$,
    where $r$ is the smallest integer 
    such that $\oneplus \sumrab \plusone \not \leq \pi$.
    
    If $\sigma \not \leq \sumra$, 
    or $\sumra \not \leq \pi$, then
    $\familysum{\sumra}$ makes no net contribution 
    to the \mob sum.  Now assume that
    $\sigma \leq \sumra \leq \pi$.
    First, we can see that
    if $\oneplus \sumrab \not \leq \pi$,
    or $\sumrab \plusone \not \leq \pi$
    then $\nsums{r+1}{\alpha} \not \leq \pi$.
    We can also see that 
    if $\oneplus \sumrab \plusone \not \leq \pi$
    then $\oneplus \left(\nsums{r+1}{\alpha}\right) \not \leq \pi$
    and $\nsums{r+1}{\alpha} \not \leq \pi$.
    The possibilities remaining are itemised
    in Table~\ref{table_family_members_in_an_interval},
    \begin{table}[btp]
        \centering
        \begin{tabular}{cccc}
            \toprule
            $\oneplus \sumrab$ &
            $\sumrab \plusone$ &
            $\nsums{r+1}{\alpha}$ &
            \mob contribution \\    
            \midrule
            $\leq \pi$ &
            $\leq \pi$ &
            $\leq \pi$ &
            $0$
            \\    
            $\leq \pi$ &
            $\leq \pi$ &
            $\not \leq \pi$ &
            $- \mobfn{\sigma}{\alpha}$
            \\    
            $\leq \pi$ &
            $\not \leq \pi$ &
            $\not \leq \pi$ &
            $0$
            \\    
            $\not \leq \pi$ &
            $\leq \pi$ &
            $\not \leq \pi$ &
            $0$
            \\    
            $\not \leq \pi$ &
            $\not \leq \pi$ &
            $\not \leq \pi$ &
            $\mobfn{\sigma}{\alpha}$
            \\    
            \bottomrule
        \end{tabular}    
        \caption{\mob contribution from family members.}
        \label{table_family_members_in_an_interval}
    \end{table}      
    where the \mob contribution is determined
    by applying 
    Lemma~\ref{lemma_pi_has_single_block}.    
    We can see that in every case 
    $\weightgen{\sigma}{\alpha}{\pi}$ provides the
    correct weight for the \mob function
    $\mobfn{\sigma}{\alpha}$.
\end{proof}

We are now in a position 
to present the main theorem for this section.
\begin{theorem}
\label{theorem_mobius_sum_bottom_level_indecomposable}
    If $\sigma$ is a sum indecomposable permutation,
    and $\order{\pi} > 3$,
    then
    \[
    \mobfn{\sigma}{\pi} 
    = 
    - \sum_{\alpha \in \contrib{\sigma}{\pi}}
    \mobfn{\sigma}{\alpha} \weightgen{\sigma}{\alpha}{\pi} .
    \] 
\end{theorem}
\begin{proof}
    Let $\alpha \leq \pi$ be an indecomposable permutation.
    
    Using Lemmas~\ref{lemma_mobius_function_is_zero}
    and~\ref{lemma_order_greater_than_three_ignore_o_opo}
    we can see that any permutations not in the set
    $\familysum{\sumra}$
    can be excluded from $\contrib{\sigma}{\pi}$, as these permutations 
    make no net contribution to the \mob sum.
    
    For every $\alpha$, 
    by Lemma~\ref{lemma_weight_of_alpha},
    $\mobfn{\sigma}{\alpha} \weightgen{\sigma}{\alpha}{\pi}$
    provides the contribution to the \mob sum
    of all families $\familysum{\sumra}$, where
    $r$ is a positive integer.  
\end{proof}
Theorem~\ref{theorem_mobius_sum_bottom_level_indecomposable}
reduces the number of permutations that need to be considered
as part of the \mob sum.  
We can see that the largest permutation in 
$\contrib{\sigma}{\pi}$ must have length less than $\order{\pi}$,
and so we can apply 
Theorem~\ref{theorem_mobius_sum_bottom_level_indecomposable}
recursively to the permutations in $\contrib{\sigma}{\pi}$
to determine their \mob values.
In this recursion, if we are attempting to determine
$\mobfn{\sigma}{\lambda}$, we can stop
if $\order{\sigma} = \order{\lambda}$
or $\order{\sigma} = \order{\lambda} - 1$, as
in these cases $\mobfn{\sigma}{\lambda}$ is $+1$ and $-1$ respectively.

\section{Increasing oscillations}\label{section-increasing-oscillations}

We now move on to increasing oscillations.
Given an indecomposable permutation
$\sigma$, 
and an increasing oscillation $\pi$,
our aim in this section is to describe
$\contrib{\sigma}{\pi}$
in precise terms.  
We will find a sum for the \mob function,
$\mobfn{\sigma}{\pi}$,
which only requires the evaluation of simple inequalities.

If $\pi$ is an increasing oscillation with 
length less than 4, then
$\mobfn{\sigma}{\pi}$ is trivial to determine
for any $\sigma$.
For the remainder of this section we assume that 
$\pi$ has length at least 4.

We partition the set of increasing oscillations
with length greater than 1 
into 
five disjoint subsets.
These subsets are
$\{ \dtwo \},\;$
$\{\ildtwokp \},\;$ 
$\{\oneil \ildtwokb\},\;$ 
$\{\ildtwokb \ilone\},\;$
and
$\{\oneil \ildtwokb \ilone \}$,
where $k$ is a positive integer.
If two increasing oscillations are in the 
same subset, then we
say that they have the same 
\emph{shape}\extindex{shape}.

We now determine what 
permutations contained in an increasing oscillation
have a non-zero contribution to the \mob sum.
\begin{lemma}
    \label{lemma_form_of_permutations_in_increasing_oscillations}
    Let $\pi$ be an increasing oscillation,
    and let $\sigma \leq \pi$ be sum indecomposable.
    Let $S$ be the subset of the permutations
    in the interval $[\sigma, \pi)$ that can be 
    written in the form
    $
    \familysum{ \nsums{r}{\familyil{\nils{k}{\dtwo}} } }
    $
    for some $k,r \geq 1$.
    If $\lambda \in [\sigma, \pi)$,
    and $\lambda \not \in S$, then
    $\mobfn{\sigma}{\lambda} = 0$.
\end{lemma}
We note here that 
$\familyil{\nils{k}{\dtwo}}$
is a set containing only increasing oscillations. 

\begin{proof}
    We start by showing that if $\pi$ is an increasing oscillation,
    and
    $\lambda = \lambda_1 \oplus \ldots \oplus \lambda_m \leq \pi$,
    where each $\lambda_i$ is sum indecomposable,
    then every $\lambda_i$
    is an increasing oscillation.
    This is trivially true if $\lambda$
    is itself an increasing oscillation,
    thus it is 
    sufficient to show
    that if  
    $\lambda$ is an increasing oscillation, 
    then
    deleting a single point results in either 
    an increasing oscillation, or
    a permutation that is the sum of
    two increasing oscillations.
    
    If $k = 1$, 
    then we can see that deleting a single point
    results in a permutation with the required 
    characteristic.
    
    Now assume that $k > 1$.
    Let $\lambda = \oneil \ildtwokb$.
    Deleting the leftmost point gives $\ildtwok$,
    and deleting the rightmost point gives
    $\oneil \left( \nils{k-1}{\dtwo} \right) \oplus 1$.
    Deleting the second point
    gives
    $\dtwo \oplus \left( \nils{k-1}{\dtwo} \right)$,
    and deleting the last-but-one point
    gives
    $\oneil \left( \nils{k-1}{\dtwo} \right) \ilone$.
    Deleting any even point $2t$ 
    except the second or second-to-last
    results in 
    $
    \left(
    \oneil \left( \nils{t-1}{\dtwo} \right) \ilone 
    \right)
    \oplus
    \left(
    \left( \nils{k-t}{\dtwo} \right)
    \right)
    $.
    Finally, deleting any odd point $2t+1$
    apart from the first or last
    results in
    $
    \left(
    \oneil \left( \nils{t-1}{\dtwo} \right)
    \right)
    \oplus
    \left(
    \oneil \left( \nils{k-t}{\dtwo} \right)
    \right)
    $.
    Thus if $\lambda = \oneil \ildtwokb$,
    then deleting a single point from $\lambda$
    results in either 
    an increasing oscillation, or
    a permutation that is the sum of
    two increasing oscillations.
    
    A similar argument applies to the other three cases,
    which we omit for brevity.
       
    To complete the proof, 
    we now see that by
    Lemma~\ref{lemma_order_greater_than_three_ignore_o_opo},
    we can ignore $\lambda = 1 $ and $\lambda = \oneplusone$.    
    If 
    $\lambda = \lambda_1 \oplus \lambda_2 \oplus \ldots \oplus \lambda_m \leq \pi$,
    then by the argument above, every $\lambda_i$
    is an increasing oscillation.
    Applying Lemma~\ref{lemma_mobius_function_is_zero} 
    completes the proof.
\end{proof}
Following Observation~\ref{observation_only_r_and_rplusone_count},
it is clear that,
if $\alpha \in \familyil{\ildtwok}$,
then for any family $\familysum{\sumra}$,
we only need consider the cases 
$\sumra$ and $\nsums{r+1}{\alpha}$
where
$r$ is the smallest integer such that
$\oneplus \sumrab \plusone \not \leq \pi$.

Given some $\pi = W_n \text{ or } M_n$, we will 
find inequalities that relate $n$, $r$ and $k$ and the
shape of $\alpha$ that will allow us to
find the values that contribute
to the \mob sum.
We know from
Lemma~\ref{lemma_form_of_permutations_in_increasing_oscillations}
the shape of the permutations that contribute to the \mob sum.
For each of the four types of increasing oscillation
($W_{2n}$, $W_{2n-1}$, $M_{2n}$ and $M_{2n-1}$),
we can examine how each shape can be embedded 
so that the unused points at the start of the increasing
oscillation are minimised.  
Figure~\ref{figure_unused_points_at_start_w2n}
shows examples of embeddings into $W_{2n}$.
This gives us an inequality relating to the start
of the embedding.
Similarly, we can find inequalities for the 
end of the embedding.
We can also find inequalities 
that relate to the interior (when $r > 1$), and
Figures~\ref{figure_unused_points_interior_not_21}
and~\ref{figure_unused_points_interior_21}
show examples of this.
We can use these inequalities to determine what
values of $k$ will allow the shape to be embedded.
For each allowable value of $k$, we
can then determine the maximum value of $r$
such that $\oneplus \sumrab \plusone \not \leq \pi$.
This then means that, by evaluating inequalities alone,
we can identify the specific permutations
that could contribute to the \mob sum.

We first have two lemmas that examine
inequalities at the start and end of an embedding.
  
\begin{lemma}
    \label{lemma_unused_points_at_start}
    If $\pi$ is an increasing oscillation,
    and $\alpha \leq \pi$ is sum indecomposable, 
    then in any embedding of an 
    element $\lambda$ of $\familysum{\sumra}$ into $\pi$,
    the minimum number of unused points at the start of $\pi$
    depends on the start of $\lambda$, and on $\pi$,
    and is as shown below:
    
    \centering
    \begin{tabular}{lcccc}
        \toprule
        Start of $\lambda$ & $\pi=W_{2n}$ & $\pi=W_{2n-1}$ & $\pi=M_{2n}$ & $\pi=M_{2n-1}$ \\
        \midrule
        $\dtwo \ldots $ & $0$ & $0$ & $0$ & $0$ \\
        $\ildtwokp \ldots $ & $0$ & $0$ & $1$ & $1$ \\
        $\oneil \left( \ildtwok \right) \ldots $ & $1$ & $1$ & $0$ & $0$ \\
        \midrule
        $\oneplus \dtwo \ldots $ & $1$ & $1$ & $1$ & $1$ \\
        $\oneplus \left( \ildtwokp \right) \ldots$ & $1$ & $1$ & $2$ & $2$ \\
        $\oneplus \oneil \left( \ildtwok \right) \ldots$ & $2$ & $2$ & $1$ & $1$ \\
        \bottomrule
    \end{tabular}        
\end{lemma}
\begin{proof}
    It is clear that if we
    minimise the number of points at the start of
    an embedding, then the number of unused points
    depends on $\pi$, 
    and the start of $\alpha$.  
    The values in Lemma~\ref{lemma_unused_points_at_start}
    are found by considering each of the possibilities.
    We illustrate some of these cases in Figure~\ref{figure_unused_points_at_start_w2n}.
    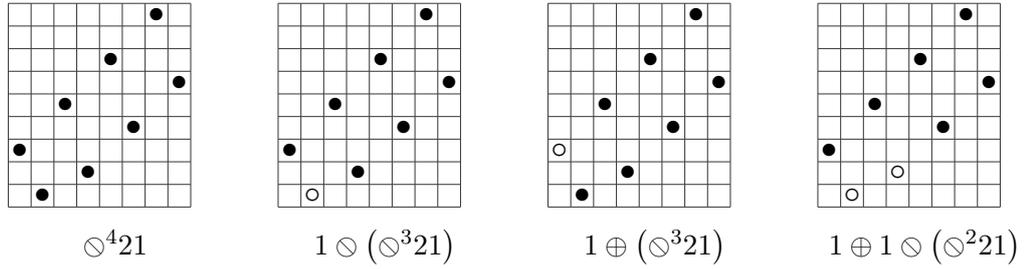
\begin{figure}      
        \centering
        \begin{subfigure}[c]{0.2\textwidth}
            \begin{tikzpicture}[scale=0.3]
            \plotgrid{8}{9}
            \plotperm{3, 1, 5, 2, 7, 4, 9, 6}
            \end{tikzpicture}
            \caption*{$\nils{4}{\dtwo}$}
        \end{subfigure}
        \quad
        \begin{subfigure}[c]{0.2\textwidth}
            \begin{tikzpicture}[scale=0.3]
            \plotgrid{8}{9}
            \plotperm{3, 1, 5, 2, 7, 4, 9, 6}
            \opendot{(2,1)}
            \end{tikzpicture}
            \caption*{$\oneil \left( \nils{3}{\dtwo} \right)$}
        \end{subfigure}
        \quad
        \begin{subfigure}[c]{0.2\textwidth}
            \begin{tikzpicture}[scale=0.3]
            \plotgrid{8}{9}
            \plotperm{3, 1, 5, 2, 7, 4, 9, 6}
            \opendot{(1,3)}
            \end{tikzpicture}
            \caption*{$\oneplus \left( \nils{3}{\dtwo} \right)$}
        \end{subfigure}
        \quad
        \begin{subfigure}[c]{0.2\textwidth}
            \begin{tikzpicture}[scale=0.3]
            \plotgrid{8}{9}
            \plotperm{3, 1, 5, 2, 7, 4, 9, 6}
            \opendot{(2,1)}
            \opendot{(4,2)}
            \end{tikzpicture}
            \caption*{$\oneplus \oneil \left( \nils{2}{\dtwo} \right)$}
        \end{subfigure}    
        \caption{Embedding the start of $\alpha$ in $W_{2n}$.}
        \label{figure_unused_points_at_start_w2n}
    \end{figure} 
\end{proof}
%
\begin{lemma}
    \label{lemma_unused_points_at_end}
    If $\pi$ is an increasing oscillation,
    and $\alpha \leq \pi$ is sum indecomposable, 
    then in any embedding of an 
    element $\lambda$ of $\familysum{\sumra}$ into $\pi$,
    the minimum number of unused points at the end of $\pi$
    depends on the end of $\lambda$, and on $\pi$,
    and is as shown below:    
    
    \centering
    \begin{tabular}{lcccc}
        \toprule
        End of $\lambda$                     & 
        $\pi=W_{2n}$ & $\pi=W_{2n-1}$ & $\pi=M_{2n}$ & $\pi=M_{2n-1}$ \\
        \midrule
        $\ldots \dtwo $ & $0$ & $0$ & $0$ & $0$ \\
        $\ldots \ildtwokp $ & $0$ & $1$ & $1$ & $0$ \\
        $\ldots \ildtwokb \ilone $ & $1$ & $0$ & $0$ & $1$ \\
        \midrule
        $\ldots \dtwo \plusone$ & $1$ & $1$ & $1$ & $1$ \\
        $\ldots \ildtwokpb \plusone $ & $1$ & $2$ & $2$ & $1$ \\
        $\ldots \ildtwokb \ilone \plusone$ & $2$ & $1$ & $1$ & $2$ \\
        \bottomrule
    \end{tabular}
\end{lemma}
\begin{proof}
    We examine all the possibilities as we did in
    Lemma~\ref{lemma_unused_points_at_start}.
\end{proof}

We now consider how closely copies of 
some sum indecomposable $\alpha$ can be 
embedded into $\pi$.
This leads to two inequalities
that relate $\alpha$, $\pi$
and the maximum number of copies of $\alpha$
that can be embedded in $\pi$.
Where $\alpha \neq \dtwo$, the shape
of $\alpha$ fixes the way the two copies
can be embedded in an increasing oscillation.
If $\alpha = \dtwo$, then we will see that
there are choices for the embedding.
\begin{lemma}
    \label{lemma_unused_points_interior}
    If $\pi$ is an increasing oscillation,
    and $\alpha \neq \dtwo$,
    and $\alpha \leq \pi$ is sum indecomposable, 
    then in any embedding of $\sumra$ into $\pi$,
    the minimum number of points 
    between the start and end of $\sumra$
    depends on  $\alpha$,
    and is as shown below:
    
    \centering
    \begin{tabular}{lccc}
        \toprule
        Shape of $\alpha$ & 
        Points in $\sumra$ &
        Unused points &
        Minimum points
        \\ 
        \midrule
        $\ildtwokp$ & $2kr $ & $2r-2$ & $2kr+2r-2$ \\
        $\oneil \ildtwokb$ & $2kr+r $ & $r-1 $ & $2kr+2r-1$ \\
        $\ildtwokb \ilone$ & $2kr+r $ & $r-1 $ & $2kr+2r-1$ \\
        $\oneil \ildtwokb \ilone$ & $2kr+2r $ & $2r-2$ & $2kr+4r-2$ \\
        \bottomrule
    \end{tabular}    
\end{lemma}
\begin{proof}
    If $r$ = 1, then there are no unused points,
    and so the minimum number of points depends solely
    on the points in $\alpha$, and the table 
    reflects this.
    
    Assume now that $r > 1$.
    If $\alpha \neq \dtwo$, then we can see that the 
    interleave fixes the layout of each copy of $\alpha$,
    so we simply pack each copy as close as possible.
    This packing clearly depends on the 
    start and end of $\alpha$, and it
    is simple to examine the four possibilities.
    Examples are shown in Figure~\ref{figure_unused_points_interior_not_21}.
    \begin{figure}
        \centering
        \captionsetup[subfigure]{justification=centering}
        \begin{subfigure}[c]{0.22\textwidth}        
            \begin{tikzpicture}[scale=0.3]
            \plotgrid{10}{12}
            \plotperm{4,1,6,3,8,5,10,7,12,9}
            \opendot{(5,8)}
            \opendot{(6,5)}
            \darkhline{6}{10}
            \darkvline{5}{12}
            \end{tikzpicture}
            \caption*{$\dtwo \interleave \dtwo \oplus$\\$\dtwo \interleave \dtwo$}
        \end{subfigure}
        \quad
        \begin{subfigure}[c]{0.22\textwidth}
            \begin{tikzpicture}[scale=0.3]
            \plotgrid{10}{12}
            \plotperm{4,1,6,3,8,5,10,7,12,9}
            \opendot{(6,5)}
            \darkhline{6}{10}
            \darkvline{4}{12}
            \end{tikzpicture}
            \caption*{$\dtwo \interleave \dtwo \oplus$\\$\oneil \dtwo \interleave \dtwo$}
        \end{subfigure}
        \quad
        \begin{subfigure}[c]{0.22\textwidth}
            \centering
            \begin{tikzpicture}[scale=0.3]
            \plotgrid{10}{12}
            \plotperm{4,1,6,3,8,5,10,7,12,9}
            \opendot{(5,8)}
            \darkhline{6}{10}
            \darkvline{6}{12}
            \end{tikzpicture}
            \caption*{$\dtwo \interleave \dtwo \ilone \oplus$\\$\dtwo \interleave \dtwo$}
        \end{subfigure}
        \quad
        \begin{subfigure}[c]{0.22\textwidth}
            \begin{tikzpicture}[scale=0.3]
            \plotgrid{10}{12}
            \plotperm{4,1,6,3,8,5,10,7,12,9}
            \opendot{(5,8)}
            \opendot{(8,7)}
            \darkhline{7}{10}
            \darkvline{6}{12}
            \end{tikzpicture}
            \caption*{$\dtwo \interleave \dtwo \ilone \oplus$\\$\oneil \dtwo$}
        \end{subfigure}
        \caption{Packing $\alpha$ as close as possible when $\alpha \neq \dtwo$.}
        \label{figure_unused_points_interior_not_21}
    \end{figure}
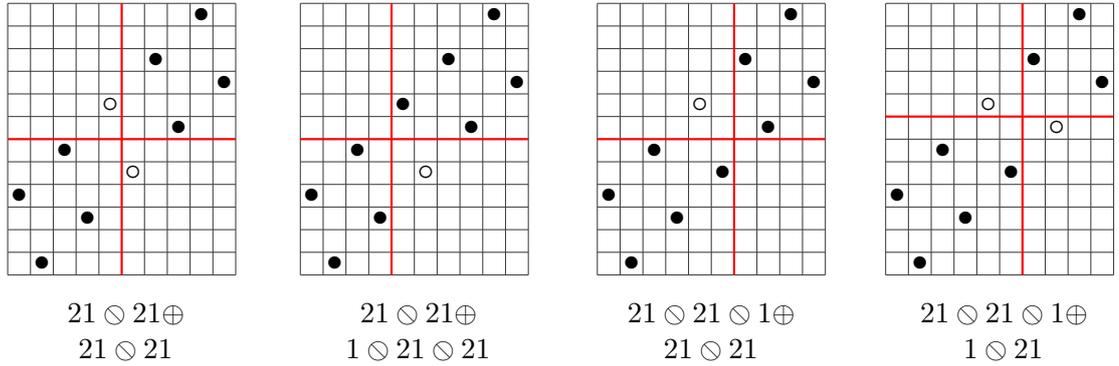 
\end{proof}

We now turn to the case where $\alpha = \dtwo$.  
This is more complex than the previous cases.
We can see that
there must be at least one point between each copy of $\alpha$.
We can insert each copy of $\dtwo$
in two ways,
one where the points are horizontally adjacent,
and one where the points are vertically adjacent.
These alternatives can be seen in 
Figure~\ref{figure_unused_points_interior_21}.
Alternating these means that there will be exactly
one point between each copy of $\alpha$,
so this embedding minimises the number of 
points between the start and end of $\sumra$.
The complication in this case relates to how we
start and end the embedding.  We illustrate this by
showing, in Figure~\ref{figure_unused_points_interior_21},
maximal embeddings where
we are embedding into $W_{8}$, $W_{10}$ and $W_{12}$.
\begin{figure}       
    \begin{center}
        \begin{subfigure}[c]{0.3\textwidth}
            \begin{center}
            \begin{tikzpicture}[scale=0.3]
            \plotgrid{8}{8}
            \plotperm{3, 1, 5, 2, 7, 4, 8, 6}
            \opendot{(4,2)}
            \opendot{(5,7)}
            \darkhline{3}{8}
            \darkhline{5}{8}
            \darkvline{2}{8}
            \darkvline{6}{8}
            \end{tikzpicture}
            \caption*{$W_{8}$}
            \end{center}
        \end{subfigure}
        \quad
        \begin{subfigure}[c]{0.3\textwidth}
            \begin{center}
            \begin{tikzpicture}[scale=0.3]
            \plotgrid{10}{10}
            \plotperm{3, 1, 5, 2, 7, 4, 9, 6, 10, 8}
            \opendot{(4,2)}
            \opendot{(5,7)}
            \opendot{(9,10)}
            \opendot{(10,8)}
            \darkhline{3}{10}
            \darkhline{5}{10}
            \darkvline{2}{10}
            \darkvline{6}{10}
            \end{tikzpicture}
            \caption*{$W_{10}$}
            \end{center}
        \end{subfigure}
        \quad
        \begin{subfigure}[c]{0.3\textwidth}
            \begin{center}
            \begin{tikzpicture}[scale=0.3]
            \plotgrid{12}{12}
            \plotperm{3, 1, 5, 2, 7, 4, 9, 6, 11, 8, 12, 10}
            \opendot{(4,2)}
            \opendot{(5,7)}
            \opendot{(10,8)}
            \opendot{(11,12)}
            \darkhline{3}{12}
            \darkhline{5}{12}
            \darkhline{9}{12}
            \darkvline{2}{12}
            \darkvline{6}{12}
            \darkvline{8}{12}
            \end{tikzpicture}
            \caption*{$W_{12}$}
            \end{center}
        \end{subfigure}
        \caption{Examples of unused points when embedding $\nsums{r}{\dtwo}$.}
        \label{figure_unused_points_interior_21}
    \end{center}
\end{figure}      
A detailed examination of each possible case
gives us our second inequality.
\begin{lemma}
    \label{lemma_unused_points_21}
    If $\pi$ is an increasing oscillation,
    and $\alpha = \dtwo$
    then for $\sumra$ to be contained
    in $\pi$
    we must have 
    $3r - 1 \leq 2n$ for $\pi \in \{W_{2n}, M_{2n} \}$,
    and
    $3r \leq 2n$ for $\pi \in \{W_{2n-1}, M_{2n-1} \}$.
\end{lemma}
\begin{proof}
    In every case we start by embedding the first $\dtwo$
    into the first two elements of the permutation.
    Thereafter, we embed each successive $\dtwo$
    as close as possible to the preceding $\dtwo$.
    The minimum number of elements to embed $r$ copies
    of $\dtwo$ will be $2r$ elements to hold the 
    points of the $\dtwo$s, and $r-1$ intermediate empty elements.
    For $W_{2n}$ and $M_{2n}$, this then gives
    $3r-1 \leq 2n$, and 
    for $W_{2n-1}$ and $M_{2n-1}$ we obtain
    $3r-1 \leq 2n-1$.
\end{proof}

We now have a complete understanding of the number of points
required to embed any 
permutation that contributes to the \mob sum
into an increasing oscillation.
The following Lemma summarises the situation.

\begin{lemma}
    \label{lemma_inequalities_for_pi}
    If $\pi$ is an increasing oscillation,
    and $\alpha \in \familyil{\ildtwok} \leq \pi$
    (so $\alpha$ is sum indecomposable),
    then for $\sumra$ to be contained
    in $\pi$,
    the inequality in the table below must be satisfied,
    where $k \geq 1$.
    
    \centering
    \begin{tabular}{l@{\phantom{xxxxxxx}}c@{\phantom{xxxxxxx}}r}
        \toprule
        $\pi$ & Shape of $\alpha$ & Inequality \\ 
        \midrule
        $W_{2n},   M_{2n}$    & $\dtwo$ &  $3r -1 \leq 2n $ \\
        $W_{2n-1}, M_{2n-1}$  & $\dtwo$ &  $3r    \leq 2n $\\
        $W_{2n}$ & $ \ildtwokp $ & $ 2kr+2r-2 \leq 2n $ \\
        $W_{2n-1}$ & $ \oneil \ildtwokb $ & $ 2kr+2r+2 \leq 2n $ \\
        $M_{2n-1}$ & $ \ildtwokb \ilone $ & $ 2kr+2r+2 \leq 2n $ \\
        $M_{2n}$ & $ \oneil \ildtwokb \ilone $ & $ 2kr+4r-2 \leq 2n $ \\ 
        $W_{2n}, W_{2n-1}, M_{2n-1}$ & $ \oneil \ildtwokb \ilone $ & $ 2kr+4r \leq 2n $ \\
        \multicolumn{2}{c}{All other cases} & $ 2kr+2r \leq 2n $ \\
        \bottomrule
    \end{tabular}
\end{lemma}
\begin{proof}
    We apply Lemmas~\ref{lemma_unused_points_at_start},
    \ref{lemma_unused_points_at_end}, 
    ~\ref{lemma_unused_points_interior} 
    and~\ref{lemma_unused_points_21} to the
    possibilities for $\pi$ and $\alpha$.
\end{proof}

As a consequence of 
Lemmas~\ref{lemma_unused_points_at_start}
and~\ref{lemma_unused_points_at_end}
we can define a relationship
between the minimum number of points
required to embed some $\sumra$,
and the minimum number of points required
to embed 
$\oneplus \sumrab$,
$\sumra \plusone$
and
$\oneplus \sumrab \plusone$.
\begin{corollary}
    \label{corollary_add_2_for_op_or_po}
    If $\pi$ is an increasing oscillation,
    and $\alpha \leq \pi$ is sum indecomposable
    and if
    the minimum number of points required to embed $\sumra$ into $\pi$
    is $C$, then
    the minimum number of points required to embed 
    $\oneplus \sumrab$ into $\pi$ is $C+2$,
    the minimum number of points required to embed 
    $\sumra \plusone$ into $\pi$ is $C+2$,
    and
    the minimum number of points required to embed 
    $\oneplus \sumrab \plusone$ into $\pi$ is $C+4$.
\end{corollary}
\begin{proof}
    We can see from Lemmas~\ref{lemma_unused_points_at_start}
    and~\ref{lemma_unused_points_at_end} that 
    adding $\oneplus {}$ 
    at the start of a permutation
    increases the number of points required by two -- 
    one for the new point, and one that is unused.
    Similarly, adding ${} \plusone$
    at the end increases the points required by two.
\end{proof}

Lemma~\ref{lemma_inequalities_for_pi} gives
us inequalities that any $\sumra$ must satisfy
to ensure that $\sumra \leq \pi$.  Further,
Corollary~\ref{corollary_add_2_for_op_or_po}
gives us inequalities that, for a given $\sumra$
allow us to determine if 
$\oneplus \sumrab \leq \pi$, 
$\sumrab \plusone \leq \pi$ and
$\oneplus \sumrab \plusone \leq \pi$.
We can therefore determine what values
of $r$ and $k$ will result in 
$\familysum{\sumra}$ contributing to the \mob function.
We now consider inequalities that
relate $\sigma$ and $\alpha$,
so that we can determine
if $\sigma \leq \alpha$ 
using an inequality.
\begin{lemma}
    \label{lemma_inequalities_for_sigma}
    If $\sigma > 1$ is an increasing oscillation,
    and $\alpha \in \familyil{\ildtwok}$ for some $k$,
    then for $\sigma$ to be contained
    in $\alpha$
    the inequality in the table below must be satisfied,
    where $k \geq 1$.
    
    \centering
    \begin{tabular}{l@{\phantom{xxxxxxx}}c@{\phantom{xxxxxxx}}r}
        \toprule
        $\sigma$ & Shape of $\alpha$ & Inequality \\ 
        \midrule
        $W_{2n-1},M_{2n},M_{2n-1}$ & $ \dtwo $ & False \\        
        $W_{2n-1}$ & $ \ildtwokb \ilone $ & $ k \geq n-1 $ \\
        $M_{2n-1}$ & $ \oneil \ildtwokb $ & $ k \geq n-1 $ \\
        $W_{2n-1}, M_{2n}, M_{2n-1}$ & $ \oneil \ildtwokb \ilone $ & $ k \geq n-1 $ \\         
        $M_{2n}$   & $ \ildtwokp $ & $ k \geq n+1 $ \\
        \multicolumn{2}{c}{All other cases} & $ k \geq n $ \\
        \bottomrule
    \end{tabular}
\end{lemma}
\begin{proof}
    We examine all possible cases.
\end{proof}

We are now nearly ready to present the main theorem 
for this section.  
Informally, for each possible shape
of permutation $\alpha$,
we will first find the
minimum and maximum values of $k$
such that $\sigma \leq \alpha \leq \pi$,
as any other values of $k$ result in 
$\alpha$ being outside the interval.
For each $\alpha$ and each $k$, we then 
determine the minimum value of $r$ such that
$\oneplus \sumrab \plusone \not\leq \pi$.
We can then use this value of $r$ (assuming it is non-zero)
to determine the weight to be applied to 
$\mobfn{\sigma}{\alpha}$.
The set of $\alpha$'s with a non-zero weight 
is a contributing set $\contrib{\sigma}{\pi}$.
At this point we can substitute a value
for any $\mobfn{\sigma}{\alpha}$
where $\order{\sigma} \leq \order{\alpha} - 1$.
We then use the same process
recursively to determine
the contributing set for the
remaining elements of $\contrib{\sigma}{\pi}$.

We first define some supporting functions.
Let $\rawmink(\sigma, \alpha)$
be the 
minimum value of $k$ that satisfies the inequality
in Lemma~\ref{lemma_inequalities_for_sigma}.
For the first inequality, which is always false,
we set $k=\order{\pi}$, as this will force
the sum, defined later in
Theorem~\ref{theorem_mobius_sum_increasing_oscillations},
to be empty.

Let $\mink(\sigma, \alpha)$ be defined as
\begin{align*}
\mink(\sigma, \alpha)
& =
\begin{cases*}
1 & If $\sigma = 1$ and $\alpha \neq \ildtwokp$, \\
2 & If $\sigma = 1$ and $\alpha = \ildtwokp$, \\
\rawmink(\sigma,\alpha) & otherwise.
\end{cases*}
\end{align*}
Observe that for any $k < \mink(\sigma, \alpha)$,
we have $\alpha < \sigma$, and so 
$\familysum{\nsums{k}{\alpha}}$ makes no net contribution to the \mob sum.

Let $\maxk(\alpha, \pi)$ be defined as
the maximum value of $k$ that satisfies the inequality in
Lemma~\ref{lemma_inequalities_for_pi},
if the shape of $\alpha$ and the shape of $\pi$ are different;
and
one less than the maximum value of $k$ that satisfies the inequality 
if the shape of $\alpha$ and the shape of $\pi$ are the same.    
For the first two inequalities, which do not involve $k$,
we set $\maxk(\alpha, \pi)=1$ if the inequality is satisfied, 
and $\maxk(\alpha, \pi)=0$ if not.
Observe here that 
for any $k > \maxk(\alpha, \pi)$ we have
$\alpha \not < \pi$, and so
$\familysum{\nsums{k}{\alpha}}$ makes no contribution to the \mob sum.

We define the weight function
for increasing oscillations, $\weightosc{\sigma}{\alpha}{\pi}$,
as
\begin{align*}
\weightosc{\sigma}{\alpha}{\pi}
& =
\begin{cases*}
1 & 
If
$
\sumrab \plusone \not \leq \pi,
$
\\
-1 & 
If
$
\sumrab \plusone \leq \pi 
\text{ and } 
\nsums{r+1}{\alpha} \not\leq \pi,
$
\\
0 &
Otherwise,
\end{cases*}
\end{align*}
where $r$ is the smallest integer 
such that $\oneplus \sumrab \plusone \not \leq \pi$.
These conditions are simpler than
those given in the weight function~(\ref{equation_general_weight_function})
for Theorem~\ref{theorem_mobius_sum_bottom_level_indecomposable} as,
by Corollary~\ref{corollary_add_2_for_op_or_po},
if 
$\sumrab \plusone \not \leq \pi$ then
$\oneplus \sumrab \not \leq \pi$ and vice-versa.
Furthermore, we will see that
this weight function
is only used when $\sigma \leq \sumra \leq \pi$.

We are now in a position to state our main theorem for 
this section.
In this theorem, we consider the contribution
to the \mob sum of each possible shape of 
some sum indecomposable $\alpha$.
There are five possible shapes, and, given that the 
expression for each shape is identical, 
we abuse notation slightly by writing our theorem
as a sum over the shapes, thus the
first sum in 
Theorem~\ref{theorem_mobius_sum_increasing_oscillations}
is over the possible shapes of $\alpha$,
where four of the shapes
have a parameter $k$.
For each shape, the limits on the interior sum
determine the minimum and maximum values
of $k$, using the summation variable $v$.  We use the 
notation $\alpha_v$ to represent the actual
permutation that has the shape $\alpha$,
where the parameter $k$ has been set to the value of $v$.
As an example, if $\alpha = \oneil \ildtwokb$,
and $v = 2$,
then $\alpha_v = \oneil \left( \nils{2}{\dtwo} \right) = 24153$.
 
\begin{theorem}
    \label{theorem_mobius_sum_increasing_oscillations}
    Let $\pi$ be an increasing oscillation,
    and let $\sigma \leq \pi$ be sum indecomposable.
    Then
    \begin{align*}
    \mobfn{\sigma}{\pi}
    & = 
    \sum_{\alpha \in \shapes} \
    \sum_{v=\mink(\sigma,\alpha)}^{\maxk(\alpha,\pi)}
    \mobfn{\sigma}{\alpha_{v}} 
    \weightosc{\sigma}{\alpha_{v}}{\pi}    
    \end{align*}
    where the first sum
    is over the possible shapes of a sum indecomposable
    permutation contained in an increasing oscillation,
    so 
    $\shapes = \{ \dtwo, \;$
    $\ildtwokp,\;$ 
    $\oneil \ildtwokb,\;$ 
    $\ildtwokb \ilone,\;$
    $\oneil \ildtwokb \ilone \}$.
\end{theorem}
\begin{proof}
    By Lemma~\ref{lemma_form_of_permutations_in_increasing_oscillations}
    the only sum-decomposable permutations contained
    in an increasing oscillation that contribute to the 
    \mob sum are $\familysum{\sumra}$,
    where $\alpha \in \shapes$.
    
    If we set $r=1$, then for each $\alpha$ in $\shapes$
    Lemma~\ref{lemma_inequalities_for_sigma} provides the
    smallest value of $k$ such that $\sigma \leq \alpha$.
    If there is no such value of $k$, then we use $\order{\pi}$,
    as the maximum value of $k$ must be smaller than this,
    and so the sum is empty.
    
    Again setting $r=1$, for each $\alpha$ in $\shapes$
    Lemma~\ref{lemma_inequalities_for_pi} provides the 
    maximum value of $k$ such that $\alpha \leq \pi$.
    If there is no value of $k$ that satisfies the 
    inequality, then we set $\maxk(\alpha,\pi) = 0$,
    thus forcing the sum to be empty.
    
    Thus the permutations $\alpha_v$ in the sum
    \[
    \sum_{\alpha \in \shapes} \
    \sum_{v=\mink(\sigma,\alpha)}^{\maxk(\alpha,\pi)}
    \] 
    are those that could contribute to the \mob sum,
    and for any $\alpha_v$ not included in the sum,
    $\familysum{\nsums{r}{\alpha_v}}$ has a zero contribution
    to the \mob sum for any $r$.
    
    Further, we can see from the construction method that
    any $\alpha_v$ included in the sum 
    has $\sigma \leq \nsums{r}{\alpha_v} \leq \pi$
    for at least one value of $r$, as if this was
    not the case, then we would have
    $\mink(\sigma,\alpha) > \maxk(\alpha,\pi)$, 
    and so the sum would be empty.
    
    We have therefore shown that the $\alpha_v$-s included 
    in the sum form a contributing set, and we could 
    therefore set $\contrib{\sigma}{\pi}$ to be those
    $\alpha_v$-s, and use
    Theorem~\ref{theorem_mobius_sum_bottom_level_indecomposable}.
    We now show that the increasing oscillation weight function
    $\weightosc{\sigma}{\alpha}{\pi}$ is equivalent
    to $\weightgen{\sigma}{\alpha}{\pi}$ as defined in the
    general case.
    
    By Corollary~\ref{corollary_add_2_for_op_or_po},
    if 
    $\sumrab \plusone \not \leq \pi$ then
    $\oneplus \sumrab \not \leq \pi$ and vice-versa,
    and so the condition for $\sumrab \plusone$ 
    also covers $\oneplus \sumrab$.
    As discussed above, 
    we know that there is at least one value of $r$ such that
    $\sigma \leq \sumra \leq \pi$, and so 
    $\weightosc{\sigma}{\alpha}{\pi}$ does not need to 
    include this condition.  
    Thus the increasing oscillation weight function
    $\weightosc{\sigma}{\alpha}{\pi}$ is equivalent
    to $\weightgen{\sigma}{\alpha}{\pi}$ as defined in the
    general case.
\end{proof}

\subsection{Example of Theorem~\ref{theorem_mobius_sum_increasing_oscillations}}
As an example of 
Theorem~\ref{theorem_mobius_sum_increasing_oscillations}
in action, we show how to determine
\[
\mobfn{3142}{315274968} = \mobfn{\nils{2}{\dtwo}}{\nils{4}{\dtwo} \ilone}.
\]
We start by considering each 
possible shape of $\alpha$,
setting $r=1$, and then using
the inequalities in
Lemmas~\ref{lemma_inequalities_for_pi}
and~\ref{lemma_inequalities_for_sigma}
to determine the minimum and maximum 
values of $k$.  This gives us

\begin{center}
\begin{tabular}{ccc}
    \toprule
    Shape of $\alpha$ & Minimum $k$ & Maximum $k$  \\ 
    \midrule
    $\dtwo$ & 1 & 1 \\
    $\ildtwok$ & 2 & 4 \\
    $\oneil \ildtwokb$ & 2 & 3 \\
    $\ildtwokb \ilone$ & 2 & 3 \\
    $\oneil \ildtwokb \ilone$ & 2 & 3 \\
    \bottomrule
\end{tabular}
\end{center}

For each shape of $\alpha$, and each value of $k$, we then
use the inequalities in
Lemma~\ref{lemma_inequalities_for_pi}
to determine the minimum value of $r$ such that
$\oneplus \sumrab \plusone \not \leq \pi$,
and we then calculate the weight using this
value of $r$.  This gives

\begin{center}
\begin{tabular}{ccr}
    \toprule
    $\alpha$ & $r$ & Weight \\ 
    \midrule
    $\dtwo$ & \multicolumn{2}{c}{No possibilities} \\
    $\nils{2}{\dtwo}$ & $2$ & $1$ \\
    $\nils{3}{\dtwo}$ & $1$ & $-1$ \\
    $\nils{4}{\dtwo}$ & $1$ & $-1$ \\
    $\oneil \left( \nils{2}{\dtwo} \right)$ & $1$ & $-1$ \\
    $\oneil \left( \nils{3}{\dtwo} \right)$ & $1$ & $-1$ \\
    $\left( \nils{2}{\dtwo} \right) \ilone$ & $2$ & $1$ \\
    $\left( \nils{3}{\dtwo} \right) \ilone$ & $1$ & $-1$ \\
    $\oneil \left( \nils{2}{\dtwo} \right) \ilone$ & $1$ & $-1$ \\    
    $\oneil \left( \nils{3}{\dtwo} \right) \ilone$ & $1$ & $-1$ \\
    \bottomrule
\end{tabular}
\end{center}

This leads to the following initial
expression:
\begin{align*}
\mobfn{\nils{2}{\dtwo}}{\nils{4}{\dtwo} \ilone}
= &
  \mobfn{\nils{2}{\dtwo}}{\nils{2}{\dtwo}}
- \mobfn{\nils{2}{\dtwo}}{\nils{3}{\dtwo}}
- \mobfn{\nils{2}{\dtwo}}{\nils{4}{\dtwo}} 
\\ &
- \mobfn{\nils{2}{\dtwo}}{\oneil \left( \nils{2}{\dtwo} \right)} 
- {} \mobfn{\nils{2}{\dtwo}}{\oneil \left( \nils{3}{\dtwo} \right)} 
\\ &
+ \mobfn{\nils{2}{\dtwo}}{\nils{2}{\dtwo} \ilone}
- \mobfn{\nils{2}{\dtwo}}{\nils{3}{\dtwo} \ilone}
\\ & 
- \mobfn{\nils{2}{\dtwo}}{\oneil \left( \nils{2}{\dtwo} \right) \ilone}
- \mobfn{\nils{2}{\dtwo}}{\oneil \left( \nils{3}{\dtwo} \right) \ilone}
\end{align*}

We know that $\mobfn{\nils{2}{\dtwo}}{\nils{2}{\dtwo}} = 1$,
and that 
\[
\mobfn{\nils{2}{\dtwo}}{\oneil \left( \nils{2}{\dtwo} \right)} =
\mobfn{\nils{2}{\dtwo}}{\nils{2}{\dtwo} \ilone} = 
-1.
\]  
Applying 
Theorem~\ref{theorem_mobius_sum_increasing_oscillations}
recursively to the other intervals eventually yields
\[
\mobfn{\nils{2}{\dtwo}}{\nils{4}{\dtwo} \ilone} = -6.
\]
\section{Concluding remarks}
\label{incosc_section-concluding-remarks}

The results in~\cite{Burstein2011} provide
two recurrences to handle the case where 
$\pi$ is decomposable.  
This work handles the case where $\sigma$ is indecomposable.
It overlaps with~\cite{Burstein2011}
when $\sigma$ is indecomposable 
and $\pi$ is decomposable.
This leaves the case where 
$\sigma$ is decomposable and 
$\pi$ is indecomposable
for further investigation.

We can see that
by symmetry
$\mobfn{\sigma}{W_{n}} = \mobfn{\sigma^{-1}}{M_{n}}$.
If we consider the value of the 
principal \mob function, $\mobfn{1}{\pi}$,
where $\pi$ is either $W_n$ or $M_n$,
then it is simple to show that the absolute value
of the principal \mob function
is bounded above by $2^n$.
The weight function for increasing oscillations
can be $\pm 1$, and we can see no obvious reason why
there should not be two distinct values, $i$ and $j$,
with the same parity, 
such that the signs of $\mobfn{1}{W_i}$ and 
$\mobfn{1}{W_j}$ were different.
We have experimental evidence,
based on the values of $W_n$ and $M_n$
for 
$n=1 \ldots 
\text{2,000,000}  
$ 
that suggests that
$\mobfn{1}{W_{2n}} < 0$, and that 
$\mobfn{1}{W_{2n-1}} > 0$.

Figure~\ref{figure_increasing_oscillation_values}
is a log-log plot of the values of $- \mobfn{1}{W_{2n}}$
from $n=\text{8,000}$ to $n=\text{10,000}$.  
As can be seen, there seems to be some 
evidence that the values
fall into distinct bands, and we 
have confirmed that this pattern
continues up to $n=\text{1,000,000}$.
Examination of the values of 
$\mobfn{1}{W_{2n-1}$} reveals the same 
patterns.
\begin{figure}
    \centering
        \centering
        \input{incosc_loglogplotofwn.tex}
        \caption{Log-Log plot of $\order{W_{2n}}$.}
        \label{figure_increasing_oscillation_values}
\end{figure}

Following discussions
at Permutation Patterns 2017,
V{\'{i}}t Jel{\'{i}}nek~\cite{Jelinek2017a} provided the following conjecture
(rephrased to reflect our notation).
\begin{conjecture}
    [{Jel{\'{i}}nek~\cite{Jelinek2017a}}]
    \label{incosc_conjecture_vit}
    Let $M(n)$ denote the absolute value of the 
    \mob function $\mobfn{1}{W_n} = \mobfn{1}{M_n}$.  
    Then for $n > 50$ we have
    \begin{align*}
    M(2n) & = n^2 
    \Longleftrightarrow  
    \text{$n+1$ is prime and $n \equiv 0 \mymod 6$} 
    \\
    M(2n) & = n^2 - 1 
    \Longleftrightarrow  
    \text{$n+1$ is prime and $n \equiv 4 \mymod 6$} 
    \\
    M(2n+1) & = n^2 - n
    \Longleftrightarrow  
    \text{$n+1$ is prime and $n \equiv 0 \mymod 6$} 
    \\
    M(2n+1) & = n^2 - n - 1 
    \Longleftrightarrow  
    \text{$n+1$ is prime and $n \equiv 4 \mymod 6$} 
    \\
    \end{align*}    
    Further, Jel{\'{i}}nek notes that 
    there does not seem to be any other small constant $k$ such that $M(n) = (n^2-k)/4$ infinitely often.    
\end{conjecture}

We also have the following conjecture relating to the banding
of the values.
\begin{conjecture}
    \label{incosc_conjecture_dwm}
    Let $M(n)$ denote the absolute value of the 
    \mob function $\mobfn{1}{W_n} = \mobfn{1}{M_n}$.  
    Let $E(n) = M(n) / (n^2)$, and
    let $O(n) = M(n)/(n^2 + n)$.
    Then, with $n \geq 1$, there exist constants 
    $0 < a < b < c < d < e < f < g < 1$ such that
    \begin{align*}
    E(12n + 10) & \in [ a , b ]  & O(12n + 11) & \in [ a , b ] \\
    E(12n +  2) & \in [ c , d ]  & O(12n +  3) & \in [ c , d ] \\
    E(12n +  6) & \in [ c , d ]  & O(12n +  7) & \in [ c , d ] \\
    E(12n +  4) & \in [ e , f ]  & O(12n +  5) & \in [ e , f ] \\
    E(12n +  8) & \in [ g , 1 ]  & O(12n +  9) & \in [ g , 1 ] \\
    E(12n     ) & \in [ g , 1 ]  & O(12n +  1) & \in [ g , 1 ] \\    
    \end{align*}
    Examining the first 2,000,000 values of 
    $\mobfn{1}{W_n}$
    gives the following
    estimates for the constants.
    \[
    \begin{array}{ccccccc}
        a & b & c & d & e & f & g \\
        0.615 & 0.680 &
        0.692 & 0.760 &
        0.821 & 0.896 &
        0.923
    \end{array}
    \]
\end{conjecture}

The 
\emph{complete nearly-layered}\extindex[permutation]{complete nearly-layered} 
permutations 
are formed by interleaving descending
permutations.  Formally,
a complete nearly-layered permutation has the form
\[
\alpha_1 \interleave \alpha_2 \interleave
\ldots
\interleave \alpha_{k-1} \interleave \alpha_k
\]
where each $\alpha_i$ is a descending permutation,
with $\alpha_i > 1$ for $i = 2, \ldots ,k-1$.
If we set $\alpha_i = 21$ for $i = 2, \ldots ,k-1$,
and $\alpha_1, \alpha_k \in \{1,21 \}$,
then we obtain the increasing oscillations.

The computational approach taken 
for increasing oscillations
could, we think, be adapted to 
complete nearly-layered permutations. 
It is clear
that the equivalent of the inequalities
in 
Lemmas~\ref{lemma_inequalities_for_pi}
and~\ref{lemma_inequalities_for_sigma}
would be somewhat more complex than those found here,
but we believe that it should be possible
to define an algorithm that could determine
the \mob function for complete nearly-layered
permutations where the lower bound is sum indecomposable.

\section{Chapter summary}

We started this chapter by saying that
our motivation was to find
a contributing set $\contrib{\sigma}{\pi}$
that is significantly smaller than
the poset interval $[\sigma, \pi)$,
and a $\{0, \pm1 \}$ weighting function $\weightgen{\sigma}{\alpha}{\pi}$
such that
\begin{align*}
\mobfn{\sigma}{\pi} 
& = 
- \sum_{\alpha \in \contrib{\sigma}{\pi}}
\mobfn{\sigma}{\alpha} \weightgen{\sigma}{\alpha}{\pi}.
\end{align*}

Our results are, essentially, computational.
By this we mean that 
if $\sigma$ is indecomposable, then 
Theorem~\ref{theorem_mobius_sum_bottom_level_indecomposable}
can be used to compute 
the value of the \mob function
on an interval $[\sigma, \pi]$
with less computational resources
than that required by the 
standard recursion of
Equation~\ref{equation_mobius_function}.
During the preparation of this thesis
the author generated 1000 random permutations
with length 14, then determined the 
poset $[1, \pi)$ defined by the random permutation ($\pi$),
and then
counted 
the permutations in the poset, 
the sum-indecomposable permutations in the poset,
and 
the skew-indecomposable permutations in the poset.
The details are summarised
in 
Table~\ref{table_indecomposable_improvement}.
\begin{table}
    \centering
    \begin{tabular}{lr}
        \toprule
        Average number of \ldots & Value \\
        \midrule
        Permutations & 3373 \\
        Sum-indecomposable permutations & 2492 \\
        Skew-indecomposable permutations & 2445 \\
        \bottomrule
    \end{tabular}    
    \caption{Statistics for 1000 posets $[1, \pi)$ defined by a random permutation of length 14,
         with figures rounded to the nearest integer.}
    \label{table_indecomposable_improvement}
\end{table}   
   
These statistics seem to indicate that
the improvement, at least for 
small intervals, does not seem to be 
as significant as the author hoped.
With hindsight this is not unexpected.
The proportion of permutations
that are sum or skew decomposable tends to zero
as the length of the permutation increases,
from which one can readily deduce that
the proportion of strongly indecomposable
permutations must tend to 1
as the length of the permutation increases.
This essentially means that the 
size of the contributing set
$\contrib{\sigma}{\pi}$
is likely to be only slightly
smaller than the size of the overall poset.

The author wrote a computer program, 
Permutation WorkShop (PWS)~\cite{Marchant2020a}, 
which can be used to investigate the \mob function
on the permutation poset.  
The author found that the overhead of identifying
which permutations were in the contributing set
$\contrib{\sigma}{\pi}$, 
and the overhead of 
calculating $\weightgen{\sigma}{\alpha}{\pi}$
meant that, in general, calculations for permutations
with an indecomposable lower bound
took longer than using the 
standard recursive definition
in Equation~\ref{equation_mobius_function}.
This observation is, however,
limited to the way in which PWS operates,
and, indeed, the hardware it runs on.
We suspect, however, that 
this observation is likely to be applicable
to other routines that
calculate the \mob function
on the permutation poset.

Despite the comments above, 
we still feel that the 
overall approach of finding
a contributing set and a weighting function
so that we can write
\[
\mobfn{\sigma}{\pi} 
= 
- \sum_{\alpha \in \contrib{\sigma}{\pi}}
\mobfn{\sigma}{\alpha} \weightgen{\sigma}{\alpha}{\pi}.
\]
is valid, and indeed, 
although it is not phrased in these terms, the
results in Chapter~\ref{chapter_2413_balloon_paper}
use this method successfully. 

By contrast, when we consider 
intervals $[\sigma, \pi]$
where $\pi$ is an increasing oscillation,
we find that the number of indecomposable permutations
contained in $W(n)$ or $M(n)$
is, apart from trivial values of $\pi$,
no greater than $2n - 4$,
and this upper bound is only achieved when $\sigma = 1$.
This is because any indecomposable permutation contained in 
an increasing oscillation must itself be a (smaller)
increasing oscillation, and there are only
two increasing oscillations of each length.
It is well-known that,
for most intervals $[1, \pi]$,
the number of permutations in the poset grows exponentially
as $\order{\pi}$ increases,
and so we believe that in this specific
case we have indeed found a contributing set
that is significantly smaller than the poset.
As supporting evidence for this claim,
we note that we were easily able to calculate
$\mobp{\pi}$ where $\pi$ was an increasing
oscillation with $\text{2,000,000}$ elements.

This gave us the raw data to
notice the banding 
shown in
Figure~\ref{figure_increasing_oscillation_values},
and led to 
Conjecture~\ref{incosc_conjecture_dwm}.
We are not aware of any other 
set of permutations 
with a simple length-based construction
where the 
values of the \mob function
fall into bands as the length of the permutation(s)
increase.
The \mob function on the permutation
pattern poset is, however, notoriously hard to compute in general,
so it is quite possible that such sets do exist,
but we do not have the understanding 
and/or the technology 
to be able to calculate values that would
exhibit banding.

We are not the only researchers to have considered
the behaviour of $\mobp{W_n}$, and
Conjecture~\ref{incosc_conjecture_vit}
comes from a personal communication
with V{\'{i}}t Jel{\'{i}}nek~\cite{Jelinek2017a}. 
We have used our computations of  
$\mobp{W_n}$ to confirm that this conjecture 
holds for $50 < n \leq \text{2,000,000}$.

We suspect that the banding behaviour 
in Conjecture~\ref{incosc_conjecture_dwm}
is a consequence of the 
link with the prime numbers in
Conjecture~\ref{incosc_conjecture_vit}.
While there are results that 
give us expressions for the 
principal \mob function value, 
we are not aware of any result,
whether relating to the value of the 
principal \mob function, or
the growth of the 
principal \mob function, 
where the result has a link to the prime numbers.
This suggest to us that
one possible area for future research
would be to develop a better understanding of
the behaviour of the principal \mob function of
increasing oscillations.
We would hope that if we could
find a relationship that accounted
for the apparent link with prime numbers,
then we would also have 
a better understanding of the
permutation pattern poset.

    \chapter{Zeros of the \mob function of permutations}
\label{chapter_oppadj_paper}

\section{Preamble}

This chapter
is based on a published 
paper~\cite{Brignall2020},
which is joint work with
Robert Brignall,
V{\'i}t Jel{\'i}nek
and
Jan Kyn{\v c}l.

In this chapter we show that if a permutation 
$\pi$ contains two intervals of length 2, 
where one interval is an 
ascent and the other a descent, 
then the \mob function $\mobfn{1}{\pi}$ of the interval $[1,\pi]$ is 
zero. As a consequence, we prove that the 
proportion of permutations of length $n$ with principal \mob 
function equal to zero is asymptotically bounded below 
by $(1-1/e)^2\ge\zpmfr$. This is the first 
result determining the value of $\mobfn{1}{\pi}$ 
for an asymptotically positive proportion of 
permutations~$\pi$. 

We further establish other
general conditions on a permutation $\pi$ that ensure 
$\mobfn{1}{\pi}=0$ including
the occurrence in $\pi$ 
of any interval of the form
$\alpha\oplus 1 \oplus\beta$.

\section{Introduction}

In this section we describe our principal results, 
and give an overview of the 
previous work in this area.  
Formal definitions are given in the next section.

In this chapter, we are mainly concerned with the 
principal \mob function.
We focus on the zeros of the principal \mob function, 
that is, on the permutations $\pi$ for which
$\mobp{\pi}=0$. We show that we can often determine 
that a permutation $\pi$ is such a \mob zero by 
examining small localities of~$\pi$.  
We formalize this idea using the notion of an 
``annihilator''. 
Informally, an annihilator is a permutation $\alpha$ 
such that any permutation 
$\pi$ containing an interval copy of $\alpha$ 
is a \mob zero. 
We will describe an infinite family of 
annihilators. 

We will also prove that any permutation 
containing an increasing as well as a decreasing 
interval of size 2 is a \mob zero. 
Based on this result, we show that the asymptotic proportion 
of \mob zeros among the permutations of a 
given length is at least $(1-1/e)^2\ge \zpmfr$. 
This is the first known result 
determining the values of the principal \mob function for an 
asymptotically positive fraction of permutations. 
We will also demonstrate how our results on the 
principal \mob function can be extended 
to intervals whose lower bound is not~$1$.

Burstein, Jel{\'{i}}nek, Jel{\'{i}}nkov{\'{a}} 
and Steingr{\'{i}}msson~\cite{Burstein2011} found
a recursion for the \mob function
for sum and skew decomposable permutations.
They used this to determine
the \mob function for separable permutations.
Their results 
for sum and skew decomposable permutations
implicitly include a result that only concerns small localities,
which is that, up to symmetry, 
if a permutation $\pi$ of length greater than two begins $12$,
then $\mobp{\pi} = 0$.

Smith~\cite{Smith2013}
found an explicit formula for the \mob function on the interval
$[1, \pi]$ for all permutations $\pi$ with a single descent.
Smith's paper includes a lemma
stating that if a permutation $\pi$
contains an interval order-isomorphic to
$123$, then $\mobp{\pi}=0$.
While the result in~\cite{Burstein2011} requires  
that the permutation starts with a particular sequence,
Smith's result is, in some sense, more general,
as the critical interval (123)
can occur in any position.
Smith's lemma may be viewed
as the first instance of an
annihilator result. 
Our results on annihilators provide a common generalization of Smith's 
lemma and the above mentioned result of Burstein et al.~\cite{Burstein2011}.

\section{Definitions and notation}
\label{sect-definitions-and-notation}

Recall that an \emph{adjacency} in a permutation is an interval of length two.
If a permutation contains a monotonic interval 
of length three or more, then each subinterval 
of length two is an adjacency.
As examples, $367249815$ has two adjacencies, $67$ and $98$;
and $1432$ also has two adjacencies, $43$ and $32$.
If an adjacency is ascending, then it is an 
\emph{up-adjacency}, otherwise it is a 
\emph{down-adjacency}.

If a permutation $\pi$ contains at least one
up-adjacency, and at least one down-adjacency,
then we say that $\pi$ has \emph{opposing adjacencies}.
An example of a permutation with
opposing adjacencies is $367249815$,
which is shown in Figure~\ref{figure-example-oppadj}.
\begin{figure}[!ht]
    \begin{center}
        \begin{tikzpicture}[scale=0.25]
                    \plotpermgrid{3,6,7,2,4,9,8,1,5};
\draw [color=blue, very thick] (1.5, 5.5) rectangle (3.5, 7.5);
\draw [color=blue, very thick] (5.5, 7.5) rectangle (7.5, 9.5);    
        \end{tikzpicture}
    \end{center}%
    \caption{A permutation with opposing adjacencies.}
    \label{figure-example-oppadj}
\end{figure}
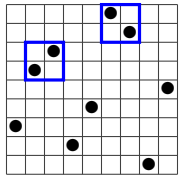

A permutation that does not 
contain any adjacencies is
\emph{adjacency-free}\extindex{adjacency-free}.
Some early papers use the term ``strongly irreducible''
for what we call adjacency-free permutations.  
See, for example, Atkinson and Stitt~\cite{Atkinson2002}.

Given a permutation $\sigma$ of length $n$, and
permutations $\alpha_1, \ldots, \alpha_n$,
not all of them equal to
the empty permutation $\emptyperm$,
the 
\emph{inflation}\extindex{inflation}
of $\sigma$ by $\alpha_1, \ldots, \alpha_n$, written as $\inflateall{\sigma}{\alpha_1, \ldots, \alpha_n}$,
is
the permutation obtained by 
removing the element $\sigma_i$
if $\alpha_i = \emptyperm$, and replacing $\sigma_i$
with an interval isomorphic to $\alpha_i$ otherwise.
Note that this is slightly different to the standard
definition of inflation, originally given in Albert and Atkinson~\cite{Albert2005},
which does not allow inflation by the empty permutation.
As examples,
$\inflateall{3624715}{1,12,1,1,21,1,1}=367249815$,
and
$\inflateall{3624715}{\emptyperm,1,1,\emptyperm,1,\emptyperm,1}=3142$.

In many cases we will be interested 
in permutations where most positions
are inflated by the singleton permutation $1$.
If $\sigma = 3624715$,
then 
we will write
$\inflateall{\sigma}{1,12,1,1,21,1,1} = 367249815$
as 
$\inflatesome{\sigma}{2,5}{12,21}$.
Formally, 
$\inflatesome{\sigma}{i_1, \ldots, i_k}{\alpha_1, \ldots, \alpha_k}$
is the inflation of $\sigma$ 
where $\sigma_{i_j}$ is inflated by $\alpha_{j}$ for
$j = 1, \ldots , k$, and all other positions of $\sigma$
are inflated by $1$. When using this notation, we always assume that the indices $i_1,\dotsc,i_k$ 
are distinct; however, we make no assumption about their relative order. 

Our aim is to study the \mob function of the permutation poset, that is, the poset of finite 
permutations ordered by containment. We are interested in describing general examples of 
intervals $[\sigma,\pi]$ such that $\mobfn{\sigma}{\pi}=0$, with particular emphasis on the case 
$\sigma=1$. We say that $\pi$ is a 
\emph{\mob zero}\extindex{\mob zero} 
(or just \emph{zero}) if $\mobp{\pi}=0$, and 
we say that $\pi$ is a 
\emph{$\sigma$-zero}\extindex{$\sigma$-zero} 
if $\mobfn{\sigma}{\pi}=0$.

It turns out that many sufficient conditions for $\pi$ to be a \mob zero can be stated in terms of 
inflations. We say that a permutation $\phi$ is an 
\emph{annihilator}\extindex{annihilator} 
if every permutation that has 
an interval copy of $\phi$ is a \mob zero; in other words, for every $\tau$ and every $i\le|\tau|$ 
the permutation $\tau_i[\phi]$ is a \mob zero. More generally, we say that $\phi$ is a 
\emph{$\sigma$-annihilator}\extindex{$\sigma$-annihilator}
 if every permutation with an interval copy of $\phi$ is a 
$\sigma$-zero.

We say that a pair of permutations $\phi$, $\psi$ is an 
\emph{annihilator pair}\extindex{annihilator pair} 
if for every 
permutation $\tau$ and every pair of distinct indices $i,j\le |\tau|$, the permutation 
$\inflatesome{\tau}{i,j}{\phi,\psi}$ is a \mob zero.

Observe that for an annihilator $\phi$, any permutation containing an interval copy of $\phi$ is 
also an annihilator. Likewise, if $\phi$ and $\psi$ form an annihilator pair then any 
permutation containing disjoint interval copies of $\phi$ and $\psi$ is an annihilator.

As our first main result, presented in Section~\ref{sec-opposing}, we show that the two 
permutations $12$ and $21$ are an annihilator pair, or equivalently, any permutation with opposing 
adjacencies is a \mob zero. Later, in Section~\ref{section-bounds-for-zn}, we use this result
to prove that \mob zeros have asymptotic density at least $(1-1/e)^2$. 

We also prove that for any two non-empty permutations $\alpha$ and $\beta$, the permutation 
$\alpha\oplus1\oplus\beta=\inflateall{123}{\alpha,1,\beta}$ is an annihilator, 
and generalize this result to a 
construction of $\sigma$-annihilators for general~$\sigma$. These results are presented in 
Section~\ref{sec-annihilator}.

Finally, in Section~\ref{sec-special}, we give several examples of annihilators and 
annihilator pairs that do not directly follow from the results in the previous sections.

\subsection{Intervals with vanishing \mob function}

We will now present several basic facts 
about the \mob function, which are valid in an arbitrary 
finite poset. 
The first fact is a simple observation following directly from the 
definition of the 
\mob function, and we present it without proof.

\begin{fact}\label{fac-del}
    Let $P$ be a finite poset with \mob function $\mu_P$,  and let $x$ and $y$ be two elements of 
    $P$ satisfying $\mobxfn{P}{x}{y}=0$. Let $Q$ be the poset obtained from $P$ by deleting the element $y$, 
    and let $\mu_Q$ be its \mob function. Then for every $z\in Q$, we have $\mobxfn{Q}{x}{z}=\mobxfn{P}{x}{z}$.
\end{fact}

Next, we introduce two types of intervals 
whose specific structure ensures that their \mob function 
is zero.

Let $[x,y]$ be a finite interval in a poset $P$. 
We say that $[x,y]$ is 
\emph{narrow-tipped}\extindex[poset]{narrow-tipped} 
if it 
contains an element $z$ different from $x$ 
such that $[x,y)=[x,z]$. The element $z$ is then called 
the 
\emph{core}\extindof{poset}{core}{(poset)}
of $[x,y]$. 

We say that the interval $[x,y]$ is 
\emph{diamond-tipped}\extindex[poset]{diamond-tipped} 
if there are three elements $z$, $z'$ 
and $w$, all different from $x$, and such that 
\begin{enumerate}
    \item $[x,y)=[x,z]\cup[x,z']$ and  
    \item $[x,z]\cap[x,z']=[x,w]$.
\end{enumerate}
Condition 2 is equivalent to $w$ being the 
greatest lower bound of $z$ and $z'$ in the interval $[x, y]$.
The triple of elements $(z,z',w)$ is again called the 
\emph{core} 
of $[x,y]$.  
Figure~\ref{fig-example-diamond-tipped-poset} shows examples
of narrow-tipped and diamond-tipped posets.
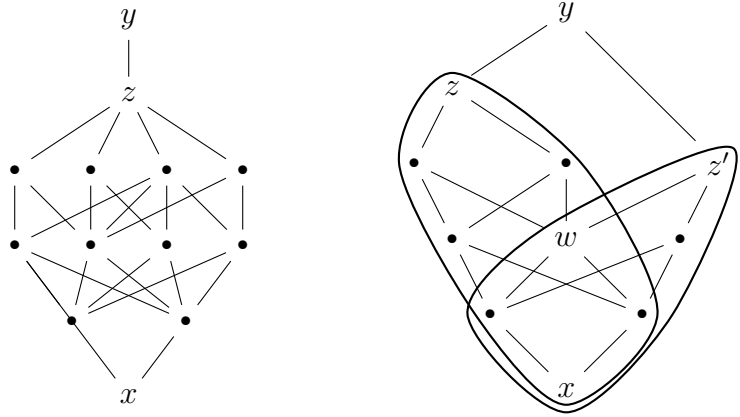
\begin{figure}
    \begin{center}
        \begin{tikzpicture}[xscale=1,yscale=1]
        \tnode{5}{0}{5}{$y$};
        \tnode{4}{0}{4}{$z$};        
        \dnode{31}{-1.5}{3};
        \dnode{32}{-0.5}{3};
        \dnode{33}{0.5}{3};
        \dnode{34}{1.5}{3};
        \dnode{21}{-1.5}{2};
        \dnode{22}{-0.5}{2};
        \dnode{23}{0.5}{2};
        \dnode{24}{1.5}{2};
        \dnode{11}{-0.75}{1};
        \dnode{12}{0.75}{1};
        \tnode{0}{0}{0}{$x$};
        \dline{5}{4};
        \dline{4}{31,32,33,34};
        \dline{31}{21,22};
        \dline{32}{22,23};
        \dline{33}{21,22,23,24};
        \dline{34}{22,24};
        \dline{21}{11,12};
        \dline{22}{11,12};
        \dline{23}{11,12};
        \dline{24}{11,12};
        \dline{12}{0};
        \dline{21}{0};
        \end{tikzpicture}		
        \qquad\qquad
        \begin{tikzpicture}[xscale=1,yscale=1]
        \tnode{1}{0}{0}{$x$};
        \dnode{12}{-1}{1};
        \dnode{21}{1}{1};
        \dnode{231}{-1.5}{2};
        \tnode{132}{0}{2}{$w$};
        \dnode{213}{1.5}{2};
        \dnode{2431}{-2}{3};
        \dnode{1342}{0}{3};
        \tnode{2143}{2}{3}{$z^\prime$};
        \tnode{13542}{-1.5}{4}{$z$};
        \tnode{214653}{0}{5}{$y$};
        \dline{1}{12};
        \dline{1}{21};
        \dline{12}{231};
        \dline{12}{132};
        \dline{12}{213};
        \dline{21}{231};
        \dline{21}{132};
        \dline{21}{213};
        \dline{231}{2431};
        \dline{231}{1342};
        \dline{132}{2431};
        \dline{132}{1342};
        \dline{132}{2143};
        \dline{213}{2143};
        \dline{2431}{13542};
        \dline{1342}{13542};
        \dline{2143}{214653};
        \dline{13542}{214653};
        \draw [thick] plot [smooth cycle] coordinates {
            (-1.5,  4.2) 
            (-2.2,  3.0)
            (-1.2,  1.0)
            ( 0.0, -0.2)
            ( 1.2,  1.0)
            ( 0.2,  3.0)
        };
        \draw [thick] plot [smooth cycle] coordinates {
            (2.2,  3.2) 
            ( -0.2,  2.2)
            (-1.3,  1.0)
            ( 0.0, -0.3)
            ( 1.4,  1.0)
        };
        \end{tikzpicture}		
    \end{center}
    \caption{Examples of narrow-tipped (left) and diamond-tipped (right) posets.}
    \label{fig-example-diamond-tipped-poset}
\end{figure}

\begin{fact}\label{fac-nd}
    Let $P$ be a poset with \mob function $\mu_P$, 
    and let $[x,y]$ be a finite interval in~$P$. If 
    $[x,y]$ is narrow-tipped or diamond-tipped, then $\mobxfn{P}{x}{y}=0$.
\end{fact}
\begin{proof}
    If $[x,y]$ is narrow-tipped with core $z$, then 
    \begin{align*}
    \mobxfn{P}{x}{y} 
    = 
    -\sum_{v\in[x,y)} \mobxfn{P}{x}{v} 
    = 
    -\sum_{v\in[x,z]} \mobxfn{P}{x}{v} 
    = 
    0.
    \end{align*}
    If $[x,y]$ is diamond-tipped with core $(z,z',w)$ then
    \begin{align*}
    \mobxfn{P}{x}{y} 
    &= -\sum_{v\in[x,y)} \mobxfn{P}{x}{v} \\
    &= -\sum_{v\in[x,z]\cup[x,z']} \mobxfn{P}{x}{v} \\
    &= -\sum_{v\in[x,z]} \mobxfn{P}{x}{v} 
    - \sum_{v\in[x,z']} \mobxfn{P}{x}{v} 
    + \sum_{v\in[x,z]\cap[x,z']} \mobxfn{P}{x}{v} \\
    &= -\sum_{v\in[x,z]} \mobxfn{P}{x}{v} 
    - \sum_{v\in[x,z']} \mobxfn{P}{x}{v} 
    + \sum_{v\in[x,w]} \mobxfn{P}{x}{v} \\
    & =0.\qedhere
    \end{align*}
\end{proof}

\subsection{Embeddings}
\label{subsect-embeddings}

Recall that an \emph{embedding} of a permutation $\sigma\in\cS_k$ 
into a permutation $\pi\in\cS_n$ is a 
function $f\colon [k]\to[n]$ with the following properties:
\begin{itemize}
    \item $1\le f(1)<f(2)<\dotsb<f(k)\le n$.
    \item For any $i,j\in[k]$, we have $\sigma_i<\sigma_j$ 
    if and only if $\pi_{f(i)}<\pi_{f(j)}$.
\end{itemize}

We let $\sE(\sigma,\pi)$ denote the set of 
embeddings of $\sigma$ into $\pi$, and $E(\sigma,\pi)$ 
denote the cardinality of $\sE(\sigma,\pi)$.

For an embedding $f$ of $\sigma$ into $\pi$, 
the 
\emph{image}\extindex[embedding]{image} 
of $f$, denoted $\Img(f)$, is the set 
$\{f(i);\;i\in[k]\}$. In particular, $|\Img(f)|=|\sigma|$. 
The permutation $\sigma$ is the 
\emph{source}\extindex[embedding]{source} 
of the embedding $f$, 
denoted $\src_\pi(f)$. When $\pi$ is clear from the context (as 
it usually will be) we write $\src(f)$ instead of $\src_\pi(f)$. 
Note that for a fixed $\pi$, the 
set $\Img(f)$ determines both $f$ and $\src_\pi(f)$ uniquely. 

We say that an embedding $f$ is 
\emph{even}\extindex[embedding]{even} 
if the cardinality of $\Img(f)$ is even, otherwise 
$f$ is 
\emph{odd}\extindex[embedding]{odd}. 
In our arguments, we will frequently consider 
\emph{sign-reversing}\extindex[embedding]{sign-reversing} 
mappings on 
sets of embeddings (with different sources), 
which are mappings that map an odd embedding to an even one and vice versa. A 
typical example of a sign-reversing mapping 
is the so-called $i$-switch, which we now define. For a 
permutation $\pi\in\cS_n$, 
let $\sE(*,\pi)$ be the set $\bigcup_{\sigma\le\pi}\sE(\sigma,\pi)$. For 
an index $i\in[n]$, the 
\emph{$i$-switch}\extindex[embedding]{$i$-switch} 
of an embedding $f\in\sE(*,\pi)$, denoted $\Delta_i(f)$, 
is the embedding $g\in\sE(*,\pi)$ 
uniquely determined by the following properties:
\begin{align*}
\Img(g)&= \Img(f)\cup\{i\} \text{ if } i\not\in\Img(f)\text{, and}\\
\Img(g)&= \Img(f)\setminus\{i\} \text{ if } i\in\Img(f).
\end{align*} 

For example, consider the permutations 
$\sigma=132$ and $\pi=41253$, and the embedding 
$f\in\sE(\sigma,\pi)$ 
satisfying $f(1)=2$, $f(2)=4$, and $f(3)=5$. 
We then have $\Img(f)=\{2,4,5\}$. 
Defining $g=\Delta_3(f)$, we see that 
$\Img(g)=\{2,3,4,5\}$, and $\src(g)$ is the permutation 
$1243$. 
Similarly, for $h=\Delta_5(g)$, we have $\Img(h)=\{2,3,4\}$ and $\src(h)=123$.

Note that for any $\pi\in\cS_n$ and any $i\in[n]$, 
the function $\Delta_i$ is a sign-reversing 
involution on the set $\sE(*,\pi)$.

Consider, for a given $\pi\in\cS_n$, 
two embeddings $f,g\in\sE(*,\pi)$. We say that $f$ 
\emph{is contained in}\extindex[embedding]{contained in} 
$g$ if $\Img(f)\subseteq \Img(g)$. 
Note that if $f$ is contained in $g$, 
then the permutation $\src(f)$ is contained in~$\src(g)$, 
and if a permutation $\lambda$ is 
contained in a permutation $\tau$, 
then any embedding from $\sE(\tau,\pi)$ contains at least one 
embedding from $\sE(\lambda,\pi)$. 
In particular, the mapping $f\mapsto \src(f)$ is a poset 
homomorphism from the set $\sE(*,\pi)$ 
ordered by containment onto the interval $[\emptyperm,\pi]$ 
in the permutation pattern poset.

\subsection{\mob function via normal embeddings}\label{sec-form}

We will now derive a general formula 
which will become useful in several subsequent arguments. The 
formula can be seen as a direct consequence of the 
well-known \mob inversion formula.
The following form of the \mob inversion formula can be deduced, 
for example, from Proposition 3.7.2
in Stanley~\cite{Stanley2012}. 
Recall that a poset is \emph{locally finite} if each of its intervals is finite.

\begin{fact}[M\"obius inversion formula]
    \label{fac-mif} 
    Let $P$ be a locally finite poset with 
    maximum element $y$, 
    let $\mu$ be the M\"obius function of $P$, and let $F\colon 
    P\to\bbR$ be a function. 
    If a function $G\colon P\to\bbR$ is defined by
    \[
    G(x)=\sum_{z\in[x,y]} F(z),
    \]
    then for every $x\in P$, we have
    \[
    F(x)=\sum_{z\in[x,y]} \mobfn{x}{z}G(z).
    \]
\end{fact}

As a consequence, we obtain the following result.

\begin{proposition}
    \label{pro-form}
    Let $\sigma$ and $\pi$ be arbitrary permutations, 
    and let $F\colon [\sigma,\pi]\to\bbR$ be a 
    function satisfying $F(\pi)=1$. We then have
    \begin{equation}
    \mobfn{\sigma}{\pi}= F(\sigma) - \sum_{\lambda\in 
        [\sigma,\pi)} 
    \mobfn{\sigma}{\lambda}\sum_{\tau\in[\lambda,\pi]} F(\tau).
    \label{eq-form}
    \end{equation}
\end{proposition}

\begin{proof}
    Fix $\sigma$, $\pi$ and $F$. 
    For $\lambda\in[\sigma,\pi]$, define
    $G(\lambda)=\sum_{\tau\in[\lambda,\pi]} 
    F(\tau)$. 
    Using Fact~\ref{fac-mif} for the poset $P=[\sigma,\pi]$, we obtain
    \[
    F(\sigma)=\sum_{\lambda\in[\sigma,\pi]} \mobfn{\sigma}{\lambda}G(\lambda).
    \]
    Substituting the definition of $G(\lambda)$ into the above identity and 
    noting that $F(\pi)=1$, we get
    \begin{align*}
    F(\sigma)
    & =\sum_{\lambda\in[\sigma,\pi]}\mobfn{\sigma}{\lambda}
    \sum_{\tau\in[\lambda,\pi]} F(\tau)\\
    &=\mobfn{\sigma}{\pi} + \sum_{\lambda\in[\sigma,\pi)}\mobfn{\sigma}{\lambda}
    \sum_{\tau\in[\lambda,\pi]} F(\tau), 
    \end{align*}
    from which the proposition follows.
\end{proof}

In our applications, the function $F(\tau)$ 
will usually be defined in terms of the number of 
embeddings of $\tau$ into $\pi$ satisfying certain additional conditions.
We call embeddings that satisfy these conditions \emph{normal embeddings}.
We briefly discussed normal embeddings
in Chapter~\ref{chapter_common_definitions}.
We extend that discussion here, as,
in this thesis, 
normal embeddings are only used in the current chapter.

The notion of normal embedding seems to originate from the work of Bj\"orner~\cite{BjornerSubword}, 
who defined normal embeddings between words, 
and showed that in the subword order of 
words over a finite alphabet, the M\"obius function of 
any interval $[x,y]$ is equal in absolute 
value to the number of normal embeddings of $x$ into~$y$. 

Bj\"orner's approach was later extended to the computation of the M\"obius 
function in the composition poset~\cite{Sagan2006}, the poset of separable 
permutations~\cite{Burstein2011}, or the poset of permutations with a fixed number of 
descents~\cite{Smith2016}. In all these cases, the authors define a notion 
of ``normal'' embeddings tailored for their poset, and then express the M\"obius 
function of an interval $[x,y]$ as the sum of weights of the ``normal'' embeddings 
of $x$ into $y$, where each normal embedding has weight $1$ or~$-1$.

For general permutations, this simple approach fails, since the M\"obius 
function $\mobfn{\sigma}{\pi}$ is sometimes larger than the number of all 
embeddings of $\sigma$ into~$\pi$. However, Smith~\cite{Smith2016a} 
introduced a notion of normal embedding applicable to arbitrary permutations, 
and proved a formula expressing $\mobfn{\sigma}{\pi}$ as a summation over certain 
sets of normal embeddings. 

For consistency, we adopt the term ``normal embedding'' in this chapter, although in 
our proofs, we will need to introduce 
several notions of normality, which are different from each 
other and from the notions of normality introduced by previous authors. 
We will always use $\sNE(\tau,\pi)$ to denote the set of embeddings of $\tau$ into $\pi$ satisfying 
the definition of normality used in the given context, 
and we let $\NE(\tau,\pi)$ be the cardinality of $\sNE(\tau,\pi)$.

The next proposition provides a general basis 
for all our subsequent applications of normal 
embeddings.

\begin{proposition}\label{pro-normal}
    Let $\sigma$ and $\pi$ be permutations. Suppose that for each $\tau\in[\sigma,\pi]$ we fix a 
    subset $\sNE(\tau,\pi)$ of $\sE(\tau,\pi)$, with the elements of $\sNE(\tau,\pi)$ being referred to 
    as \emph{normal embeddings} of $\tau$ into~$\pi$. Assume that $\sNE(\pi,\pi)=\sE(\pi,\pi)$, that is, 
    the unique embedding of $\pi$ into $\pi$ is normal. For each
    $\lambda\in[\sigma,\pi)$, define the two sets of 
    embeddings
    \begin{align*}
    \sNElo&=\bigcup_{\substack{\tau\in[\lambda,\pi]\\ |\tau| 
            \text{ odd}}}\sNE(\tau,\pi)\quad\text{and}\\
    \sNEle&=\bigcup_{\substack{\tau\in[\lambda,\pi]\\ |\tau| 
            \text{ even}}}\sNE(\tau,\pi).
    \end{align*}
    If for every $\lambda\in[\sigma,\pi)$ such that $\mobfn{\sigma}{\lambda}\neq 0$, we have the 
    identity
    \begin{equation}
    \left|\sNElo\right|= \left|\sNEle\right|,
    \label{eq-cancel}
    \end{equation}
    then $\mobfn{\sigma}{\pi} = (-1)^{|\pi|-|\sigma|}\NE(\sigma,\pi)$.
\end{proposition}
\begin{proof}
    The trick is to define the function $F(\tau)=(-1)^{|\pi|-|\tau|}\NE(\tau,\pi)$ and apply 
    Proposition~\ref{pro-form}. This yields
    \begin{align*}
    \mobfn{\sigma}{\pi}&= F(\sigma) - \sum_{\lambda\in 
        [\sigma,\pi)} \mobfn{\sigma}{\lambda}\sum_{\tau\in[\lambda,\pi]} F(\tau)\\
    &=F(\sigma) - \sum_{\lambda\in 
        [\sigma,\pi)} \mobfn{\sigma}{\lambda}\sum_{\tau\in[\lambda,\pi]} (-1)^{|\pi|-|\tau|}\NE(\tau,\pi)\\
    &=F(\sigma) -\sum_{\lambda\in 
        [\sigma,\pi)} \mobfn{\sigma}{\lambda}(-1)^{|\pi|}\bigl( 
    \left|\sNEle\right|-\left|\sNElo\right|\bigr)\\
    &=F(\sigma)\\
    &=(-1)^{|\pi|-|\sigma|}\NE(\sigma,\pi),
    \end{align*}
    as claimed.
\end{proof}

We remark that the general formula of 
Proposition~\ref{pro-form} can be useful even in situations 
where the more restrictive assumptions of 
Proposition~\ref{pro-normal} fail. An example of such an
application of Proposition~\ref{pro-form} 
appears in
Jel{\'{i}}nek, Kantor, Kyn{\v{c}}l and Tancer~\cite{Jelinek2020}.

\section{Permutations with opposing adjacencies}\label{sec-opposing}

In this section, we show that if a permutation 
has opposing adjacencies, then the value of the 
principal \mob function is zero. 
\begin{theorem}
    \label{theorem-PMF-opposing-adjacencies}
    If $\pi$ has opposing adjacencies, then $\mobp{\pi} = 0$.
\end{theorem}
For this theorem, we are able to give two proofs. One of them is based on the notion of 
diamond-tipped intervals, and the other uses the approach of normal embeddings. As both these 
approaches will later be adapted to more complicated settings, we find it instructive to include 
both proofs here.

\begin{proof}[Proof via diamond-tipped posets] For contradiction, suppose that the theorem fails, 
    and let $\pi$ be a shortest permutation with opposing adjacencies such that $\mobp{\pi}\neq0$. Since 
    $\pi$ has opposing adjacencies, there is a permutation $\tau$ and indices $i,j\le|\tau|$ such that 
    $\pi=\tau_{i,j}[12,21]$. Define $\phi=\tau_{i,j}[1,21]$ and $\phi'=\tau_{i,j}[12,1]$. 
    
    We claim that the interval $[1,\pi]$ can be transformed into a diamond-tipped interval with core 
    $(\phi,\phi',\tau)$ by deleting a set of \mob zeros from the interior of $[1,\pi]$. Since by 
    Fact~\ref{fac-del}, the deletion of \mob zeros does not affect the value of $\mobfn{1}{\pi}$, and 
    since diamond-tipped intervals have zero \mob function by Fact~\ref{fac-nd}, this claim will imply 
    that $\mobfn{1}{\pi}=0$, a contradiction.
    
    To prove the claim, note first that any permutation $\lambda\in[1,\pi)$ with opposing adjacencies is 
    a \mob zero, since $\pi$ is a minimal counterexample to the theorem. Choose any $\lambda\in[1,\pi)$. 
    Observe that if $\lambda$ has no up-adjacency, then $\lambda\le\phi$, and symmetrically, if 
    $\lambda$ has no down-adjacency, then $\lambda\le\phi'$. Thus, any $\lambda\in[1,\pi)$ not 
    belonging to $[1,\phi]\cup[1,\phi']$ has opposing adjacencies and can be deleted from $[1,\pi]$. 
    
    Next, suppose that a permutation $\lambda$ is in $[1,\phi]\cap[1,\phi']$ but not in $[1,\tau]$. 
    Observe that any permutation in $[1,\phi]\setminus[1,\tau]$ has a down-adjacency, while any 
    permutation in $[1,\phi']\setminus[1,\tau]$ has an up-adjacency. It follows that $\lambda$ 
    has opposing adjacencies and can again be deleted from $[1,\pi]$.
    
    After these deletions, the remaining poset is diamond-tipped with core $(\phi,\phi',\tau)$ as 
    claimed, hence $\mobfn{1}{\pi}=0$, a contradiction.
\end{proof}

\begin{proof}[Proof via normal embeddings] Suppose again that $\pi\in\cS_n$ is a shortest 
    counterexample. Suppose that $\pi$ has an up-adjacency at positions $i$, $i+1$, and a 
    down-adjacency at positions $j$, $j+1$. Note that the positions $i$, $i+1$, $j$ and $j+1$ are 
    all distinct, and in particular $n\ge 4$.
    
    We will say that an embedding $f\in\sE(*,\pi)$ is \emph{normal} if $\Img(f)$ is a superset of 
    $[n]\setminus\{i,j\}$. In other words, $\Img(f)$ contains all positions of $\pi$ with the possible 
    exception of $i$ and~$j$. Thus, there are four normal embeddings. 
    
    We will use Proposition~\ref{pro-normal} with the above notion of normal embeddings and with 
    $\sigma=1$. Clearly, we have $\sE(\pi,\pi)=\sNE(\pi,\pi)$. The main task is to verify equation 
    \eqref{eq-cancel}, that is, to show that for every $\lambda\in[1,\pi)$ such that $\mobp{\lambda}\neq0$ 
    we have $|\sNElo|=|\sNEle|$. To prove this identity, we let $\sNEl$ denote the set 
    $\sNElo\cup\sNEle$, and we will provide a sign-reversing involution on~$\sNEl$.
    
    Choose a $\lambda\in[1,\pi)$ with $\mobp{\lambda}\neq0$. It follows that $\lambda$ does 
    not have opposing adjacencies, otherwise it would be a counterexample shorter than~$\pi$. Without 
    loss of generality, assume that $\lambda$ has no up-adjacency. We will prove that the 
    $i$-switch operation $\Delta_i$ is a sign-reversing involution on~$\sNEl$. 
    
    It is clear that $\Delta_i$ is sign-reversing. We need to demonstrate that for every $f\in\sNEl$, the 
    embedding $g=\Delta_i(f)$ is again in $\sNEl$. It is clear that $g$ is normal. It remains to argue 
    that $\src(g)$ contains $\lambda$, or in other words, that there is an embedding of $\lambda$ into 
    $\pi$ contained in~$g$. Let $h$ be a (not necessarily normal) embedding of $\lambda$ into 
    $\pi$ contained in~$f$. If $i$ is not in $\Img(h)$, then $h$ is also contained in $g$, and 
    we are done. Suppose now that $i\in\Img(h)$. Then $i+1\not\in\Img(h)$, because $i$ and $i+1$ form an 
    up-adjacency in $\pi$ while $\lambda$ has no up-adjacency. We modify the embedding 
    $h$ so that the element mapped to $i$ will be mapped to $i+1$ instead, and the mapping of the 
    remaining elements is unchanged; let $h'$ be the resulting embedding (formally, we have 
    $\Delta_i(\Delta_{i+1}(h))=h'$). Since $i$ and $i+1$ form an adjacency in $\pi$, we have 
    $\src(h')=\src(h)=\lambda$. Since $i+1$ is in the image of all normal embeddings, we see that $h'$ 
    is contained in $g$, and so $g\in\sNEl$. This shows that $\Delta_i$ is the required sign-reversing 
    involution on $\sNEl$, verifying the assumptions of Proposition~\ref{pro-normal}.
    
    Proposition~\ref{pro-normal} then gives us that $\mobfn{1}{\pi}=(-1)^{n-1}\NE(1,\pi)$. Since every normal embedding into $\pi$ contains both $i+1$ and $j+1$ in its image, there is 
    clearly no normal embedding of 1 into $\pi$ and therefore we get $\mobfn{1}{\pi}=0$.
\end{proof}

\section{\texorpdfstring{A general construction of $\sigma$-annihilators}%
    {A general construction of sigma-annihilators}}
\label{sec-annihilator}

Let $\sigma$ be a fixed non-empty lower bound permutation (the case $\sigma=1$ being the most 
interesting). Recall that a permutation $\phi$ is a 
\emph{$\sigma$-zero} 
if $\mobfn{\sigma}{\phi}=0$, 
and $\phi$ is a 
\emph{$\sigma$-annihilator} 
if every permutation with an interval copy of
$\phi$ is a $\sigma$-zero. Clearly, any $\sigma$-annihilator is also a $\sigma$-zero. Our goal in 
this section is to present a general construction of an infinite family of $\sigma$-annihilators.

A permutation $\phi$ is 
\emph{$\sigma$-narrow}\extindex[permutation]{$\sigma$-narrow} 
if $\phi$ contains a permutation $\phi^-$ of size 
$|\phi|-1$ such that every permutation in the set $[1,\phi)\setminus [1,\phi^-]$ is a  
$\sigma$-annihilator. In this situation, we call $\phi^-$ a 
\emph{$\sigma$-core of $\phi$}\extindex{$\sigma$-core of $\phi$}. 

Note that if $\phi$ is $\sigma$-narrow with $\sigma$-core $\phi^-$, then the interval $[1,\phi]$ can be 
transformed into a narrow-tipped interval by a deletion of $\sigma$-annihilators. Our first goal is 
to show that, with a few exceptions, all $\sigma$-narrow permutations are $\sigma$-annihilators. 

\begin{proposition}\label{pro-narrow}
    If a permutation $\phi$ is $\sigma$-narrow with a $\sigma$-core $\phi^-$, and if $\sigma$ has no 
    interval copy of $\phi$ or of $\phi^-$, then $\phi$ is a $\sigma$-annihilator.
\end{proposition}
\begin{proof}
    Let $\phi$ be $\sigma$-narrow with a $\sigma$-core~$\phi^-$. Let $\pi$ be a permutation with an interval 
    copy of~$\phi$, that is, $\pi=\tau_i[\phi]$ for some $\tau$ and~$i$. We show that 
    $\mobfn{\sigma}{\pi}=0$. We may assume that $\sigma\le\pi$, otherwise $\mobfn{\sigma}{\pi}=0$ trivially. 
    Let $\pi^-$ be the permutation $\tau_i[\phi^-]$. Note that $\sigma\neq\pi$ and $\sigma\neq\pi^-$, 
    since $\sigma$ has no interval copy of~$\phi$ or of~$\phi^-$. 
    
    The key step of the proof is to show that any permutation in $[\sigma,\pi)\setminus [\sigma,\pi^-]$ 
    is a $\sigma$-zero. After we have proved this, we may use Fact~\ref{fac-del} to remove all such 
    $\sigma$-zeros from the interval $[\sigma,\pi]$ without affecting the value of $\mobfn{\sigma}{\pi}$; 
    note that $\sigma$ itself is clearly not a $\sigma$-zero, so it will not be removed, implying that 
    $\sigma<\pi^-$. After the removal of $[\sigma,\pi)\setminus [\sigma,\pi^-]$, the remainder of the 
    interval $[\sigma,\pi]$ is a narrow-tipped poset with core $\pi^-$, yielding 
    $\mobfn{\sigma}{\pi}=0$ by Fact~\ref{fac-nd}. 
    
    Therefore, to prove that $\mobfn{\sigma}{\pi}=0$ for a particular $\pi=\tau_i[\phi]$, it is enough to 
    show that all the permutations in $[\sigma,\pi)\setminus [\sigma,\pi^-]$ are $\sigma$-zeros. We 
    prove this by induction on~$|\tau|$.
    
    If $|\tau|=1$, we have $\pi=\phi$ and $\pi^-=\phi^-$. Then all the permutations in 
    $[1,\pi)\setminus[1,\pi^-]$ are $\sigma$-annihilators (and therefore $\sigma$-zeros) by definition 
    of $\sigma$-narrowness, and in particular, restricting our attention to permutations containing 
    $\sigma$, we see that all the permutations in $[\sigma,\pi)\setminus[\sigma,\pi^-]$ are 
    $\sigma$-zeros, as claimed.

    Suppose that $|\tau|>1$. Consider a permutation $\gamma \in [\sigma,\pi)\setminus [\sigma,\pi^-]$. 
    Since $\gamma$ is contained in $\pi=\tau_i[\phi]$, it can be expressed as 
    $\gamma=\btau_j[\bphi]$ for some $\emptyperm\le\bphi\le \phi$ and $1\le\btau\le\tau$, where $\btau$ 
    has an embedding into $\tau$ which maps $j$ to $i$. Note that $\bphi$ cannot be contained 
    in~$\phi^-$, because in such case we would have $\gamma\le \pi^-$. Moreover, if $\bphi=\phi$, then 
    necessarily $\btau<\tau$, and by induction $\gamma$ is a $\sigma$-zero. Finally, if $\bphi$ is in 
    $[1,\phi)\setminus [1,\phi^-]$, then $\bphi$ is a $\sigma$-annihilator by the $\sigma$-narrowness of 
    $\phi$, and hence $\gamma$ is a $\sigma$-zero.
\end{proof}

With the help of Proposition~\ref{pro-narrow}, we can now provide an explicit general construction 
of $\sigma$-annihilators.

\begin{proposition}\label{pro-sum}
    Let $\alpha$ and $\beta$ be non-empty permutations. Assume that $\sigma$ does not contain any 
    interval copy of a permutation of the form $\alpha'\oplus\beta'$ with $1\le\alpha'\le \alpha$ and 
    $1\le\beta'\le\beta$ 
    (in particular, $\sigma$ has no up-adjacency). Then $\alpha\oplus1\oplus\beta$ is 
    $\sigma$-narrow with $\sigma$-core $\alpha\oplus\beta$, and $\alpha\oplus1\oplus\beta$ is a 
    $\sigma$-annihilator. 
\end{proposition}

\begin{proof} We proceed by induction on $|\alpha|+|\beta|$. Suppose first that 
    $\alpha=\beta=1$. Then trivially $\alpha\oplus1\oplus\beta=123$ is $\sigma$-narrow with $\sigma$-core 
    $\alpha\oplus\beta=12$, since the set $[1,123)\setminus[1,12]$ is empty. Moreover, by assumption, 
    $\sigma$ has no interval copy of $12$, and therefore also no interval copy of $123$, 
    hence $123$ is a $\sigma$-annihilator by Proposition~\ref{pro-narrow}.
    
    Suppose now that $|\alpha|+|\beta|>2$. Define $\phi=\alpha\oplus1\oplus\beta$ and 
    $\phi^-=\alpha\oplus\beta$. To prove that $\phi$ is $\sigma$-narrow with $\sigma$-core $\phi^-$, we will show 
    that any permutation $\gamma\in[1,\phi)\setminus[1,\phi^-]$ is a $\sigma$-annihilator. Such a 
    $\gamma$ has the form $\alpha'\oplus 1\oplus\beta'$ for some $1\le\alpha'\le\alpha$ and 
    $1\le\beta'\le\beta$, with $|\alpha'|+|\beta'|<|\alpha|+|\beta|$; note that we here exclude the 
    cases $\alpha'=\emptyperm$ and $\beta'=\emptyperm$, because in these cases $\gamma$ would be 
    contained 
    in~$\phi^-$. By induction, $\gamma$ is $\sigma$-narrow, with $\sigma$-core $\gamma^-=\alpha'\oplus\beta'$. 
    Moreover, $\sigma$ has no interval isomorphic to $\gamma$ or $\gamma^-$: observe that if $\sigma$ 
    had an interval isomorphic to $\gamma$, it would also have an interval isomorphic to 
    $\alpha'\oplus1$, which is forbidden by our assumptions on~$\sigma$. Thus, we may apply 
    Proposition~\ref{pro-narrow} to conclude that $\gamma$ is a $\sigma$-annihilator, and in particular 
    $\phi$ is $\sigma$-narrow with $\sigma$-core $\phi^-$, as claimed. Proposition~\ref{pro-narrow} then gives us 
    that $\phi$ is a $\sigma$-annihilator.
\end{proof}

Focusing on the special case $\sigma=1$, which satisfies the assumptions of 
Proposition~\ref{pro-sum} trivially, we obtain the following result.

\begin{corollary}\label{cor-sum}
    For any non-empty permutations $\alpha$ and $\beta$, the permutation $\alpha\oplus1\oplus\beta$ is 
    an annihilator.
\end{corollary}

\section{The density of zeros}
\label{section-bounds-for-zn}

Our goal is to find an asymptotic positive lower bound on the proportion of permutations of 
length $n$ whose principal \mob function is 
zero. The key step is the following lemma.

\begin{lemma}\label{lem-incdec}
    Let $s_n$ be the number of permutations of size $n$ with opposing adjacencies. Then 
    \[
    \frac{s_n}{n!}=\left(1-\frac{1}{e}\right)^2+\cO\left(\frac{1}{n}\right).
    \]
\end{lemma}

\begin{proof}
    Let $a_n$ be the number of permutations of size $n$ that have no up-adjacency, and let $b_n$ be the 
    number of permutations of size $n$ that have neither an up-adjacency nor a down-adjacency.
    
    The numbers $a_n$ (sequence A000255 in the OEIS~\cite{sloane}) have already 
    been studied by Euler~\cite{Euler}, and it is known~\cite{RumneyPrimrose}
    that they satisfy $a_n/n! =e^{-1}+\cO(n^{-1})$. 
    
    The numbers $b_n$ (sequence A002464 in the OEIS~\cite{sloane}) satisfy the 
    asymptotics $b_n/n!= e^{-2} + \cO(n^{-1})$, which follows from the results of 
    Kaplansky~\cite{Kaplansky1945} (see also~Albert et al.~\cite{Albert2003}).
    
    We may now express the number $s_n$ of permutations with opposing adjacencies by inclusion-exclusion 
    as follows: we subtract from 
    $n!$ the number of permutations having no up-adjacency and the number 
    of permutations having no down-adjacency, and then we add back the 
    number of permutations having no adjacency at all. This yields $s_n=n!-2a_n 
    +b_n$, from which the lemma follows by the above-mentioned asymptotics of $a_n$ 
    and~$b_n$.
\end{proof}

Combining Theorem~\ref{theorem-PMF-opposing-adjacencies} with Lemma~\ref{lem-incdec} we obtain the 
following consequence, which is the main result of this section.

\begin{corollary}\label{cor-incdec}
    For a given $n$ and for $\pi$ a uniformly random permutation of length $n$, 
    the probability that $\mobp{\pi}=0$ is at least 
    \[
    \left(1-\frac{1}{e}\right)^2-\cO\left(\frac{1}{n}\right).
    \]
\end{corollary}

\section{More complicated examples}\label{sec-special}

We will now construct several specific examples of annihilators and annihilator pairs, which are 
not covered by the general results obtained in the previous sections. We begin with a construction 
of four new annihilator pairs, which we will later use to construct new annihilators.

\begin{theorem}\label{thm-pair}
    The two permutations $213$ and $2431$ form an annihilator pair.
\end{theorem}
\begin{proof}
    Our proof is based on the concept of normal embeddings and follows a similar structure as the normal 
    embedding proof of Theorem~\ref{theorem-PMF-opposing-adjacencies}.
    
    Suppose for contradiction that there is a permutation $\pi$ that contains an interval 
    isomorphic to $213$ as well as an interval isomorphic to~$2431$, and that $\mobp{\pi}\neq0$.
    Fix a smallest possible $\pi$, and let $n$ be its length. Note that an interval 
    isomorphic to $213$ is necessarily disjoint from an interval isomorphic to~$2431$, and in 
    particular, $n\ge 7$.
    
    Let $i$, $i+1$ and $i+2$ be three positions of $\pi$ containing an interval copy of $213$, and let 
    $j$, $j+1$, $j+2$ and $j+3$ be four positions containing an interval copy of~$2431$. We will apply 
    the approach of Proposition~\ref{pro-normal}, with $\sigma=1$. We will say that an embedding 
    $f\in\sE(*,\pi)$ is \emph{normal} if $\Img(f)$ is a superset of $[n]\setminus\{i+2, j+2, j+3\}$.  
    Informally, the image of a normal embedding contains all the positions of $\pi$, except possibly 
    some of the three positions that correspond to the value $3$ of $213$ or the values $3$ and $1$ 
    of~$2431$ in the chosen interval copies of $213$ and $2431$,
    as shown in Figure~\ref{figure-intervals-in-213-2431}. In particular, there are eight normal 
    embeddings. 
    \begin{figure}[!ht]
        \begin{center}
            \begin{tikzpicture}[scale=0.3]
            \sqat{3}{0}{0};
            \normaldot{(1,2)};
            \normaldot{(2,1)};
            \opendot{(3,3)};
            \sqat{4}{5}{5};
            \normaldot{(6,7)};
            \normaldot{(7,9)};
            \opendot{(8,8)};
            \opendot{(9,6)};
            \path [draw=gray,wavy]	(-1,4.5) -- (10,4.5);
            \path [draw=gray,wavy]	(4.5,-1) -- (4.5,10);
            \end{tikzpicture}
        \end{center}
        \caption{The intervals $213$ and $2431$ in Theorem~\ref{thm-pair}.  Normal embeddings may omit some of the hollow points.}
        \label{figure-intervals-in-213-2431}
    \end{figure}
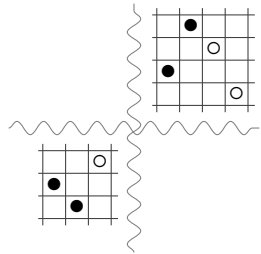
    
    We now verify the assumptions of Proposition~\ref{pro-normal}. We obviously have 
    $\sNE(\pi,\pi)=\sE(\pi,\pi)$. The main task is to verify, for a given $\lambda\in[1,\pi)$ with 
    $\mobp{\lambda}\neq0$, the identity~\eqref{eq-cancel} of Proposition~\ref{pro-normal}, that is, the 
    identity $\left|\sNElo\right|= \left|\sNEle\right|$. 
    
    Fix a $\lambda\in[1,\pi)$ such that $\mobp{\lambda}\neq0$, and let $\sNEl$ be the set 
    $\sNElo\cup \sNEle$. We will describe a sign-reversing involution $\Phi_\lambda$ on $\sNEl$. 
    The involution $\Phi_\lambda$ will always be equal to a switch operation $\Delta_k$, where the 
    choice of $k$ will depend on~$\lambda$.
    
    Suppose first that $\lambda$ does not contain any down-adjacency. We claim that 
    $\Delta_{j+2}$ is an involution on the set $\sNEl$. To see this, choose $f\in\sNEl$ and define 
    $g=\Delta_{j+2}(f)$. It is clear that $g$ is a normal embedding. 
    
    To prove that $g$ belongs to $\sNEl$, it remains to show that $\src(g)$ contains~$\lambda$, or  
    equivalently, that there is an embedding of $\lambda$ into $\pi$ that is contained in~$g$. 
    Let $h$ be an embedding of $\lambda$ into $\pi$ which is contained in~$f$. If $j+2\not\in\Img(h)$, 
    then $h$ is also contained in $g$ and we are done. 
    
    Suppose then that $j+2\in\Img(h)$. This means that $j+1$ is not in $\Img(h)$, because $\pi$ has a 
    down-adjacency at positions $j+1$ and $j+2$, while $\lambda$ has no down-adjacency. We 
    now modify $h$ in such a way that the element previously mapped to $j+2$ will be mapped to $j+1$, 
    while the mapping of the remaining elements remains unchanged. Let $h'$ be the embedding obtained 
    from $h$ by this modification; formally, we have $h'=\Delta_{j+1}(\Delta_{j+2}(h))$. Since the two 
    elements $\pi_{j+1}$ and $\pi_{j+2}$ form an adjacency, we have $\src(h')=\src(h)=\lambda$. 
    Moreover, $h'$ is contained in $g$ (recall that $g$ is normal, and therefore $\Img(g)$ contains 
    $j+1$). Consequently, $g$ is in $\sNEl$, as claimed.
    
    We now deal with the case when $\lambda$ contains a down-adjacency. Since 
    $\mobp{\lambda}\neq0$, it follows by Theorem~\ref{theorem-PMF-opposing-adjacencies} that $\lambda$ 
    has no up-adjacency. We distinguish two subcases, depending on whether $\lambda$ 
    contains an interval copy of~$2431$.
    
    Suppose that $\lambda$ contains an interval copy of~$2431$. We will prove that in this case, 
    $\Delta_{i+2}$ is a sign-reversing involution on~$\sNEl$. We begin by observing that $\lambda$ has 
    no interval copy of $213$, otherwise $\lambda$ would be a counterexample to Theorem~\ref{thm-pair}, 
    contradicting the minimality of~$\pi$. Fix again an embedding $f\in\sNEl$, and define 
    $g=\Delta_{i+2}(f)$. As in the previous case, $g$ is clearly normal, and we only need to show that 
    there is an embedding of $\lambda$ into $\pi$ contained in~$g$. Let $h$ be an embedding of $\lambda$ 
    into $\pi$ contained in $f$. If $i+2\not\in\Img(h)$, then $h$ is contained in $g$ and we are done, 
    so suppose $i+2\in\Img(h)$. If at least one of the two positions $i$ and $i+1$ belongs to $\Img(h)$, 
    then $\lambda$ contains an up-adjacency or an interval copy of $213$, contradicting our 
    assumptions. Therefore, we can modify $h$ so that the element mapped to $i+2$ is mapped to $i$ 
    instead, obtaining an embedding of $\lambda$ contained in $g$ and showing that~$g\in\sNEl$.
    
    Finally, suppose that $\lambda$ has no interval copy of $2431$. In this case, we prove that 
    $\Delta_{j+3}$ is the required involution on $\sNEl$. As in the previous cases, we fix $f\in\sNEl$, 
    define $g=\Delta_{j+3}(f)$, and let $h$ be an embedding of $\lambda$ contained in~$f$; we again 
    want to modify $h$ into an embedding $\lambda$ contained in~$g$. Let $\alpha$ be the subpermutation 
    of $\lambda$ formed by those positions that are mapped into the set $J=\{j,j+1,j+2,j+3\}$ by~$h$. 
    Recall that the positions in $J$ induce an interval copy of $2431$ in~$\pi$. In particular, 
    $\alpha\le 2431$, and $\lambda$ has an interval copy of~$\alpha$. We know that $\alpha\neq 2431$, 
    since we assume that $\lambda$ has no interval copy of $2431$. Also, $\alpha\neq 321$, since $321$ 
    is an annihilator by Corollary~\ref{cor-sum}, while $\mobp{\lambda}\neq0$. Finally, $\alpha \neq 
    231$, since $\lambda$ has no up-adjacency. This implies that $\alpha\le 132$, and we can 
    modify $h$ so that all the positions originally mapped into $J$ will get mapped into 
    $J\setminus\{j+3\}$, obtaining an embedding of $\lambda$ into $\pi$ contained in $g$.
    
    Having thus verified the assumptions of Proposition~\ref{pro-normal}, we can conclude that 
    $\mobp{\pi}=(-1)^{|\pi|-1}\NE(1,\pi)=0$, a contradiction.
\end{proof}

The following three results are 
proved using similar methods to those used in the
proof of Theorem~\ref{thm-pair}.
\begin{theorem}\label{thm-pair2}
    The permutations $2143$ and $2431$ form an annihilator pair.
\end{theorem}
\begin{theorem}\label{thm-pair3}
    The permutations $312$ and $23514$ form an annihilator pair.
\end{theorem}
\begin{theorem}\label{thm-pair4}
    The permutations $25134$ and $23514$ form an annihilator pair.
\end{theorem}
We omit the proofs here, 
as they were not included in the published paper~\cite{Brignall2020}
on which this chapter is based.
They can be found in
\url{https://arxiv.org/abs/1810.05449v1}~\cite{OppAdjPreviousVersion}.

\begin{theorem}\label{thm-annihil}
    %
    %
    Each of the three permutations $215463$, $236145$ and $214653$ is a \mob annihilator.
\end{theorem}
\begin{proof}
    We first present the proof for the permutation $215463$.
    Let $\alpha=215463$, $\beta=\alpha_1[\epsilon]=14352$, $\beta'=\alpha_6[\epsilon]=21435$ and $\gamma=\alpha_{1,6}[\epsilon,\epsilon]=1324$. From 
    Figure~\ref{fig-annihil} (left) we see that, after the removal of the annihilators $\alpha_3[\epsilon], \alpha_4[\epsilon]$ and $\alpha_5[\epsilon]$, the 
    interval $[1,\alpha]$ becomes diamond-tipped with core $(\beta,\beta',\gamma)$. Hence by Facts~\ref{fac-del} and~\ref{fac-nd} we have $\mu[1,\alpha]=0$.
    
    Let $\pi$ be a permutation of the form $\tau_i[\alpha]$ for some $\tau$ and $i\le|\tau|$. We will 
    show, by induction on $|\tau|$, that $\pi$ is a zero. The case $|\tau|=1$ has been proved in the 
    previous paragraph.
    
    Assume that $|\tau|>1$. We will demonstrate that we can remove some zeros from the interval $[1,\pi]$ to 
    end up with a diamond-tipped interval with core $(\tau_i[\beta],\tau_i[\beta'],\tau_i[\gamma])$. 
    Choose a $\lambda\in[1,\pi)$. We can then write $\lambda$ as $\lambda=\btau_j[\alpha^*]$ for some 
    $\btau\le \tau$ and some (possibly empty) $\alpha^*\le\alpha$, where $\btau$ has an embedding into 
    $\tau$ mapping $j$ to~$i$. 
    
    If $\alpha^*$ is an annihilator, then $\lambda$ is a zero and can be removed. If $\alpha^*=\alpha$, 
    then $|\btau|<|\tau|$, and by induction, $\lambda$ is a zero and can be removed. In all the other 
    cases, we have $\alpha^*\le \beta$ or $\alpha^*\le \beta'$, and in particular, $\lambda$ belongs to 
    $[1,\tau_i[\beta]]\cup[1,\tau_i[\beta']]$.
    
    Suppose now that $\lambda$ is in $[1,\tau_i[\beta]]\cap[1,\tau_i[\beta']]$ but not in 
    $[1,\tau_i[\gamma]]$. Since $\lambda\le\tau_i[\beta]$, we can write it as 
    $\lambda=\tau_j^L[\beta^L]$, for some $\tau^L\le\tau$ and $\beta^L\le\beta$, where $\tau^L$ has an 
    embedding into $\tau$ mapping $j$ to~$i$. Since $\lambda\not\le\tau_i[\gamma]$, we know that 
    $\beta^L\not\le\gamma$. This means that 
    $\beta^L\in[1,\beta]\setminus[1,\gamma]=\{14352,\allowbreak3241,1342,231\}$. 
    Similarly, $\lambda\in[1,\tau_i[\beta']]\setminus[1,\tau_i[\gamma]]$ means that $\lambda$ can be 
    written as 
    $\lambda=\tau_k^R[\beta^R]$, with $\beta^R\in\{21435,2143\}$. Since $\lambda$ has an interval copy 
    of $\beta^L$ 
    as well as an interval copy of $\beta^R$, 
    Theorem~\ref{theorem-PMF-opposing-adjacencies} shows that $\lambda$ is a zero if $\beta^L \in \{1342,231\}$, and Theorem~\ref{thm-pair2} shows that $\lambda$ is a zero if $\beta^L \in \{14352,3241\}$ (using that $3241$ is a diagonal reflection of $2431$). Therefore $\lambda$ can be removed. 
    
    After the removal described above, $[1,\pi]$ is transformed into a diamond-tipped interval, showing 
    that $\pi$ is a zero.
    
    The arguments for the other two permutations are completely analogous.  For
    $236145$ we have $\alpha=236145$, $\beta=25134$, $\beta'=23514$, $\gamma=2413$, $\beta^L\in \{25134,1423\}$ and $\beta^R\in\{23514,2314\}$, and use Theorems~\ref{thm-pair},~\ref{thm-pair3} and~\ref{thm-pair4}.
    For $214653$ we have $\alpha=214653$, $\beta=13542$, $\beta'=2143$, $\gamma=132$, $\beta^L\in 
    \{13542, 2431,\allowbreak 1342, 231\}$ and $\beta^R\in\{2143,213\}$, and use 
    Theorems~\ref{theorem-PMF-opposing-adjacencies}, \ref{thm-pair} and~\ref{thm-pair2}.
\end{proof}

\begin{figure}[!ht]
    %
    %
    %
    \centering
    \begin{subfigure}[t]{0.32\textwidth}
        \begin{tikzpicture}[scale=0.15]
        \tikzmath{\y=0;};
        \permnode{ 0}{\y}{1};
        \tikzmath{\y=8;};
        \permnode{-4}{\y}{12};
        \permnode{ 4}{\y}{21};
        \tikzmath{\y=\y+9;};
        \permnode{-6}{\y}{231};
        \permnode{ 0}{\y}{213};
        \permnode{ 6}{\y}{132};
        \tikzmath{\y=\y+10;};
        \permnode{-10.5}{\y}{3241};
        \permnode{ -3.5}{\y}{1342};
        \permnode{  3.5}{\y}{1324};
        \permnode{ 10.5}{\y}{2143};
        \tikzmath{\y=\y+11;};
        \permnode{ -6}{\y}{14352};
        \permnode{  6}{\y}{21435};
        \tikzmath{\y=\y+12;};
        \permnode{  0}{\y}{215463};
        \link{12}{1};
        \link{21}{1};
        \link{231}{12};
        \link{231}{21};
        \link{213}{12};
        \link{213}{21};
        \link{132}{12};
        \link{132}{21};
        \link{3241}{231};
        \link{3241}{213};
        \link{1342}{231};
        \link{1342}{132};
        \link{1324}{213};
        \link{1324}{132};
        \link{2143}{213};
        \link{2143}{132};
        \link{14352}{3241};
        \link{14352}{1342};
        \link{14352}{1324};
        \link{21435}{1324};
        \link{21435}{2143};
        \link{215463}{14352}
        \link{215463}{21435}
        \end{tikzpicture}					
    \end{subfigure}
    \quad    
    \begin{subfigure}[t]{0.26\textwidth}
        \begin{tikzpicture}[scale=0.15]
        \tikzmath{\y=0;};
        \permnode{ 0}{\y}{1};
        \tikzmath{\y=8;};
        \permnode{-4}{\y}{12};
        \permnode{ 4}{\y}{21};
        \tikzmath{\y=\y+9;};
        \permnode{-9}{\y}{132};
        \permnode{-3}{\y}{312};
        \permnode{ 3}{\y}{213};
        \permnode{ 9}{\y}{231};
        \tikzmath{\y=\y+10;};
        \permnode{ -8}{\y}{1423};
        \permnode{  0}{\y}{2413};
        \permnode{  8}{\y}{2314};
        \tikzmath{\y=\y+11;};
        \permnode{ -6}{\y}{25134};
        \permnode{  6}{\y}{23514};
        \tikzmath{\y=\y+12;};
        \permnode{  0}{\y}{236145};
        \link{12}{1};
        \link{21}{1};
        \link{132}{12};
        \link{132}{21};
        \link{231}{12};
        \link{231}{21};
        \link{213}{12};
        \link{213}{21};
        \link{312}{12};
        \link{312}{21};
        \link{1423}{132};
        \link{1423}{312};
        \link{2413}{132};
        \link{2413}{231};
        \link{2413}{213};
        \link{2413}{312};
        \link{2314}{213};
        \link{2314}{231};
        \link{25134}{1423};
        \link{25134}{2413};
        \link{23514}{2413};
        \link{23514}{2314};
        \link{236145}{25134}
        \link{236145}{23514}
        \end{tikzpicture}					
    \end{subfigure}
    \quad      
    \begin{subfigure}[t]{0.26\textwidth}
        \begin{tikzpicture}[scale=0.15]
        \tikzmath{\y=0;};
        \permnode{ 0}{\y}{1};
        \tikzmath{\y=8;};
        \permnode{-4}{\y}{12};
        \permnode{ 4}{\y}{21};
        \tikzmath{\y=\y+9;};
        \permnode{-6}{\y}{231};
        \permnode{ 0}{\y}{132};
        \permnode{ 6}{\y}{213};
        \tikzmath{\y=\y+10;};
        \permnode{-8}{\y}{2431};
        \permnode{ 0}{\y}{1342};
        \permnode{ 8}{\y}{2143};
        \tikzmath{\y=\y+11;};
        \permnode{-6}{\y}{13542};
        \tikzmath{\y=\y+12;};
        \permnode{ 0}{\y}{214653};
        \link{12}{1};
        \link{21}{1};
        \link{231}{12};
        \link{231}{21};
        \link{213}{12};
        \link{213}{21};
        \link{132}{12};
        \link{132}{21};
        \link{2431}{231};
        \link{2431}{132};
        \link{1342}{231};
        \link{1342}{132};
        \link{2143}{132};
        \link{2143}{213};
        \link{13542}{2431};
        \link{13542}{1342};
        \link{214653}{13542}
        \link{214653}{2143}
        \end{tikzpicture}		
    \end{subfigure}
    \caption{The three annihilators from Theorem~\ref{thm-annihil}, and the posets of their           
        subpermutations. The figures omit the permutations with opposing adjacencies, as well as 
        the permutations with an interval copy of a permutation of the form 
        $\alpha\oplus1\oplus\beta$.}
    \label{fig-annihil}
\end{figure}
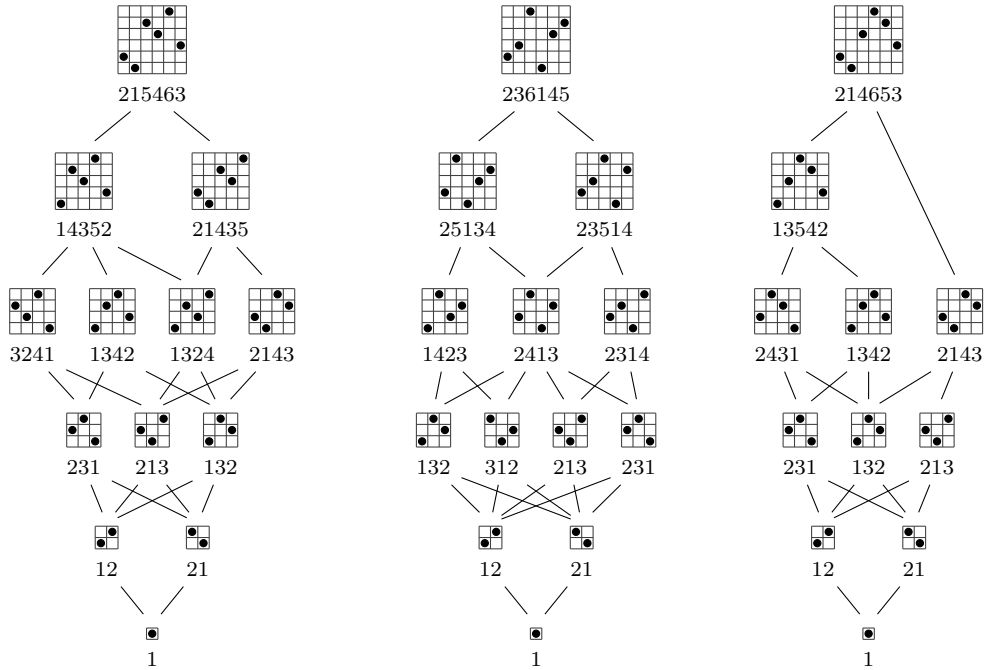

The annihilator $215463$ of Theorem~\ref{thm-annihil} can be written as a sum of two intervals, 
namely $215463=21\oplus 3241$. One might wonder whether the two summands are in fact an annihilator 
pair. This, however, is not the case, as evidenced by the permutation $32417685=3241\oplus3241$, which 
is not a \mob zero. An analogous example applies to $214653=21\oplus 2431$.

In the proof of Theorem~\ref{thm-annihil}, it was crucial that for each 
$\alpha\in\{215463,\allowbreak236145,\allowbreak214653\}$, the interval $[1,\alpha]$ becomes 
diamond-tipped after the removal of some annihilators. However, this property alone is not 
sufficient to make a permutation $\alpha$ an annihilator. Consider, for instance, the permutation 
$\alpha=214635$. We may routinely check that by removing some annihilators, the interval 
$[1,\alpha]$ can be made diamond-tipped with core $(\beta=13524,\beta'=21435,\gamma=1324)$. This 
implies that $\alpha$ is a \mob zero by Facts~\ref{fac-del} and~\ref{fac-nd}; however, it does not 
imply that $\alpha$ is an annihilator. In fact, $\alpha$ is not an annihilator, as demonstrated by 
the permutation
\begin{align*}
\pi&=582741936_{2,4,5}[\beta,\alpha,\beta']\\
&=9,17,19,21,18,20,2,12,11,14,16,13,15,5,4,7,6,8,1,22,3,10,
\end{align*}
whose principal \mob function is 1, not 0. This example also shows that not all \mob zeros are 
annihilators.

In fact, among permutations of size at most 6, there are up to symmetry four \mob zeros that are 
not annihilators. Apart from the permutation $214635$ pointed out above, there are three 
more examples: $235614$, $254613$ and $465213$. To see that these three permutations are not 
annihilators, it suffices to check that for any $\alpha\in\{235614,254613,465213\}$, the 
permutation $24153_{2}[\alpha]$ has non-zero principal \mob function. We verified, with the help 
of a computer, that all the \mob zeros of size at most 6 that are not symmetries of the four 
examples above can be shown to be annihilators by our results. 
This data is available at \url{https://iuuk.mff.cuni.cz/~jelinek/mf/zeros.txt}.

\section{Concluding remarks}

Given Theorem~\ref{theorem-PMF-opposing-adjacencies},
it is natural to wonder if we can find a  
similar result that applies 
to a permutation with multiple adjacencies, 
but no opposing adjacencies.
One difficulty here is that 
there are permutations that have 
multiple adjacencies, and do not
have
opposing adjacencies, where 
the principal \mob function value 
is non-zero.  
As an example, any permutation 
$\pi = 2,1,4,3,\dots, 2k,2k-1 = \nsums{k}{21}$
has
$\mobp{\pi} = -1$
by the results of Burstein et al.~\cite[Corollary 3]{Burstein2011}.

Let $d_n$ be the ``density of zeros'' of the \mob function, that is, the probability that 
$\mobp{\pi}=0$ for a uniformly random permutation $\pi$ of size~$n$. The asymptotic behaviour 
of~$d_n$ is still elusive.

\begin{problem}
    Does the limit $\lim_{n\to\infty} d_n$ exist? And if it does, what is its value?
\end{problem}

Corollary~\ref{cor-incdec} implies that $\liminf_{n\to\infty} d_n \ge (1-1/e)^2\ge 0.3995$. We 
have no upper bound on $d_n$ apart from the trivial bound $d_n\le 1$, but computational data 
suggest that simple permutations very often (though not always) have non-zero principal \mob 
function. 
Since a random permutation is simple with probability approaching 
$1/e^2$~\cite{Albert2003}, this would suggest that $\limsup_{n\to\infty} d_n$ is at 
most~$1-1/e^2\approx 0.8647$.

\begin{table}[!ht]
    \[
    \begin{array}{lcr}
    \begin{array}{lr}
    n & d_n \\
    \midrule
    1 & 0.0000 \\ 
    2 & 0.0000 \\
    3 & 0.3333 \\
    4 & 0.4167 \\ 
    5 & 0.4833 \\
    6 & 0.5361 \\
    7 & 0.5742 \\
    \end{array}
    & \phantom{xxx} &
    \begin{array}{lr}
    n & d_n \\
    \midrule
    8 & 0.5942 \\
    9 & 0.6019 \\
    10 & 0.6040 \\
    11 & 0.6034 \\
    12 & 0.6021 \\
    13 & 0.6006
    \end{array}
    \end{array}
    \]
    \caption{The density of \mob zeros among permutations of length $n$, with $n = 1, \ldots, 13$.}
    \label{table-zn-one-to-thirteen}
\end{table}

Table~\ref{table-zn-one-to-thirteen} lists the values of $d_n$ for  $n=1, \ldots, 13$.
The values are based on data supplied by Smith~\cite{Smith2018} for $1 \leq n \leq 
9$, and calculations performed by the author of this thesis.
Data files with the values of the
principal \mob function for all 
permutations of length twelve or less are available from
\url{https://doi.org/10.21954/ou.rd.7171997.v2}.
Based on this somewhat limited numeric 
evidence, we make the following conjecture:

\begin{conjecture}
    \label{conjecture-PMF-zero-61} The values $d_n$ are
    bounded from above by 0.6040.
\end{conjecture}

It is natural to look for further ways to identify \mob zeros
and \mob annihilators.
Characterizing all the \mob zeros would be an ambitious goal, 
since $\mobp{\pi}$ might be zero as a result of 
``accidental'' cancellations with no deeper structural 
significance for~$\pi$. 

An 
\emph{annihilator multiset}\extindex{annihilator multiset}
is a multiset of permutations
$A = \{ \alpha_1, \ldots, \alpha_n \}$
such that any permutation $\pi$ that contains disjoint interval
copies of the permutations 
$\alpha_1, \ldots, \alpha_n$
has $\mobp{\pi} = 0$.

If $A = \{ \alpha_1, \ldots, \alpha_n \}$
and
$B = \{ \beta_1, \ldots, \beta_m \}$
are annihilator multisets, 
then we say that $A$ \emph{contains} $B$
if $A \neq B$ and
we can find the elements of $B$
as disjoint interval copies in the elements of $A$.
An annihilator multiset $A$ is 
\emph{minimal}\extindex{minimal annihilator multiset}
if there is no annihilator multiset contained in $A$.

Using Corollary 3 of~\cite{Burstein2011},
which implies $\mobp{\pi} = \mobp{\pi \oplus \pi}$
for $\pi \neq 1$,
it is simple to show that the permutations
in a minimal annihilator multiset are, in fact,
all distinct, and so 
we can refer to 
\emph{minimal annihilator sets}\extindex{minimal annihilator sets}
of permutations.

\begin{problem}
    Which permutations are \mob annihilators? 
    Are there infinitely many minimal annihilator sets 
    that contain just one element, 
    and are not of the form $\alpha\oplus1\oplus\beta$?
\end{problem}

It seems likely to us that the proofs of Theorems~\ref{thm-pair}--\ref{thm-pair4} might be extended to give several more annihilator pairs, such as $(312,235614)$. However, we do not see any general pattern in these examples yet.

\begin{problem}
    Are there infinitely many minimal annihilator sets with two elements?
\end{problem}

\begin{problem}
    Are there any minimal annihilator sets with more than two elements?
\end{problem}

\section{Chapter summary}

Prior to the publication
of the paper on which this chapter is based,
the proportion of 
permutations where we had a (computationally) simple
way to determine the value of the principal \mob function
was, asymptotically, zero.  
This chapter presents two main results.

The first, 
Theorem~\ref{theorem-PMF-opposing-adjacencies},
tells us that if
a permutation has opposing adjacencies, then 
the value of the principal \mob function 
is zero.  
It is possible to determine if 
a permutation has opposing adjacencies
in time proportional to the length of
the permutation, so this is a 
test that is simple to implement.
This is potentially useful
to anyone wanting to compute values
of the principal \mob function.

The second main result
comes from
Corollary~\ref{cor-incdec}.
In essence this gives us that
the proportion of permutations
where the principal \mob function 
is zero is at least \zpmfr.
From a computational perspective,
this implies that there are significant benefits
from using 
Theorem~\ref{theorem-PMF-opposing-adjacencies}
when determining the value of the 
principal \mob function, 
as we have a linear time algorithm
which, asymptotically,
gives a positive result for nearly 40\%
of permutations.

\subsection{\texorpdfstring{Improving the lower bound for the density of zeros, $d_n$}
    {Improving the lower bound for the density of zeros}}

Conjecture~\ref{conjecture-PMF-zero-61}
suggests, based on some
rather limited numerical evidence, that 
the density of zeros, $d_n$, of the 
principal \mob function 
is bounded above by $0.6040$.
We know, from Corollary~\ref{cor-incdec},
that asymptotically $d_n$ is bounded below
by \zpmfr.  
One area for further research would be to consider
if we can find better bounds on the behaviour of $d_n$.
One difficulty with this is that we would need to
find a result
where the number of permutations covered by the result
is proportional to $n!$, where $n$ is the length 
of the permutation.  Any relationship that was 
less than factorial (for example, 
exponential or polynomial) would
mean that the proportion of 
permutations covered
would be, asymptotically, zero.  

In the concluding remarks above, 
we note that 
it is natural to wonder if we can find a  
similar result that applies 
where a permutation has multiple adjacencies, 
but no opposing adjacencies.
Such a result
would have to account for the 
permutations that have 
multiple non-opposing adjacencies where 
the principal \mob function value 
is non-zero.  
Table~\ref{table-count-of-non-opp-adj-permutations}
shows, for lengths $4, \ldots, 12$, the number of 
permutations with multiple non-opposing adjacencies
broken down by whether
the value of the 
principal \mob function is zero or not.
\begin{table}
    \[
    \begin{array}{lrr}
    \toprule
    \text{Length} 
    &  =0 & \neq 0 \\
    \midrule
    4 &        6 &        2 \\
    5 &       30 &        8 \\
    6 &      170 &       38 \\
    7 &     1154 &      212 \\
    8 &     8954 &     1502 \\
    9 &    78006 &    13088 \\
    10 &   757966 &   130066 \\
    11 &  8132206 &  1436296 \\
    12 & 95463532 & 17403612 \\
    \bottomrule
    \end{array}
    \]
    \caption{Number of permutations with non-opposing adjacencies,
        classified by the value of the 
        principal \mob function.}
    \label{table-count-of-non-opp-adj-permutations}
\end{table}

This suggests that it might be possible to find
a result similar to
Theorem~\ref{theorem-PMF-opposing-adjacencies}
for some or all of these cases, although,
as noted,
any such result will clearly need some 
additional criteria that will exclude
permutations that have a non-zero
principal \mob function value.

We remark that the non-opposing adjacency case
may be important because 
the proportion of permutations
of length $n$ that have non-opposing adjacencies
is, asymptotically, non-zero, as we show now.
\begin{theorem}
    The proportion of permutations
    that have non-opposing adjacencies is,
    asymptotically, bounded below
    by 0.1944.
\end{theorem}
\begin{proof}
    We find a lower bound 
    by counting 
    permutations that have
    non-opposing adjacencies.
    
    We will need to use
    \begin{theorem}[{%
            Albert, Atkinson and Klazar~\cite[Theorem 5]{Albert2003}}]
        \label{AAK-number-of-simple-permutations}
        The number of simple permutations of length $n$,
        $S(n)$,
        is given by
        \[
        S(n) 
        =
        \dfrac{n!}{\e^2}
        \left(
        1 - \dfrac{1}{n} + \dfrac{2}{n(n-1)} +O(n^{-3})
        \right).
        \]
    \end{theorem}
    
    Let $Z^\prime(n)$ be the proportion
    of permutations of length $n$
    that have non-opposing adjacencies.
    Let $n \geq 6$ be an integer;
    and let $k$ be an integer in the
    range $2, \ldots, \lfloor n/2 \rfloor$.
    Let $\sigma$ be a simple permutation with length $n-k$.
    We will count the number of ways 
    we can inflate $\sigma$ with $k$ adjacencies
    to obtain a permutation with length $n$
    that has non-opposing adjacencies.
    We can choose the positions to inflate in 
    $\binom{n-k}{k}$ ways.  The positions chosen
    can be inflated by either 
    $12$ or $21$, so there are just $2$
    distinct inflations by adjacencies.
    It follows that 
    the number of ways to inflate $\sigma$
    that result in a permutation 
    with non-opposing adjacencies is given by
    \[
    2 \binom{n-k}{k}.
    \]
    Since we are inflating
    simple permutations, it follows from
    Albert and Atkinson~\cite[Proposition 2]{Albert2005}
    that the inflations are unique.
    
    For an inflation to contain a non-opposing
    adjacency, we need to inflate at least two points.
    Further, to obtain a permutation
    of length $n$ by inflating with adjacencies
    we can, at most, inflate $\lfloor n/2 \rfloor$ positions.
    Now using Theorem~\ref{AAK-number-of-simple-permutations}
    we can say that  
    \[
    Z^\prime(n) 
    \geq
    \dfrac{1}{n!}
    \sum_{k=2}^{\lfloor n/2 \rfloor}
    S(n-k)
    2 \binom{n-k}{k}.
    \]
    
    Since we are only interested in the
    asymptotic behaviour of $Z^\prime(n)$,
    we can assume that $n > 20$,
    and so we write
    \begin{align*}
    \lim_{n \to \infty}
    Z^\prime(n) 
    & \geq
    \lim_{n \to \infty}    
    \dfrac{1}{n!}
    \sum_{k=2}^{\lfloor n/2 \rfloor}
    S(n-k)
    2 \binom{n-k}{k} \\
    & \geq
    \lim_{n \to \infty}    
    \dfrac{1}{n!}
    \sum_{k=2}^{10}
    S(n-k)
    2 \binom{n-k}{k}  \\
    & \geq 0.1944.
    \end{align*}
\end{proof}
Thus if we could show that
asymptotically, some fixed proportion of 
the permutations of length $n$ with
non-opposing adjacencies
were all \mob zeros,
then we could improve the 
bounds given by 
Corollary~\ref{cor-incdec}.

\subsection{Extending the ``opposing adjacencies'' theorem}

It is natural to ask if we can extend 
Theorem~\ref{theorem-PMF-opposing-adjacencies}
to handle cases where 
the lower bound of the 
interval is not $1$.
This is not possible in general, as
if we take any permutation $\sigma \neq 1$,
and inflate any two distinct points in positions
$\ell$ and $r$ by $12$ and $21$
respectively, then
$\pi = \inflatesome{\sigma}{\ell,r}{12,21}$
has opposing adjacencies,
but 
$\mobfn{\sigma}{\pi} = 1$, as
can
be deduced from 
Figure~\ref{figure-extending-PFM-oppadj-fails}.
\begin{figure}
    \begin{center}
        \begin{tikzpicture}[xscale=1,yscale=1]
        \node (n12-21) at ( 0, 3) {$\inflatesome{\sigma}{\ell,r}{12,21}$};
        \node (n12-1)  at (-2, 2) {$\inflatesome{\sigma}{\ell,r}{12,1}$};  
        \node (n1-21)  at ( 2, 2) {$\inflatesome{\sigma}{\ell,r}{1,21}$};  
        \node (n1-1)   at ( 0, 1) {$\inflatesome{\sigma}{\ell,r}{1,1}$};  
        \draw (n12-21) -- (n1-21);
        \draw (n12-21) -- (n12-1);
        \draw (n12-1) --  (n1-1);
        \draw (n1-21) --  (n1-1);
        \end{tikzpicture}
        \qquad
        \begin{tikzpicture}[xscale=1,yscale=1]
        \node (n12-21) at ( 0, 3) {$1$};
        \node (n12-1)  at (-2, 2) {$-1$};  
        \node (n1-21)  at ( 2, 2) {$-1$};  
        \node (n1-1)   at ( 0, 1) {$1$};  
        \draw (n12-21) -- (n1-21);
        \draw (n12-21) -- (n12-1);
        \draw (n12-1) --  (n1-1);
        \draw (n1-21) --  (n1-1);
        \end{tikzpicture}
    \end{center}
    \caption{The Hasse diagram of the interval 
        $[\sigma, \inflatesome{\sigma}{\ell,r}{12,21}]$,
        and the corresponding values of the \mob function.}
    \label{figure-extending-PFM-oppadj-fails}
\end{figure}
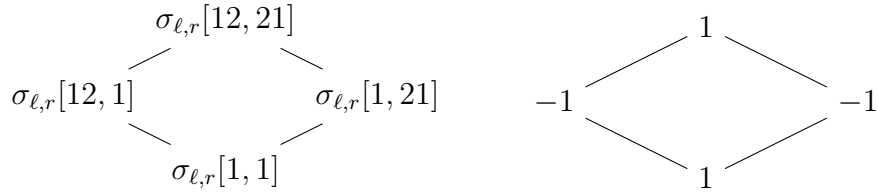

Although we do not have a general extension of
Theorem~\ref{theorem-PMF-opposing-adjacencies},
we can show that:
\begin{theorem}
    If $\sigma$ is adjacency-free,
    and $\pi$ contains an interval 
    order-isomorphic to a symmetry of $1243$,
    then $\mobfn{\sigma}{\pi} = 0$.
\end{theorem}
\begin{proof}
    First note that if $\sigma \not\leq \pi$,
    then $\mobfn{\sigma}{\pi} = 0$ 
    from the definition of the \mob function.
    Further, since $\sigma$ is adjacency-free,
    we cannot have $\sigma = \pi$.
    
    We can now assume that $\sigma < \pi$.
    Without loss of generality 
    we can also assume, by symmetry, that the
    interval in $\pi$ is
    order-isomorphic to $1243$.
    
    We start by claiming that, for any permutation
    $\sigma$ which is adjacency-free, 
    and any $c$ with $1 \leq c \leq \order{\sigma}$,
    we have 
    $\mobfn{\sigma}{\inflatesome{\sigma}{c}{1243}} = 0$.
    The Hasse diagram of the interval 
    $[\sigma, \inflatesome{\sigma}{c}{1243}]$ is shown in
    Figure~\ref{figure-hasse-interval-inflate-1243}.
    \begin{figure}
        \begin{center}
            \begin{tikzpicture}[xscale=1,yscale=1.1]
            \node (n1243) at ( 0, 3) {$\inflatesome{\sigma}{c}{1243}$};
            \node (n123)  at (-2, 2) {$\inflatesome{\sigma}{c}{123}$};  
            \node (n132)  at ( 2, 2) {$\inflatesome{\sigma}{c}{132}$};  
            \node (n12)   at (-2, 1) {$\inflatesome{\sigma}{c}{12}$};  
            \node (n21)   at ( 2, 1) {$\inflatesome{\sigma}{c}{21}$};  
            \node (n1)    at ( 0, 0) {$\inflatesome{\sigma}{c}{1}$};  
            \draw (n1243) -- (n123);
            \draw (n1243) -- (n132);
            \draw (n123) --  (n12);
            \draw (n132) --  (n12);
            \draw (n132) --  (n21);
            \draw (n12)  --  (n1);
            \draw (n21)  --  (n1);            
            \end{tikzpicture}
        \end{center}
        \caption{The Hasse diagram of the interval 
            $[\sigma, \inflatesome{\sigma}{c}{1243}]$.}
        \label{figure-hasse-interval-inflate-1243}
    \end{figure}
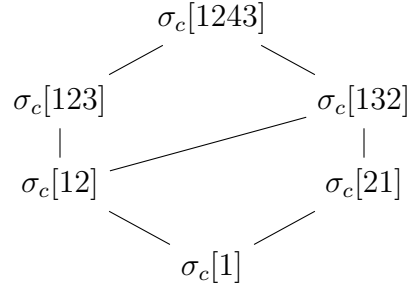
    From the definition of the \mob function, we have
    $\mobfn{\sigma}{\inflatesome{\sigma}{c}{1}} = 1$,
    $\mobfn{\sigma}{\inflatesome{\sigma}{c}{12}} = -1$, 
    $\mobfn{\sigma}{\inflatesome{\sigma}{c}{21}} = -1$, 
    $\mobfn{\sigma}{\inflatesome{\sigma}{c}{123}} = 0$,
    and
    $\mobfn{\sigma}{\inflatesome{\sigma}{c}{132}} = 1$,
    and so
    $\mobfn{\sigma}{\inflatesome{\sigma}{c}{1243}} = 0$,
    and thus our claim is true.
    
    Our argument now follows a similar pattern to
    that used by the proof of
    Theorem~\ref{theorem-PMF-opposing-adjacencies},
    and we restrict ourselves to highlighting the differences.
    
    Assume that $\pi$ is a proper inflation of $\sigma$,
    with length greater than $\order{\sigma} + 4$, 
    and $\pi$ contains an interval order-isomorphic to $1243$.
    Let $\gamma$ be the permutation formed by replacing
    an occurrence of $1243$ in $\pi$ by $12$, so if 
    $\ell$ is the position of the first point 
    of the $1243$ selected, then
    $\inflatesome{\gamma}{\ell,\ell+1}{12,21} = \pi$.
    Let $\lambda = \inflatesome{\gamma}{\ell,\ell+1}{12,1}$, 
    and let $\rho = \inflatesome{\gamma}{\ell,\ell+1}{1,21}$;
    Define sets 
    $L = [\sigma, \lambda]$,
    $R = [\sigma, \rho]$,
    $G_\gamma = [\sigma, \gamma]$,
    $G_x = L \cap R \setminus G_\gamma$, and
    $T = [\sigma, \pi) \setminus (L \cup R)$.
    
    Similarly to Theorem~\ref{theorem-PMF-opposing-adjacencies}, 
    we have 
    \[
    \mobfn{\sigma}{\pi} = 
    - \sum_{\tau \in L} \mobfn{\sigma}{\tau}
    - \sum_{\tau \in R} \mobfn{\sigma}{\tau}
    - \sum_{\tau \in T} \mobfn{\sigma}{\tau}
    + \sum_{\tau \in G_\gamma} \mobfn{\sigma}{\tau}
    + \sum_{\tau \in G_x} \mobfn{\sigma}{\tau},
    \]
    and the sums over the sets 
    $L$, $R$ and $G_\gamma$ are obviously zero.
    Using similar arguments to 
    Theorem~\ref{theorem-PMF-opposing-adjacencies}, 
    we can see that every permutation $\tau$
    in $T$ or $G_x$ contains
    an interval order-isomorphic to $1243$,
    and so by the inductive hypothesis,
    has $\mobfn{\sigma}{\tau} = 0$,
    and thus we have 
    $\mobfn{\sigma}{\pi} = 0$.
\end{proof}    

Although we cannot find a general 
extension to
Theorem~\ref{theorem-PMF-opposing-adjacencies},
we can find a necessary condition for 
a proper inflation of certain permutations
to have a \mob function value of zero.
This is
\begin{lemma}
    If $\sigma$ is adjacency-free,
    and 
    $\pi = \inflateall{\sigma}{\alpha_1, \ldots, \alpha_n}$ 
    is a proper inflation of $\sigma$,
    then $\mobfn{\sigma}{\pi} = 0$ implies that
    at least one $\alpha_i \not \in \{1, 12, 21\}$.    
\end{lemma}
\begin{proof}
    Assume that every $\alpha_i \in \{1, 12, 21\}$.
    Let $k$ be the number of $\alpha_i$-s that are not equal to 1,
    and 
    let $j_1, \ldots , j_k$
    be the indexes ($i$-s) where $\alpha_i \neq 1$,
    so 
    $\pi = 
    \inflatesome{\sigma}{j_1, \ldots, j_k}
    {\alpha_{j_1}, \ldots, \alpha_{j_k}}$.
    
    Then every permutation in the interval
    $[\sigma, \pi]$ has a unique representation
    as 
    $\inflatesome{\sigma}{j_1, \ldots, j_k}{\upsilon_1, \ldots , \upsilon_k}$,
    where
    \begin{align*}
    \upsilon_i & \in
    \begin{cases}
    \{ 1, 12 \} & \text{if } \alpha_{j_i} = 12, \\
    \{ 1, 21 \} & \text{if } \alpha_{j_i} = 21. \\
    \end{cases}
    \end{align*}
    So each position $j_i$ can be inflated 
    by one of two permutations, and thus
    there is an obvious isomorphism
    between permutations in the interval
    $[\sigma, \pi]$  
    and binary numbers with $k$ bits.
    It follows that the poset can be represented as
    a Boolean algebra, and so by a 
    well-known result 
    (see, for instance,
    Example 3.8.3 in Stanley~\cite{Stanley2012}),
    $\mobfn{\sigma}{\pi} = (-1)^{\order{\pi} - \order{\sigma}}$.
    Thus if $\mobfn{\sigma}{\pi} = 0$, at least one 
    $\alpha_i \not \in \{1, 12, 21\}$.
\end{proof}

    \chapter{2413-balloons and the growth of the \mob function}
\label{chapter_2413_balloon_paper}

\section{Preamble}

This chapter
is based on a published paper~\cite{Marchant2020}
which is sole work by the author.

In this chapter
we 
show that the growth of the 
principal \mob 
function on the permutation poset
is exponential.  
This improves on previous work, which 
has shown that the growth 
is at least polynomial.

We define a method of constructing 
a permutation from a smaller permutation
which we call
``ballooning''.  
We show that if $\beta$ is a 2413-balloon,
and $\pi$ is the 2413-balloon of $\beta$,
then
$\mobfn{1}{\pi} = 2 \mobfn{1}{\beta}$.
This allows us to construct a sequence
of permutations $\pi_1, \pi_2, \pi_3\ldots$,
with lengths $n, n+4, n+8, \ldots$     
such that $\mobfn{1}{\pi_{i+1}} = 2 \mobfn{1}{\pi_{i}}$,
and this gives us exponential growth
of the principal \mob function.
Further, our construction method 
gives permutations that lie within
a hereditary class with
finitely many simple permutations.

We also find an expression for the value
of $\mobfn{1}{\pi}$, where $\pi$ is a 2413-balloon,
with no restriction
on the permutation
being ballooned.

\section{Introduction}

In the concluding remarks to their seminal paper,
Burstein, Jel{\'{i}}nek, Jel{\'{i}}nkov{\'{a}} 
and Steingr{\'{i}}msson~\cite{Burstein2011}
ask whether the principal \mob function
is unbounded, which is the first reference 
to the growth of the \mob function in the literature.
They show that 
if $\pi$ is a separable permutation,
then $\mobp{\pi} \in \{0, \pm 1 \}$, 
and thus is bounded.
The separable permutations lie in a  
hereditary class which only contains
the simple permutations
$1$, $12$ and $21$.
They ask (Question 27) for which classes 
is $\mobp{\pi}$ bounded?

Smith~\cite{Smith2013}
found an explicit case-wise formula for the principal \mob 
function for all permutations with a single descent.
For certain sets of permutations
with a single descent, the 
associated formula is, up to a sign,
\[
\mobp{\pi} = \binom{k}{2},
\]
where $k$ is a linear expression in the length of $\pi$.
This shows that
the growth of the \mob function
is at least quadratic.
Jel{\'{i}}nek, Kantor, Kyn{\v{c}}l and Tancer~\cite{Jelinek2020}
show how to construct a sequence of permutations
where the absolute value of the \mob function
grows according to the seventh power of the length.

%
%
We show that, given some permutation $\beta$,
we can construct a permutation 
that we call the ``2413-balloon'' of $\beta$.
This permutation will have four more points than $\beta$.
We then show that if 
$\pi$ is a 2413-balloon of $\beta$,
and $\beta$ is itself a 2413-balloon, 
then 
$\mobp{\pi} = 2 \mobp{\beta}$.
From this we deduce that
the growth of the principal \mob function is exponential.
If $\beta = 25314$ (which is a 2413-balloon),
then we can construct a hereditary class
that contains only the simple permutations
$\{ 1,12,21,2413,25314\}$, 
where the growth of the 
principal \mob function is exponential,
answering questions in
Burstein et al~\cite{Burstein2011}
and 
Jel{\'{i}}nek et al~\cite{Jelinek2020}.

%
%
We start by recalling some essential definitions and notation in
Section~\ref{section-definitions-and-notation},
where we also provide some extensions of 
existing results.
We formally define a 2413-balloon in 
Section~\ref{section-define-2413-balloon},
and we provide some
results which will be used
in the remainder of this chapter.
In Section~\ref{section-2413-double-balloons},
we derive an expression for the 
value of $\mobp{\pi}$ when $\pi$ is a double 2413-balloon,
and following this
we show that the growth of the \mob function
is exponential in
Section~\ref{section-growth-rate-of-mu}.
We return to the topic of 2413-balloons in
Section~\ref{section-2413-balloons}, 
and derive an expression for the 
value of $\mobp{\pi}$ when $\pi$ is
any 2413-balloon.
Finally, we discuss the generalization
of the balloon operator in
Section~\ref{section-concluding-remarks}.
We also ask some questions regarding
the growth of the \mob function.

\section{Essential definitions, notation, and results}
\label{section-definitions-and-notation}

In this section we recall some standard definitions
and notation that we will use, and add some 
simple definitions and consequences of known results.  

%
%
If $\pi$ is a permutation with length $n$,
then
the number of 
\emph{corners}\extindex[permutation]{corners} 
of $\pi$
is the number of points of $\pi$ that are extremal in both
position and value,
that is, $\pi_1 \in \{1, n\}$ or $\pi_n \in \{1, n\}$.
It is easy to see that any permutation 
with length 2 or more can have at most two corners.
We adopt the convention that the permutation
$1$ has one corner.

%
%
Recall that if a permutation $\pi$ can be written as 
$\oneplus\oneplus\tau$,
$\oneminus\oneminus\tau$,
$\tau\plusone\plusone$, or
$\tau\minusone\minusone$,
where $\tau$ is non-empty 
(so $\order{\pi} \geq 3$),
then we say that $\pi$
has a \emph{long corner}.

%
%
We now have
\begin{lemma}
    \label{lemma-oneplus-oneplus}
    If $\pi$ has a long corner,
    then
    $\mobp{\pi} = 0$.	
\end{lemma}
\begin{lemma}
    \label{lemma-oneplus}
    If $\pi$ can be written as 
    $\pi = \oneplus\tau$,
    or
    $\pi = \tau\plusone$
    or 
    $\pi = \oneminus\tau$
    or
    $\pi = \tau\minusone$,
    and does not have a long corner,
    then
    $\mobp{\pi} = - \mobp{\tau}$.
\end{lemma}
These are well-known consequences of
Propositions~1~and~2 of
Burstein, 
Jel{\'{i}}nek, 
Jel{\'{i}}nkov{\'{a}} and 
Steingr{\'{i}}msson~\cite{Burstein2011},
and we refrain from
providing proofs here.
The reader is directed to    
Lemma~\ref{lemma_mobius_function_is_zero}
on page~\pageref{lemma_mobius_function_is_zero}
in Chapter~\ref{chapter_incosc_paper}
for a proof of Lemma~\ref{lemma-oneplus-oneplus}.
Lemma~\ref{lemma-oneplus}
is a trivial extension of Corollary 3 in~\cite{Burstein2011}.  

%
%
Recall that a triple adjacency is a monotonic interval of length 3.
Smith shows that
\begin{lemma}[{%
    Smith~\cite[Lemma 1]{Smith2013}}]
    \label{lemma-triple-adjacencies}
    If a permutation $\pi$
    contains a triple adjacency then $\mobp{\pi} = 0$.
\end{lemma}
A trivial corollary to 
Lemma~\ref{lemma-triple-adjacencies} is
\begin{corollary}
    \label{corollary-monotonic-interval}
    If a permutation contains a monotonic interval
    with length 3 or more, then $\mobp{\pi} = 0$.    
\end{corollary}

%
%
Recall that Hall's Theorem~\cite[Proposition 3.8.5]{Stanley2012} says 
that 
\[
\mobfn{\sigma}{\pi} = 
\sum_{c \in \chainc(\sigma, \pi)} (-1)^{\order{c}}  =
\sum_{i=1}^{\order{\pi} - 1} (-1)^i K_i
\]
where $\chainc(\sigma, \pi)$ is the set of chains 
in the poset interval $[\sigma, \pi]$ which  contain both $\sigma$ and $\pi$,
and $K_i$ is the number of chains of length $i$.

%
%
We can also use Hall's Theorem if we have a subset of chains
that meet a specific criteria:
\begin{lemma}
    \label{lemma-hall-sum-second-element-psi}
    Let $\pi$ be any permutation with length three or more.
    Let $\psi$ be a permutation with $1 < \psi < \pi$.
    Let $\chainc$ be the subset of chains in 
    the poset interval $[1, \pi]$ where the second-highest element is $\psi$.
    Then 
    \[\sum\limits_{c \in \chainc}  (-1)^{\order{c}} = - \mobp{\psi}.\]
\end{lemma}	
\begin{proof}
    If we remove $\pi$ from the chains in $\chainc$, then
    we have all of the chains in the poset interval $[1, \psi]$,
    and the Hall sum of these chains is, by definition,
    $\mobp{\psi}$.  It follows that the Hall sum
    of the chains in $\chainc$ is $ - \mobp{\psi}$.	
\end{proof}
\begin{corollary}
    \label{corollary-hall-sum-second-highest-set}
    Given a permutation $\pi$,
    and a set of permutations $S$
    where every $\sigma \in S$ satisfies
    $1 < \sigma < \pi$,
    then 
    if $\chainc$ is the set of chains in the poset interval
    $[1, \pi]$ where the second-highest element is in $S$,
    then the Hall sum of $\chainc$ is
    $- \sum_{\sigma \in S} \mobp{\sigma}$.    
\end{corollary}
\begin{proof}	
    First, partition $\chainc$ based on the second-highest element,
    and then
    apply Lemma~\ref{lemma-hall-sum-second-element-psi} to each
    partition.
\end{proof}

\section{2413-Balloons}
\label{section-define-2413-balloon}

In this section we define the vocabulary and notation
specific to this chapter.
We also present some general results which will be used
in later sections.

%
%
Given a non-empty permutation $\beta$,
the 
\emph{2413-balloon}\extindex{2413-balloon} 
of $\beta$
is the permutation formed by
inserting $\beta$ into the centre of $2413$,
which we write as $\ball{2413}{\beta}$.
Formally, we have
\begin{align*}
(\ball{2413}{\beta})_i
& =
\begin{cases}
2 & \text{if $i = 1$}\\
\order{\beta} + 4 & \text{if $i = 2$}\\
\beta_{i-2} + 2 & \text{if $i > 2$ and $i \leq \order{\beta} + 2$ }\\
1 & \text{if $i = \order{\beta} + 3$}\\
\order{\beta} + 3 & \text{if $i = \order{\beta} + 4$}\\ 
\end{cases}
\end{align*}
Figure~\ref{figure-2413-balloons-a} 
shows $\ball{2413}{\beta}$.
\begin{figure}
    \begin{center}
        \begin{subfigure}[t]{0.35\textwidth}
            \begin{center}
                \begin{tikzpicture}[scale=0.3]
                    \foreach \i in {0,1,2,5,6,7}{
                        \draw [color=lightgray] ({\i+0.5}, 0.5)--({\i+0.5}, {7.5});
                    };
                    \foreach \i in {0,1,2,5,6,7}{
                        \draw [color=lightgray] (0.5, {\i+0.5})--({7.5}, {\i+0.5});
                    };
                \normaldot{(1,2)};
                \normaldot{(2,7)};
                \scell{4}{4}{$\beta$};
                \normaldot{(6,1)};
                \normaldot{(7,6)};
                \end{tikzpicture}
            \end{center}
            \caption{}
            \label{figure-2413-balloons-a} 
        \end{subfigure}
        \begin{subfigure}[t]{0.55\textwidth}
            \begin{center}
                \begin{tikzpicture}[scale=0.3]
                \foreach \i in {0,1,2,3,4,7,8,9,10,11}{
                    \draw [color=lightgray] ({\i+0.5}, 0.5)--({\i+0.5}, {11.5});
                };
                \foreach \i in {0,1,2,3,4,7,8,9,10,11}{
                    \draw [color=lightgray] (0.5, {\i+0.5})--({11.5}, {\i+0.5});
                };
                \normaldot{(1,2)};
                \normaldot{(2,11)};
                \normaldot{(3,4)};
                \normaldot{(4,9)};
                \scell{6}{6}{$\gamma$};                
                \normaldot{(8,3)};
                \normaldot{(9,8)};
                \normaldot{(10,1)};
                \normaldot{(11,10)};
                \end{tikzpicture}
            \end{center}
            \caption{}
            \label{figure-2413-balloons-b} 
        \end{subfigure}
    \end{center}
    \caption{
        (a) The 2413-balloon $\ball{2413}{\beta}$ and
        (b) the double 2413-balloon $\ball{2413}{\ball{2413}{\gamma}}$.}
    \label{figure-2413-balloons} 
\end{figure}
Throughout this chapter we will be 
discussing permutations that contain
an interval copy of a smaller permutation.
Examples of this smaller permutation 
are $\beta$ in $\ball{2413}{\beta}$,
and $\gamma$ in $\ball{2413}{\ball{2413}{\gamma}}$,
as shown in Figure~\ref{figure-2413-balloons}.
In figures where this is the case, 
the permutation plot
scale will be non-linear so that
the cell containing the interval copy
($\beta$ and $\gamma$ in our examples)
is larger than the other cells.

%
%
The balloon operation as defined has to be
right-associative and the definition given does not
support overriding right-associativity.
In other words, 
$\ball{2413}{\ball{2413}{\beta}}$
must be
$\ball{2413}{(\ball{2413}{\beta})}$,
and
$\ball{(\ball{2413}{2413})}{\beta}$ is not defined.
In 
Section~\ref{section-concluding-remarks} we  
suggest how the balloon operation could be generalized.

%
%
Given some $\pi = \ball{2413}{\beta}$,
if $\beta$ is itself a 2413-balloon,
so $\pi = \ball{2413}{\ball{2413}{\gamma}}$,
then we say that $\pi$ is a 
\emph{double 2413-balloon}\extindex{double 2413-balloon}.
Figure~\ref{figure-2413-balloons-b} shows a double 2413-balloon.

%
%
\begin{remark}
    We 
    can write $\ball{2413}{\beta}$
    as the inflation  
    $25314 [1,1,\beta,1,1]$. 
    We refer the reader to Albert and Atkinson~\cite{Albert2005}
    for further details of inflations.
    In this chapter we use balloon notation,
    as we feel that this leads to a simpler exposition.
\end{remark}

%
%
If we have $\pi = \ball{2413}{\beta}$,
and we have some $\sigma$ 
with $\beta \leq \sigma < \pi$,
we will frequently want to 
represent $\sigma$
in terms of 
sub-permutations of $2413$ and 
the permutation $\beta$.
We start by colouring 
the extremal points of $\pi$ red,
and all remaining points black.  
Note that the red points are a 2413 permutation,
and the black points are $\beta$.

Now consider 
a specific embedding of $\sigma$ into $\pi$,
where we use all of the black points ($\beta$).
If the embedding is 
monochromatic ($\sigma = \beta$) then
we require no special notation.
If the embedding is not monochromatic,
then it must be the case that
only some of the red points are used.
We take $2413$, and mark the red points that are 
unused with an overline, and then write
$\sigma$ using our balloon notation.
As an example of this, 
if
$\pi = \ball{2413}{21} = 264315$, and $\sigma = 213$,
then we could represent $\sigma$ as
$\ball{\ex{2}\ex{4}\ex{1}3}{21}$.
This example is shown in Figure~\ref{figure-overbar-notation}.
\begin{figure}
    \begin{center}
        \begin{tikzpicture}[scale=0.3]
        \plotpermgrid{2,6,4,3,1,5}
        \opendot{(1,2)};
        \opendot{(2,6)};
        \opendot{(5,1)};
        \end{tikzpicture}
    \end{center}
    \caption{An embedding of $213 = \ball{\ex{2}\ex{4}\ex{1}3}{21}$ in $264315 = \ball{2413}{21}$.}
    \label{figure-overbar-notation}
\end{figure}
We can see that if 
$\beta \leq \sigma < \ball{2413}{\beta}$, 
and $\beta$ is not monotonic (i.e., not the 
identity permutation or its reverse), then
there is a unique way to represent
$\sigma$ using this notation.

%
%
If we have $\pi = \ball{2413}{\beta}$, 
and $\sigma$ is a permutation such that
$\beta \leq \sigma < \pi$,
then we say that
$\sigma$ is a 
\emph{reduction}\extindex{reduction (of a 2413 balloon)}
of $\pi$.
%
If $\sigma$ is a reduction of $\pi = \ball{2413}{\beta}$,
and there is no $\eta$ with $\order{\eta} <\order{\beta}$
such that
either $\sigma$ is equal to $\ball{2413}{\eta}$,
or $\sigma$ is a reduction of $\ball{2413}{\eta}$,
then we say that $\sigma$ is a
\emph{proper reduction}\extindof{2413-balloon}{proper reduction}{(of a 2413-balloon)}
of $\pi$.
A reduction of $\pi$ that is not a proper reduction
is an 
\emph{improper reduction}\extindex{improper reduction (of a 2413 balloon)}.

The following case-by-case analysis shows the
improper reductions (of $\pi$)
based on the form of $\beta$.

\begin{itemize}
    \item If $\beta$ is a 2413-balloon,
    then
    $\beta$ is the only improper reduction of $\pi$.
    \item If $\beta$ is not a 2413-balloon,
    and $\beta$ has no corners,
    then
    there are no improper reductions of $\pi$.
    \item If $\beta$ has one corner,
    then there are four improper reductions of $\pi$.
    As an example, if $\beta = 1 \oplus \gamma$, then the 
    improper reductions of $\pi$ are
    $\ball{\ex{2}\ex{4}13}{\beta}$,
    $\ball{\ex{2}\ex{4}1\ex{3}}{\beta}$,
    $\ball{\ex{2}\ex{4}\ex{1}3}{\beta}$,
    and
    $\beta$.  
    \item If $\beta$ has two corners, then
    there are seven improper reductions of $\pi$.
    As an example, if $\beta = 1 \oplus \gamma \oplus 1$,
    then the improper reductions are
    $\ball{\ex{2}\ex{4}13}{\beta}$,
    $\ball{24\ex{1}\ex{3}}{\beta}$,
    $\ball{2\ex{4}\ex{1}\ex{3}}{\beta}$,
    $\ball{\ex{2}4\ex{1}\ex{3}}{\beta}$,
    $\ball{\ex{2}\ex{4}1\ex{3}}{\beta}$,
    $\ball{\ex{2}\ex{4}\ex{1}3}{\beta}$, 
    and    
    $\beta$.
\end{itemize}

The set of permutations that are 
proper reductions of $\pi$
is written as $\redset$.
Figure~\ref{figure-2413-reductions}
shows all the reductions (proper and improper)
of $\pi = \ball{2413}{\beta}$.

\begin{figure}
    \begin{center}
        \begin{subfigure}[t]{0.2\textwidth}
            \centering
            \begin{tikzpicture}[scale=0.3]
            \foreach \i in {0,1,4,5,6}{
                \draw [color=lightgray] ({\i+0.5}, 0.5)--({\i+0.5}, {6.5});
            };
            \foreach \i in {0,1,4,5,6}{
                \draw [color=lightgray] (0.5, {\i+0.5})--({6.5}, {\i+0.5});
            };
            \normaldot{(1,6)};
            \scell{3}{3}{$\beta$};
            \normaldot{(5,1)};
            \normaldot{(6,5)};
            \end{tikzpicture}
            \caption*{$\ball{\ex{2}413}{\beta}$}
        \end{subfigure}
        \begin{subfigure}[t]{0.2\textwidth}
            \centering
            \begin{tikzpicture}[scale=0.3]
            \foreach \i in {0,1,4,5,6}{
                \draw [color=lightgray] ({\i+0.5}, 0.5)--({\i+0.5}, {6.5});
            };
            \foreach \i in {0,1,2,5,6}{
                \draw [color=lightgray] (0.5, {\i+0.5})--({6.5}, {\i+0.5});
            };
            \normaldot{(1,2)};
            \scell{3}{4}{$\beta$};
            \normaldot{(5,1)};
            \normaldot{(6,6)};
            \end{tikzpicture}
            \caption*{$\ball{2\ex{4}13}{\beta}$}
        \end{subfigure}
        \begin{subfigure}[t]{0.2\textwidth}
            \centering
            \begin{tikzpicture}[scale=0.3]
            \foreach \i in {0,1,2,5,6}{
                \draw [color=lightgray] ({\i+0.5}, 0.5)--({\i+0.5}, {6.5});
            };
            \foreach \i in {0,1,4,5,6}{
                \draw [color=lightgray] (0.5, {\i+0.5})--({6.5}, {\i+0.5});
            };
            \normaldot{(1,1)};
            \normaldot{(2,6)};
            \scell{4}{3}{$\beta$};
            \normaldot{(6,5)};
            \end{tikzpicture}
            \caption*{$\ball{24\ex{1}3}{\beta}$}
        \end{subfigure}
        \begin{subfigure}[t]{0.2\textwidth}
            \centering
            \begin{tikzpicture}[scale=0.3]
            \foreach \i in {0,1,2,5,6}{
                \draw [color=lightgray] ({\i+0.5}, 0.5)--({\i+0.5}, {6.5});
            };
            \foreach \i in {0,1,2,5,6}{
                \draw [color=lightgray] (0.5, {\i+0.5})--({6.5}, {\i+0.5});
            };
            \normaldot{(1,2)};
            \normaldot{(2,6)};
            \scell{4}{4}{$\beta$};
            \normaldot{(6,1)};
            \end{tikzpicture}
            \caption*{$\ball{241\ex{3}}{\beta}$}
        \end{subfigure}	
        \vspace{1\baselineskip}
    \end{center} 
    \begin{center}
        \begin{subfigure}[t]{0.15\textwidth}
            \centering
            \begin{tikzpicture}[scale=0.3]
            \foreach \i in {0,3,4,5}{
                \draw [color=lightgray] ({\i+0.5}, 0.5)--({\i+0.5}, {5.5});
            };
            \foreach \i in {0,1,4,5}{
                \draw [color=lightgray] (0.5, {\i+0.5})--({5.5}, {\i+0.5});
            };
            \scell{2}{3}{$\beta$};
            \normaldot{(4,1)};
            \normaldot{(5,5)};
            \end{tikzpicture}
            \caption*{$\ball{\ex{2}\ex{4}13}{\beta}$}
        \end{subfigure}
        \begin{subfigure}[t]{0.15\textwidth}
            \centering
            \begin{tikzpicture}[scale=0.3]
            \foreach \i in {0,1,4,5}{
                \draw [color=lightgray] ({\i+0.5}, 0.5)--({\i+0.5}, {5.5});
            };
            \foreach \i in {0,3,4,5}{
                \draw [color=lightgray] (0.5, {\i+0.5})--({5.5}, {\i+0.5});
            };
            \scell{3}{2}{$\beta$};
            \normaldot{(1,5)};
            \normaldot{(5,4)};
            \hiddendot{(1,1)}; 
            \end{tikzpicture}
            \caption*{$\ball{\ex{2}4\ex{1}3}{\beta}$}
        \end{subfigure}
        \begin{subfigure}[t]{0.15\textwidth}
            \centering
            \begin{tikzpicture}[scale=0.3]
            \foreach \i in {0,1,4,5}{
                \draw [color=lightgray] ({\i+0.5}, 0.5)--({\i+0.5}, {5.5});
            };
            \foreach \i in {0,1,4,5}{
                \draw [color=lightgray] (0.5, {\i+0.5})--({5.5}, {\i+0.5});
            };
            \normaldot{(1,5)};
            \scell{3}{3}{$\beta$};
            \normaldot{(5,1)};
            \end{tikzpicture}
            \caption*{$\ball{\ex{2}41\ex{3}}{\beta}$}
        \end{subfigure}
        \begin{subfigure}[t]{0.15\textwidth}
            \centering
            \begin{tikzpicture}[scale=0.3]
            \foreach \i in {0,1,4,5}{
                \draw [color=lightgray] ({\i+0.5}, 0.5)--({\i+0.5}, {5.5});
            };
            \foreach \i in {0,1,4,5}{
                \draw [color=lightgray] (0.5, {\i+0.5})--({5.5}, {\i+0.5});
            };
            \normaldot{(1,1)};
            \scell{3}{3}{$\beta$};
            \normaldot{(5,5)};
            \end{tikzpicture}
            \caption*{$\ball{2\ex{4}\ex{1}3}{\beta}$}
        \end{subfigure}	
        \begin{subfigure}[t]{0.15\textwidth}
            \centering
            \begin{tikzpicture}[scale=0.3]
            \foreach \i in {0,1,4,5}{
                \draw [color=lightgray] ({\i+0.5}, 0.5)--({\i+0.5}, {5.5});
            };
            \foreach \i in {0,1,2,5}{
                \draw [color=lightgray] (0.5, {\i+0.5})--({5.5}, {\i+0.5});
            };
            \normaldot{(1,2)};
            \scell{3}{4}{$\beta$};
            \normaldot{(5,1)};
            \end{tikzpicture}
            \caption*{$\ball{2\ex{4}1\ex{3}}{\beta}$}
        \end{subfigure}	
        \begin{subfigure}[t]{0.15\textwidth}
            \centering
            \begin{tikzpicture}[scale=0.3]
            \foreach \i in {0,1,2,5}{
                \draw [color=lightgray] ({\i+0.5}, 0.5)--({\i+0.5}, {5.5});
            };
            \foreach \i in {0,1,4,5}{
                \draw [color=lightgray] (0.5, {\i+0.5})--({5.5}, {\i+0.5});
            };
            \normaldot{(1,1)};
            \normaldot{(2,5)};
            \scell{4}{3}{$\beta$};
            \end{tikzpicture}
            \caption*{$\ball{24\ex{1}\ex{3}}{\beta}$}
        \end{subfigure}	
        \vspace{1\baselineskip}
    \end{center}
    \begin{center}
        \begin{subfigure}[t]{0.15\textwidth}
            \centering
            \begin{tikzpicture}[scale=0.3]
            \foreach \i in {0,1,4}{
                \draw [color=lightgray] ({\i+0.5}, 0.5)--({\i+0.5}, {4.5});
            };
            \foreach \i in {0,1,4}{
                \draw [color=lightgray] (0.5, {\i+0.5})--({4.5}, {\i+0.5});
            };
            \normaldot{(1,1)};
            \scell{3}{3}{$\beta$};           
            \end{tikzpicture}
            \caption*{$\ball{2\ex{4}\ex{1}\ex{3}}{\beta}$}
        \end{subfigure}
        \begin{subfigure}[t]{0.15\textwidth}
            \centering
            \begin{tikzpicture}[scale=0.3]
            \foreach \i in {0,1,4}{
                \draw [color=lightgray] ({\i+0.5}, 0.5)--({\i+0.5}, {4.5});
            };
            \foreach \i in {0,3,4}{
                \draw [color=lightgray] (0.5, {\i+0.5})--({4.5}, {\i+0.5});
            };
            \normaldot{(1,4)};
            \scell{3}{2}{$\beta$};
            \hiddendot{(1,1)};             
            \end{tikzpicture}
            \caption*{$\ball{\ex{2}4\ex{1}\ex{3}}{\beta}$}
        \end{subfigure}
        \begin{subfigure}[t]{0.15\textwidth}
            \centering
            \begin{tikzpicture}[scale=0.3]
            \foreach \i in {0,3,4}{
                \draw [color=lightgray] ({\i+0.5}, 0.5)--({\i+0.5}, {4.5});
            };
            \foreach \i in {0,1,4}{
                \draw [color=lightgray] (0.5, {\i+0.5})--({4.5}, {\i+0.5});
            };
            \scell{2}{3}{$\beta$};
            \normaldot{(4,1)};
            \end{tikzpicture}
            \caption*{$\ball{\ex{2}\ex{4}1\ex{3}}{\beta}$}
        \end{subfigure}	
        \begin{subfigure}[t]{0.15\textwidth}
            \centering
            \begin{tikzpicture}[scale=0.3]
            \foreach \i in {0,3,4}{
                \draw [color=lightgray] ({\i+0.5}, 0.5)--({\i+0.5}, {4.5});
            };
            \foreach \i in {0,3,4}{
                \draw [color=lightgray] (0.5, {\i+0.5})--({4.5}, {\i+0.5});
            };
            \scell{2}{2}{$\beta$};
            \normaldot{(4,4)};
            \hiddendot{(1,1)};             
            \end{tikzpicture}
            \caption*{$\ball{\ex{2}\ex{4}\ex{1}3}{\beta}$}
        \end{subfigure}
        \begin{subfigure}[t]{0.15\textwidth}
            \centering
            \begin{tikzpicture}[scale=0.3]
            \foreach \i in {0,3}{
                \draw [color=lightgray] ({\i+0.5}, 0.5)--({\i+0.5}, {3.5});
            };
            \foreach \i in {0,3}{
                \draw [color=lightgray] (0.5, {\i+0.5})--({3.5}, {\i+0.5});
            };
            \scell{2}{2}{$\beta$};
            \hiddendot{(1,1)};             
            \end{tikzpicture}
            \caption*{$\beta$}
        \end{subfigure}
    \end{center}
    \caption{Reductions of $\pi = \ball{2413}{\beta}$.  Some may not be proper reductions, depending on $\beta$.}	      
    \label{figure-2413-reductions}
\end{figure}
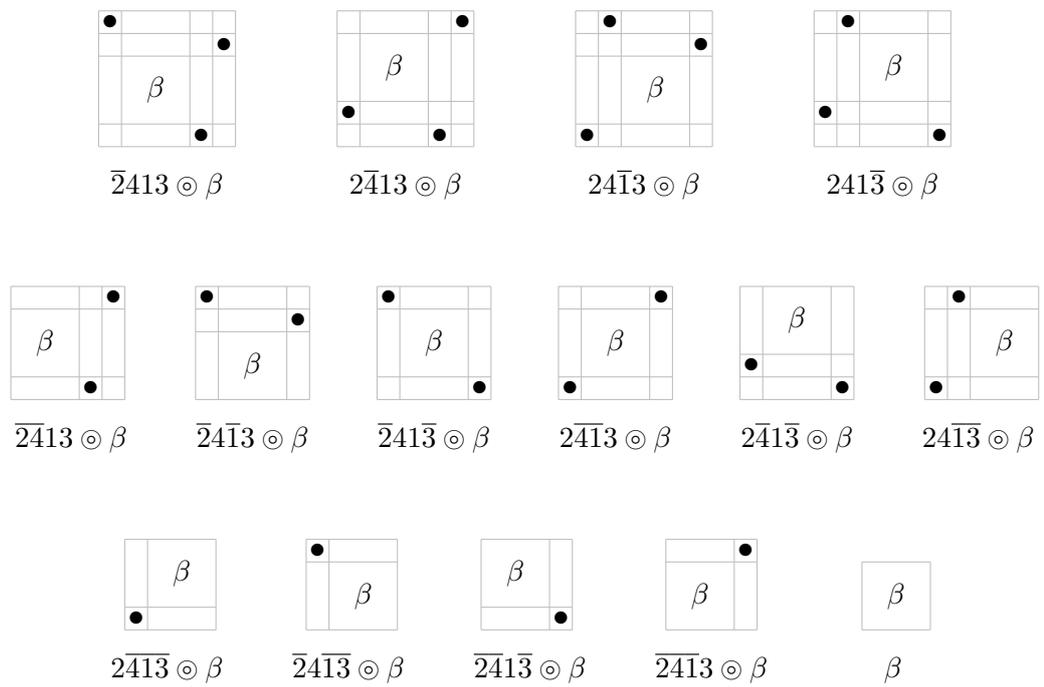

%
%
The strategy that we will use 
in 
Sections~\ref{section-2413-double-balloons} 
and~\ref{section-2413-balloons} is
to partition the chains in the poset interval 
$[1,\pi]$ into three sets,
$\chainr$, 
$\chaing$, and
$\chainb$. 
We then show that there are parity-reversing involutions
on the sets 
$\chaing$ and
$\chainb$,
and therefore, by 
Corollary~\ref{corollary-halls-corollary},
the Hall sum for each of these sets is zero,
and so $\mobp{\pi}$ is given by the 
Hall sum of the set $\chainr$.
Finally, we show that the Hall sum of $\chainr$
can be written in terms of $\mobp{\beta}$.

%
%
The chains in $\chainr$ are those chains
where the second-highest element 
is a proper reduction of $\pi$,
so if $\kappa_c$ is the second-highest
element of a chain $c$, then 
$c \in \chainr$ if and only if
$\kappa_c \in \redset$.
Note that, as mentioned earlier,
the members of $\redset$,
and hence the chains in $\chainr$,
depend on the form of $\pi$.
It is easy to see that for any permutation $\sigma \in \redset$
we must have $\order{\sigma} \geq \order{\beta}$.

We have some results that
are independent of $\redset$, 
and, once we have given some further 
definitions, we present these in
the current section
to avoid repetition.

Let $\pi$ be a 2413-balloon, and let $c$ be any chain
in the poset interval $[1, \pi]$.

%
%
Since the top of the chain is, by definition, a 2413-balloon,
it follows that $c$ has a unique maximal
segment that includes the element $\pi$,
where every element in the segment is a 2413-balloon.
We call the smallest element 
in this segment 
the 
\emph{least 2413-balloon}\extindex{least 2413-balloon}\footnote{The name 
    should really be ``least 2413-balloon in the chain that has only 2413-balloons above it''.}.

%
%
Further, since the permutation 1 is not a 2413-balloon,
it follows that $c$ has an element that is immediately below
the least 2413-balloon in the chain,
and we call this element the 
\emph{pivot}\extindof{2413-balloon}{pivot}{(in a 2413-balloon)}.

%
%
We define $\phi_c$ to be the least 2413-balloon in $c$,
$\psi_c$ to be the pivot in $c$,
$\tau_c$ to be the permutation that satisfies
$\ball{2413}{\tau_c} = \phi_c$,
and $\kappa_c$ to be the second-highest element of $c$.
Note that $\phi_c$ and $\psi_c$ must be distinct,
but we can have $\tau_c = \psi_c$.
Further, $\kappa_c$ is independent, and 
may be the same as $\phi_c$, $\psi_c$ or $\tau_c$.
Figure~\ref{figure-example-chains} shows some
example chains, highlighting these elements.
%
%
\begin{figure}
    \begin{center}
        \begin{tikzpicture}[]                
        \draw [] (0,5) -- (0,4);
        \draw [dotted] (0,4) -- (0,3);
        \draw [] (0,3) -- (0,2);
        \draw [dotted] (0,2) -- (0,1);
        \spoint{0}{5};
        \spoint{0}{4};
        \spoint{0}{3};
        \spoint{0}{2};
        \node [right] at (0.1, 5.0) {$\pi = \ball{2413}{\beta}$};
        \node [right] at (0.1, 4.0) {$\kappa_c$};
        \node [right] at (0.1, 3.0) {$\phi_c= \ball{2413}{\tau_c}$};
        \node [right] at (0.1, 2.0) {$\psi_c$};
        \draw [] (4,5) -- (4,4);
        \draw [] (4,4) -- (4,3);
        \draw [dotted] (4,3) -- (4,2);
        \spoint{4}{5};
        \spoint{4}{4};
        \spoint{4}{3};
        \node [right] at (4.1, 5.0) {$\pi = \ball{2413}{\beta}$};
        \node [right] at (4.1, 4.0) {$\kappa_c = \phi_c$};
        \node [right] at (4.1, 3.6) {$\phantom{\kappa_c} = \ball{2413}{\tau_c}$};
        \node [right] at (4.1, 3.0) {$\psi_c$};
        \draw [] (8,5) -- (8,3);
        \draw [dotted] (8,3) -- (8,2);
        \spoint{8}{5};
        \spoint{8}{3};
        \node [right] at (8.1, 5.0) {$\pi = \ball{2413}{\beta}$};
        \node [right] at (8.1, 4.6) {$\phantom{\pi} = \phi_c$};
        \node [right] at (8.1, 4.2) {$\phantom{\pi} = \ball{2413}{\tau_c}$};
        \node [right] at (8.1, 3.0) {$\kappa_c = \psi_c$};
        \end{tikzpicture}
    \end{center}
    \caption{Examples of chains, showing some possible relationships between $\pi$, $\kappa_c$, $\phi_c$, and $\psi_c$.}            
    \label{figure-example-chains}
\end{figure}
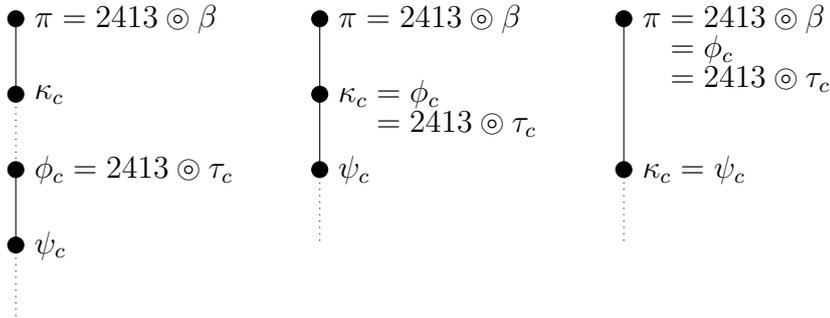    

%
%
We are now in a position to give a definition of the 
sets
$\chainr$,
$\chaing$, and
$\chainb$.
This definition
depends on the set of proper reductions of $\pi$, $\redset$,
which, as stated earlier, 
depends on the form of $\beta$.

Let $\chainc$ be the set of chains in the poset interval $[1, \pi]$.
We define subsets of $\chainc$ as follows:
\begin{align*}
\chainr &= \{ c : c \in \chainc \text{ and } \kappa_c \in \redset \}, \\
\chaing &= \{ c : c \in \chainc \setminus \chainr \text{ and } \psi_c \leq 2413 \}, \\
\chainb &= \{ c : c \in \chainc \setminus (\chainr \cup \chaing) \}.
\end{align*}
Clearly, every chain in $[1, \pi]$ is included
in exactly one of these subsets, 
and so these sets are a partition of the chains.

%
%
Given a pivot $\psi_c$, there is a unique 
permutation $\eta_c$ which we call the 
\emph{core}\extindof{2413-balloon}{core}{(2413-balloon)}
of $\psi_c$.  
In essence, $\eta_c$ is the smallest permutation
such that $\psi_c < \ball{2413}{\eta_c}$.
To determine the core,
we use the following algorithm:
\begin{align*}
\text{If $\psi_c$ can be written as~~}
& 1 \ominus ( ( \eta \ominus 1 ) \oplus 1) 
\text{~~or~~}
( ( 1 \oplus \eta ) \ominus 1 ) \oplus 1 \\
\text{or~~}
&
1 \oplus ( 1 \ominus ( \eta \oplus 1 ) ) 
\text{~~or~~}
( 1 \oplus ( 1 \ominus \eta ) ) \ominus 1, \\
\text{then set~~}
& \eta_c = \eta.\\
\text{Otherwise, if $\psi_c$ can be written as~~}
& ( \eta \ominus 1 ) \oplus 1 
\text{~~or~~}
1 \ominus ( \eta \oplus 1 ) \\
\text{~~or~~}
&1 \ominus \eta \ominus 1 
\text{~~or~~}
1 \oplus \eta \oplus 1 \\
\text{~~or~~}
& ( 1 \oplus \eta ) \ominus 1 
\text{~~or~~}
1 \oplus ( 1 \ominus \eta ), \\
\text{then set~~}
& \eta_c = \eta. \\
\text{Otherwise, if $\psi_c$ can be written as~~}
& 1 \oplus \eta 
\text{~~or~~}
1 \ominus \eta 
\text{~~or~~}
\eta \ominus 1 
\text{~~or~~}
\eta \oplus 1, \\
\text{then set~~}
& \eta_c = \eta.\\
\text{Otherwise, set~~}
& \eta_c = \psi_c.\\
\end{align*}

Since we have $\psi_c < \phi_c = \ball{2413}{\tau_c}$,
it is easy to see that 
$\eta_c \leq \tau_c$.
Note that $\ball{2413}{\eta_c}$ is the smallest 2413-balloon
that contains $\psi_c$.

%
%
We now define two functions, one for each of
$\chaing$ and
$\chainb$,
which will give us parity-reversing involutions.

\begin{align*}
\Phi_{\chaing}(c) & =
\begin{cases}
c \setminus \{2413\} & \text{If $\psi_c = 2413$} \\
c \cup \{2413\} & \text{If $\psi_c < 2413$} \\
\end{cases} \\
\Phi_{\chainb}(c) & =
\begin{cases}
c \setminus \{\ball{2413}{\eta_c}\} & \text{If $\eta_c = \tau_c$} \\
c \cup \{\ball{2413}{\eta_c}\} & \text{If $\eta_c < \tau_c$} \\
\end{cases} \\
\end{align*}

\begin{remark}
    If we were to allow the ballooning of the empty permutation $\epsilon$,
    and then treat $2413$ as $\ball{2413}{\epsilon}$
    then $\Phi_{\chaing}(c)$ is subsumed by 
    $\Phi_{\chainb}(c)$.  Doing this, however, introduces 
    additional complications in later proofs, 
    and so we prefer two involutions.
\end{remark}

For $\Phi_{\chaing}(c)$ to be a 
parity-reversing involution on $\chaing$,
we need to show that
if $c \in \chaing$, then
$\Phi_{\chaing}(c)$ is a chain, 
that $\Phi(c) \in \chaing$, 
and that
$c$ and $\Phi(c)$ have different parities.
It is easy to see that 
this last condition is true.
A similar comment applies to 
$\Phi_{\chainb}(c)$ and $\chainb$.

For $\Phi_{\chaing}(c)$ we can show that 
all the conditions hold for any $\redset$,
regardless of the form of $\beta$.
For $\Phi_{\chainb}(c)$
we show that some weaker conditions hold
for an arbitrary subset of the reductions of $\pi$,
and then, when we have an explicit set of proper reductions,
we show that all conditions hold.
The following Lemma
gives us a result that applies to 
$\Phi_{\chaing}(c)$ and $\Phi_{\chainb}(c)$
for any $\redset$,
and we
will use this result in both 
Section~\ref{section-2413-double-balloons}
and
Section~\ref{section-2413-balloons}.

%
%
\begin{lemma}
    \label{lemma-2413-balloon-phi-gb}
    Let $\pi = \ball{2413}{\beta}$, 
    with $\order{\beta} > 4$, 
    and let 
    $\chainr$,
    $\chaing$, and
    $\chainb$
    be as defined above.
    
    \begin{enumerate}[label=(\alph*)]
        \item \label{enum-lemma-balloon-g}
        If $c \in \chaing$, 
        then $\Phi_{\chaing}(c) \in \chaing$.
        
        \item \label{enum-lemma-balloon-b-eq}
        If $c \in \chainb$, 
        with $\eta_c = \tau_c$,
        and $\Phi_{\chainb}(c)$ is a chain,
        then $\Phi_{\chainb}(c) \in \chainb \cup \chainr$.
        
        \item \label{enum-lemma-balloon-b-lt}
        If $c \in \chainb$, 
        with $\eta_c < \tau_c$,
        then $\Phi_{\chainb}(c) \in \chainb \cup \chainr$.
    \end{enumerate}
\end{lemma}
\begin{proof}
    \textbf{Case~\ref{enum-lemma-balloon-g}.}
    First, assume that $c \in \chaing$ with $\psi_c = 2413$.
    Then $c$ contains a segment 
    $2413 < \ball{2413}{\tau_c}$,
    and $\cprime = \Phi_{\chaing}(c) = c \setminus \{ 2413 \}$.
    We can see that $\cprime$ is a chain,
    as 2413 is neither the smallest nor the largest entry in $\cprime$.
    Further, $\psi_{\cprime} < 2413$.  
    Since $\order{\beta} > 4$, 
    and $\order{\psi_{\cprime}} < 4$
    we must have $\cprime \not \in \chainr$,
    and therefore $\cprime \in \chaing$.
    
    Now assume that $c \in \chaing$ with $\psi_c < 2413$.
    Then $c$ contains a segment 
    $\psi_c < \ball{2413}{\tau_c}$,
    and $\cprime = \Phi_{\chaing}(c) = c \cup \{ 2413 \}$.
    We can see that $\cprime$ is a chain,
    since $\psi_c < 2413 < \ball{2413}{\tau_c}$,
    and further, $\psi_{\cprime} = 2413$.
    Since $\order{\beta} > 4$, 
    and $\order{\psi_{\cprime}} = 4$
    we must have $\cprime \not \in \chainr$,
    and therefore $\cprime \in \chaing$. 
    
    \textbf{Case~\ref{enum-lemma-balloon-b-eq}.}
    Let $c$ be a chain in $\chainb$, 
    with $\eta_c = \tau_c$.
    Then $c$ contains a segment
    $\psi_c < \ball{2413}{\tau_c}$,
    and $\cprime = \Phi_{\chainb}(c) = c \setminus \{ \ball{2413}{\tau_c} \}$.
    If $\tau_c = \beta$, then $\cprime$ is not a chain,
    so we must have $\tau_c < \beta$, and therefore
    $\cprime$ is a chain that contains a segment
    $\psi_c < \ball{2413}{\gamma}$,
    with $\tau_c < \gamma$.
    Now, $\psi_c$ is the pivot of $\cprime$,
    so we cannot have $\cprime \in \chaing$
    as this would imply that $c \in \chaing$,
    which is a contradiction.
    Thus either
    $\cprime \in \chainr$ or $\cprime \in \chainb$.
    
    \textbf{Case~\ref{enum-lemma-balloon-b-lt}.}
    Let $c$ be a chain in $\chainb$, 
    with $\eta_c < \tau_c$.
    Then $c$ contains a segment
    $\psi_c < \ball{2413}{\tau_c}$,
    and $\cprime = \Phi_{\chainb}(c) = c \cup \{ \ball{2413}{\eta_c} \}$.
    We can see that $\cprime$ is a chain since
    $\psi_c < \ball{2413}{\eta_c} < \ball{2413}{\tau_c}$.
    Now, $\psi_c$ is the pivot of $\cprime$,
    so we cannot have $\cprime \in \chaing$
    as this would imply that $c \in \chaing$,
    which is a contradiction.
    So either
    $\cprime \in \chainr$ or $\cprime \in \chainb$.
\end{proof}

%
%
We now have
\begin{observation}
    \label{observation-all-we-have-to-do}
    If $\pi = \ball{2413}{\beta}$,
    with $\order{\beta} > 4$,
    then to show that
    $\Phi_{\chainb}$ is a parity-reversing involution on $\chainb$
    it is sufficient to show that:    
    \begin{enumerate}[label=(\alph*)]
        \item \label{enum-observation-all-we-have-to-do-b-eq}
        If $c \in \chainb$ 
        and $\eta_c = \tau_c$, 
        then $\Phi_{\chainb}(c)$ is a chain, and
        $\Phi_{\chainb}(c) \not\in \chainr$.
        
        \item \label{enum-observation-all-we-have-to-do-b-chain}
        If $c \in \chainb$,
        and $\eta_c < \tau_c$, 
        then $\Phi_{\chainb}(c) \not\in \chainr$.
    \end{enumerate}
    
    Further,
    if
    $\Phi_{\chainb}$ is a parity-reversing involution on the set $\chainb$,
    then 
    $\mobp{\pi} = - \sum_{\sigma \in \redset} \mobp{\sigma}$.    
\end{observation}
\begin{proof}
    Combining~\ref{enum-observation-all-we-have-to-do-b-eq}
    and~\ref{enum-observation-all-we-have-to-do-b-chain}
    above
    with cases~\ref{enum-lemma-balloon-b-eq}
    and~\ref{enum-lemma-balloon-b-lt}
    of Lemma~\ref{lemma-2413-balloon-phi-gb}
    gives us that 
    $\Phi_{\chainb}$ is a parity-reversing involution on $\chainb$.
    
    This now gives us that
    $\sum_{c \in \chainb} (-1)^{\order{c}} = 0$.     
    By Lemma~\ref{lemma-2413-balloon-phi-gb}, we know that 
    $\sum_{c \in \chaing} (-1)^{\order{c}} = 0$, 
    so we must have
    $\mobp{\pi} = \sum_{c \in \chainr} (-1)^{\order{c}}$.
    Since the chains in $\chainr$ are defined by the 
    second-highest element ($\kappa_c$) being in $\redset$, 
    the final part of the observation follows
    by applying Corollary~\ref{corollary-hall-sum-second-highest-set}.
\end{proof}

\section[The principal \mob function of double 2413-balloons]{The principal \mob function of double\\ 2413-balloons}
\label{section-2413-double-balloons}

We are now able to state and prove our first major result.
\begin{theorem}
    \label{theorem-2413-balloon-beta-is-a-balloon}
    Let $\pi = \ball{2413}{\beta}$,
    where $\beta$ is a 2413-balloon,
    Then $\mobp{\pi} = 2 \mobp{\beta}$.	
\end{theorem}
\begin{proof}
    Note that $\beta \not \in \redset$, and further
    that $\order{\beta} > 4$, 
    since $\beta$ is a 2413-balloon.
    
    Using Observation~\ref{observation-all-we-have-to-do},
    we will show that
    $\Phi_{\chainb}$ is a parity-reversing involution on $\chainb$.
    Once we have shown that we have parity-reversing involutions,
    we will then show how to express the Hall sum of $\chainr$
    in terms of $\mobp{\beta}$.
    
    %
    %
    \begin{subproof}[Proof that $\Phi_{\chainb}$ is a parity-reversing involution on $\chainb$.]
        Let $c$ be a chain in $\chainb$.
        
        %
        First, assume that $\eta_c = \tau_c$.
        If $\tau_c = \beta$, then either 
        $\psi_c$ is a proper reduction of $\pi$,
        or $\psi_c = \beta$.
        In the first case, $c \in \chainr$,
        and in the second case 
        $\psi_c$ is a 2413-balloon, 
        and these are both contradictions.
        Thus
        we must have $\tau_c < \beta$,
        and so there is at least one permutation
        in $c$ greater than $\phi_c$.  
        It follows that $\cprime$ is a chain.
        We now show that $\cprime \not\in \chainr$.
        Assume, to the contrary, that $\cprime \in \chainr$
        which implies that $\psi_c$ is a proper reduction of $\pi$.
        But now we have $\eta_c = \beta$, which is a contradiction,
        so $\psi_c$ is not a proper reduction of $\pi$,
        therefore $\cprime \not \in \chainr$.
        
        %
        Now assume that $\eta_c < \tau_c$.
        Let $\cprime = \Phi_{\chainb}(c) = c \cup \{ \ball{2413}{\eta_c} \}$.  
        We know by Lemma~\ref{lemma-2413-balloon-phi-gb}
        that this is a chain.
        Either $\kappa_c = \kappa_{\cprime}$,
        or $\kappa_{\cprime}$ 
        is a 2413-balloon.
        If $\kappa_c = \kappa_{\cprime}$,
        then $\cprime \not \in \chainr$.
        If $\kappa_{\cprime}$ is a 2413-balloon,
        then $\kappa_{\cprime} \not \in \redset$,
        so $\cprime \not \in \chainr$.
        Thus we must have $\cprime \not\in \chainr$.
        
        So now we have that 
        if $c \in \chainb$ 
        and $\eta_c = \tau_c$, 
        then $\Phi_{\chainb}(c)$ is a chain;
        and that for any $c \in \chainb$,
        $\Phi_{\chainb}(c) \in \chainb$.
        It follows that
        $\Phi_{\chainb}$ is a parity-reversing involution
        on $\chainb$.        
    \end{subproof}
    
    We now have that
    $\Phi_{\chaing}$ and
    $\Phi_{\chainb}$ 
    are parity-reversing involutions on
    $\chaing$ and
    $\chainb$ respectively.
    It follows from 
    Observation~\ref{observation-all-we-have-to-do}
    that
    $
    \mobp{\pi} = - \sum_{\sigma \in \redset} \mobp{\sigma}.
    $
    We now show how to express $\mobp{\sigma}$, where $\sigma \in \redset$,
    in terms of $\mobp{\beta}$.
    
    We start by noting that since $\beta$ is a 2413-balloon,
    then $\beta$ has no corners.
    Now,
    take the case where $\sigma = \ball{\ex{2}413}{\beta}$,
    which is the first permutation in 
    Figure~\ref{figure-2413-reductions}.
    Note that we can write 
    $\sigma = \oneminus ((\beta \minusone) \plusone)$.
    Applying Lemma~\ref{lemma-oneplus} to the 
    outermost three points in $\sigma$ 
    (those from the $\ex{2}413$), we find that
    $\mobp{\sigma} = -\mobp{\beta}$.
    The other cases are similar, and 
    this gives us:\footnote{This table is slightly redundant,
        as the entries are determined by the parity of the red points.
        We include it as later results have similar tables where some
        values of $\mobp{\sigma}$ are zero, and this gives a consistent presentation.}
    \begin{center}
        $
        \begin{array}{ccccc}
        \begin{array}{lr}
        \sigma & \mobp{\sigma} \\
        \midrule
        \ball{\ex{2}413}{\beta} & - \mobp{\beta} \\
        \ball{2\ex{4}13}{\beta} & - \mobp{\beta} \\
        \ball{24\ex{1}3}{\beta} & - \mobp{\beta} \\
        \ball{241\ex{3}}{\beta} & - \mobp{\beta} \\
        \phantom{x} & \phantom{x} \\		    
        \phantom{x} & \phantom{x} \\		    
        \end{array} 
        & \phantom{xxx} &
        \begin{array}{lr}
        \sigma & \mobp{\sigma} \\
        \midrule
        \ball{\ex{2}\ex{4}13}{\beta} & \mobp{\beta} \\
        \ball{\ex{2}4\ex{1}3}{\beta} & \mobp{\beta} \\
        \ball{\ex{2}41\ex{3}}{\beta} & \mobp{\beta} \\
        \ball{2\ex{4}\ex{1}3}{\beta} & \mobp{\beta} \\
        \ball{2\ex{4}1\ex{3}}{\beta} & \mobp{\beta} \\
        \ball{24\ex{1}\ex{3}}{\beta} & \mobp{\beta} \\
        \end{array} 
        & \phantom{xxx} &
        \begin{array}{lr}
        \sigma & \mobp{\sigma} \\
        \midrule
        \ball{2\ex{4}\ex{1}\ex{3}}{\beta} & - \mobp{\beta} \\
        \ball{\ex{2}4\ex{1}\ex{3}}{\beta} & - \mobp{\beta} \\
        \ball{\ex{2}\ex{4}1\ex{3}}{\beta} & - \mobp{\beta} \\
        \ball{\ex{2}\ex{4}\ex{1}3}{\beta} & - \mobp{\beta} \\
        \phantom{x} & \phantom{x} \\		    
        \phantom{x} & \phantom{x} \\		    
        \end{array} \\		  		    
        \end{array}
        $
    \end{center}
    It is now easy to see that 
    \[
    \sum_{\sigma \in \redset} \mobp{\sigma} = - 2 \mobp{\beta}
    \]
    and the result follows directly.	
\end{proof}

\section{The growth of the \mob function}
\label{section-growth-rate-of-mu}

We define
$\mobmaxn = \max \{ \order{\mobp{\pi}} : \order{\pi} = n \}$.
Previous work in~\cite{Jelinek2020} and~\cite{Smith2013}
has shown that the growth of $\mobmaxn$ is at least polynomial.  
We will show that the growth is at least exponential.
We have
\begin{theorem}
    \label{theorem-growth-of-mobius-function}
    For all $n$, 
    $\mobmaxn \geq 2^{\lfloor n/4 \rfloor - 1 }$.
\end{theorem}    
\begin{proof}
    We start by defining a function
    to construct a permutation of length $n$.
    \[
    \pi^{(n)} =
    \begin{cases}
    1                   & \text{If $n = 1$} \\
    12                  & \text{If $n = 2$} \\
    132                 & \text{If $n = 3$} \\
    2413                & \text{If $n = 4$} \\
    \ball{2413}{\pi^{(n-4)}} & \text{Otherwise}
    \end{cases}
    \]
    Note that for $n > 8$, $\pi^{(n)}$ is a double 2413-balloon.
    It is simple to calculate $\mobp{\pi^{(n)}}$ for $n=1, \ldots, 8$,
    and these values are given below.
    
    \begin{center}
        \begin{tabular}{lcl}
            $\mobp{\pi^{(1)}} = \mobp{1} = 1$,  &
            \phantom{xxxxx} & 
            $\mobp{\pi^{(5)}} = \mobp{25314} =  4$, \\
            $\mobp{\pi^{(2)}} = \mobp{12} = -1$, && 
            $\mobp{\pi^{(6)}} = \mobp{263415} = -1$, \\
            $\mobp{\pi^{(3)}} = \mobp{132} = 1$, &&
            $\mobp{\pi^{(7)}} = \mobp{2735416} = 1$, \\
            $\mobp{\pi^{(4)}} = \mobp{2413} = -3$, && 
            $\mobp{\pi^{(8)}} = \mobp{28463517}  = -6$.
        \end{tabular}    
    \end{center}
    These values match
    Theorem~\ref{theorem-growth-of-mobius-function},
    and so this is true for $n \leq 8$.
    For $n > 8$, 
    $\mobp{\pi^{(n)}} = 2 \mobp{\pi^{(n-4)}}$ 
    by Theorem~\ref{theorem-2413-balloon-beta-is-a-balloon},
    and the result follows immediately.    
\end{proof}
\begin{remark}
    \label{remark-simples-in-2413-balloons}
    It is easy to see that, with the definitions above,
    the only simple permutations that can be contained
    in $\pi^{(n)}$ are
    $1$, 
    $12$, 
    $21$,
    $2413$, and
    $25314$.
    This answers
    Problem 4.4 in~\cite{Jelinek2020},
    which asks
    whether $\mobp{\pi}$ is bounded on a hereditary 
    class which contains only finitely many simple permutations,
    as, by Theorem~\ref{theorem-growth-of-mobius-function},
    we have unbounded growth, but only finitely many simple 
    permutations.
\end{remark}

If we repeat the ballooning process, as we do in $\pi^{(n)}$,
then the permutation plot is rather striking.  
We illustrate this in 
Figure~\ref{figure-multiple-2413-ballons},
which shows $\pi^{(21)}$.
\begin{figure}
    \begin{center}
        \begin{tikzpicture}[scale=0.3]
        \plotpermgrid{02,21,04,19,06,17,08,15,10,13,11,09,12,07,14,05,16,03,18,01,20}
        \end{tikzpicture}
        \caption{A permutation plot showing $\pi^{(21)}$.}
        \label{figure-multiple-2413-ballons}
    \end{center}
\end{figure}
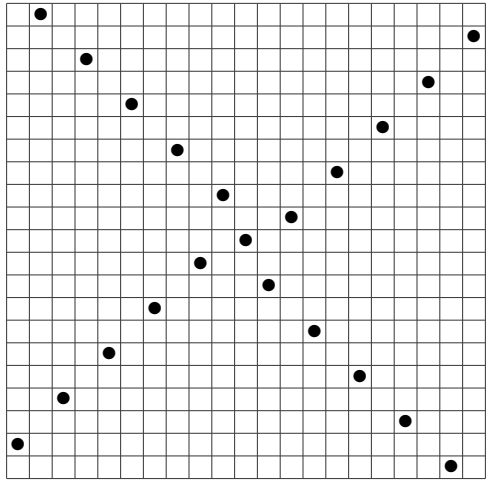

\section{The principal \mob function of 2413-balloons}
\label{section-2413-balloons}

Theorem~\ref{theorem-2413-balloon-beta-is-a-balloon} 
gives us an expression for the value of the \mob function
$\mobp{\pi}$ when $\pi$ is a double 2413-balloon.
We expand on this to find an expression for 
the \mob function $\mobp{\pi}$ when
$\pi$ is any 2413-balloon.

We start with a Lemma
that handles the case where
$\beta$ is not a 2413-balloon,
and has more than four points.
The structure of our proof is similar to that of 
Theorem~\ref{theorem-2413-balloon-beta-is-a-balloon},
but we present a complete
argument to aid readability.

We will show
\begin{lemma}
    \label{lemma-2413-balloon-beta-not-a-balloon}
    Let $\pi = \ball{2413}{\beta}$,
    where $\beta$ is not a 2413-balloon,
    and $\order{\beta} > 4$.
    Then $\mobp{\pi} = \mobp{\beta}$.	
\end{lemma}
\begin{proof}
    First note that if $\beta$ is monotonic, then
    by Corollary~\ref{corollary-monotonic-interval}
    we have $\mobp{\beta} = 0 = \mobp{\pi}$.
    For the remainder of this proof, we assume that
    $\beta$ is not monotonic.
    
    If $\beta$ has one corner, then
    without loss of generality, we can assume by symmetry that
    $\beta = \oneplus \gamma$.
    Similarly, if $\beta$ has two corners,
    then we can assume that $\beta = \oneplus \gamma \plusone$.
    
    As before, we will use
    Observation~\ref{observation-all-we-have-to-do}.
    We will show that
    $\Phi_{\chainb}$ is a parity-reversing involution on $\chainb$.
    Once we have shown that we have parity-reversing involutions,
    we will then show how to express the Hall sum of $\chainr$
    in terms of $\mobp{\beta}$.    
    
    The proper reductions of $\pi$ depend on the number of corners of $\beta$.
    Below we list the improper reductions of $\pi$ for each case.
    
    \begin{center}    
        \begin{tabular}{lr}
            \toprule
            Corners in $\beta$ & Improper reductions of $\pi$  \\
            \midrule
            No corners & 
            None. \\
            One corner ($\beta = 1 \oplus \gamma$) &
            $\ball{\ex{2}\ex{4}13}{\beta}$,
            $\ball{\ex{2}\ex{4}1\ex{3}}{\beta}$,
            $\ball{\ex{2}\ex{4}\ex{1}3}{\beta}$,
            and
            $\beta$. \\
            Two corners ($\beta = 1 \oplus \gamma  \oplus 1$) &
            $\ball{\ex{2}\ex{4}13}{\beta}$,
            $\ball{24\ex{1}\ex{3}}{\beta}$,
            \\ & 
            $\ball{2\ex{4}\ex{1}\ex{3}}{\beta}$,
            $\ball{\ex{2}4\ex{1}\ex{3}}{\beta}$,
            $\ball{\ex{2}\ex{4}1\ex{3}}{\beta}$,
            $\ball{\ex{2}\ex{4}\ex{1}3}{\beta}$, 
            \\ & and    
            $\beta$. \\
            \bottomrule
        \end{tabular}        
    \end{center}
    
    %
    %
    \begin{subproof}[Proof that $\Phi_{\chainb}$ is a parity-reversing involution on $\chainb$.]
        Let $c$ be a chain in $\chainb$.
        
        First, assume that $\eta_c = \tau_c$, so
        $\cprime = c \setminus \{ \ball{2413}{\tau_c} \}$.
        We start by showing that $\cprime$ is a valid chain.
        Assume otherwise, which implies $\tau_c = \beta$.
        
        If $\beta$ has no corners, 
        then
        $\psi_c \in \redset$, so $c \in \chainr$,
        which is a contradiction.
        
        If $\beta$ has one corner,
        then
        either $\psi_c \in \redset$, 
        which is a contradiction,
        or $\psi_c \not \in \redset$.
        In the latter case,
        assume, without loss of generality,
        that $\beta = 1 \oplus \gamma$.
        Then 
        $\ball{\ex{2}\ex{4}13}{\beta}       = \ball{2\ex{4}13}{\gamma}$,
        $\ball{\ex{2}\ex{4}1\ex{3}}{\beta}  = \ball{2\ex{4}1\ex{3}}{\gamma}$,
        $\ball{\ex{2}\ex{4}\ex{1}3}{\beta}  = \ball{2\ex{4}\ex{1}3}{\gamma}$,
        and
        $\beta                              = \ball{2\ex{4}\ex{1}\ex{3}}{\gamma}$.
        Thus in all cases where 
        $\psi_c \not \in \redset$,
        we have that $\eta_c$ is not minimal, 
        which is a contradiction.
        
        Finally, if $\beta$ has two corners,
        then
        either $\psi_c \in \redset$,
        which is a contradiction,
        or $\psi_c \not \in \redset$.
        The latter case implies that
        $\psi_c = \beta$,
        and then we have that
        either 
        $\psi_c = 1 \oplus \gamma \oplus 1 = \ball{2\ex{4}\ex{1}3}{\gamma}$
        or
        $\psi_c = 1 \ominus \gamma \ominus 1 = \ball{\ex{2}41\ex{3}}{\gamma}$,
        so
        $\eta_c$ is not minimal,
        which is a contradiction.
        
        Thus we have that $\cprime$ must be a chain,
        and, moreover, $\tau_c \neq \beta$.
        
        We now show that $\cprime \not\in \chainr$.
        Assume, to the contrary, that $\cprime \in \chainr$
        which implies that $\psi_c$ is a proper reduction of $\pi$.
        But now we have $\eta_c = \beta$, 
        but this would give $\tau_c = \beta$,
        which is a contradiction,
        therefore $\cprime \not \in \chainr$.
        
        Now assume that $\eta_c < \tau_c$.
        Let $\cprime = \Phi_{\chainb}(c) = c \cup \{ \ball{2413}{\eta_c} \}$,
        and we know from Lemma~\ref{lemma-2413-balloon-phi-gb}
        that $\cprime$ is a chain.
        Now either $\kappa_c = \kappa_{\cprime}$,
        or $\kappa_{\cprime}=\ball{2413}{\eta_c}$ is a 2413-balloon.
        In either case we have $\cprime \not \in \chainr$.
        
        So if $c \in \chainb$, then $\Phi_{\chainb}(c)$
        is a chain in $\chainb$,
        and thus $\Phi_{\chainb}$ is a parity-reversing involution.
    \end{subproof}    
    
    We have shown that
    $\Phi_{\chaing}$ and
    $\Phi_{\chainb}$ 
    are parity-reversing involutions on
    $\chaing$ and
    $\chainb$ respectively.
    It follows from 
    Observation~\ref{observation-all-we-have-to-do}
    that
    $
    \mobp{\pi} = - \sum_{\sigma \in \redset} \mobp{\sigma}.
    $
    We now show how to express $\mobp{\sigma}$, where $\sigma \in \redset$
    in terms of $\mobp{\beta}$.
    We use a similar mechanism to that used
    in Theorem~\ref{theorem-2413-balloon-beta-is-a-balloon}.
    There are some additional considerations
    where $\beta$ has one or two corners.
    
    As an example, 
    take the case where $\sigma = \ball{2\ex{4}13}{\beta}$,
    and $\beta$ has one corner, and so, by our assumption,  can be
    written as $\oneplus \gamma$.
    We can write 
    $\sigma = ((\oneplus \beta) \minusone) \plusone$,
    and expanding $\beta$ we have
    $\sigma = ((\oneplus \oneplus \gamma) \minusone) \plusone$,    
    Applying Lemma~\ref{lemma-oneplus} to the 
    outermost two points in $\sigma$, 
    we find that
    $\mobp{\sigma} = \mobp{\oneplus \oneplus \gamma}$,
    and by 
    Lemma~\ref{lemma-oneplus-oneplus} we now have
    $\mobp{\sigma} = 0$.
    Because of this, our analysis depends on the number of corners of $\beta$,
    and we consider each case separately below.
    
    If $\beta$ has no corners, then we have
    \begin{center}
        $
        \begin{array}{ccccc}
        \begin{array}{lr}
        \sigma & \mobp{\sigma} \\
        \midrule
        \ball{\ex{2}413}{\beta} & - \mobp{\beta} \\
        \ball{2\ex{4}13}{\beta} & - \mobp{\beta} \\
        \ball{24\ex{1}3}{\beta} & - \mobp{\beta} \\
        \ball{241\ex{3}}{\beta} & - \mobp{\beta} \\
        \phantom{x} & \phantom{x} \\		    
        \phantom{x} & \phantom{x} \\		    
        \end{array} 
        & \phantom{xxx} &
        \begin{array}{lr}
        \sigma & \mobp{\sigma} \\
        \midrule
        \ball{\ex{2}\ex{4}13}{\beta} & \mobp{\beta} \\
        \ball{\ex{2}4\ex{1}3}{\beta} & \mobp{\beta} \\
        \ball{\ex{2}41\ex{3}}{\beta} & \mobp{\beta} \\
        \ball{2\ex{4}\ex{1}3}{\beta} & \mobp{\beta} \\
        \ball{2\ex{4}1\ex{3}}{\beta} & \mobp{\beta} \\
        \ball{24\ex{1}\ex{3}}{\beta} & \mobp{\beta} \\
        \end{array} 
        & \phantom{xxx} &
        \begin{array}{lr}
        \sigma & \mobp{\sigma} \\
        \midrule
        \ball{2\ex{4}\ex{1}\ex{3}}{\beta} & - \mobp{\beta} \\
        \ball{\ex{2}4\ex{1}\ex{3}}{\beta} & - \mobp{\beta} \\
        \ball{\ex{2}\ex{4}1\ex{3}}{\beta} & - \mobp{\beta} \\
        \ball{\ex{2}\ex{4}\ex{1}3}{\beta} & - \mobp{\beta} \\
        \phantom{x} & \phantom{x} \\		    
        \beta & \mobp{\beta} \\		    
        \end{array} \\		  		    
        \end{array}
        $
    \end{center}
    
    If $\beta$ has one corner, under our assumption
    that $\beta$ = $\oneplus \gamma$, we have
    \begin{center}
        $
        \begin{array}{ccccc}
        \begin{array}{lr}
        \sigma & \mobp{\sigma} \\
        \midrule
        \ball{\ex{2}413}{\beta} & - \mobp{\beta} \\
        \ball{2\ex{4}13}{\beta} & 0 \\
        \ball{24\ex{1}3}{\beta} & - \mobp{\beta} \\
        \ball{241\ex{3}}{\beta} & - \mobp{\beta} \\
        \phantom{x} & \phantom{x} \\		    
        \end{array} 
        & \phantom{xxx} &
        \begin{array}{lr}
        \sigma & \mobp{\sigma} \\
        \midrule
        \ball{\ex{2}4\ex{1}3}{\beta} & \mobp{\beta} \\
        \ball{\ex{2}41\ex{3}}{\beta} & \mobp{\beta} \\
        \ball{2\ex{4}\ex{1}3}{\beta} & 0 \\
        \ball{2\ex{4}1\ex{3}}{\beta} & 0 \\
        \ball{24\ex{1}\ex{3}}{\beta} & \mobp{\beta} \\
        \end{array} 
        & \phantom{xxx} &
        \begin{array}{lr}
        \sigma & \mobp{\sigma} \\
        \midrule
        \ball{2\ex{4}\ex{1}\ex{3}}{\beta} & 0 \\
        \ball{\ex{2}4\ex{1}\ex{3}}{\beta} & - \mobp{\beta} \\
        \phantom{x} & \phantom{x} \\ 
        \phantom{x} & \phantom{x} \\
        \phantom{x} & \phantom{x} \\		    
        \end{array} \\		  		    
        \end{array}
        $
    \end{center}
    
    Finally, if $\beta$ has two corners, under
    our assumption that $\beta = \oneplus \gamma \plusone$,
    we have
    \begin{center}
    	$
	    \begin{array}{ccc}
		    \begin{array}{lr}
			    \sigma & \mobp{\sigma} \\
    			\midrule
	     		\ball{\ex{2}413}{\beta} & - \mobp{\beta} \\
		    	\ball{2\ex{4}13}{\beta} & 0 \\
			    \ball{24\ex{1}3}{\beta} & 0 \\
    			\ball{241\ex{3}}{\beta} & - \mobp{\beta} \\
	    	\end{array} 
		    & \phantom{xxx} &
    		\begin{array}{lr}
	    		\sigma & \mobp{\sigma} \\
		    	\midrule
    			\ball{\ex{2}4\ex{1}3}{\beta} & 0 \\
	    		\ball{\ex{2}41\ex{3}}{\beta} & \mobp{\beta} \\
		    	\ball{2\ex{4}\ex{1}3}{\beta} & 0 \\
    			\ball{2\ex{4}1\ex{3}}{\beta} & 0 \\
 	    	\end{array} 
    	\end{array}
    	$
    \end{center}
    In all three cases we have    
    \[
    \sum_{\sigma \in \redset} \mobp{\sigma} = - \mobp{\beta}
    \]
    and the result follows directly.
\end{proof}	

We are now in a position to state and prove the main Theorem
for this section.
\begin{theorem}
    \label{theorem-2413-balloons}
    Let $\pi = \ball{2413}{\beta}$.  Then
    \begin{align*}
    \mobp{\pi} = 
    \begin{cases}
    4 & \text{If $\beta = 1$} \\
    -6 & \text{If $\beta = 2413$} \\
    2 \mobp{\beta} & \text{If $\beta$ is a 2413-balloon} \\
    \mobp{\beta}  & \text{Otherwise}.	
    \end{cases}
    \end{align*}
\end{theorem}
\begin{proof}
    The value of $\mobp{\ball{2413}{\beta}}$
    for the symmetry classes of $\beta$ 
    with $\order{\beta} \leq 4$ are shown below.
    \begin{center}
        $
        \begin{array}{ccc}
        \begin{array}{lrr}
        \beta & \mobp{\beta} &\mobp{\ball{2413}{\beta}} \\
        \midrule
        1    &  1 &  4 \\
        12   & -1 & -1 \\
        123  &  0 &  0 \\
        132  &  1 &  1 \\
        1234 &  0 &  0 \\
        1243 &  0 &  0 \\
        \end{array} 
        & \phantom{xxx} &
        \begin{array}{lrr}
        \beta & \mobp{\beta} &\mobp{\ball{2413}{\beta}} \\
        \midrule
        1324 & -1 & -1 \\
        1342 & -1 & -1 \\
        1432 &  0 &  0 \\
        2143 & -1 & -1 \\
        2413 & -3 & -6 \\
        \phantom{x} & \phantom{x} \\		    
        \end{array} \\		  		    
        \end{array}
        $
    \end{center}        
    It is easy to see that
    these values meet Theorem~\ref{theorem-2413-balloons}.
    We now combine
    Theorem~\ref{theorem-2413-balloon-beta-is-a-balloon} and
    Lemma~\ref{lemma-2413-balloon-beta-not-a-balloon}
    to complete the proof.
\end{proof}

\section{Concluding remarks}
\label{section-concluding-remarks}

\subsection{Generalising the balloon operator}

Given two permutations $\alpha$ and $\beta$,
with lengths $a$ and $b$ respectively,
and two integers $i,j$ which satisfy
$0 \leq i,j \leq a$, the
\emph{$i,j$-balloon}\extindex{$i,j$-balloon} 
of $\beta$ by $\alpha$,
written as 
$\ballgen{i,j}{\alpha}{\beta}$,
is the 
permutation formed by inserting
the permutation $\beta$ into $\alpha$
between the $i$-th and $i+1$-th columns of $\alpha$,
and 
between the $j$-th and $j+1$-th rows of $\alpha$.
The integers $i$ and $j$ are, collectively, the
\emph{indexes}\extindex[balloon]{indexes} 
of the balloon.

Formally, we have
\begin{align*}
(\ballgen{i,j}{\alpha}{\beta})_x
& =
\begin{cases}
\alpha_x 
& \text{if $x \leq i$ and $\alpha_x \leq j $}\\
\alpha_x + \order{\beta} 
& \text{if $x \leq i$ and $\alpha_x > j $}\\
\beta_{x-i} + j 
& \text{if $x > i $ and $x \leq i+\order{\beta}$ }\\
\alpha_{x-\order{\beta}} 
& \text{if $x > i+\order{\beta}$ and $\alpha_{x-\order{\beta}} \leq j $}\\
\alpha_{x-\order{\beta}} + \order{\beta} 
& \text{if $x > i+\order{\beta}$ and $\alpha_{x-\order{\beta}} > j $}\\ 
\end{cases}
\end{align*}
As before, the balloon notation is not associative.
Unlike 2413-balloons, which have to be interpreted
as right-associative, 
generalized balloons can use brackets to define associativity.
Note that the 2413-balloon defined in Section~\ref{section-definitions-and-notation}
is written as $\ballgen{2,2}{2413}{\beta}$ in our
generalized notation.

We remark that for any $\alpha$ and any $\beta$, we have
$\ballgen{0,0}{\alpha}{\beta} = \alpha \oplus \beta$,
and we can easily determine $\mobp{\alpha \oplus \beta}$
using results from
Propositions~1~and~2 of
Burstein, 
Jel{\'{i}}nek, 
Jel{\'{i}}nkov{\'{a}} and 
Steingr{\'{i}}msson~\cite{Burstein2011}.

\subsection{Generalised 2413-balloons}

If we restrict $\alpha$ to 2413,
then, up to symmetry, there are seven 
possible values for the indexes:
$(0,0)$,
$(0,1)$,
$(0,2)$,
$(1,0)$,
$(1,1)$,
$(1,2)$, and
$(2,2)$.
Theorem~\ref{theorem-2413-balloons} 
handles the case where the indexes are $(2,2)$,
and~\cite{Burstein2011} 
handles the case where the indexes are $(0,0)$.
For the other indexes, we have
\begin{conjecture}
    \label{conjecture-most-2413-balloons}
    Let $\pi = \ballgen{i,j}{2413}{\beta}$,
    where 
    $(i,j) \in 
    \{
    (0,1),
    (0,2),
    (1,1),
    (1,2)
    \}$.
    Then
    \[
    \mobp{\pi} =
    \begin{cases}
    0 & \text{If $(i,j) = (0,1)$ and $\beta = \tau \plusone$} \\
    0 & \text{If $(i,j) = (0,2)$ and $\beta = \tau \minusone$} \\
    0 & \text{If $(i,j) = (1,1)$ and $\beta = \oneminus \tau$ or $\beta = 12$} \\
    0 & \text{If $(i,j) = (1,2)$ and $\beta = \oneplus \tau$} \\
    \mobp{\beta} & \text{Otherwise.}
    \end{cases}
    \]
\end{conjecture}
and
\begin{conjecture}
    \label{conjecture-1-0-2413-balloons}
    Let $\pi = \ballgen{1,0}{2413}{\beta}$.
    Then
    \[
    \mobp{\pi} =
    \begin{cases}
    6 & \text{If $\beta = 1$} \\
    -2 & \text{If $\beta = 21$} \\
    0 & \text{If $\beta = 312$} \\
    2 \mobp{\beta} & \text{If $\beta = \ballgen{1,0}{2413}{\gamma}$} \\
    \mobp{\beta} & \text{Otherwise.}
    \end{cases}
    \]
\end{conjecture}
Lemma~\ref{lemma-wedge-is-sum-of-reductions}
in Chapter~\ref{chapter_balloon_permutations_preprint} 
can be applied to the cases where
the indexes are 
$(0,1)$,
$(0,2)$, or
$(1,0)$.
Using this lemma, together with
some of the techniques used earlier in this chapter,
it is easy to show that
Conjecture~\ref{conjecture-most-2413-balloons}
is true in these cases.
For brevity, we do not provide proofs here.

\subsection[Bounding the \mob function on hereditary classes]{Bounding the \mob function on hereditary\\ classes}

Corollary 24 in 
Burstein, Jel{\'{i}}nek, Jel{\'{i}}nkov{\'{a}} 
and Steingr{\'{i}}msson~\cite{Burstein2011}
gives us that if $\pi$ is separable,
then $\mobp{\pi} \in \{ 0, \pm1 \}$.  
The simple permutations in the 
hereditary class of separable permutations
are 
$1$,
$12$, and
$21$.
In Remark~\ref{remark-simples-in-2413-balloons}
we have unbounded growth where the 
simple permutations in the hereditary class are just
$1$, 
$12$, 
$21$,
$2413$, and
$25314$,
so adding $2413$ and $25314$ to the simple permutations
moves us from bounded growth
to unbounded growth.
This then leads to:
\begin{question}
    If $C$ is a hereditary class containing just the simple permutations
    $1$, 
    $12$, 
    $21$ and
    $2413$,
    and $\pi \in C$, then
    is $\mobp{\pi}$ bounded?
    Further, if $D$ is a hereditary class containing just the simple permutations
    $1$, 
    $12$, 
    $21$,
    $2413$, and
    $3142$,
    and $\pi \in D$, then
    is $\mobp{\pi}$ bounded?    
\end{question}

\section{Chapter summary}

The main result from this paper is a proof that 
the growth of 
$\mobmaxn = \max \{ \order{\mobp{\pi}} : \order{\pi} = n \}$.
is at least exponential.
This is proved by finding explicit recursions 
for the principal \mob function
of double 2413-balloons.

This result is not unexpected.  Indeed, 
while the main result from 
Jel{\'{i}}nek, Kantor, Kyn{\v{c}}l and Tancer~\cite{Jelinek2020}
is that the growth of $\mobmaxn$ is 
bounded below by an order-7 polynomial,
in the final section of their paper
they define a set of permutations $\kappa_n$ as
\[
\kappa_n 
=
n + 1, n + 3, \ldots, 3n - 1, 1, 3n + 1, 2, 3n + 2, \ldots, n, 4n, n + 2, n + 4, \ldots, 3n;
\]
and then they conjecture that 
\emph{``the absolute value of $\mobp{\kappa_n}$ is exponential in $n$''}.
The author computed the value of the principal \mob function
for $\kappa_1, \ldots, \kappa_7$, and the results are shown in
Table~\ref{table_kappa_n_values}.
\begin{table}
    \[
    \begin{array}{lrr}
    \toprule
    n & \mobp{\kappa_n} & \mobp{\kappa_{n}} / \mobp{\kappa_{n-1}} \\
    \midrule
    1  &     -1 & \text{--}  \\
    2  &    -27 & 27         \\
    3  &   -117 &  4.333     \\
    4  &   -509 &  4.350     \\
    5  &  -2389 &  4.694     \\
    6  & -10946 &  4.582     \\
    7  & -51210 &  4.678     \\
    \bottomrule
    \end{array}
    \]
    \caption{Values of the principal \mob function of $\kappa_1, \ldots, \kappa_7$.}
    \label{table_kappa_n_values}
\end{table}
Examining these values we can see two things:
\begin{itemize}
    \item The ratio $\mobp{\kappa_{n}} / \mobp{\kappa_{n-1}}$ is not an 
    integer.  In contrast, 2413 double-balloons grow by a factor of 2 
    at each iteration.
    \item $\kappa_n$ appears to grow at more than
    double the rate of a double 2413-balloon.
\end{itemize}    
The author feels that it is very unlikely that 
$\mobmaxn$ is 
bounded below by something that grows faster than
an exponential function.  

The only mechanism currently available to
determine $\mobmaxn$ is essentially to 
calculate the value 
of the principal \mob function for every permutation
of length $n$.  While this has been done 
for lengths $1, \ldots, 13$, 
the computational effort required to 
determine the value of $\mobmaxx{14}$
is very large\footnote{The author estimates the effort 
    on his HEDT PC would be 
    around half a million CPU hours.}
and it is therefore unreasonable
to expect this data to become available
in the near future.

The technique used to construct 2413-balloon permutations
can, as mentioned earlier, be thought of
in terms of inflations of $25314$.
We commented earlier that we 
used balloon notation,
as it is our belief that 
this gave us a simpler exposition.
In Section~\ref{section-concluding-remarks}
we generalized the ballooning process, and defined
$i,j$-balloons.  
Although $\ballgen{i,j}{\alpha}{\beta}$
can be considered as an inflation of a permutation
that is $\alpha$ plus a single new point, 
we think that viewing
permutations through the ``balloon'' lens
gives a sufficiently different view
that this technique could well have a wider application.
Initial investigations by the author and others
seem to indicate that
most $i,j$-balloons behave ``regularly''.  Typically
we find that $\mobp{\ballgen{i,j}{\alpha}{\beta}}$ is a 
multiple of $\mobp{\beta}$.  


    \chapter{The principal \mob function of balloon permutations}
\label{chapter_balloon_permutations_preprint}

\section{Preamble}

This chapter
is based on independent research by the author
that is, at the time of writing, still in progress.
The author intends that this chapter will
form the basis of work that will be submitted for publication,
and at present this will be a single-author paper.

Generalised balloon permutations
were introduced in 
Section~\ref{section-concluding-remarks}
of Chapter~\ref{chapter_2413_balloon_paper}.
In this chapter we drop the ``generalised'' qualifier.
A balloon permutation is formed from the merge of two 
permutations, $\alpha$ and $\beta$,
and has the property that 
$\beta$ occurs as an interval copy
in the balloon permutation.

In this chapter we find an expression
for the principal \mob function
of balloon permutations
in terms of a sum over a set of permutations, 
plus a correction factor.

We then show that for certain types of 
balloon permutation (``wedge permutations'')
the correction factor is always zero.
Further, we show that the
principal \mob function of a wedge permutation
is always a multiple of 
the principal \mob function
of $\beta$.

\section{Introduction}

In this chapter we recall the definition
of a balloon permutation from 
Chapter~\ref{chapter_2413_balloon_paper},
and provide a second, equivalent, definition.

We show how, given two permutations
$\alpha$ and $\beta$ we can construct
a balloon permutation $\ballij{\alpha}{\beta}$.
We then provide some examples of
balloon permutations,
which include direct sums and skew sums,
and introduce ``block diagrams'',
which show how $\alpha$ and $\beta$
are related.

We discuss how to represent permutations 
contained in a balloon permutation
in terms of $\alpha$ and $\beta$.
This includes discussing how to resolve
ambiguities that arise when a
permutation contained in a balloon permutation
has multiple embeddings.

We show that the principal \mob function of
$\ballij{\alpha}{\beta}$ can be expressed as a sum 
of the principal \mob function over a certain
subset of permutations contained in $\ballij{\alpha}{\beta}$,
plus a correction factor that is calculated 
from a specific set of
chains.  
The subset of permutations has
the property that they all contain 
$\beta$ as an interval copy,
although we note that not every permutation
that contains $\beta$ as an interval copy
is included in the set of permutations.

One of the types of balloon permutation
is a wedge permutation.
We show that the correction factor
for a wedge permutation
is always zero.
We further show that 
the value of the principal \mob function
of a wedge permutation
is a multiple of the 
principal \mob function of
$\beta$.

We conclude this chapter with a brief summary
of the results found.
We briefly discuss 
certain phenomena which have been observed
when the wedge operation is iterated.

\section{Definitions, examples and notation}

\subsection{Balloon permutations}

Generalised balloon permutations were introduced
in Chapter~\ref{chapter_2413_balloon_paper},
and we recall that they are defined as
\begin{align}
\label{equation-recall-balloon-definition}
(\ballgen{i,j}{\alpha}{\beta})_x
& =
\begin{cases}
\alpha_x 
& \text{if $x \leq i$ and $\alpha_x \leq j $}\\
\alpha_x + \order{\beta} 
& \text{if $x \leq i$ and $\alpha_x > j $}\\
\beta_{x-i} + j 
& \text{if $x > i $ and $x \leq i+\order{\beta}$ }\\
\alpha_{x-\order{\beta}} 
& \text{if $x > i+\order{\beta}$ and $\alpha_{x-\order{\beta}} \leq j $}\\
\alpha_{x-\order{\beta}} + \order{\beta} 
& \text{if $x > i+\order{\beta}$ and $\alpha_{x-\order{\beta}} > j $}\\ 
\end{cases}
\end{align}
with $0 \leq i,j \leq \order{\alpha}$.

If $\alpha = 1$, then
$\ballgen{i,j}{\alpha}{\beta}$
is one of 
$1 \oplus \beta$,
$1 \ominus \beta$,
$\beta \oplus 1$, or
$\beta \ominus 1$.
In these cases 
it is trivial to find $\mobp{\ballgen{i,j}{\alpha}{\beta}}$
using results from 
Burstein, 
Jel{\'{i}}nek, 
Jel{\'{i}}nkov{\'{a}} and 
Steingr{\'{i}}msson~\cite{Burstein2011},
and we exclude these cases from further consideration in
this chapter by requiring that $\order{\alpha} > 1$.

Given two permutations, $\alpha$ and $\beta$,
we say that some permutation $\pi$ 
is a 
\emph{merge}\extindex[permutation]{merge}
of $\alpha$ and $\beta$
if the points of $\pi$
can be coloured red or blue
so that the red points 
are order-isomorphic to $\alpha$,
and the blue points are order-isomorphic to $\beta$.

We can alternatively define a 
\emph{balloon}\extindex[permutation]{balloon} 
permutation
as the merge of two non-empty permutations
$\alpha$ and $\beta$,
which we write as $\ballij{\alpha}{\beta}$,
with the additional requirement that
the blue points (order-isomorphic
to $\beta$) must be an interval copy of $\beta$.
One consequence of this last requirement 
is that 
\[
\ballij{\alpha}{\eta} < \ballij{\alpha}{\zeta}
\;\;\text{~if and only if~}\;\;
\eta < \zeta.
\]
We call this the 
\emph{nesting}\extindex{nesting condition} 
condition.

Our alternative definition of a balloon 
permutation is somewhat abstract,
and, as we will see, certain values of $i$ and/or $j$
make a significant difference to our results.
Before we continue with our discussion of notation,
we first present several varieties of balloon permutations,
together with a diagrammatic way of understanding
how balloon permutations are structured.

\subsection{Types of balloon permutations}

When discussing  types of balloon
permutations $\ballij{\alpha}{\beta}$, we will generally give a 
description of how $\alpha$ and $\beta$ are merged.
We will also provide what we term a 
\emph{block diagram}\extindex{block diagram}.
A block diagram is used to give a visual
representation of the merge, showing the 
relative locations of $\alpha$ and $\beta$ in
$\ballij{\alpha}{\beta}$.  In all the examples
we show, there are parts of the permutation plot
of $\ballij{\alpha}{\beta}$ that are guaranteed to be empty.   
In a block diagram,
we highlight
these empty regions by shading them.

\subsubsection{Balloon permutations}

Figure~\ref{figure_generalised-balloon}
shows the block diagram for any balloon permutation.
\begin{figure}[!h]
    \begin{center}
        \begin{tikzpicture}[scale=1]
        \fill [color=lightgray] (0.5,1.5) rectangle (1.5,2.5);
        \fill [color=lightgray] (1.5,0.5) rectangle (2.5,1.5);
        \fill [color=lightgray] (1.5,2.5) rectangle (2.5,3.5);
        \fill [color=lightgray] (2.5,1.5) rectangle (3.5,2.5);
        \plotgrid{3}{3}
        \scell{1}{1}{$\alpha_{BL}$};
        \scell{1}{3}{$\alpha_{TL}$};
        \scell{2}{2}{$\beta$};
        \scell{3}{1}{$\alpha_{BR}$};
        \scell{3}{3}{$\alpha_{TR}$};
        \end{tikzpicture}
    \end{center}
    \caption{Block diagram for the generalised balloon permutation $\ballgen{i,j}{\alpha}{\beta}$.}
    \label{figure_generalised-balloon} 
\end{figure}
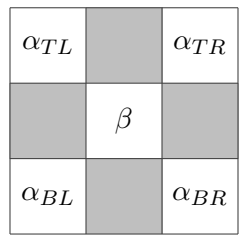
Note that, by design, a 
block diagram does not include the 
indices used (the $i$ and $j$ in $\ballgen{i,j}{\alpha}{\beta}$),
as the purpose of a block diagram is to provide 
a high-level view of the construction.

There are two sub-types of balloon permutations,
which we describe now.

\subsubsection{Direct sums and skew sums}

The simplest examples of balloon permutations, 
which have already appeared extensively in the literature,
are the direct sum of two permutations, $\alpha \oplus \beta$,
and the skew sum, $\alpha \ominus \beta$.
Figure~\ref{figure-sum-and-skew-as-blocks}
shows the block diagram for 
direct sums and skew sums.
Direct sums occur when
we have
$\ballgen{0,0}{\alpha}{\beta}$ or
$\ballgen{\order{\alpha},\order{\alpha}}{\alpha}{\beta}$.
Skew sums occur when we have
$\ballgen{0,\order{\alpha}}{\alpha}{\beta}$ or
$\ballgen{\order{\alpha},0}{\alpha}{\beta}$.
\begin{figure}[!h]
    \begin{center}
        \begin{subfigure}[t]{0.45\textwidth}
            \begin{center}
                \begin{tikzpicture}[scale=1]
                \fill [color=lightgray] (0.5,1.5) rectangle (1.5,2.5);
                \fill [color=lightgray] (1.5,0.5) rectangle (2.5,1.5);
                \plotgrid{2}{2}
                \scell{1}{1}{$\alpha$};
                \scell{2}{2}{$\beta$};
                \end{tikzpicture}                
            \end{center}
            \caption{$\alpha \oplus \beta = \ballgen{\order{\alpha},\order{\alpha}}{\alpha}{\beta}$.}
        \end{subfigure}
        \begin{subfigure}[t]{0.45\textwidth}
            \begin{center}
                \begin{tikzpicture}[scale=1]
                \fill [color=lightgray] (0.5,0.5) rectangle (1.5,1.5);
                \fill [color=lightgray] (1.5,1.5) rectangle (2.5,2.5);
                \plotgrid{2}{2}
                \scell{1}{2}{$\alpha$};
                \scell{2}{1}{$\beta$};
                \end{tikzpicture}                
            \end{center}
            \caption{$\alpha \ominus \beta = \ballgen{\order{\alpha},0}{\alpha}{\beta}$.}
        \end{subfigure}
    \end{center}
    \caption{Block diagrams for direct and skew sums.}
    \label{figure-sum-and-skew-as-blocks} 
\end{figure}

\subsubsection{Wedge permutations}
\label{subsection-wedge-permutations}

Our second type of balloon permutation
is the 
\emph{wedge permutation}\extindex[permutation]{wedge}.
Wedge permutations occur when we have
$\ballgen{i,0}{\alpha}{\beta}$ 
or
$\ballgen{i, \order{\alpha}}{\alpha}{\beta}$ 
or
$\ballgen{0,j}{\alpha}{\beta}$ 
or
$\ballgen{\order{\alpha},j}{\alpha}{\beta}$.
There are thus four (symmetric) ways in which
we can construct a wedge permutation,
For our purposes, we only need 
consider one symmetry,
and we choose the version
defined by 
$\ballgen{i, \order{\alpha}}{\alpha}{\beta}$.
We write this is as $\ballwedgex{i}{\alpha}{\beta}$.
Henceforth, we will refer to this as a wedge permutation
without qualification.

Note that if
$i \in \{ 0 , \order{\alpha} \}$,
then
$\ballwedgex{i}{\alpha}{\beta}$
is a direct or skew sum of $\alpha$ and $\beta$.
We will occasionally want to refer to 
wedge permutations that are not
direct or skew sums, and we call
these 
\emph{proper wedge permutations}\extindex[permutation]{proper wedge}.

Another way to understand the construction
of wedge permutations
is to define $\ballwedgek{\alpha}{\beta}$
as the permutation formed by taking the first 
$k$ points of $\alpha$, then appending 
all of the points from $\beta$, with their values 
increased by $\order{\alpha}$,
and finally appending the remaining points of $\alpha$.
Figure~\ref{figure_example_wedge_generic}
shows the block diagram for a wedge permutation,
where $\alpha_L$ represents the first $k$
points for $\alpha$, and $\alpha_R$ 
represents the remaining points of $\alpha$.
\begin{figure}[!h]
    \begin{center}
        \begin{tikzpicture}[scale=1]
        \fill [color=lightgray] (0.5,1.5) rectangle (1.5,2.5);
        \fill [color=lightgray] (1.5,0.5) rectangle (2.5,1.5);
        \fill [color=lightgray] (2.5,1.5) rectangle (3.5,2.5);        
        \plotgrid{3}{2}
        \scell{1}{1}{$\alpha_{L}$};
        \scell{2}{2}{$\beta$};
        \scell{3}{1}{$\alpha_{R}$};
        \end{tikzpicture}
    \end{center}
    \caption{Block diagram for the wedge permutation $\ballwedgek{\alpha}{\beta}$.}
    \label{figure_example_wedge_generic} 
\end{figure}

\subsection{Notation}

We have said that a balloon permutation will be written 
as $\ballij{\alpha}{\beta}$.
We now consider how we will represent some permutation 
$\sigma$, where $\sigma \leq \ballij{\alpha}{\beta}$.  
Our aim is to define a notation where
most permutations contained in $\ballij{\alpha}{\beta}$
have a unique representation.

Since $\ballij{\alpha}{\beta}$ is a merge of $\alpha$ and $\beta$,
we can colour the points of $\alpha$ red,
and the points of $\beta$ blue.
We note that for a merge in general there may be 
several possible colourings.
By contrast, there is only one possible
colouring for a balloon permutation.
This is because
$\beta$ is an interval at a fixed
position within $\ballij{\alpha}{\beta}$,
and so we know that 
the first point of $\beta$ 
is in column $i+1$,
and the lowest point of $\beta$ is in row $j+1$.
Since $\beta$ is an interval,
this then means that 
there is only one possible choice for 
the blue points,
and so we have
a unique colouring.

When we discuss permutations contained in $\ballij{\alpha}{\beta}$,
we will sometimes want to discuss how these permutations
can be found as embeddings.
Recall that 
if $\sigma$ is contained in $\pi$,
then an embedding of $\sigma$ in $\pi$
is a set of points of $\pi$, with cardinality $\order{\sigma}$,
that is order-isomorphic to $\sigma$.
Further, note that an embedding is
not necessarily unique.  
As a (trivial) example, if $\order{\pi} = n$,
then there are $n$ distinct embeddings of
the permutation 1 in $\pi$.

Now let $\pi = \ballij{\alpha}{\beta}$,
and assume that we have a permutation $\sigma < \pi$,
and we want to describe an embedding $\omega$ 
of $\sigma$ in $\pi$ in terms of 
the points of $\alpha$ and $\beta$.
The points of $\omega$ in $\ballij{\alpha}{\beta}$ can be partitioned
into two sets -- those that are red, and those that are blue.
The blue points used in $\omega$ will form 
a permutation $\eta$, 
where we allow $\eta$ to be the empty permutation $\emptyperm$.
In general, we will not be concerned
with exactly which red points are used.
We know that the red points are a 
(possibly improper or empty) subset of the points of $\alpha$.
We represent the embedding as
\[
\omega = \ballij{\ex{\alpha}}{\eta}.
\]
This representation must be thought of as
first finding the complete permutation $\ballij{\alpha}{\eta}$,
where $\eta$ uses the blue points chosen,
and then removing some or possibly all of the red points.
With this notation, we can clearly understand 
the permutation formed by the blue points, 
as this is $\eta$, but we do not know
which red points have been used.
This notation does not distinguish between
two embeddings where the blue points of
each embedding are equivalent to the same permutation.
This is deliberate, 
as when we consider embeddings,
we will only be concerned with the permutation formed by the blue points,
not exactly which blue points are used,
and we think that the ambiguity in notation
is more than offset by the increase in clarity
in the discourse.

We now partition the permutations 
contained in $\pi = \ballij{\alpha}{\beta}$ into four subsets.

\subsubsection{Complete permutations}

Our first set of permutations are those where
there is an embedding that uses all of the red points.
If $\sigma$ is such a permutation, then
it is possible to write
$\sigma = \ballij{\alpha}{\eta}$
for some (possibly empty) permutation $\eta$.
The permutation $\eta$ is unique,
since if we had
$\sigma 
= 
\ballij{\alpha}{\eta}
= 
\ballij{\alpha}{\zeta}$, 
then by the nesting condition we must have $\eta = \zeta$.
We call these permutations 
\emph{complete}\extindex[permutation]{complete}, 
as they
have an embedding which uses the complete set of red points.
We will always write them in the form
$\ballij{\alpha}{\eta}$.
The permutation $\pi = \ballij{\alpha}{\beta}$ is, of course, complete.

\subsubsection{Proper reductions}

Our second set of permutations
are those where every embedding uses all of the blue points, 
excluding the permutation $\pi$, 
which, as noted above, is complete.
For these permutations, we will not be interested
in understanding exactly which red points are used in 
any embedding.
We write these permutations in the form
$\ballij{\ex{\alpha}}{\beta}$.  
This representation must be thought of as
first finding the complete permutation $\ballij{\alpha}{\beta}$,
and then removing some or possibly all of the red points.
We call these permutations 
\emph{proper reductions}\extindof{balloon}{proper reduction}{(of a balloon)}.
Given $\pi = \ballij{\alpha}{\beta}$,
we denote the set of permutations that are proper reductions
as $\redset$.
Note that 
no proper reduction can be complete.

\subsubsection{Matryoshka permutations}

Our third set of permutations 
are a subset of the permutations that are
neither complete, nor proper reductions,
and satisfy a specific 
condition (the ``matryoshka'' condition).

Given any two permutations $\sigma$ and $\pi$, 
there is a set of all possible
embeddings of $\sigma$ into $\pi$, 
which we write $\sigma(\pi)$.
Since we are only interested in
cases where $\pi = \ballij{\alpha}{\beta}$,
and where $\sigma$ is neither complete, 
nor a proper reduction,
we can extend our notation to write
\[
\sigma(\pi) = 
\{
\ballij{\ex{\alpha}^1}{\eta^1},
\ldots,
\ballij{\ex{\alpha}^n}{\eta^n}
\}.
\]
Note that there may
be cases where there are two
embeddings 
$\ballij{\ex{\alpha}^k}{\eta^k}$,
and
$\ballij{\ex{\alpha}^\ell}{\eta^\ell}$,
with $k \neq \ell$, where the permutations
$\eta^k$ and $\eta^\ell$ are the same.  
We are interested in understanding  
which permutations occur as $\eta$ in 
$\ballij{\ex{\alpha}}{\eta}$
in the set of embeddings
$\sigma(\pi)$,
so 
our next step is to form the set of permutations
that occur as some $\eta$ in $\sigma(\pi)$.
We define
\[
E_{\sigma(\pi)} =
\{
\zeta :
\sigma(\pi) \text{ contains an embedding }
\ballij{\ex{\alpha}}{\zeta}
\}.
\]
and, given some $E_{\sigma(\pi)}$,  we label the
elements as
$
E_{\sigma(\pi)} =
\{
\zeta_1, \ldots, \zeta_m
\}.
$
Note that $E_{\sigma(\pi)}$ is a set of permutations,
and that this set can include the empty permutation $\emptyperm$.
Now,
if there is an integer
$k$, with $1 \leq k \leq m$
such that for all $\ell = 1, \ldots, m$, and $\ell \neq k$ we have
$\zeta_k < \zeta_\ell$,
then we say that the permutation
$\sigma$ is 
\emph{matryoshka}\extindex[permutation]{matryoshka},
and we will write these permutations
in the form
$\ballij{\ex{\alpha}}{\zeta}$,
where $\zeta = \zeta_k$.
As before, this representation must be thought of as
first finding the complete permutation $\ballij{\alpha}{\zeta}$,
and then removing some or possibly all of the red points.
The set of matryoshka permutations
forms our third set.

If a permutation $\sigma$
is not complete, and is not a proper reduction, 
and has an embedding 
$\ballij{\ex{\alpha}}{\eta}$,
where $\eta \in \{ \emptyperm, 1\}$,
then it is easy to see that $\sigma$ must be matryoshka.

\subsubsection{Defective permutations}

The remaining permutations 
are not complete permutations,
proper reductions,
or matryoshka.
We say that these permutations
are \emph{defective}\extindex[permutation]{defective}.  
We will not need to concern
ourselves with a unique representation 
of a defective permutation.

\subsubsection{Notation for elements of a chain}

Our main arguments will be based on 
partitioning the chains in the poset, 
and then applying 
Corollary~\ref{corollary-halls-corollary}.
Given $\pi = \ballij{\alpha}{\beta}$,
we now introduce the terminology
and notation we will use in handling the
chains in the poset $[1, \pi]$.

Let $c$ be a chain in the interval $[1, \pi]$.
We start by noting that the top element 
of every chain in the interval $[1, \pi]$
is $\pi = \ballij{\alpha}{\beta}$,
and the bottom element is the permutation $1$.
Recall that we have $\order{\alpha} > 1$, 
so the permutation $1$ cannot be written as 
$\ballij{\alpha}{\eta}$ for any $\eta$.
It follows that every chain contains a largest permutation
that cannot be written as $\ballij{\alpha}{\eta}$.
Note that since we cannot write this permutation as
$\ballij{\alpha}{\eta}$ for any $\eta$,
this permutation is not complete. 
As in Chapter~\ref{chapter_2413_balloon_paper},
we call this permutation the 
\emph{pivot}\extindof{balloon}{pivot}{(in a balloon)}, 
and denote it by $\psi_c$.
Since $\psi_c$ is not complete, and the highest element
of the chain, $\pi$, is complete,
there must be a permutation above $\psi_c$ in the chain.
We call the permutation above $\psi_c$ in the chain
$\phi_c$.
Finally, since every chain has at least two elements,
there is always a second-highest permutation in the chain,
and we call this $\kappa_c$.

We remark that
$\psi_c$ cannot be a complete permutation,
and that
$\phi_c$, 
and every permutation above $\phi_c$ in the chain,  
must be a complete permutation.

We will partition the chains in the poset into three
sets.
The first set, $\chainr$, consists of those chains
where $\kappa_c$ is a proper reduction.  
The second set, $\chainm$, comprises chains $c$
that are not in $\chainr$,
where $\psi_c$ is matryoshka.
Our final set $\chaing$ contains the remaining chains.

Formally, we define
\begin{align}
\label{defn_block_chains_rmg}
\begin{split}
\chainc & = \text{The set of all chains in $[1, \pi]$} \\
\chainr & =
\{
c : 
c \in \chainc; 
\kappa_c \text{~is a proper reduction} 
\}
\\
\chainm & =
\{
c : 
c \in \chainc \setminus \chainr; 
\psi_c \text{~is matryoshka}
\}
\\
\chaing & = 
\{
c :
c \in \chainc \setminus (\chainr \cup \chainm)
\}.
\end{split}
\end{align}    

We now have all the terminology and notation to 
prove our first result in this chapter.

\section{The principal \mob function of balloon permutations}

Let $\pi = \ballij{\alpha}{\beta}$.  
Our aim in this section
is to derive an expression for $\mobp{\pi}$ as follows.
\begin{theorem}
    \label{theorem-pmf-balloon-permutations}
    If $\pi = \ballij{\alpha}{\beta}$, 
    and 
    $\chaing$ is as defined in Equation~\ref{defn_block_chains_rmg}, then
    \[
    \mobp{\pi}
    =
    - \sum_{\lambda \in \redset} \mobp{\lambda}
    +
    \sum_{c \in \chaing} (-1)^{\order{c}}.
    \]   
\end{theorem}
\begin{proof}
    Since the sets 
    $\chainr$,
    $\chainm$, and
    $\chaing$ partition the chains in the poset,
    we can write
    \[
    \mobp{\pi}
    =
    \sum_{c \in \chainr} (-1)^{\order{c}}
    +
    \sum_{c \in \chainm} (-1)^{\order{c}}
    +
    \sum_{c \in \chaing} (-1)^{\order{c}}.
    \]
    
    We start by showing that the Hall sum for the 
    set $\chainm$ is zero.
    
    Let $c$ be a chain in $\chainm$,
    $\psi_c = \ballij{\ex{\alpha}}{\eta}$,
    and     
    $\phi_c = \ballij{\alpha}{\tau}$.
    
    Define a function $\Phi$ as follows:
    \[
    \Phi(c) =
    \begin{cases}
    c \setminus \{ \ballij{\alpha}{\eta} \} & \text{if $\eta = \tau$,} \\
    c \cup \{ \ballij{\alpha}{\eta} \} & \text{otherwise.}     
    \end{cases}
    \]
    
    We have two cases to consider.  Either $\eta = \tau$, 
    or $\eta \neq \tau$.
    
    Case 1: $\eta = \tau$. 
    
    The chain $c$ has a segment
    $\ballij{\ex{\alpha}}{\eta} < \ballij{\alpha}{\eta}$.    
    If $\eta$ = $\beta$, then $\psi_c = \kappa_c$.
    Further, $\kappa_c$ is a proper reduction
    since $\eta$ is minimal,
    so $c \in \chainr$, 
    thus we must have $\eta < \beta$, and
    so $\Phi(c) = \cprime$ is a chain.  
    
    Now, since $\eta \neq \beta$,
    it follows that $\Phi(c)$ contains a segment
    $\ballij{\ex{\alpha}}{\eta} < \ballij{\alpha}{\zeta}$
    for some $\zeta \leq \beta$.
    Since $\eta < \beta$, $\ballij{\ex{\alpha}}{\eta}$
    is not a reduction.  
    If $\psi_{\cprime} = \kappa_{\cprime}$, 
    then $\cprime \not\in \chainr$.
    If $\psi_{\cprime} \neq \kappa_{\cprime}$, then
    we must have $\kappa_{\cprime} = \kappa_c$,
    and so again $\cprime \not\in \chainr$.
    Since we have $\psi_{\cprime} = \psi_c$, 
    it then follows that
    $c \in \chainm$.
    
    Case 2: $\eta \neq \tau$.
    
    The chain $c$ has a segment
    $\ballij{\ex{\alpha}}{\eta} < \ballij{\alpha}{\tau}$.
    Clearly, 
    $\ballij{\ex{\alpha}}{\eta} < \ballij{\alpha}{\eta}$.
    so $\cprime = \Phi(c)$ can fail to be 
    a chain if and only if 
    $\ballij{\alpha}{\eta} \not< \ballij{\alpha}{\tau}$ or
    equivalently, from the nesting condition,
    $\eta \not< \tau$.
    
    We show that $\eta < \tau$ by assuming otherwise,
    and showing that this leads to a contradiction.
    \begin{figure}
        \begin{center}
            \begin{subfigure}{0.22\textwidth}
                \centering
                \begin{tikzpicture}[scale=0.3]
                \fill [color=lightgray] (0.5,2.5) rectangle (3.5,6.5);
                \fill [color=lightgray] (3.5,0.5) rectangle (7.5,2.5);
                \fill [color=lightgray] (3.5,6.5) rectangle (7.5,9.5);
                \fill [color=lightgray] (7.5,2.5) rectangle (9.5,6.5);        
                \plotgrid{9}{9};
                \embedast{(1,2)}{blue};
                \embedast{(2,8)}{blue};
                \embeddot{(3,9)}{black};
                \embedast{(4,5)}{blue};
                \embedast{(5,3)}{blue};
                \embedast{(6,6)}{blue};
                \embeddot{(7,4)}{black};
                \embeddot{(8,1)}{black};
                \embeddot{(9,7)}{black};
                \end{tikzpicture}
                \caption{}
                \label{figure-block-case-2-1}
                $\ballij{\ex{\alpha}^1}{\eta} < \ballij{\alpha}{\tau}$.              
            \end{subfigure}
            \phantom{x}
            \begin{subfigure}{0.22\textwidth}
                \centering
                \begin{tikzpicture}[scale=0.3]
                \fill [color=lightgray] (0.5,2.5) rectangle (3.5,6.5);
                \fill [color=lightgray] (3.5,0.5) rectangle (7.5,2.5);
                \fill [color=lightgray] (3.5,6.5) rectangle (7.5,9.5);
                \fill [color=lightgray] (7.5,2.5) rectangle (9.5,6.5);        
                \plotgrid{9}{9};
                \embedcir{(1,2)}{black};
                \embedcir{(2,8)}{black};
                \embeddot{(3,9)}{black};
                \embedast{(4,5)}{blue};
                \embedast{(5,3)}{blue};
                \embedast{(6,6)}{blue};
                \embeddot{(7,4)}{black};
                \embeddot{(8,1)}{black};
                \embeddot{(9,7)}{black};
                \end{tikzpicture}
                \caption{}
                \label{figure-block-case-2-2}
                $\eta < \ballij{\ex{\alpha}^2}{\tau}$.
            \end{subfigure}
            \phantom{x}
            \begin{subfigure}{0.22\textwidth}
                \begin{tikzpicture}[scale=0.3]
                \fill [color=lightgray] (0.5,2.5) rectangle (3.5,6.5);
                \fill [color=lightgray] (3.5,0.5) rectangle (7.5,2.5);
                \fill [color=lightgray] (3.5,6.5) rectangle (7.5,9.5);
                \fill [color=lightgray] (7.5,2.5) rectangle (9.5,6.5);        
                \plotgrid{9}{9};
                \embedcir{(1,2)}{black};            
                \embedcir{(2,8)}{black};
                \embeddot{(3,9)}{black};
                \embedast{(4,5)}{blue};
                \embedast{(5,3)}{blue};
                \embeddot{(6,6)}{black};
                \embeddot{(7,4)}{black};
                \embeddot{(8,1)}{black};
                \embedast{(9,7)}{blue};
                \end{tikzpicture}
                \caption{}
                \label{figure-block-case-2-3}                
                $\eta < \ballij{\ex{\alpha}^3}{\tau}$.
            \end{subfigure}
            \phantom{x}
            \begin{subfigure}{0.22\textwidth}
                \begin{tikzpicture}[scale=0.3]
                \fill [color=lightgray] (0.5,2.5) rectangle (3.5,6.5);
                \fill [color=lightgray] (3.5,0.5) rectangle (7.5,2.5);
                \fill [color=lightgray] (3.5,6.5) rectangle (7.5,9.5);
                \fill [color=lightgray] (7.5,2.5) rectangle (9.5,6.5);        
                \plotgrid{9}{9};
                \embedast{(1,2)}{blue};
                \embedast{(2,8)}{blue};
                \embeddot{(3,9)}{black};
                \embedast{(4,5)}{blue};
                \embedast{(5,3)}{blue};
                \embeddot{(6,6)}{black};
                \embeddot{(7,4)}{black};
                \embeddot{(8,1)}{black};
                \embedast{(9,7)}{blue};                
                \end{tikzpicture}
                \caption{}
                \label{figure-block-case-2-4}                
                $\ballij{\ex{\alpha}^4}{\eta^\prime} < \ballij{\alpha}{\tau}$.
            \end{subfigure}
        \end{center}
        \caption{Examples to illustrate the steps in case 2.  
            The left-hand side of the inequality 
            is shown as {$\textcolor{blue}\ast$}.}
    \end{figure}
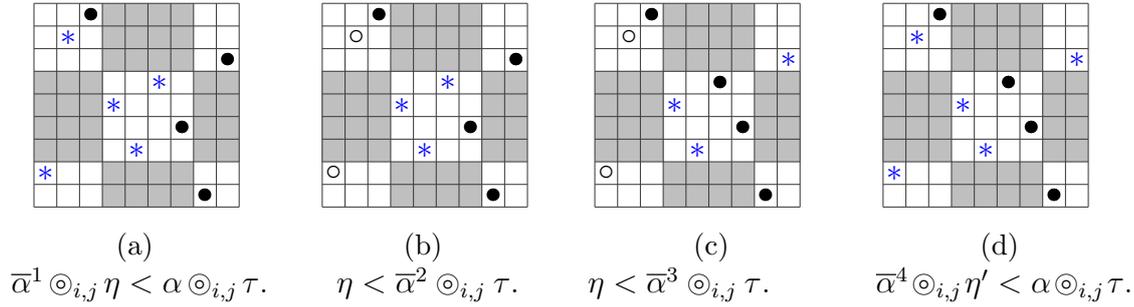
    To aid the reader, we provide a running example
    using a generalised balloon, 
    where $\alpha = 24513$, $\tau = 3142$, 
    $\psi_c = 15324$, and $\psi_c$ is matryoshka, with
    representation $\ballij{\ex{\alpha}}{213}$.
    Of course, in this example 
    the representation of $\psi_c$ is not 
    minimal, and in addition $\eta < \tau$, 
    but despite this
    we feel that the 
    diagrams are a helpful aid in understanding
    the steps in our argument.  
    
    The chain $c$ contains a segment
    $\ballij{\ex{\alpha}^1}{\eta} < \ballij{\alpha}{\tau}$.
    Figure~\ref{figure-block-case-2-1} shows how 
    $\psi_c = \ballij{\ex{\alpha}^1}{\eta}$
    might be embedded in 
    $\phi_c = \ballij{\alpha}{\tau}$.
    
    Since $\ballij{\ex{\alpha}^1}{\eta} < \ballij{\alpha}{\tau}$,
    the points of $\,\ex{\alpha}^1$ in
    $\,\ballij{\ex{\alpha}^1}{\eta}$ are a proper
    subset of the points of $\alpha$ in
    $\ballij{\alpha}{\tau}$.
    We can remove 
    the points of $\,\ex{\alpha}^1$ from
    both sides of the inequality 
    and we obtain
    $\eta < \ballij{\ex{\alpha}^2}{\tau}$.
    This is shown in Figure~\ref{figure-block-case-2-2}.
    
    Since, by assumption, 
    $\eta \not< \tau$, it follows that we must be able to
    find an embedding of $\eta$ in 
    $\ballij{\ex{\alpha}^2}{\tau}$ that
    uses at least one point that is not in $\tau$, and so
    we can write
    $\eta < \ballij{\ex{\alpha}^3}{\tau}$,
    where the points in $\tau$ 
    that are used is a proper subset of 
    the points of $\tau$ that are used
    in $\eta < \ballij{\ex{\alpha}^2}{\tau}$. 
    An example is shown in Figure~\ref{figure-block-case-2-3}.
    
    Now note that the points from $\alpha$ in this new embedding
    of $\eta$,    
    $\ex{\alpha}^3$, must be disjoint from
    the points from $\alpha$ in the original embedding,
    $\ex{\alpha}^1$.
    It follows that we can write
    $\ballij{\ex{\alpha}^1}{\eta}$ 
    as
    $\ballij{\ex{\alpha}^4}{\eta^\prime}$,
    where the points in $\ex{\alpha}^4$
    are the union of the points in
    $\ex{\alpha}^1$ and $\ex{\alpha}^3$.
    Further, we must have 
    $\order{\eta^\prime} < \order{\eta}$.
    This is shown in Figure~\ref{figure-block-case-2-4}.
    
    We now have $\psi_c = \ballij{\ex{\alpha}^4}{\eta^\prime}$,
    where $\order{\eta^\prime} < \order{\eta}$,
    but this is a contradiction, since
    $\ballij{\ex{\alpha}^1}{\eta}$
    is matryoshka, and so
    $\eta$ is minimal.
    Therefore our assumption must be wrong,
    and so we have  $\eta < \tau$,
    and thus 
    $\cprime = \Phi(c)$ is a chain.
    
    It remains to show that $\cprime \in \chainm$.
    Now, $\cprime$ has a segment
    $\ballij{\ex{\alpha}}{\eta} < 
    \ballij{\alpha}{\eta} <
    \ballij{\alpha}{\tau}$.
    Since $\psi_c = \psi_{\cprime}$,
    and $\psi_c$ is matryoshka,
    the only way for $\cprime \not\in \chainm$
    is if 
    $\kappa_{\cprime}$ is a proper reduction.
    
    Since $\psi_{\cprime}$ is the pivot of $\cprime$,
    and there are at least two permutations
    above $\psi_{\cprime}$ in $\cprime$,
    it follows that $\kappa_{\cprime}$
    is a complete permutation, and such a permutation
    cannot be a proper reduction.
    It follows, therefore, that
    $\cprime \not\in \chainr$.
    Since $\psi_{\cprime} = \psi_c$, 
    this then means that $\cprime \in \chainm$.
    
    We now have that
    $\Phi(c)$ is a parity-reversing involution
    on $\chainm$, and so,
    by Corollary~\ref{corollary-halls-corollary},
    we have
    $\sum_{c \in \chainm} (-1)^{\order{c}} = 0$.
    
    We now have 
    \[
    \mobp{\pi}
    =
    \sum_{c \in \chainr} (-1)^{\order{c}}
    +
    \sum_{c \in \chaing} (-1)^{\order{c}}.
    \]
    The chains in $\chainr$ are characterised by the second-highest
    element being an element of $\redset$.
    Using Corollary~\ref{corollary-hall-sum-second-highest-set}
    from Chapter~\ref{chapter_2413_balloon_paper} 
    on the sum over chains in $\chainr$
    completes the proof.
    \end{proof}

Theorem~\ref{theorem-pmf-balloon-permutations}
is hard to use in practice, because of the 
second term, 
$\sum_{c \in \chaing} (-1)^{\order{c}}$.

There is some numerical evidence, based
on analysing permutations with length 12 or less,
that this second term is, in fact, zero in many cases.

Clearly, one way to handle the difficulties 
of this second term would be to ensure
that the sum was zero.  
The easiest case is where
$\chaing = \emptyset$.
This occurs when every permutation
in the poset is matryoshka.
We state this formally as
\begin{corollary}[to Theorem~\ref{theorem-pmf-balloon-permutations}]
    \label{corollary-balloon-is-matryoshka}
    If $\pi = \ballij{\alpha}{\beta}$, 
    and every permutation in 
    $[1, \pi)$ is matryoshka,
    then
    \[
    \mobp{\pi} = - \sum_{\lambda \in \redset} \mobp{\lambda}.
    \]
\end{corollary}    
\begin{proof}
    If every permutation in 
    $[1, \pi)$ is matryoshka,
    then $\chaing = \emptyset$,
    and the result follows immediately.
\end{proof}

It is easy to see that 
if $\pi$ is a direct or skew sum,
then every permutation contained in 
$\pi$ is matryoshka.
In the following section we show
that if $\pi$ is a wedge permutation,
then every permutation contained 
in $\pi$ is matryoshka.

\section{The principal \mob function of wedge permutations}
\label{section-pmf-wedge-permutations}

We will start by showing:
\begin{theorem}
    \label{theorem-block-wedges-are-matryoshka}
    If $\pi$ is a wedge permutation, as defined in
    Sub-section~\ref{subsection-wedge-permutations}, 
    then every permutation in $[1, \pi)$
    is matryoshka.
\end{theorem}
\begin{proof}
    To prove Theorem~\ref{theorem-block-wedges-are-matryoshka},
    it is sufficient to show that an arbitrary
    permutation contained in a wedge permutation
    is matryoshka.
    
    Let $\pi = \ballwedgek{\alpha}{\beta}$,
    and let $\sigma < \pi$.
    Consider any embedding of $\sigma$ in $\pi$.
    Then if there are $n$ blue points in the embedding,
    these, from the construction method, 
    will represent the top $n$ points
    of the permutation $\sigma$.
    
    Now assume we have two embeddings of $\sigma$, say
    $\ballij{\ex{\alpha}^1}{\eta}$ and 
    $\ballij{\ex{\alpha}^2}{\zeta}$.
    If $\order{\eta} = \order{\zeta}$,
    then we have $\eta = \zeta$.
    Assume now, without loss of generality, that
    $\order{\eta} < \order{\zeta}$.
    Then $\eta$ is contained in $\zeta$,
    as the top $\order{\eta}$ points of $\zeta$
    are order-isomorphic to $\eta$.

    For any $\sigma$ there will be some 
    $\eta$ that is minimal,
    noting that this may mean that $\eta = \emptyperm$.
    
    This then means that  any
    permutation contained in a wedge permutation
    is matryoshka.
\end{proof}

We now have
\begin{lemma}
    \label{lemma-wedge-is-sum-of-reductions}
    If $\pi = \ballwedgek{\alpha}{\beta}$,
    then
    \[
    \mobp{\pi} = - \sum_{\lambda \in \redset} \mobp{\lambda}.
    \]
\end{lemma}
\begin{proof}
    By Theorem~\ref{theorem-block-wedges-are-matryoshka}
    every permutation in a wedge permutation is matryoshka.
    Applying Corollary~\ref{corollary-balloon-is-matryoshka}    
    then gives us the result.
\end{proof}

We have shown that every permutation in a wedge permutation
is matryoshka.  This is not the case for balloon permutations
in general.
If 
\begin{align*}
\alpha & = 4,6,3,5,8,9,2,12,10,13,11,7,1 \\
\beta & = 2,4,1,3,7,5,8,6 \\
\pi & = \ballgen{5,8}{\alpha}{\beta} \\
    & = 4,6,3,5,8,10,12,9,11,15,13,16,14,17,2,20,18,21,19,7,1
\intertext{and}
\sigma & = (2413 \oplus 1 \oplus 3142) \ominus 21 \\
       & = 4,6,3,5,7,10,8,11,9,2,1
\end{align*}
as shown in Figure~\ref{figure-balloon-not-matryoshka},
then 
$\sigma$ is not complete, 
as $\alpha$ has 13 points, but $\sigma$ only has 11.
Further, $\sigma$ is not a proper reduction,
as it does not contain an interval copy of $\beta$.
Finally, 
$
E_{\sigma(\pi)}
=
\{
2413,
3142,
24135,
13524
\}
$
contains two permutations of length 4,
$2413$ and $3142$, and no permutations 
with length less than 4.
It follows that $\sigma$ is not matryoshka,
and therefore $\sigma$ must be defective.
\begin{figure}
    \centering
    \begin{tikzpicture}[scale=0.3]
    \fill [color=lightgray] (0.5,8.5) rectangle (5.5,16.5);      
    \fill [color=lightgray] (5.5,0.5) rectangle (13.5,8.5);
    \fill [color=lightgray] (5.5,16.5) rectangle (13.5,21.5);
    \fill [color=lightgray] (13.5,8.5) rectangle (21.5,16.5);      
    \plotpermgrid{4,6,3,5,8,10,12,9,11,15,13,16,14,17,2,20,18,21,19,7,1}
    \end{tikzpicture}
    \qquad
    \begin{tikzpicture}[scale=0.3]
    \plotpermgrid{4,6,3,5,7,10,8,11,9,2,1}
    \end{tikzpicture}
    \caption{The permutations $\pi$ and $\sigma$ 
        that demonstrate that $\sigma$ is 
        defective in the interval $[1, \pi]$.}
    \label{figure-balloon-not-matryoshka}
\end{figure}
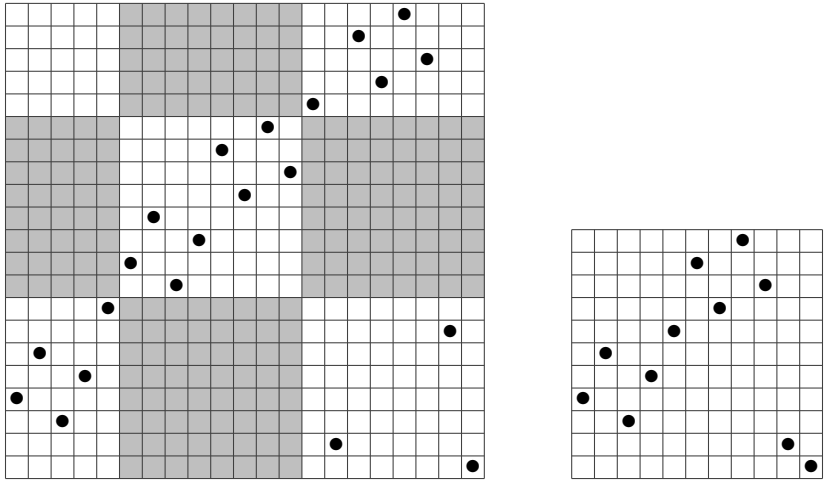

We do not claim that $\pi$ and $\sigma$ above are minimal,
although a (fairly restricted) computer search
failed to find any smaller counter-examples.

We now show that
\begin{theorem}
    \label{theorem-wedge-permutations-multiple-of-beta}
    If $\pi = \ballwedgek{\alpha}{\beta}$,
    then
    $\mobp{\pi} = c \cdot \mobp{\beta}$,
    where $c$ is an integer.
\end{theorem}
\begin{proof}[Proof of Theorem~\ref{theorem-wedge-permutations-multiple-of-beta}]
    Our proof will use Lemmas~\ref{lemma-oneplus-oneplus} 
    and~\ref{lemma-oneplus} from Chapter~\ref{chapter_2413_balloon_paper},
    which we restate here for convenience.
    \begin{lemma}[%
        Lemma~\ref{lemma-oneplus-oneplus} from Chapter~\ref{chapter_2413_balloon_paper}]
        \label{lemma-oneplus-oneplus-balloon}
        If $\pi$ has a long corner,
        then
        $\mobp{\pi} = 0$.	    
    \end{lemma}
    \begin{lemma}[%
        Lemma~\ref{lemma-oneplus} from Chapter~\ref{chapter_2413_balloon_paper}]
        \label{lemma-oneplus-balloon}
        If $\pi$ can be written as 
        $\pi = \oneplus\tau$,
        or
        $\pi = \tau\plusone$
        or 
        $\pi = \oneminus\tau$
        or
        $\pi = \tau\minusone$,
        and does not have a long corner,
        then
        $\mobp{\pi} = - \mobp{\tau}$.
    \end{lemma}

    Let $\pi = \ballwedgek{\alpha}{\beta}$.    
    
    First, consider the case when $\order{\alpha} = 1$,
    so either 
    $\pi = 1 \oplus \beta$ 
    or
    $\pi = \beta \ominus 1$.
    In the first sub-case, if $\beta$ begins 1, then
    $\rho$ has a long corner, and so
    by Lemma~\ref{lemma-oneplus-oneplus-balloon}
    $\mobp{\pi} = 0$.
    If $\beta$ does not begin $1$, then
    by Lemma~\ref{lemma-oneplus-balloon}
    $\mobp{\pi} = - \mobp{\beta}$.
    The argument for the second sub-case is similar.
    Thus we have that 
    Theorem~\ref{theorem-wedge-permutations-multiple-of-beta}
    is true if $\order{\alpha} = 1$.
    
    Now assume that
    Theorem~\ref{theorem-wedge-permutations-multiple-of-beta}
    is true for all $\alpha$ with $\order{\alpha} < m$,
    for some $m \geq 1$,
    and assume that $\order{\alpha} = m$.
    Let $\rho$ be a reduction of $\pi$.
    From the definition of a reduction we have that
    $\rho = \ballwedgex{\ell}{\alpha^\prime}{\beta}$ or
    $\rho = \beta$,    
    where $\order{\alpha^\prime} < \order{\alpha}$.
    In the first case, by the inductive hypothesis,
    we have that $\mobp{\rho}$ is a multiple of $\mobp{\beta}$.
    In the second case, trivially, $\mobp{\rho} = \mobp{\beta}$.
    Now summing over all reductions, we can see that
    $\mobp{\pi} = c \cdot \mobp{\beta}$ as required.
\end{proof}

\section{Chapter summary}

The main results from this chapter
are expressions for the principal \mob function
of balloon permutations, and a somewhat simpler expression
for the principal \mob function of wedge permutations.
We also have that if $\pi$ is a wedge permutation
that can be written as $\ballwedgek{\alpha}{\beta}$,
then $\mobp{\pi}$ is a multiple of $\mobp{\beta}$.

We have observed that 
if we repeatedly
iterate the wedge construction,
that is, we define
\begin{align*}
W_{k,\alpha, \beta, 0}  & = \beta, \\
W_{k,\alpha, \beta , 1} & = \ballwedgek{\alpha}{\beta} \\
W_{k,\alpha, \beta , n} & = \ballwedgek{\alpha}{(W_{k,\alpha, \beta, n-1})} 
\; \text{for $n > 1$}
\end{align*}
then the sequence given by
\[
\mobp{W_{k,\alpha, \beta, 0}}, \;
\mobp{W_{k,\alpha, \beta, 1}}, \;
\mobp{W_{k,\alpha, \beta, 2}}, \;
\mobp{W_{k,\alpha, \beta, 3}}, \;
\ldots
\]
follows one of the following patterns:
\begin{align*}
    &0,0,0,0,0, \ldots &
    &c,0,0,0,0, \ldots &
    &c,c,0,0,0, \ldots \\
    &c, -c, 0, 0, 0 \ldots &
    &c, 2c, 4c, 8c, 16c, \ldots &
    &c, c, 2c, 4c, 8c, \ldots \\
    &c, -c, c, -c, c, \ldots &
    &c, c, c, c, c, \ldots &
    &c, d, d, d, d, \ldots \\
\end{align*}
where $c = \mobp{\beta}$,
and, where appropriate, 
$d = \mobp{\ballwedgek{\alpha}{\beta}}$.

For some examples, we have outline proofs that
these patterns will continue, 
based on a specific analysis of the proper reductions.
However, we do not yet have a way to characterise
wedge permutations in general, so that we can predict
their behaviour without conducting a complete
analysis of the proper reductions.
We remark that 
this is one of the aspects that
we are still researching.

    \chapter{Conclusion}
\label{chapter_conclusion}

\section{A review of our results}

Our journey through the
\mob function on the
permutation pattern poset
started, 
in Chapter~\ref{chapter_incosc_paper},
by looking at ways in which we could 
calculate the value of
the \mob function in a more efficient
way than using the recursive
definition  of Equation~\ref{equation_mobius_function},
and here we found that we could reduce
the number of permutations that needed to be considered
slightly in the general case,
and by a very significant number 
in the case of increasing oscillations.

Chapter~\ref{chapter_oppadj_paper} continues this
theme by showing that if a permutation has 
opposing adjacencies, then
the value of the principal \mob function is zero.
We also describe other cases
where, if $\sigma$ meets certain criteria, 
then any permutation $\pi$ that contains 
$\sigma$ as an interval
has $\mobp{\pi} = 0$.
The main result from this chapter is, however,
not the results that provide a way to determine
the value of the \mob function, but the result
that, asymptotically, \zpmfp~of permutations
are \mob zeros.  This represents a move
away from finding ways to determine the value
of the \mob function towards ways to
better understand the permutation pattern poset.

Chapter~\ref{chapter_2413_balloon_paper}
finds a recursion for the 
value of the principal \mob function
of 2413-balloons, and uses this to show that
$\mobmaxn$ grows at least exponentially.
In this chapter the main result is the exponential growth,
and to a large extent the results for the 
values of the principal \mob function of 2413-balloons
is the mechanism we use to prove it.

Our results from Chapter~\ref{chapter_balloon_permutations_preprint}
show that for balloon permutations generally
the value of the principal \mob function
is, essentially, related to the permutations $\beta$
plus a correction factor.  For
wedge permutations, this correction factor is guaranteed to be zero.

If we consider the results directly
relating to the \mob function
from Chapter~\ref{chapter_oppadj_paper},
one aspect that could be used to distinguish
them is that they are all related
to finding a set of permutations $S$
with the property that if a larger permutation contains each 
permutation in $S$ as an interval copy,
then the value of the \mob function is zero.

Now compare this with Chapter~\ref{chapter_2413_balloon_paper}.
Given some permutation
$\pi = \ball{2413}{\beta}$, 
first note that $\beta$ is an interval in $\pi$,
and so we can (somewhat trivially) claim that
$\pi$ contains an interval copy of $\beta$.
Our results, with some small exceptions,
all give the value of 
$\mobp{\pi}$ as a multiple of $\mobp{\beta}$.

We also note that 
Conjectures~\ref{conjecture-most-2413-balloons}
and~\ref{conjecture-1-0-2413-balloons}
relate to permutations
$\pi = \ballij{2413}{\beta}$,
and again here we see that
$\beta$ occurs as an interval copy in $\pi$.

Finally, if we look 
at the results from Chapter~\ref{chapter_balloon_permutations_preprint}, 
again we can see that $\beta$ occurs as an interval copy  
in $\ballij{\alpha}{\beta}$ and, trivially, in $\ballwedgek{\alpha}{\beta}$.

Our suggestion here is that,
in some ill-defined sense,
intervals in permutations 
have a marked effect on the 
value of the principal \mob function.
In some cases, the presence of 
a copy interval or intervals guarantees that the 
value of the principal \mob function is zero.
In other cases, where we have an interval copy of $\beta$,
we see that the 
value of the principal \mob function
can be expressed in terms of $\mobp{\beta}$,
with, possibly, some correction factor, 
although for 2413-balloons and wedge permutations, 
the correction factor
is zero for all but trivial cases.

\section{Further research into the \mob function}

In the chapter summaries 
we have already mentioned several possible avenues 
for future research.
We now consider 
more general avenues for further research, and we divide these  
into two main areas.

The first area is research that will give
expressions or recursions for the value of the 
\mob function on some interval.  
Current results, and indeed our active research,
tends to examine permutations that have some ``structure''.
We do not intend to try and formally define structure;
rather we claim that 
it is a property that is generally
recognisable when it is seen.
The classic example of a set of permutations
with structure is, we suggest,
the decomposable and separable permutations,
studied by 
Burstein, Jel{\'{i}}nek, Jel{\'{i}}nkov{\'{a}} 
and Steingr{\'{i}}msson~\cite{Burstein2011},
as, in some ill-defined sense, these permutations
have a lot of structure.
Our own research into permutations 
with opposing adjacencies, and 
permutations that are 2413-balloons again 
looks at permutations with structure.

The second area is research into
what we call ``global'' properties
of the permutation pattern poset.
The result from Chapter~\ref{chapter_oppadj_paper}
that \zpmfp~of permutations are \mob zeros is one example,
as is the exponential growth rate
of the principal \mob function proved in 
Chapter~\ref{chapter_2413_balloon_paper}.

\subsection{The \mob function of simple permutations}

Recall that the simple permutations are permutations
that only contain trivial intervals.
In other areas of permutation patterns,
such as the enumeration of permutation classes,
the simple permutations 
underpin many results.
For further details,
we refer the reader to
the survey article by Brignall~\cite{Brignall2010}.

By contrast, 
little is known about the principal \mob function
of simple permutations.
The first reference to this area 
occurs in 
the concluding remarks of
Burstein, 
Jel{\'{i}}nek, 
Jel{\'{i}}nkov{\'{a}} and 
Steingr{\'{i}}msson~\cite{Burstein2011}.
Here
(using the terminology of this thesis)
they give 
a sequence of values of $\mobmaxn$
for $n = 1, \ldots, 11$,
and they note that
there is, up to symmetry,
a unique permutation $\pi_n$ of length $n$
such that $\order{\mobp{\pi_n}} = \mobmaxn$.
They further note that
$\pi_n$ is simple
except for the case $n=3$, 
(but there are no simple permutations of length 3).
The author has confirmed 
that this is also the case  
for permutations of length $12$ and $13$
(see Table~\ref{table_values_of_true_mob_min_max} on
page~\pageref{table_values_of_true_mob_min_max}).

To understand the principal \mob function
of simple permutations,
we begin by considering some 
examples that are easy to describe,
as they have recognisable structure.

A simple parallel alternation is a permutation $\pi$
with even length $2 n$, where
$
\pi = 2, 4, \ldots, 2n, 1, 3, \ldots, 2n-1,
$
or any symmetry of this sequence.
We show an example of a simple
parallel alternation 
in Figure~\ref{figure_examples_simple_parallel_alternation}.
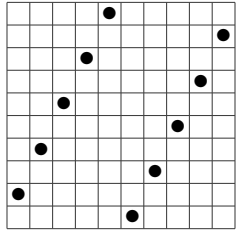
\begin{figure}
    \begin{center}
        \begin{subfigure}[t]{0.30\textwidth}
            \centering
            \begin{tikzpicture}[scale=0.3]
            \plotpermgrid{2,4,6,8,10,1,3,5,7,9}
            \end{tikzpicture}
        \end{subfigure}
    \end{center}
    \caption{A simple parallel alternation.}	      
    \label{figure_examples_simple_parallel_alternation}
\end{figure}
Smith's paper on permutations with one 
descent~\cite{Smith2013} covers 
simple parallel alternations, and so we have 
an explicit expression
for their principal \mob function value.
If $\pi$ is a simple parallel alternation of length $n$, then
\[
\mobp{\pi} = -\dbinom{\frac{n}{2} + 1}{2}.
\]
Increasing oscillations are also simple permutations.
Our paper on permutations with an 
indecomposable lower bound~\cite{Brignall2017a},
on which Chapter~\ref{chapter_incosc_paper} is based,
includes a recursion for 
the principal \mob function of increasing oscillations.
We are unaware of any other published results 
for families of simple permutations.

We now provide several
examples of simple permutations
with a recognisable structure,
and
give conjectures for the 
value of the principal 
\mob function.

We have already described simple parallel alternations.
We extend our vocabulary to define
an 
\emph{alternation}\extindex[permutation]{alternation} 
to be a permutation
where every odd entry is to the right of every even entry,
or a symmetry of such a permutation.
A 
\emph{wedge alternation}\extindex[permutation]{wedge alternation} 
is then an
alternation where the two sets of entries
point in opposite directions, 
and an example is shown in 
Figure~\ref{figure_examples_of_type_1_2_wedge_simples}.
Wedge alternations are not simple,
but a single point can be added 
in one of two ways to form a simple
permutation.
These are called type 1 and type 2 wedge simples, 
written $W_1(n)$ and $W_2(n)$,
where we require $n > 3$.
The wedge simples appear to have been 
introduced in~\cite{Brignall2008a},
where we have
\begin{theorem}[{%
        Brignall, 
        Huczynska and 
        Vatter 
        \cite[Theorem 3]{Brignall2008a}}]
    For any fixed $k$, every sufficiently long simple permutation
    contains either a proper pin sequence of length at least $k$, 
    a parallel alternation of length at least $k$, 
    or a wedge
    simple permutation of length at least $k$.
\end{theorem}
We refer the interested reader to~\cite{Brignall2008a}
for a definition of ``proper pin sequence''.

The type 1 and type 2 wedge simples have the form
\begin{align*}
W_1(n)
& =
\begin{cases}
3, 5, \ldots, n-1, 1, n, n-2, \ldots , 2 & \text{If $n$ is even,} \\
3, 5, \ldots, n,   1, n-1, n-3, \ldots, 2 & \text{If $n$ is odd,} \\
\end{cases}
\intertext{and}
W_2(n) 
& = 
\begin{cases}
2, 4, \ldots, n-2, n, n-3, n-5, \ldots, 1, n-1 & \text{if $n$ is even,} \\
2, 4, \ddots, n-3, n, n-2, n-4, \ldots, 1, n-1 & \text{if $n$ is odd,} \\
\end{cases}
\end{align*}
and every symmetry of these permutations.
Examples of type 1 and type 2 wedge simples
are shown in
Figure~\ref{figure_examples_of_type_1_2_wedge_simples}.
\begin{figure}
    \begin{center}
        \begin{subfigure}[t]{0.30\textwidth}
            \centering
            \begin{tikzpicture}[scale=0.3]
            \plotpermgrid{2,4,6,8,10,9,7,5,3,1}
            \end{tikzpicture}
            \caption*{A wedge alternation.}	      
        \end{subfigure}
        \begin{subfigure}[t]{0.30\textwidth}
            \centering
            \begin{tikzpicture}[scale=0.3]
            \plotpermgrid{3,5,7,9,1,10,8,6,4,2}
            \end{tikzpicture}
            \caption*{A type 1 wedge simple}
        \end{subfigure}
        \begin{subfigure}[t]{0.30\textwidth}
            \centering
            \begin{tikzpicture}[scale=0.3]
            \plotpermgrid{2,4,6,8,10,7,5,3,1,9}
            \end{tikzpicture}
            \caption*{A type 2 wedge simple}
        \end{subfigure}	
    \end{center}
    \caption{Examples of a wedge alternation, and type 1 and type 2 wedge simple permutations.}	      
    \label{figure_examples_of_type_1_2_wedge_simples}
\end{figure}
We calculated the value of the 
principal \mob function of
$W_1(n)$ and $W_2(n)$ for $n = 4, \dots, 30$, and the 
results are shown in 
Tables~\ref{table-values-of-w1}
and~\ref{table-values-of-w2}
in Appendix~\ref{chapter_appendix-a}.

These values suggest the following conjecture:
\begin{conjecture}
    For all $n > 3$,
    \begin{align*}
    \mobp{W_1(n)} 
    & = 
    (-1)^{n} 3(3-n)  \\
    \mobp{W_2(n)} 
    & = 
    (-1)^{n} (1-n).
    \end{align*}
\end{conjecture}    

Schmerl and Trotter~\cite{Schmerl1993}
show that every simple permutation
of length $n$ contains a simple permutation
of length $n-1$ or $n-2$.
Further, they show that the only 
\emph{exceptional}\extindex[permutation]{exceptional}
simple permutations, which are those
permutations of length $n$ that
do not contain a simple sub permutation of length $n-1$,
are the parallel alternations, which we discussed above.
The 
\emph{nearly-exceptional}\extindex[permutation]{nearly-exceptional} 
simple permutations
are permutations which are not exceptional, but
there is only a single point which, when deleted,
results in a smaller simple permutation; so
deleting any other point results in a non-simple permutation.
The nearly-exceptional permutations have not,
as far as we are aware, appeared 
in any publication.
They were described in a talk given by
Robert Brignall 
at Permutation Patterns 2010~\cite{Brignall2010a}.
There are three types of
nearly exceptional simple permutations, 
which we refer to as
$E_1 (2n, k)$,
$E_2 (2n)$ 
and
$O (2n+1, k)$.
The first parameter gives the length of the permutation,
which, for simplicity, we require to be greater than 5.
For $E_{1} (2n, k)$, the second parameter, $k$ must satisfy $1 \leq k \leq n-2$,
and 
for $O (2n+1, k)$, $k$ must satisfy $1 \leq k \leq n-1$.
Formally, up to symmetry, we have
\begin{align*}
E_{1}(2n,k) ={ } &
n+1, 
1, 
n+2, 
2, 
\ldots, \\&
k, 
n+k+1, 
n, 
n+k+2, 
n-1, 
\ldots, \\&
2n, 
k+1;
\\
\intertext{~}
E_2 (2n) ={ } & 
n, 
1, 
n+1, 
2, 
\ldots, 
2n-2, 
2n, 
n-1, 
2n-1;
\\
\intertext{and}
O (2n+1, k) ={ } &
n-k+1, 
2n+1, 
n-k+2, 
2n,   
\ldots, \\&
2n+2 - k, 
n + 1,
n-k, 
n + 2, 
n-k-1, 
\ldots, \\&
1, 
2n-k+1.
\end{align*}
Examples of 
$E_1 (2n, k)$,
$E_2 (2n)$ 
and
$O (2n+1, k)$
are shown in
Figure~\ref{figure_examples_of_nearly_exceptional_simples}.
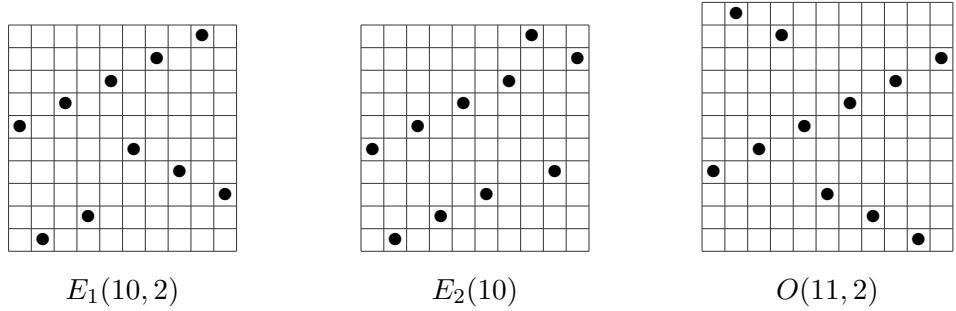
\begin{figure}
    \begin{center}
        \begin{subfigure}[t]{0.30\textwidth}
            \centering
            \begin{tikzpicture}[scale=0.3]
            \plotpermgrid{6,1,7,2,8,5,9,4,10,3}
            \end{tikzpicture}
            \caption*{$E_{1}(10,2)$}
        \end{subfigure}
        \begin{subfigure}[t]{0.30\textwidth}
            \centering
            \begin{tikzpicture}[scale=0.3]
            \plotpermgrid{5,1,6,2,7,3,8,10,4,9}
            \end{tikzpicture}
            \caption*{$E_{2} (10)$}
        \end{subfigure}	    
        \begin{subfigure}[t]{0.30\textwidth}
            \centering
            \begin{tikzpicture}[scale=0.3]
            \plotpermgrid{4,11,5,10,6,3,7,2,8,1,9}
            \end{tikzpicture}
            \caption*{$O(11,2)$}
        \end{subfigure}	
    \end{center}
    \caption{Examples of nearly exceptional simple permutations.}	      
    \label{figure_examples_of_nearly_exceptional_simples}
\end{figure}
We calculated the value of the 
principal \mob function of
$E_1(2n,k)$,
$E_2(2n)$,
and
$O(2n+1,k)$,
for $n = 3, \ldots, 15$, and all valid values of $k$,
and the 
results are shown in 
Tables~\ref{table-values-of-e1-3-12}, 
\ref{table-values-of-e2},
and~\ref{table-values-of-o}
in Appendix~\ref{chapter_appendix-a}.

These values suggest the following conjectures:
\begin{conjecture}
    For all $n \geq 5$
    and
    for all $k$ with $1 \leq k \leq n-2$,
    \[
    \mobp{E_{1} (2n, k)} 
    = 
    - \dfrac{(4n - 2) + k^2 - k}{2}.    
    \]
\end{conjecture}
\begin{conjecture}
    For all $n \geq 5$,
    \[
    \mobp{E_2 (2n)}     
    = 
    \dfrac{n - n^2 - 4}{2}.    
    \]
\end{conjecture}    
\begin{conjecture}
    For all $n \geq 5$
    and
    for all $k$ with $1 \leq k \leq n-1$,
    \[
    \mobp{O (2n+1, k)} 
    = 
    2n.  
    \]
\end{conjecture}

We remark that these
permutations do exhibit a significant amount of structure,
and we suggest that finding an expression
for the principal \mob function of 
other, less-structured,
simple permutations 
will be difficult.

\section{Further research into global properties of the poset}
The second area for future research
is to examine global attributes of the 
permutation pattern poset.  
Our results for the 
proportion of permutations that
are \mob zeros in Chapter~\ref{chapter_oppadj_paper}
based on~\cite{Brignall2020},
and our result on the growth of $\mobmaxn$
in Chapter~\ref{chapter_2413_balloon_paper}
based on~\cite{Marchant2020} are examples of 
existing results.

\subsection{\texorpdfstring{%
        Extremal values of $\mobp{\pi}$ as a function of the length of the permutations}{%
        Extremal values of the principal \mob function as a function of the length of the permutations}}

We think that further examination of the behaviour of $\mobmaxn$
is one interesting area for research.  
We could also define
$\truemobmax(n) = \max \{ \mobp{\pi} : \order{\pi} = n \}$,
and
$\truemobmin(n) = \min \{ \mobp{\pi} : \order{\pi} = n \}$,
and then ask how these two functions behave
as functions of $n$.
Trivially, we have that, for $n > 4$,
$\truemobmax(n+1) \geq - \truemobmin(n)$
and
$\truemobmin(n+1) \leq - \truemobmax(n)$.
Table~\ref{table_values_of_true_mob_min_max} shows the first
thirteen values of these functions.
\begin{table}
    \[
    \begin{array}{lrr}
    \toprule
    n & \truemobmin(n) & \truemobmax(n) 
    \\
    \midrule
    1  &     1 &  1 \\
    2  &   -1 &  -1 \\
    3  &    0 &   1 \\
    4  &   -3 &   0 \\
    5  &    0 &   6 \\
    6  &  -11 &   1 \\
    7  &   -2 &  15 \\
    8  &  -27 &  14 \\
    9  &  -50 &  39 \\
    10 &  -58 &  55 \\
    11 &  -81 & 143 \\ 
    12 & -261 & 183 \\
    13 & -330 & 261 \\
    \bottomrule
    \end{array}
    \]
    \caption{Values of $\truemobmax(n)$ and $\truemobmin(n)$ for $n=1, \ldots, 13$.}
    \label{table_values_of_true_mob_min_max}
\end{table}

\subsection{\texorpdfstring{%
        Characterising permutations where $\mobp{\pi}$ achieves extreme values}{
        Characterising permutations where the \mob  function achieves extreme values}}
    
The functions discussed above, 
$\mobmaxn$, 
$\truemobmax(n)$, and
$\truemobmin(n)$,
concern themselves with the minimum or maximum values 
of the principal \mob function.
We could also ask for the characteristics
of the permutations that achieve the
minimum or maximum values.
This question was, of course,
first asked implicitly 
in
Burstein, 
Jel{\'{i}}nek, 
Jel{\'{i}}nkov{\'{a}} and 
Steingr{\'{i}}msson~\cite{Burstein2011},
where there is a remark that
up to length 11,
if $\pi$ is a permutation where
$\order{\mobp{\pi}} = \mobmaxn$,
then $\pi$ is simple, and, up to symmetry, unique.
We remark here that our own intuition is that 
if $\order{\mobp{\pi}} = \mobmaxn$,
then $\pi$ is likely to be simple,
but we suspect, for sufficiently large
lengths, there may be more than one 
canonical permutation that meets the equality.

A slightly different approach would be to consider
the set of canonical permutations that achieve
$\truemobmax(n)$ or
$\truemobmin(n)$.
We have results up to length 13,
and these are shown in 
Table~\ref{table_canonical_min_max}
in Appendix~\ref{chapter_appendix-b}.

We remark that, as with many aspects of the
permutation pattern poset, results for small permutations
are atypical, and, in our opinion, the
entries in the table for $n \leq 7$
certainly fall into the atypical category.

We can see that $\truemobmax(13)$ is attained by
three permutations.  
If we set 
\[\theta = 4,7,2,10,5,1,12,8,3,11,6,9\]
then the first two permutations listed achieving
$\truemobmax(13)$ are
$1 \oplus \theta$, and 
$1 \oplus ((\theta)^R)^{-1}$
respectively.
With this exception, all of the permutations
that achieve $\truemobmin(n)$ or $\truemobmax(n)$
for $n  > 7$ are simple.

We suggest that,
for lengths greater than 7,
the permutations that achieve
$\truemobmax(n)$ are either simple,
or they are a sum (either direct or skew) 
of 1 and a permutation that achieves $\truemobmin(n-1)$.
Conversely,
the permutations that achieve
$\truemobmin(n)$ are either simple,
or they are a sum  
of 1 and a permutation that achieves $\truemobmax(n-1)$.

    \appendix
    
    \chapter{\texorpdfstring{Values of $\mobp{\pi}$ where $\pi$ is a specific simple permutation}
    {Values of mu(pi) where pi is a specific simple permutation}}
\label{chapter_appendix-a}

\section{\texorpdfstring
    {Values of $\mobp{W_1(n)}$}
    {Values of the principal \mob function of W1(n)}}

\begin{table}[ht!]
    \begin{center}
        \begin{tabular}{ccccc}
            \begin{tabular}{lr}
                \toprule
                $n$ & $\mobp{W_1(n)}$  \\
                \midrule
                4  &  -3 \\    
                5  &   6 \\
                6  &  -9 \\
                7  &  12 \\
                8  & -15 \\
                9  &  18 \\
                10 & -21 \\
                11 &  24 \\
                12 & -27 \\
                \bottomrule
            \end{tabular}    
            &
            \phantom{xxx}
            &
            \begin{tabular}{lr}
                \toprule
                $n$ & $\mobp{W_1(n)}$  \\
                \midrule
                13 &  30 \\
                14 & -33 \\
                15 &  36 \\
                16 & -39 \\
                17 &  42 \\    
                18 & -45 \\
                19 &  48 \\
                20 & -51 \\
                21 &  54 \\
                \bottomrule
            \end{tabular}    
            &
            \phantom{xxx}        
            &
            \begin{tabular}{lr}
                \toprule
                $n$ & $\mobp{W_1(n)}$  \\
                \midrule
                22 & -57 \\
                23 &  60 \\
                24 & -63 \\
                25 &  66 \\
                26 & -69 \\
                27 &  72 \\
                28 & -75 \\
                29 &  78 \\
                30 & -81 \\
                \bottomrule
            \end{tabular}    
        \end{tabular}
        \caption{Values of $\mobp{W_1(n)}$ for $n = 4, \ldots, 30$.}
        \label{table-values-of-w1}   
    \end{center}
\end{table}

\section{\texorpdfstring
    {Values of $\mobp{W_2(n)}$}
    {Values of the principal \mob function of W2(n)}}

\begin{table}[ht!]
    \begin{center}
        \begin{tabular}{ccccc}
            \begin{tabular}{lr}
                \toprule
                $n$ & $\mobp{W_2(n)}$  \\
                \midrule
                4  &  -3 \\    
                5  &   4 \\
                6  &  -5 \\
                7  &   6 \\
                8  &  -7 \\
                9  &   8 \\
                10 &  -9 \\
                11 &  10 \\
                12 & -11 \\
                \bottomrule
            \end{tabular}    
            &
            \phantom{xxx}
            &
            \begin{tabular}{lr}
                \toprule
                $n$ & $\mobp{W_2(n)}$  \\
                \midrule
                13 &  12 \\
                14 & -13 \\
                15 &  14 \\
                16 & -15 \\
                17 &  16 \\    
                18 & -17 \\
                19 &  18 \\
                20 & -19 \\
                21 &  20 \\
                \bottomrule
            \end{tabular}    
            &
            \phantom{xxx}
            &
            \begin{tabular}{lr}
                \toprule
                $n$ & $\mobp{W_2(n)}$  \\
                \midrule
                22 & -21 \\
                23 &  22 \\
                24 & -23 \\
                25 &  24 \\
                26 & -25 \\
                27 &  26 \\
                28 & -27 \\
                29 &  28 \\
                30 & -29 \\
                \bottomrule
            \end{tabular}    
        \end{tabular}
        \caption{Values of $\mobp{W_2(n)}$ for $n = 4, \ldots, 30$.}
        \label{table-values-of-w2}   
    \end{center}
\end{table}

\section{\texorpdfstring
    {Values of $\mobp{E_1(2n,k)}$}
    {Values of the principal \mob function of E1(2n,k)}}

\begin{table}[ht!]
    \ContinuedFloat*
    \begin{center}
        \begin{tabular}{ccc}
            \begin{tabular}{llr}
                \toprule
                $n$ & $k$ & $\mobp{E_1(2n,k)}$  \\
                \midrule
                3  &  1 &  -5 \\    
                &    &     \\                
                4  &  1 &  -7 \\    
                4  &  2 &  -8 \\    
                &    &     \\    
                5  &  1 &  -9 \\    
                5  &  2 & -10 \\    
                5  &  3 & -12 \\    
                &    &     \\    
                6  &  1 & -11 \\    
                6  &  2 & -12 \\    
                6  &  3 & -14 \\    
                6  &  4 & -17 \\    
                &    &     \\    
                7  &  1 & -13 \\    
                7  &  2 & -14 \\    
                7  &  3 & -16 \\    
                7  &  4 & -19 \\    
                7  &  5 & -23 \\    
                &    &     \\    
                &    &     \\    
                &    &     \\    
                &    &     \\    
                \bottomrule
            \end{tabular}    
            &
            \begin{tabular}{llr}
                \toprule
                $n$ & $k$ & $\mobp{E_1(2n,k)}$  \\
                \midrule                        
                8  &  1 & -25 \\    
                8  &  2 & -16 \\    
                8  &  3 & -18 \\    
                8  &  4 & -21 \\    
                8  &  5 & -25 \\    
                8  &  6 & -30 \\    
                &    &     \\    
                9  &  1 & -17 \\    
                9  &  2 & -18 \\    
                9  &  3 & -20 \\    
                9  &  4 & -23 \\    
                9  &  5 & -27 \\    
                9  &  6 & -32 \\    
                9  &  7 & -38 \\    
                &    &     \\    
                10 &  1 & -19 \\    
                10 &  2 & -20 \\    
                10 &  3 & -22 \\    
                10 &  4 & -25 \\    
                10 &  5 & -29 \\    
                10 &  6 & -34 \\    
                10 &  7 & -40 \\    
                10 &  8 & -47 \\    
                \bottomrule
            \end{tabular}    
            &
            \begin{tabular}{llr}
                \toprule
                $n$ & $k$ & $\mobp{E_1(2n,k)}$  \\
                \midrule            
                11 &  1 & -21 \\    
                11 &  2 & -22 \\    
                11 &  3 & -24 \\    
                11 &  4 & -27 \\    
                11 &  5 & -31 \\    
                11 &  6 & -36 \\    
                11 &  7 & -42 \\    
                11 &  8 & -49 \\    
                11 &  9 & -57 \\    
                &    &     \\    
                12 &  1 & -23 \\    
                12 &  2 & -24 \\    
                12 &  3 & -26 \\    
                12 &  4 & -29 \\    
                12 &  5 & -33 \\    
                12 &  6 & -38 \\    
                12 &  7 & -44 \\    
                12 &  8 & -51 \\    
                12 &  9 & -59 \\    
                12 & 10 & -68 \\  
                &    &     \\    
                &    &     \\    
                &    &     \\    
                \bottomrule
            \end{tabular}    
        \end{tabular}
        \caption{Values of $\mobp{E_1(2n,k)}$ for $n = 3, \ldots, 12$, and all valid values of $k$.}
        \label{table-values-of-e1-3-12}   
    \end{center}
\end{table}    
\begin{table}[ht!]    
    \ContinuedFloat
    \begin{center}
        \begin{tabular}{ccc}
            \begin{tabular}{llr}
                \toprule
                $n$ & $k$ & $\mobp{E_1(2n,k)}$  \\
                \midrule   
                13 &  1 & -25 \\    
                13 &  2 & -26 \\    
                13 &  3 & -28 \\    
                13 &  4 & -31 \\    
                13 &  5 & -35 \\    
                13 &  6 & -40 \\    
                13 &  7 & -46 \\    
                13 &  8 & -53 \\    
                13 &  9 & -61 \\    
                13 & 10 & -70 \\    
                13 & 11 & -80 \\    
                &    &     \\    
                &    &     \\                
                \bottomrule
            \end{tabular}    
            &
            \begin{tabular}{llr}
                \toprule
                $n$ & $k$ & $\mobp{E_1(2n,k)}$  \\
                \midrule            
                14 &  1 & -27 \\    
                14 &  2 & -28 \\    
                14 &  3 & -30 \\    
                14 &  4 & -33 \\    
                14 &  5 & -37 \\    
                14 &  6 & -42 \\    
                14 &  7 & -48 \\    
                14 &  8 & -55 \\    
                14 &  9 & -63 \\    
                14 & 10 & -72 \\    
                14 & 11 & -82 \\    
                14 & 12 & -93 \\   
                &    &     \\                
                \bottomrule
            \end{tabular}    
            &
            \begin{tabular}{llr}
                \toprule
                $n$ & $k$ & $\mobp{E_1(2n,k)}$  \\
                \midrule            
                15 &  1 & -29 \\    
                15 &  2 & -30 \\    
                15 &  3 & -32 \\    
                15 &  4 & -35 \\    
                15 &  5 & -39 \\    
                15 &  6 & -44 \\    
                15 &  7 & -50 \\    
                15 &  8 & -57 \\    
                15 &  9 & -65 \\    
                15 & 10 & -74 \\    
                15 & 11 & -84 \\    
                15 & 12 & -95 \\    
                15 & 13 &-107 \\    
                \bottomrule
            \end{tabular}    
        \end{tabular}
        \caption{Values of $\mobp{E_1(2n,k)}$ for $n = 13, 14, 15$, and all valid values of $k$.}
    \end{center}
\end{table}

\section{\texorpdfstring
    {Values of $\mobp{E_2(2n)}$}
    {Values of the principal \mob function of E2(2n)}}
    
\begin{table}[ht!]
    \begin{center}
        \begin{tabular}{ccccc}
            \begin{tabular}{lr}
                \toprule
                $n$ & $\mobp{E_2(2n)}$  \\
                \midrule
                3  &  -5 \\    
                4  &  -8 \\    
                5  & -12 \\
                6  & -17 \\
                7  & -23 \\
                \bottomrule
            \end{tabular}    
            &
            \phantom{xxx}
            &
            \begin{tabular}{lr}
                \toprule
                $n$ & $\mobp{E_2(2n)}$  \\
                \midrule
                8  & -30 \\
                9  & -38 \\
                10 & -47 \\
                11 & -57 \\
                12 & -68 \\
                \bottomrule
            \end{tabular}    
            &
            \phantom{xxx}
            &
            \begin{tabular}{lr}
                \toprule
                $n$ & $\mobp{E_2(2n)}$  \\
                \midrule
                13 & -80 \\
                14 & -93 \\
                15 &-107 \\
                & \\
                & \\
                \bottomrule
            \end{tabular}    
        \end{tabular}
        \caption{Values of $\mobp{E_2(2n)}$ for $n = 3, \ldots, 15$.}
        \label{table-values-of-e2}   
    \end{center}
\end{table}

\section{\texorpdfstring
    {Values of $\mobp{O(2n+1,k)}$}
    {Values of the principal \mob function of O(2n+1,k)}}

\begin{table}[ht!]
    \begin{center}
        \begin{tabular}{ccc}
            \begin{tabular}{llr}
                \toprule
                $n$ & $k$ & $\mobp{O(2n+1,k)}$  \\
                \midrule
                3  & $1, 2$          &   6 \\    
                4  & $1,2,3$         &   8 \\    
                5  & $1, \ldots,  4$ &  10 \\
                6  & $1, \ldots,  5$ &  12 \\
                7  & $1, \ldots,  6$ &  14 \\
                8  & $1, \ldots,  7$ &  16 \\
                9  & $1, \ldots,  8$ &  18 \\
                \bottomrule
            \end{tabular}    
            &
            \phantom{xxx}
            &
            \begin{tabular}{llr}
                \toprule
                $n$ & $k$ & $\mobp{O(2n+1,k}$  \\
                \midrule
                10 & $1, \ldots,  9$ &  20 \\
                11 & $1, \ldots, 10$ &  22 \\
                12 & $1, \ldots, 11$ &  24 \\
                13 & $1, \ldots, 12$ &  26 \\
                14 & $1, \ldots, 13$ &  28 \\
                15 & $1, \ldots, 14$ &  30 \\
                &                 &     \\
                \bottomrule
            \end{tabular}    
        \end{tabular}
        \caption{Values of $\mobp{O(2n+1,k)}$ for $n = 3, \ldots, 15$, and all valid values of $k$.}
        \label{table-values-of-o}   
    \end{center}
\end{table}

    \chapter{\texorpdfstring
    {Canonical permutations that achieve $\truemobmax(n)$ or $\truemobmin(n)$}
    {Canonical permutations that achieve MaxMu(n) or MinMu(n)}}
\label{chapter_appendix-b}


\begin{table}[ht!]
    \[
    \begin{array}{lrr}
    \toprule
    n & 
    \mobp{\pi} = \truemobmin(n) & 
    \mobp{\pi} = \truemobmax(n) \\
    \midrule
    1 & 
    \mathbf{1} & 
    \mathbf{1} 
    \\ 
    2 & 
    \mathbf{12} & 
    \mathbf{12} 
    \\ 
    3 & 
    123 & 
    132 
    \\ 
    4 & 
    \mathbf{2413} & 
    1234; \;1243; \;1432 
    \\ 
    \multirow{3}{*}{5 \Bigg\{ } & 
    12345; \;12354; \;12435; \;12453; &  \\ &
    12543; \;13452; \;14325; \;14532; & \mathbf{24153} \\ &
    15432; \;21354; \;21453; \;21543 \phantom{;} & 
    \\ 
    6 & 
    \mathbf{351624} &
    231564 
    \\
    7 & 
    2547163; \;3416725 & 
    \mathbf{2461735} 
    \\
    8 & 
    \mathbf{35172846} & 
    \mathbf{36184725} 
    \\
    9 & 
    \mathbf{472951836} & 
    \mathbf{357182946} 
    \\
    10 & 
    \mathbf{4{,}6{,}8{,}1{,}9{,}2{,}10{,}3{,}5{,}7} & 
    \mathbf{4{,}7{,}9{,}1{,}10{,}6{,}2{,}8{,}3{,}5} 
    \\
    11 & 
    \mathbf{3{,}5{,}8{,}10{,}1{,}7{,}11{,}2{,}9{,}4{,}6} & \mathbf{3{,}6{,}1{,}9{,}4{,}11{,}7{,}2{,}10{,}5{,}8} 
    \\
    12 & 
    \mathbf{4{,}7{,}2{,}10{,}5{,}1{,}12{,}8{,}3{,}11{,}6{,}9} & \mathbf{5{,}10{,}2{,}7{,}12{,}4{,}9{,}1{,}6{,}11{,}318} 
    \\
    \multirow{3}{*}{13 \Bigg\{ } && 
    1{,}5{,}8{,}3{,}11{,}6{,}2{,}13{,}9{,}4{,}12{,}7{,}10 ; \\ &
    \mathbf{6{,}2{,}9{,}4{,}11{,}1{,}7{,}13{,}3{,}10{,}5{,}12{,}8} & 1{,}8{,}11{,}5{,}13{,}9{,}3{,}12{,}6{,}2{,}10{,}4{,}7 ; \\ &&
    \mathbf{4{,}7{,}2{,}10{,}5{,}13{,}1{,}12{,}8{,}3{,}11{,}6{,}9}  \\
    \bottomrule
    \end{array}
    \]
    \caption{Canonical permutations of length $n$ where the principal \mob function has a minimum / maximum value.  Simple permutations are highlighted.}
    \label{table_canonical_min_max}
\end{table}

    \backmatter

    \cleardoublepage
    \phantomsection
    \addcontentsline{toc}{chapter}{Bibliography}
    \bibliographystyle{dwmplain}

    \printindex       
\end{document}